\documentclass[11pt]{article}

\usepackage{amsfonts,amssymb,amsmath,graphicx,graphics,color}
\usepackage[english]{babel}
 \def\botcaption#1#2{\medskip\centerline{{\scshape #1.}\kern8pt
 {\rm #2}}\bigskip}

 \makeatletter
 \@addtoreset{equation}{section}
 \makeatother

 \makeatletter
 \@addtoreset{enunciato}{section}
 \makeatother

 \newcounter{enunciato}[section]

 \newtheorem{ittheorem}{Theorem}
 \newtheorem{ithypothesis}{Hypothesis}
 \newtheorem{itlemma}{Lemma}
 \newtheorem{itproposition}{Proposition}
 \newtheorem{itconjecture}{Conjecture}
 \newtheorem{itdefinition}{Definition}
 \newtheorem{itassumption}{Assumption}
 \newtheorem{itremark}{Remark}
 \newtheorem{itclaim}{Claim}
 \newtheorem{itcorollary}{Corollary}

 \newenvironment{theorem}{\addtocounter{enunciato}{1}
 \begin{ittheorem}}{\end{ittheorem}}

 \newenvironment{corollary}{\addtocounter{enunciato}{1}
 \begin{itcorollary}}{\end{itcorollary}}

 \newenvironment{lemma}{\addtocounter{enunciato}{1}
 \begin{itlemma}}{\end{itlemma}}

 \newenvironment{proposition}{\addtocounter{enunciato}{1}
 \begin{itproposition}}{\end{itproposition}}
 
 \newenvironment{conjecture}{\addtocounter{enunciato}{1}
 \begin{itconjecture}}{\end{itconjecture}}

 \newenvironment{definition}{\addtocounter{enunciato}{1}
 \begin{itdefinition}}{\end{itdefinition}}

 \newenvironment{assumption}{\addtocounter{enunciato}{1}
 \begin{itassumption}}{\end{itassumption}}

 \newenvironment{hypothesis}{\addtocounter{enunciato}{1}
 \begin{ithypothesis}}{\end{ithypothesis}}

 \newenvironment{remark}{\addtocounter{enunciato}{1}
 \begin{itremark}}{\end{itremark}}

 \newenvironment{claim}{\addtocounter{enunciato}{1}
 \begin{itclaim}}{\end{itclaim}}

 \newenvironment{proof}{\noindent {\bf Proof.\,}
 }{\hspace*{\fill}$\square$\medskip}

 \newcommand{\be}[1]{\begin{equation}\label{#1}}
 \newcommand{\ee}{\end{equation}}

 \newcommand{\bl}[1]{\begin{lemma}\label{#1}}
 \newcommand{\el}{\end{lemma}}

 \newcommand{\br}[1]{\begin{remark}\label{#1}}
 \newcommand{\er}{\end{remark}}

 \newcommand{\bt}[1]{\begin{theorem}\label{#1}}
 \newcommand{\et}{\end{theorem}}

 \newcommand{\bd}[1]{\begin{definition}\label{#1}}
 \newcommand{\ed}{\end{definition}}

 \newcommand{\ba}[1]{\begin{assumption}\label{#1}}
 \newcommand{\ea}{\end{assumption}}

 \newcommand{\bh}[1]{\begin{hypothesis}\label{#1}}
 \newcommand{\eh}{\end{hypothesis}}

 \newcommand{\bcl}[1]{\begin{claim}\label{#1}}
 \newcommand{\ecl}{\end{claim}}

 \newcommand{\bp}[1]{\begin{proposition}\label{#1}}
 \newcommand{\ep}{\end{proposition}}
 
 \newcommand{\bconj}[1]{\begin{conjecture}\label{#1}}
 \newcommand{\econj}{\end{conjecture}}

 \newcommand{\bc}[1]{\begin{corollary}\label{#1}}
 \newcommand{\ec}{\end{corollary}}

 \newcommand{\bpr}{\begin{proof}}
 \newcommand{\epr}{\end{proof}}

 \newcommand{\bi}{\begin{itemize}}
 \newcommand{\ei}{\end{itemize}}

 \newcommand{\ben}{\begin{enumerate}}
 \newcommand{\een}{\end{enumerate}}

 \def\botcaption#1#2{\medskip\centerline{{\scshape #1.}\kern8pt
 {\rm #2}}\bigskip}

 \def \ba {\begin{array}}
 \def \ea {\end{array}}

 \def \Z {{\mathbb Z}}
 \def \R {{\mathbb R}}
 
 \def \N {{\mathbb N}}
 \def \P {{\mathbb P}}
 \def \E {{\mathbb E}}

 \def \cX {{\mathcal X}}
 \def \cC {{\mathcal C}}
 \def \cO {{\mathcal O}}
 \def \cW {{\mathcal W}}
 \def \cD {{\mathcal D}}
 \def \cL {{\mathcal L}}
 \def \cI {{\mathcal I}}
 \def \cJ {{\mathcal J}}
 \def \cR {{\mathcal R}}
 \def \cA {{\mathcal A}}
 \def \cM {{\mathcal M}}
 \def \cN {{\mathcal N}}
 \def \cE {{\mathcal E}}
 \def \cP {{\mathcal P}}
 \def \cQ {{\mathcal Q}}
 \def \cB {{\mathcal B}}
 \def \cF {{\mathcal F}}
 
 \def \cC {{\mathcal C}}
 \def \cU {{\mathcal U}}
 \def \cV {{\mathcal V}}
 \def \cH {{\mathcal H}}
 \def \cT {{\mathcal T}}
 
 \def \cS {{\mathcal S}}

 \def \AB {\mathrm{int}}
 \def \nAB {\mathrm{nint}}

 \def \ind {{1}}
 \def \gep {{\varepsilon}}

 \def \DOM {{\hbox{\footnotesize\rm DOM}}}
 \def \CONE {{\hbox{\footnotesize\rm CONE}}}
 \def \EIGH {{\hbox{\footnotesize\rm EIGH}}}

\begin{document}

\title{Phase diagram for a copolymer in a micro-emulsion}

\author{\renewcommand{\thefootnote}{\arabic{footnote}}
F.\ den Hollander
\footnotemark[1]
\\
\renewcommand{\thefootnote}{\arabic{footnote}}
N.\ P\'etr\'elis
\footnotemark[2]
}

\footnotetext[1]
{ Mathematical Institute, Leiden University, P.O.\ Box 9512,
2300 RA Leiden, The Netherlands.\\
E-mail adress: denholla@math.leidenuniv.nl\\
URL: https://www.math.leidenuniv.nl/~denholla/ 
}\,

\footnotetext[2]
{Laboratoire de Math\'ematiques Jean Leray UMR 6629,
Universit\'e de Nantes, 2 Rue de la Houssini\`ere,\\ 
BP 92208, F-44322 Nantes Cedex 03, France.\\
E-mail adress: nicolas.petrelis@univ-nantes.fr\\
URL: http://www.math.sciences.univ-nantes.fr/~petrelis/
}\,

\maketitle


\begin{abstract}
In this paper we study a model describing a copolymer in a micro-emulsion. The 
copolymer consists of a random concatenation of hydrophobic and hydrophilic 
monomers, the micro-emulsion consists of large blocks of oil and water arranged 
in a percolation-type fashion. The interaction Hamiltonian assigns energy $-\alpha$ 
to hydrophobic monomers in oil and energy $-\beta$ to hydrophilic monomers in 
water, where $\alpha,\beta$ are parameters that without loss of generality are taken 
to lie in the cone $\{(\alpha,\beta) \in\R^2\colon\,\alpha \geq |\beta|\}$. Depending on 
the values of these parameters, the copolymer either stays close to the oil-water 
interface (localization) or wanders off into the oil and/or the water (delocalization). 
Based on an \emph{assumption} about the strict concavity of the free energy of a copolymer 
near a linear interface, we derive a variational formula for the quenched free energy per monomer that 
is \emph{column-based}, i.e., captures what the copolymer does in columns of
different type. We subsequently transform this into a variational formula that is 
\emph{slope-based}, i.e., captures what the polymer does as it travels at different 
slopes, and we use the latter to identify the phase diagram in the $(\alpha,\beta)$-cone. 
There are two regimes: \emph{supercritical} (the oil blocks percolate) and \emph{subcritical} 
(the oil blocks do not percolate). The supercritical and the subcritical phase diagram 
each have two localized phases and two delocalized phases, separated by four 
critical curves meeting at a quadruple critical point. The different phases correspond 
to the different ways in which the copolymer can move through the micro-emulsion. 
The analysis of the phase diagram is based on \emph{three hypotheses} about the
possible frequencies at which the oil blocks and the water blocks can be visited. 
We show that these three hypotheses are plausible, but do not 
provide a proof.

\vskip 0.5truecm
\noindent
\emph{AMS} 2000 \emph{subject classifications.} 60F10, 60K37, 82B27.
\emph{Acknowledgment:} The research in this paper is supported by ERC Advanced 
Grant 267356-VARIS. NP is grateful for hospitality at the Mathematical Institute of 
Leiden University during extended visits in 2011, 2012 and 2013 within the framework 
of this grant. FdH and NP are grateful for hospitality at the Institute for Mathematical 
Sciences at the National University of Singapore in May of 2015.\\
\emph{Remark:} The part of this paper dealing with the ``column-based'' variational 
formula for the free energy has appeared as a preprint on the mathematics archive: 
arXiv:1204.1234.
\end{abstract}

\newpage


\setcounter{section}{-1}

\section{Outline}

In Section~\ref{Intro}, we introduce our model for a copolymer in a micro-emulsion and present 
a variational formula for the quenched free energy per monomer, which we refer to as the 
\emph{slope-based variational formula}, involving the fractions of time the copolymer moves 
at a given slope in the interior of the two solvents and the fraction of time it moves along the 
interfaces between the two solvents. This variational formula is the corner stone of our analysis. 
In Section~\ref{phdiag}, we identify the phase diagram. There are two regimes: \emph{supercritical} 
(the oil blocks percolate) and \emph{subcritical} (the oil blocks do not percolate). We obtain the \emph{general structure} of the phase diagram, and state a number of properties that exhibit the 
\emph{fine structure} of the phase diagram as well. The latter come in the form of theorems and 
conjectures, and are based on three hypotheses.

In Section~\ref{keyingr}, we introduce a \emph{truncated} version of the model in which the
copolymer is not allowed to travel more than $M$ blocks upwards or downwards in each 
column, where $M\in\N$ is arbitrary but fixed. We give a precise definition of the various 
ingredients that are necessary to state the slope-based variational formula for the truncated
model, including various auxiliary quantities that are needed for its proof. Among these is 
the quenched free energy per monomer of the copolymer crossing a block column of a given 
type, whose existence and variational characterization are given in Section~\ref{proofprop1}. 
In Section~\ref{proofofgene}, we derive an auxiliary variational formula for the quenched free 
energy per monomer in the truncated model, which we refer to as the \emph{column-based 
variational formula}, involving both the free energy per monomer and the fraction of time 
spent inside single columns of a given type, summed over the possible types. At the end of 
Section~\ref{proofofgene}, we show how the truncation can be removed by letting $M\to\infty$.
In Section~\ref{varfo2}, we use the column-based variational formula to derive the slope-based variational formula. In Section~\ref{phdiagsupsub} we use the slope-based variational formula 
to prove our results for the phase diagram.  Appendices~\ref{Path entropies}--\ref{appA} 
collect several technical results that are needed along the way.

For more background on random polymers with disorder we refer the reader to the monographs 
by Giacomin~\cite{G07} and den Hollander~\cite{dH09}, and to the overview paper by Caravenna, 
den Hollander and P\'etr\'elis~\cite{CdHP12}.


\section{Model and slope-based variational formula}
\label{Intro}

In Section~\ref{Mfe} we define the model, in Section~\ref{newvarfor} we state the slope-based
variational formula, in Section~\ref{discus} we place this formula in the proper context.


\subsection{Model}
\label{Mfe}

To build our model, we distinguish between three scales: (1) the \emph {microscopic} 
scale associated with the size of the monomers in the copolymer ($=1$, by convention); 
(2) the {\emph {mesoscopic}} scale associated with the size of the droplets in the 
micro-emulsion ($L_n\gg1$); (3) the {\emph {macroscopic}} scale associated with the 
size of the copolymer ($n\gg L_n$). 

\medskip\noindent
{\bf Copolymer configurations.}
Pick $n\in \N$ and let $\cW_n$ be the set of $n$-step \emph{directed self-avoiding paths} 
starting at the origin and being allowed to move \emph{upwards, downwards and to the right}, 
i.e., 
\begin{align}
\label{defw}
\nonumber
\cW_n &= \big\{\pi=(\pi_i)_{i=0}^n \in (\N_0\times\Z)^{n+1}\colon\,\pi_0=(0,1),\\
&\qquad \pi_{i+1}-\pi_i\in \{(1,0),(0,1),(0,-1)\}\,\,\forall\,0\leq i< n,
\,\pi_i\neq \pi_j\,\,\forall\,0\leq i<j \leq n\big\}.
\end{align}
The copolymer is associated with the path $\pi$. The $i$-th monomer is associated with 
the bond $(\pi_{i-1},\pi_i)$. The starting point $\pi_0$ is chosen to be $(0,1)$ for 
convenience.

\begin{figure}[htbp]
\begin{center}
\vspace{-.4cm}
\includegraphics[width=.20\textwidth]{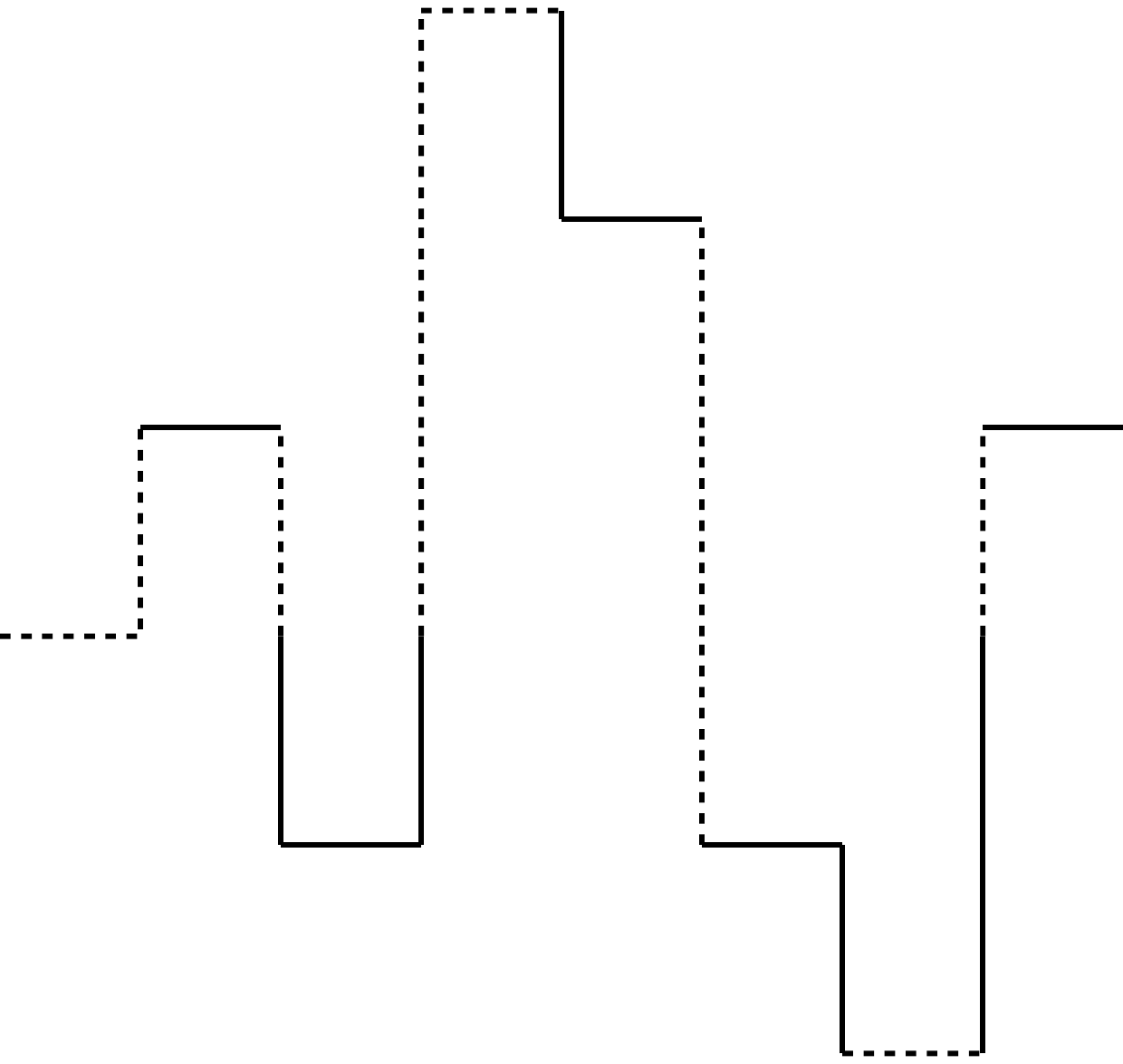}
\vspace{-1cm}
\caption{Microscopic disorder $\omega$ in the copolymer. Dashed bonds represent 
monomers of type $A$ (hydrophobic), drawn bonds represent monomers of type $B$
(hydrophilic).}
\label{fig-micrdis}
\end{center}
\end{figure}

\medskip\noindent
{\bf Microscopic disorder in the copolymer.} 
Each monomer is randomly labelled $A$ (hydrophobic) or $B$ (hydrophilic), with probability 
$\frac{1}{2}$ each, independently for different monomers. The resulting labelling is 
denoted by
\be{bondlabel}
\omega = \{\omega_i \colon\, i \in \N\} \in \{A,B\}^\N
\ee
and represents the \emph{randomness of the copolymer}, i.e., $\omega_i=A$ and $\omega_i=B$ 
mean that the $i$-th monomer is of type $A$, respectively, of type $B$ (see 
Fig.~\ref{fig-micrdis}). We denote by $\P_\omega$ the law of the microscopic disorder.

\begin{figure}[htbp]
\begin{center}
\includegraphics[width=.42\textwidth]{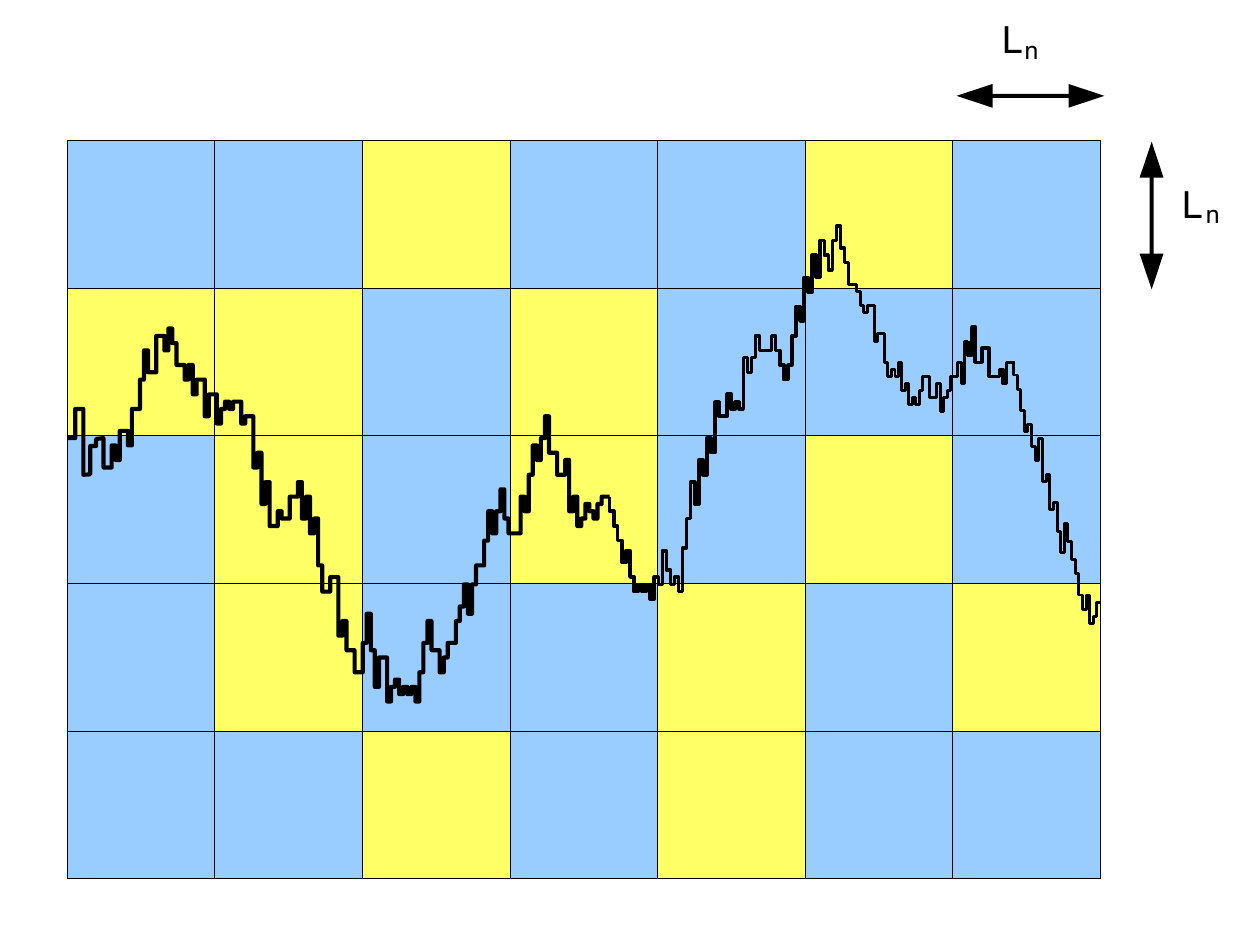}
\caption{Mesoscopic disorder $\Omega$ in the micro-emulsion. Light shaded blocks 
represent droplets of type $A$ (oil), dark shaded blocks represent droplets of type $B$ 
(water). Drawn is also the copolymer, but without an indication of the microscopic disorder 
$\omega$ attached to it.} 
\label{fig-mesdis}
\end{center}
\end{figure} 

\medskip\noindent
{\bf Mesoscopic disorder in the micro-emulsion.}
Fix $p \in (0,1)$ and $L_n \in \N$. Partition $(0,\infty)\times\R$ into square blocks 
of size $L_n$:
\be{blocks}
(0,\infty)\times \R= \bigcup_{x \in \N_0 \times \Z} \Lambda_{L_n}(x), \qquad
\Lambda_{L_n}(x) = xL_n + (0,L_n]^2.
\ee
Each block is randomly labelled $A$ (oil) or $B$ (water), with probability $p$, respectively, 
$1-p$, independently for different blocks. The resulting labelling is denoted by
\be{blocklabel}
\Omega = \{\Omega(x)\colon\,x \in \N_0\times \Z\} \in \{A,B\}^{\N_0\times \Z}
\ee
and represents the \emph{randomness of the micro-emulsion}, i.e., $\Omega(x)=A$ and 
$\Omega(x)=B$ mean that the $x$-th block is of type $A$, respectively, of type $B$
(see Fig.~\ref{fig-mesdis}). The law of the mesoscopic disorder is denoted by $\P_\Omega$ 
and is independent of $\P_\omega$. The size of the blocks $L_n$ is assumed to be 
non-decreasing and to satisfy
\be{speed}
\lim_{n\to \infty} L_n = \infty \quad\text{and} \quad 
\lim_{n\to \infty} \frac{\log n}{n} \,L_n= 0,
\ee  
i.e., the blocks are large compared to the monomer size but small compared to the copolymer 
size. For convenience we assume that if an $A$-block and a $B$-block are next to each other, 
then the interface belongs to the $A$-block.

\medskip\noindent 
{\bf Hamiltonian and free energy.} Given $\omega,\Omega$ and $n$, with each path 
$\pi \in \cW_n$ we associate an \emph{energy} given by the Hamiltonian
\be{Hamiltonian}
H_{n,L_n}^{\omega,\Omega}(\pi;\alpha,\beta)
=  \sum_{i=1}^n \Big(\alpha\, 1\Big\{\omega_i=\Omega^{L_n}_{(\pi_{i-1},\pi_i)}=A\Big\}
+ \beta\, 1\left\{\omega_i=\Omega^{L_n}_{(\pi_{i-1},\pi_i)}=B\right\}\Big),
\ee
where $\Omega^{L_n}_{(\pi_{i-1},\pi_i)}$ denotes the label of the block the step 
$(\pi_{i-1},\pi_i)$ lies in. What this Hamiltonian does is count the number of 
$AA$-matches and $BB$-matches and assign them energy $\alpha$ and $\beta$, respectively, 
where $\alpha,\beta\in\R$. (Note that the interaction is assigned to bonds rather than 
to sites, and that we do not follow the convention of putting a minus sign in front of 
the Hamiltonian.) Similarly to what was done in our earlier papers \cite{dHW06}, 
\cite{dHP07b}, \cite{dHP07c}, \cite{dHP07a}, without loss of generality we may 
restrict the interaction parameters to the cone
\be{defcone}
\CONE = \{(\alpha,\beta)\in\R^2\colon\,\alpha\geq |\beta|\}.
\ee
For $n\in \N$, the free energy per monomer is defined as
\be{partfunc}
f_{n}^{\omega,\Omega}(\alpha,\beta)
=\tfrac{1}{n}\log Z_{n,L_n}^{\omega,\Omega}(\alpha,\beta) 
\quad \text{with}\quad Z_{n,L_n}^{\omega,\Omega}(\alpha,\beta)
=\sum_{\pi \in \cW_n} e^{H_{n,L_n}^{\omega,\Omega}(\pi;\, \alpha,\beta)},
\ee
and in the limit as $n\to\infty$ the free energy per monomer is given by
\be{felimdeff}
f(\alpha,\beta;p) = \lim_{n\to \infty}f_{n,L_n}^{\omega,\Omega}(\alpha,\beta), 
\ee
provided this limit exists $\omega,\Omega$-a.s. 

Henceforth, we subtract the term $\alpha \sum_{i=1}^n 1\{\omega_i=A\}$ from the Hamiltonian, 
which by the law of large numbers $\omega$-a.s.\ is $\tfrac{\alpha}{2}n(1+o(1))$ as $n\to\infty$ 
and corresponds to a shift of $-\tfrac{\alpha}{2}$ in the free energy. The latter transformation allows 
us to lighten the notation, starting with the Hamiltonian in \eqref{Hamiltonian}, which becomes
\be{Hamiltonianalt}
H_{n,L_n}^{\omega,\Omega}(\pi;\alpha,\beta)
= \sum_{i=1}^n \Big(\beta\, 1\left\{\omega_i=B\right\}
-\alpha\,1\left\{\omega_i=A\right\}\Big)\,  
1\left\{\Omega^{L_n}_{(\pi_{i-1},\pi_i)}=B\right\}.
\ee


\subsection{The slope-based variational formula for the quenched free energy per step}
\label{newvarfor}

Theorem~\ref{varformula2} below gives a  variational formula for the free energy per step 
in \eqref{felimdeff}. This variational formula, which is the \emph{corner stone} of our paper, 
involves the fractions of time the copolymer moves at a given slope through the interior of 
solvents $A$ and $B$ and the fraction of time it moves along $AB$-interfaces. \emph{This 
variational formula will be crucial to identify the phase diagram}, i.e., to identify the typical 
behavior of the copolymer in the micro-emulsion as a function of the parameters $\alpha,
\beta,p$ (see Section~\ref{phdiag} for theorems and conjectures). Of particular interest is 
the distinction between \emph{localized phases}, where the copolymer stays close to the 
$AB$-interfaces, and \emph{delocalized phases}, where it wanders off into the solvents 
$A$ and/or $B$. We will see that there are several such phases.  

To state Theorem~\ref{varformula2} we need to introduce some further notation. With 
each $l\in\R_+=[0,\infty)$ we associate two numbers $v_{A,l},v_{B,l} \in [1+l,\infty)$ 
indicating how many steps per horizontal step the copolymer takes when traveling at 
slope $l$ in solvents $A$ and $B$, respectively. We further let $v_{\cI}\in [1,\infty)$ 
denote the number of steps per horizontal step the copolymer takes when traveling along
$AB$-interfaces. These numbers are gathered into the set 
\be{newB}
\bar{\cB}=\{v=(v_A,v_B,v_\cI) \in \cC \times \cC \times [1,\infty)\}
\ee
with
\be{newC}
\cC=\big\{l \mapsto u_l \text{ on } \R_+\colon\, \text{ continuous with } 
u_l\geq 1+l\,\,\,\forall\,l\in\R_{+}\big\}.
\ee

Let $\tilde\kappa(u,l)$ be the entropy per step carried by trajectories moving at 
slope $l$ with the constraint that the total number of steps divided by the total 
number of horizontal steps is equal to $u \in [1+l,\infty)$ (for more details, see 
Section~\ref{pathentr}). Let $\phi_{\cI}(u;\,\alpha,\beta)$ be the free energy per 
step when the copolymer moves along an $AB$-interface, with the constraint that the 
total number of steps divided by the total number of horizontal steps is equal to 
$u \in [1,\infty)$ (for more details, see Section~\ref{interf}). Let $\bar\rho = (\rho_A,
\rho_B,\rho_{\cI}) \in \cM_1(\R_+\times\R_+\times\{\cI\})$, where $\bar\rho_A(dl)$ 
and $\bar \rho_B(dl)$ denote the fractions of horizontal steps at which the copolymer 
travels through solvents $A$ and $B$ at a slope that lies between $l$ and $l+dl$, 
and $\rho_{\cI}$ denotes the fraction of horizontal steps at which the copolymer 
travels along $AB$-interfaces. The possible $\bar\rho$ form a set
\be{susub}
\bar \cR_p\subset \cM_1\big(\R_+\times\R_+\times\{\cI\}\big)
\ee
that depends on $p$ (for more details, see Section~\ref{Percofreq}). With these 
ingredients we can now state our \emph{slope-based variational formula}.

\begin{theorem} {\rm {\bf [slope-based variational formula]}}
\label{varformula2}
For every $(\alpha,\beta)\in\CONE$ and $p \in (0,1)$ the free energy in \eqref{felimdeff} 
exists for $\P$-a.e.\ $(\omega,\Omega)$ and in $L^1(\P)$, and is given by
\be{genevar}
f(\alpha,\beta;p) =\sup_{\bar{\rho}\in \bar{\cR}_p}\,
\sup_{v\,\in\,\bar{\cB}}\,\,\frac{\bar N(\bar{\rho},v)}{\bar D(\bar{\rho},v)},
\ee
where 
\begin{align}
\label{varformod}
\nonumber 
\bar N(\bar\rho,v) 
&= \int_{0}^{\infty} v_{A,l}\,\tilde{\kappa}(v_{A,l},l)\,\bar\rho_{A}(dl)
+ \int_{0}^{\infty} v_{B,l}\,\big[\tilde{\kappa}(v_{B,l},l)
+\tfrac{\beta-\alpha}{2}\big]\,\bar\rho_{B}(dl)
+ v_{\cI}\,\phi_{\cI}(v_{\cI};\alpha,\beta)\,\bar\rho_{\cI},\\
\bar D(\bar{\rho},v)
&=\int_{0}^{\infty} v_{A,l}\,\bar{\rho}_{A}(dl) 
+ \int_{0}^{\infty} v_{B,l}\,\bar{\rho}_{B}(dl)
+ v_{\cI}\,\bar{\rho}_{\cI},
\end{align}
with the convention that $\bar N(\bar{\rho},v)/\bar D(\bar{\rho},v)=-\infty$ when 
$\bar D(\bar{\rho},v)=\infty$.
\end{theorem}

\br{hypothesis}
{\rm In order to obtain \eqref{varformod}, we need to assume strict concavity of 
an auxiliary free energy, involving a copolymer in the vicinity of a single linear 
interface. This is the object of Assumption~\ref{assu} in Section~\ref{interf}, 
which is supported by a brief discussion.} 
\er


\subsection{Discussion}
\label{discus}

The variational formula in (\ref{genevar}--\ref{varformod}) is tractable, to the extent that 
the $\tilde\kappa$-function is known explicitly, the $\phi_{\cI}$-function has been studied 
in depth in the literature (and much is known about it), while the set $\bar{\cB}$ is simple. 
\emph{The key difficulty of {\rm (\ref{genevar}--\ref{varformod})} resides in the set 
$\bar{\cR}_p$, whose structure is not easy to control}.  A detailed study of this set
is not within the scope of our paper. Fortunately, it turns out that we need to know 
\emph{relatively little} about $\bar{\cR}_p$ in order to identify the \emph{general structure} 
of the phase diagram (see Section~\ref{phdiag}). With the help of three \emph{hypotheses} 
on $\bar\cR_p$, each of which is plausible, we can also identify the \emph{fine structure} 
of the phase diagram (see Section~\ref{phdiagdetailed}).
 
We expect that the supremum in \eqref{genevar} is attained at a unique $\bar \rho \in 
\bar \cR_p$ and a unique $v\in \bar \cB$. This maximizer corresponds to the copolymer 
having a specific way to configure itself optimally within the micro-emulsion. 

\medskip\noindent
{\bf Column-based variational formula.}
The \emph{slope-based variational formula} in Theorem~\ref{varformula2} will be 
obtained by combining two auxiliary variational formulas. Both formulas involve the 
free energy per step $\psi(\Theta,u_\Theta;\alpha,\beta)$ when the copolymer 
crosses a block column of a given type $\Theta$, taking values in a type space 
$\overline\cV$, for a given $u_\Theta \in \R^+$ that indicates how many steps 
on scale $L_n$ the copolymer makes in this column type. A precise definition of 
this free energy per block column will be given in Section~\ref{newsec}.

The first auxiliary variational formula is stated in Section~\ref{keyingr} (Proposition~\ref{energ}) 
and gives an expression for $\psi(\Theta,u_\Theta;\alpha,\beta)$ that involves the entropy 
$\tilde \kappa(\cdot,l)$ of the copolymer moving at a given slope $l$ and the quenched free 
energy per monomer $\phi_\cI$ of the copolymer near a \emph{single linear interface}. 
Consequently, the free energy of our model with a \emph{random geometry} is directly linked 
to the free energy of a model with a \emph{non-random geometry}. This will be crucial for our 
analysis of the phase diagram in Section~\ref{phdiag}. \emph{The microscopic disorder 
manifests itself only through the free energy of the linear interface model}. 

The second auxiliary variational formula is stated in Section~\ref{proofofgene} 
(Proposition~\ref{varformula}). It is referred to as the \emph{column-based variational formula}, 
and provides an expression for $f(\alpha,\beta;p)$ by using the block-column free energies 
$\psi(\Theta,u_\Theta;\alpha,\beta)$ for $\Theta\in \overline\cV$ and by weighting each 
column type with the frequency $\rho(d\Theta)$ at which it is visited by the copolymer. The 
numerator is the total free energy, the denominator is the total number of monomers (both 
on the mesoscopic scale). The variational formula optimizes over $(u_\Theta)_{\Theta\in
\overline\cV} \in\cB_{\,\overline \cV}$ and $\rho\in \cR_p$. The reason why these two 
suprema appear in \eqref{genevarold} is that, as a consequence of assumption \eqref{speed}, 
the \emph{mesoscopic scale carries no entropy}: all the entropy comes from the microscopic 
scale, through the free energy per monomer in single columns. \emph{The mesoscopic 
disorder manifests itself only through the presence of the set} $\cR_p$.   

\medskip\noindent
{\bf Removal of the corner restriction.} 
In our earlier papers \cite{dHW06}, \cite{dHP07b}, \cite{dHP07c}, \cite{dHP07a}, we allowed 
the configurations of the copolymer to be given by the subset of $\cW_n$ consisting of those 
paths that enter pairs of blocks through a common corner, exit them at one of the two corners 
diagonally opposite and in between stay confined to the two blocks that are seen upon entering. 
The latter is an {\it unphysical restriction} that was adopted to simplify the model. In these 
papers we derived a variational formula for the free energy per step that had a \emph{much
simpler structure}. We analyzed this variational formula as a function of $\alpha,\beta,p$ and 
found that there are two regimes, \emph{supercritical} and \emph{subcritical}, depending on 
whether the oil blocks percolate or not along the \emph{coarse-grained self-avoiding path}
running along the corners. In the supercritical regime the phase diagram turned out to have 
two phases, in the subcritical regime it turned out to have four phases, meeting at two tricritical 
points. 

In Section~\ref{phdiag} we show how the variational formula in Theorem~\ref{varformula2} 
can be used to identify the phase diagram. It turns out that there are two types of phases: 
\emph{localized phases} (where the copolymer spends a positive fraction of its time near 
the $AB$-interfaces) and \emph{delocalized phases} (where it spends a zero fraction near 
the $AB$-interfaces). Which of these phases occurs depends on the parameters $\alpha,\beta,p$. 
It is energetically favorable for the copolymer to stay close to the $AB$-interfaces, where it has 
the possibility of placing more than half of its monomers in their preferred solvent (by switching 
sides when necessary), but this comes with a loss of entropy. The competition between energy 
and entropy is controlled by the energy parameters $\alpha,\beta$ (determining the reward of 
switching sides) and by the density parameter $p$ (determining the density of the $AB$-interfaces). 
It turns out that the phase diagram is different in the supercritical and the subcritical regimes, where 
the $A$-blocks percolate, respectively, do not percolate. The phase diagram is richer than for the 
model with the corner restriction. 
 
\begin{figure}[htbp]
\vspace{.3cm}
\begin{center}
\includegraphics[width=.3\textwidth]{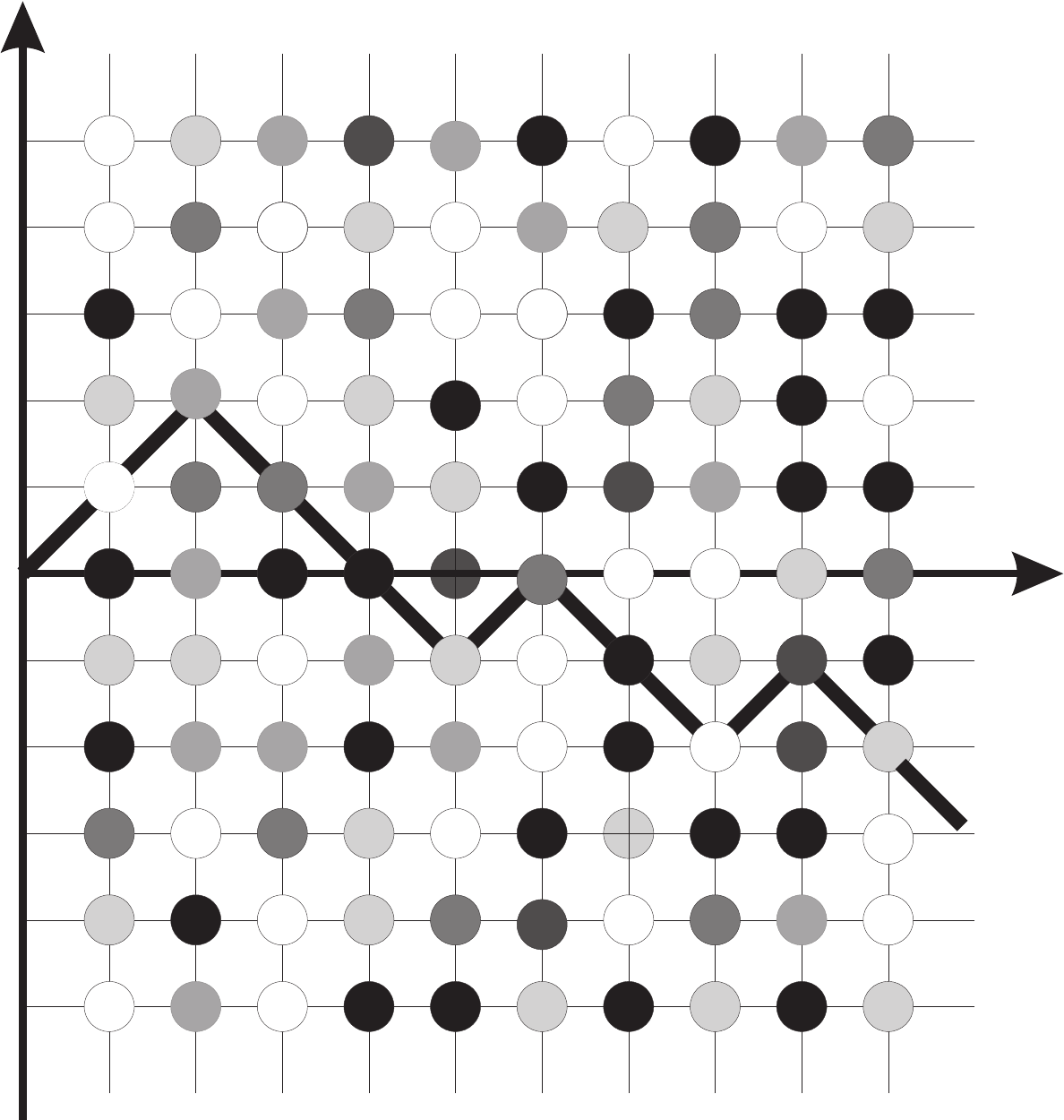}
\caption{Picture of a directed polymer with bulk disorder. The different shades
of black, grey and white represent different values of the disorder.} 
\label{fig-dirpolbulk}
\end{center}
\end{figure} 
\vspace{-.2cm}

\medskip\noindent 
{\bf Comparison with the directed polymer with bulk disorder.}
A model of a polymer with disorder that has been studied intensively in the literature 
is the \emph{directed polymer with bulk disorder}. Here, the set of paths is 
\be{dpre1}
\cW_n = \big\{\pi=(i,\pi_i)_{i=0}^n \in (\N_0\times\Z^d)^{n+1}\colon\,\pi_0=0,\,
\|\pi_{i+1}-\pi_i\|=1\,\,\forall\,0 \leq i<n\big\},
\ee
where $\|\cdot\|$ is the Euclidean norm on $\Z^d$, and the Hamiltonian is  
\be{dpre2}
H^\omega_n(\pi) = \lambda \sum_{i=1}^n \omega(i,\pi_i),
\ee
where $\lambda>0$ is a parameter and $\omega = \{\omega(i,x)\colon\,i\in\N,\,x\in\Z^d\}$ 
is a field of i.i.d.\ $\R$-valued random variables with zero mean, unit variance and finite 
moment generating function, where $\N$ is time and $\Z^d$ is space (see Fig.~\ref{fig-dirpolbulk}). 
This model can be viewed as a version of a copolymer in a micro-emulsion where the 
droplets are of the \emph{same} size as the monomers. For this model a variational 
formula for the free energy has been derived by Rassoul-Agha, Sepp\"al\"ainen and 
Yilmaz~\cite{RASY12}, \cite{RASY15}. However, the variational formula is abstract and 
therefore does not lead to a quantitative understanding of the phase diagram. Most
of the analysis in the literature relies on the application of martingale techniques (for 
details, see e.g.\ den Hollander~\cite{dH09}, Chapter 12).

In our model (which is restricted to $d=1$ and has self-avoiding paths that may move 
north, south and east instead of north-east and south-east), the droplets are much larger 
than the monomers. This causes a \emph{self-averaging of the microscopic disorder}, 
both when the copolymer moves inside one of the solvents and when it moves near an 
interface. Moreover, since the copolymer is much larger than the droplets, also 
\emph{self-averaging of the mesoscopic disorder} occurs. This is why the free energy 
can be expressed in terms of a variational formula, as in Theorem~\ref{varformula2}.
This variational formula acts as a \emph{jumpboard} for a detailed analysis of the phase 
diagram. Such a detailed analysis is lacking for the directed polymer with bulk disorder.

The directed polymer in random environment has two phases: a \emph{weak disorder
phase} (where the quenched and the annealed free energy are asymptotically comparable) 
and a \emph{strong disorder phase} (where the quenched free energy is asymptotically 
smaller than the annealed free energy). The strong disorder phase occurs in dimension 
$d=1,2$ for all $\lambda>0$ and in dimension $d \geq 3$ for $\lambda>\lambda_c$, with
$\lambda_c \in [0,\infty]$ a critical value that depends on $d$ and on the law of the 
disorder. It is predicted that in the strong disorder phase the copolymer moves within 
a narrow corridor that carries sites with high energy (recall our convention of not
putting a minus sign in front of the Hamiltonian), resulting in \emph{superdiffusive} 
behavior in the spatial direction. We expect a similar behavior to occur in the localized 
phases of our model, where the polymer targets the $AB$-interfaces. It would be interesting 
to find out how far the coarsed-grained self-avoiding path in our model travels vertically as 
a function of $n$.


\section{Phase diagram}
\label{phdiag}

In Section~\ref{phdiaggeneral} we identify the \emph{general structure} of the phase diagram. 
In particular, we show that there is a localized phase $\cL$ in which $AB$-localization occurs, 
and a delocalized phase $\cD$ in which no $AB$-localization occurs. In Section~\ref{phdiagdetailed}, 
we obtain various results for the \emph{fine structure} of the phase diagram, both for the 
supercritical regime $p>p_c$ and for the subcritical regime $p<p_c$, where $p_c$ denotes 
the \emph{critical threshold for directed bond percolation in the positive quandrant of} $\Z^2$. 
This fine structure comes in the form of theorems and conjectures, and is based on three
hypotheses, which we discuss in Section~\ref{hypo}.


\subsection{General structure}
\label{phdiaggeneral}

To state the general structure of the phase diagram, we need to define a reduced version 
of the free energy, called the \emph{delocalized free energy} $f_{\cD}$, obtained by taking into 
account those trajectories that, when moving along an $AB$-interface, are delocalized in the 
$A$-solvent. The latter amounts to replacing the linear interface free energy $\phi_{\cI}(v_\cI;
\alpha,\beta)$ in \eqref{genevar} by the entropic constant lower bound $\tilde{\kappa}(v_\cI,0)$. 
Thus, we define
\be{genevarD1d}
f_{\cD}(\alpha,\beta;p) = \sup_{\bar{\rho}\in\bar{\cR}_{p}}\,\sup_{v\in\,\bar\cB}\,\,
\frac{\bar{N}_{\cD}(\bar{\rho},v)}{\bar{D}_{\cD}(\bar{\rho},v)}
\ee
with 
\begin{align}
\label{ffd}
\bar{N}_{\cD}(\bar{\rho},v) 
&= \int_{0}^{\infty} v_{A,l}\,\tilde \kappa(v_{A,l},l)\,
[\bar{\rho}_{A}+\bar{\rho}_\cI\,\delta_0](dl)
+ \int_{0}^{\infty} v_{B,l}\,\big[\tilde\kappa(v_{B,l},l)+\tfrac{\beta-\alpha}{2}\big]
\,\bar \rho_{B}(dl),\\
\bar{D}_{\cD}(\bar{\rho},v) 
&= \int_{0}^{\infty} v_{A,l}\,
[\bar{\rho}_{A}+\bar{\rho}_\cI\,\delta_0](dl)
+ \int_{0}^{\infty} v_{B,l}
\,\bar \rho_{B}(dl),
\end{align}
provided $\bar D_{\cD}(\bar{\rho},v)<\infty$. Note that  $f_{\cD}(\alpha,\beta;p)$ depends 
on $(\alpha,\beta)$ through $\alpha-\beta$ only. 

We partition the $\CONE$ into the two phases $\cD$ and $\cL$ defined by 
\begin{equation}
\label{defD}
\begin{aligned}
\cL&=\{(\alpha,\beta)\in \CONE\colon f(\alpha,\beta;\,p)>f_{\cD}(\alpha,\beta;p)\},\\
\cD&=\{(\alpha,\beta)\in \CONE\colon f(\alpha,\beta;\,p)=f_{\cD}(\alpha,\beta;p)\}.
\end{aligned}
\end{equation}
The localized phase $\cL$ corresponds to large values of $\beta$, for which the energetic 
reward to spend some time travelling along $AB$-interfaces exceeds the entropic penalty 
to do so. The delocalized phase $\cD$, on the other hand, corresponds to small values of 
$\beta$, for which the energetic reward does not exceed the entropic penalty.  

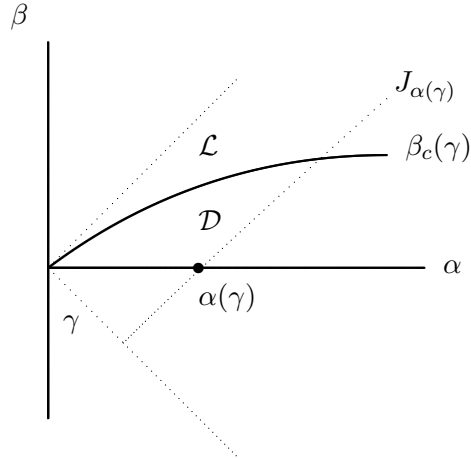
\begin{figure}[htbp]
\begin{center}
\setlength{\unitlength}{0.5cm}
\begin{picture}(10,10)(0,-2)
{\thicklines
\qbezier(0,-4)(0,0)(0,6)
\qbezier(0,0)(4,0)(10,0)
\qbezier(0,0)(4,3)(9,3)
}
\qbezier[40](0,0)(3,3)(5,5)
\qbezier[40](0,0)(3,-3)(5,-5)
\qbezier[60](2,-2)(4,0)(9,4.5)
\put(10.5,-.15){$\alpha$}
\put(-1,6.5){$\beta$}
\put(4,1){$\cD$}
\put(4,3){$\cL$}
\put(9.5,3){$\beta_c(\gamma)$}
\put(.4,-1.6){$\gamma$}
\put(4,-1){$\alpha(\gamma)$}
\put(4,0){\circle*{.25}}
\put(9.2,4.7){$J_{\alpha(\gamma)}$}
\end{picture}
\end{center}
\vspace{1cm}
\caption{Qualitative picture of the phase diagram in $\CONE$. The curve $\gamma 
\mapsto \beta_c(\gamma)$ separates the localized phase $\cL$ from the delocalized 
phase $\cD$. The parameter $\gamma$ measures the distance between the origin 
and the point on the lower boundary of $\CONE$ from which the line with slope $1$ 
hits the curve at height $\beta(\gamma)$. Note that $\alpha(\gamma)=\gamma\sqrt{2}$ 
is the value where this line crosses the horizontal axis.}
\label{fig-phdiaggen}
\end{figure}

For $\alpha\geq 0$, let $J_\alpha$ be the halfline in $\CONE$ defined by (see
Fig.~\ref{fig-phdiaggen})
\be{jalpha}
J_\alpha=\{(\alpha+\beta,\beta)\colon\, \beta\in [-\tfrac{\alpha}{2},\infty)\}.
\ee

\bt{phgene}
(a) There exists a curve $\gamma \mapsto \beta_c(\gamma)$, lying strictly inside 
the upper quadrant, such that 
\begin{equation}
\label{intjalpha1}
\begin{aligned}
\cL \cap J_\alpha
&=\{(\alpha+\beta,\beta)\colon\,\beta\in (\beta_c(\gamma(\alpha)),\infty)\},\\
\cD \cap J_\alpha
&=\big\{(\alpha+\beta,\beta)\colon\,\beta\in [-\tfrac{\alpha}{2},\beta_c(\gamma(\alpha))]\big\},
\end{aligned}
\end{equation}
for all $\alpha \in (0,\infty)$ with $\gamma(\alpha) = \alpha/\sqrt{2}$.\\
(b) Inside phase $\cD$ the free energy $f$ is a function of $\alpha-\beta$ only, i.e.,
$f$ is constant on $J_\alpha\cap \cD$ for all $\alpha\in (0,\infty)$.
\et


\subsection{Fine structure}
\label{phdiagdetailed}

This section is organized as follows. In Section \ref{phdiagsup}, we consider the supercritical 
regime $p>p_c$, and state a theorem. Subject to two hypotheses, we show that the delocalized 
phase $\cD$ (recall \eqref{defD}) splits into two subphases $\cD=\cD_1\cup\cD_2$ such that 
the fraction of monomers placed by the copolymer in the $B$ solvent is strictly positive inside 
$\cD_1$ and equals $0$ in  $\cD_2$. Thus,  $\cD_1$ and $\cD_2$ are  said to be non-saturated, 
respectively, saturated. We give a characterization of the critical curve $\alpha\mapsto 
\beta_c(\alpha)$ (recall \eqref{intjalpha1}) in terms of the single linear free energy and state 
some properties of this curve. Subsequently, we formulate a conjecture stating that the localized 
phase $\cL$ also splits into two subphases $\cL=\cL_1\cup\cL_2$, which are non-saturated, 
respectively, saturated. In Section~\ref{phdiagsub}, we consider the subcritical regime $p<p_c$, 
and obtain similar results. 

For $p\in (0,1)$ and $(\alpha,\beta)\in \CONE$, let $\cO_{p,\alpha,\beta}$ denote the 
subset of $\bar \cR_{p}$ containing those $\bar\rho$ that maximize the variational 
formula in \eqref{genevar}, i.e., 
\be{rap}
\cO_{p,\alpha,\beta}=\bigg\{\bar \rho \in \bar \cR_{p}\colon\, f(\alpha,\beta;p)
=\sup_{v\in \bar \cB}\,\frac{\bar N(\bar\rho,v)}{\bar D(\bar\rho,v)}\bigg\}.
\ee
Throughout the remainder of this section we need the following hypothesis:

\bh{hyp1}
For all $p \in (0,1)$ and $\alpha \in (0,\infty)$ there exists a $\bar\rho\in
\cO_{p,\alpha,0}$ such that $\bar\rho_{\cI}>0$. 
\eh

\noindent
This hypothesis will allow us to derive an expression for $\beta_c(\gamma)$ in \eqref{intjalpha1}.

\br{hyp1r1} 
{\rm Hypothesis \ref{hyp1} will be discussed in Section~\ref{hypo}. The existence of 
$\bar\rho$ is proven in Appendix~\ref{appT} for a truncated version of our model, 
introduced in Section~\ref{Mtruncation}. This truncated model approximates the full 
model as the truncation level diverges (see Proposition~\ref{pr:formimppp}).}
\er
 
For $c\in (0,\infty)$, define $v(c)=(v_A(c),v_B(c),v_\cI(c))\in \bar\cB$ as
\begin{eqnarray}
\label{defucp1}
v_{A,l}(c)
&\hspace{-1.25cm} = \chi_l^{-1}(c), &l\in [0,\infty),\\
\label{defucp2} 
v_{B,l}(c)
&= \chi_l^{-1}
\big(c+\tfrac{\alpha-\beta}{2}\big), &l\in [0,\infty),\\
\label{defucp3} 
v_\cI(c)
&\hspace{-2.1cm} = z, &\partial^-_u (u\,\phi_{\cI}(u))(z) 
\geq c\geq \partial^+_u (u\,\phi_{\cI}(u))(z),
\end{eqnarray}
where 
\be{}
\chi_l(v) = \big(\partial_u (u\,\tilde{\kappa}(u,l)\big)(v)
\ee
and $\chi_l^{-1}$ denotes the inverse function. Lemma~\ref{l:lemconv2}(v-vi) ensures 
that $v \mapsto \chi_l(v)$ is one-to-one between $(1+l,\infty)$ and $(0,\infty)$. The 
existence and uniqueness of $z$ in (\ref{defucp3}) follow from the strict concavity of 
$u\mapsto u\phi_{\cI}(u)$ (see Assumption~\ref{assu}) and Lemma~\ref{l:lemconv} 
(see (\ref{l3}--\ref{l4})). We will prove in Proposition~\ref{propa} that the maximizer 
$v\in \bar \cB$ of \eqref{genevar} necessarily belongs to the familly $\{v(c)\colon\,c\in
(0,\infty)\}$.

For $\bar\rho \in \bar\cR_p$, define
\be{}
K_A(\bar\rho) = \int_0^\infty (1+l)\bar\rho_A(dl),
\qquad K_B(\bar\rho) = \int_0^\infty (1+l)\bar\rho_B(dl).
\ee


\subsubsection{Supercritical regime}
\label{phdiagsup}

\begin{figure}[htbp]
\begin{center}
\includegraphics[width=.65\textwidth]{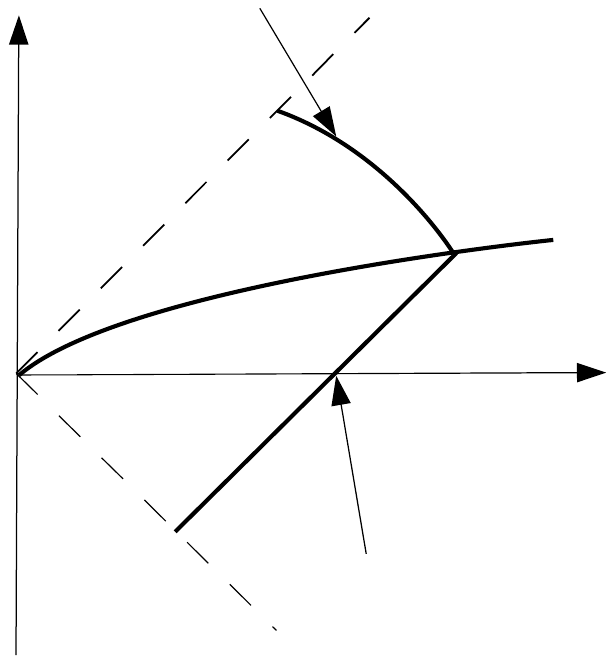}
\vspace{-7cm}
\caption{Qualitative picture of the phase diagram in the supercritical regime $p>p_c$.} 
\label{fig-phasediagsup}
\end{center}
\begin{picture}(10,0)(-150,-60)
\put(100,60){$\cD_2$}
\put(27,60){$\cD_1$}
\put(50,110){$\cL_1$ }
\put(100,130){$\cL_2$ }
\put(140,68){$\alpha$}
\put(-20,145){$\beta$}
\put(5,45){$\gamma$}
\put(74,24){$\alpha^*$}
\put(125,105){$\beta_c(\gamma)$}
\put(32,167){$\beta^*_c(\gamma)$}
\end{picture}
\end{figure} 


\paragraph{Splitting of the $\cD$-phase.}
\label{cD1}

We partition $\cD$ into two phases: $\cD=\cD_1\cup\cD_2$. To that end we introduce the 
\emph{delocalized $A$-saturated free energy}, denoted by $f_{\cD_2}(p)$, which is obtained 
by restricting the supremum in \eqref{genevarD1d} to those $\bar\rho\in\bar{\cR}_{p}$ 
that do not charge $B$. Such $\bar\rho$, which we call $A$-saturated, exist because $p>p_c$, 
allowing for trajectories that do not visit $B$-blocks. Thus, $f_{\cD_2}(p)$ is defined 
as
\be{genevarD2d}
f_{\cD_2}(p) = \sup_{ {\bar{\rho}\in\bar{\cR}_{p}} \atop {K_B(\bar\rho)=0} }\,
\sup_{v\in\,\bar{\cB}}\,\,\frac{\bar{N}_{\cD_2}(\bar{\rho},v)}{\bar{D}_{\cD}(\bar{\rho},v)}
\ee
with
\begin{equation}
\bar{N}_{\cD_2}(\bar{\rho},v) 
= \int_{0}^{\infty} v_{A,l}\,\tilde \kappa(v_{A,l},l)\,
[\bar{\rho}_{A}+\bar{\rho}_\cI\,\delta_0](dl),
\end{equation}
provided $D_\cD(\bar{\rho},v)<\infty$. Note that $f_{\cD_2}(p)$ is a constant that does 
not depend on $(\alpha,\beta)$.

With the help of this definition, we can split the $\cD$-phase defined in \eqref{defD} 
into two parts (see Fig.~\ref{fig-phasediagsup}):
\begin{itemize}
\item  
The $\cD_1$-phase corresponds to small values of $\beta$ and small to moderate values 
of $\alpha$. In this phase there is no $AB$-localization and no $A$-saturation. For the 
variational formula in \eqref{genevar} this corresponds to the restriction where the 
$AB$-localization term disappears while the $A$-block term and the $B$-block term 
contribute, i.e.,
\be{genevarD1}
\cD_1 = \big\{(\alpha,\beta)\in \CONE\colon\,f(\alpha,\beta;\,p) 
= f_{\cD}(\alpha,\beta;\,p)>f_{\cD_2}(p)\big\}.
\ee
\item 
The $\cD_2$-phase corresponds to small values of $\beta$ and large values of $\alpha$. 
In this phase there is no $AB$-localization but $A$-saturation occurs. For the variational 
formula in \eqref{genevar} this corresponds to the restriction where the $AB$-localization 
term disappears and the $B$-block term as well, i.e.,
\be{genevarD21}
\cD_2 = \big\{(\alpha,\beta)\in \CONE\colon\,f(\alpha,\beta;\,p) = f_{\cD_2}(p)\}.
\ee
\end{itemize}

Let $\cT_{p}$ be the subset of $\bar \cR_{p}$ containing those $\bar\rho$ that have a 
strictly positive $B$-component and are relevant for the variational formula 
in \eqref{genevar}, i.e., 
\be{tueo}
\cT_{p}=\big\{\bar\rho\in \bar \cR_{p}\colon\,K_B(\bar\rho)>0,\, 
K_A(\bar\rho) + K_B(\bar\rho)<\infty\big\}.
\ee
Note that $\cT_{p}$ does not depend on $(\alpha,\beta)$. To state our main result for the 
delocalized part of the phase diagram we need the following hypothesis:

\bh{hyp2} 
For all $p>p_c$,
\be{impo}
\sup_{\bar \rho\in \cT_{p}}
\frac{\int_{0}^\infty g_A(l)\,[\bar\rho_A + \bar\rho_\cI\,\delta_0](dl)}
{K_B(\bar\rho)}<\infty,
\ee
where 
\be{gdef}
g_A(l) = v_{A,l}(c)\,\big[\tilde \kappa(v_{A,l}(c),l)-c\big]\Big|_{c\,=\,f_{\cD_2}}
\ee
with $v_{A,l}(c)$ as defined in \eqref{defucp1}.
\eh

\noindent

\br{rhyp2}
{\rm Hypothesis~\ref{hyp2} will allow us to show that $\cD_1$ and $\cD_2$ are non-empty.
This hypothesis, which will be discussed further in Section~\ref{hypo}, relies on the fact that,
in the supercritical regime, large subcritical clusters typically have a diameter that is of the 
same size as their circumference.}
\er
   
\begin{remark}
\label{rr2}
{\rm The function $g_A$ has the following properties: (1) $g_A(0)>0$; (2) $g_A$ is strictly 
decreasing on $[0,\infty)$; (3) $\lim_{l \to \infty} g_A(l)=-\infty$. Property (2) follows from
Lemma~\ref{l:lemconv2}(ii) and the fact that $u\mapsto u\tilde{\kappa}(u,l)$ is concave 
(see Lemma \ref{l:lemconv2}(i)). Property (3) follows from $f_{\cD_2}>0$, 
Lemma~\ref{l:lemconv2}(iv) and the fact that $v_{A,l}(f_{\cD_2}) \geq 1+l$ for $l\in [0,\infty)$. 
Property (1) follows from property (2) because $\int_0^\infty g_A(l)[\hat\rho_A+\hat\rho_I
\delta_0](dl)=0$ for all $\hat\rho$ maximizing \eqref{genevarD2d}.}
\end{remark}

Let 
\be{alpha}
\alpha^{*} = \sup\{\alpha\geq 0\colon\,f_{\cD}(\alpha,0;\,p)>f_{\cD_2}(p)\}.
\ee

\bt{phdgsup}
Assume Hypotheses {\rm \ref{hyp1}} and {\rm \ref{hyp2}}. Then the 
following hold:\\
(a) $\alpha^*\in (0,\infty)$.\\
(b) For every  $\alpha \in [0,\alpha^*)$,
\be{intjalphai}
J_\alpha\cap \cD_1 = J_\alpha\cap \cD 
=\{(\alpha+\beta,\beta)\colon\,\beta\in [-\tfrac{\alpha}{2},\,\beta_c(\gamma(\alpha))].
\ee
(c) For every $\alpha \in [\alpha^*,\infty)$,  
\be{intjalphai2}
J_\alpha\cap \cD_2 = J_\alpha\cap \cD=\{(\alpha+\beta,\beta)\colon\,
\beta\in [-\tfrac{\alpha}{2},\beta_c(\gamma(\alpha))]\}.
\ee
(d) For every $\alpha\in [0,\infty)$,
\be{bc**}
\beta_c(\gamma(\alpha))=\inf\big\{\beta>0\colon\,\phi_{\cI}(\bar v_{A,0};\alpha+\beta,\beta)
>\tilde \kappa (\bar v_{A,0},0)\big\}\quad \text{with}\  \ \bar v=v(f_\cD(\alpha,0;\,p)).
\ee
(e) On $[\alpha^*,\infty)$, $\alpha\mapsto\beta_c(\gamma(\alpha))$ is concave, continuous, 
non-decreasing and bounded from above.\\
(f) Inside phase $\cD_1$ the free energy $f$ is a function of $\alpha-\beta$ only, 
i.e., $f$ is constant on $J_\alpha\cap \cD_1$  for all $\alpha\in [0,\alpha^*]$.\\
(g) Inside phase $\cD_2$ the free energy $f$ is constant.
\et


\paragraph{Splitting of the $\cL$-phase.}

We partition $\cL$ into two phases: $\cL=\cL_1\cup\cL_2$. To that end we introduce the 
\emph{localized $A$-saturated free energy}, denoted by $f_{\cL_2}$, which is obtained by
restricting the supremum in \eqref{genevar} to those $\bar\rho \in\bar{\cR}_{p}$ that 
do not charge $B$, i.e.,
\be{genevarL2}
f_{\cL_2}(\alpha,\beta;\,p) 
= \sup_{ {\bar{\rho}\in\bar{\cR}_{p}} \atop {K_B(\bar\rho)=0} }\,
\sup_{v\in\,\bar{\cB}}\,\, \frac{\bar{N}(\bar{\rho},v)}{\bar{D}(\bar{\rho},v)},
\ee
provided $D(\bar{\rho},v)<\infty$.

With the help of this definition, we can split the $\cL$-phase defined in \eqref{defD} 
into two parts (see Fig.~\ref{fig-phasediagsup}):
\begin{itemize}
\item 
The $\cL_1$-phase corresponds to small to moderate values of $\alpha$ and large values 
of $\beta$. In this phase $AB$-localization occurs, but $A$-saturation does not, so 
that the free energy is given by the variational formula in \eqref{genevar} without 
restrictions, i.e.,
\be{genevarL1}
\cL_1 = \big\{(\alpha,\beta)\in \CONE\colon\,f(\alpha,\beta;p) 
>\max \{f_{\cL_2}(\alpha,\beta;p),\,f_{\cD}(\alpha,\beta;p)\}\big\}.
\ee
\item 
The $\cL_2$-phase corresponds to large values of $\alpha$ and $\beta$. In this phase 
both $AB$-localization and $A$-saturation occur. For the variational formula in
\eqref{genevar} this corresponds to the restriction where the contribution of 
$B$-blocks disappears, i.e.,
\be{genevarL21}
\cL_2 = \big\{(\alpha,\beta)\in\CONE\colon\,
f(\alpha,\beta;p) = f_{\cL_2}(\alpha,\beta;p)>f_{\cD}(\alpha,\beta;\,p)\big\}.
\ee
\end{itemize}

\bconj{phdgsup*}
(a) There exists a curve $\gamma\mapsto\beta^*_c(\gamma)$, lying above the
curve $\gamma\mapsto\beta_c(\gamma)$, such that 
\begin{equation}
\label{intjalpha1alt}
\begin{aligned}
\cL_1 \cap J_\alpha
&=\{(\alpha+\beta,\beta)\colon\,
\beta\in (\beta_c(\gamma(\alpha)),\beta^*_c(\gamma(\alpha))]\},\\
\cL_2 \cap J_\alpha
&=\big\{(\alpha+\beta,\beta)\colon\,\beta\in [\beta^*_c(\gamma(\alpha)),\infty)\big\}.
\end{aligned}
\end{equation}
for all $\alpha\in (0,\alpha^*]$.\\ 
(b) $\cL_1 \cap J_\alpha = \emptyset$ for all $\alpha\in (\alpha^*,\infty)$.
\econj
 

\subsubsection{Subcritical regime}
\label{phdiagsub}

\begin{figure}[htbp]
\begin{center}
\includegraphics[width=.65\textwidth]{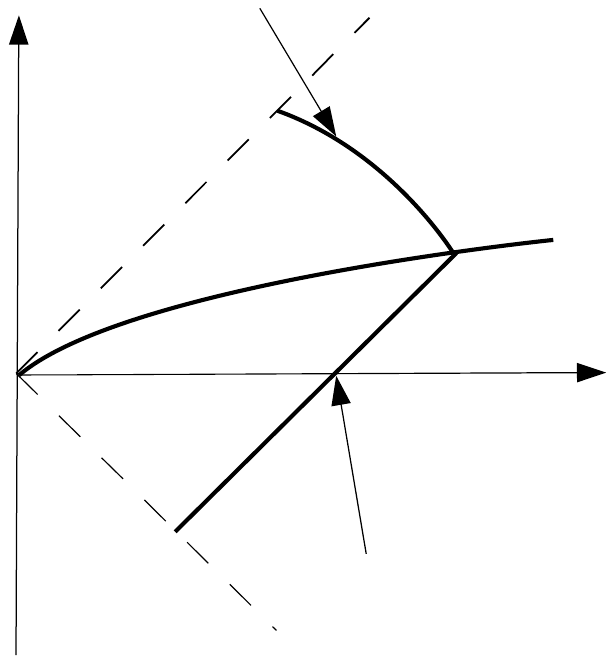}
\vspace{-7cm}
\caption{Qualitative picture of the phase diagram in the subcritical regime $p<p_c$.}  
\label{fig-phasediagsub}
\end{center}
\begin{picture}(10,0)(-150,-60)
\put(100,60){$\cD_2$}
\put(27,60){$\cD_1$}
\put(50,110){$\cL_1$ }
\put(100,130){$\cL_2$ }
\put(140,68){$\alpha$}
\put(-20,145){$\beta$}
\put(5,45){$\gamma$}
\put(74,24){$\bar\alpha^*$}
\put(125,105){$\beta_c(\gamma)$}
\put(32,167){$\beta^*_c(\gamma)$}
\end{picture}
\end{figure} 

\paragraph{Splitting of the $\cD$-phase.}

Let
\be{}
K_p = \inf_{\bar\rho\in\bar\cR_p} K_B(\bar\rho). 
\ee
Note that $K_p>0$ because $p<p_c$. We again partition $\cD$ into two phases: 
$\cD=\cD_1\cup\cD_2$. To that end we introduce the \emph{delocalized maximally 
$A$-saturated free energy}, denoted by $f_{\cD_2}(p)$, which is obtained by restricting 
the supremum in \eqref{genevarD1d} to those $\bar\rho\in\bar{\cR}_p$ achieving 
$K_p$. Thus, $f_{\cD_2}(p)$ is defined as
\be{genevarD2d*}
f_{\cD_2}(\alpha,\beta;\,  p) = \sup_{ {\bar{\rho}\in\bar{\cR}_{p}} \atop {K_B(\bar\rho)=K_p} }\,
\sup_{v\in\,\bar{\cB}}\,\,\frac{\bar{N}_{\cD}(\bar{\rho},v)}{\bar{D}_{\cD}(\bar{\rho},v)},
\ee
provided $D_\cD(\bar{\rho},v)<\infty$. Note that, contrary to what we had in the supercritical regime,
 $f_{\cD_2}(\alpha,\beta; p)$ depends on $\alpha-\beta$.

With the help of this definition, we can split the $\cD$-phase defined in \eqref{defD} 
into two parts (see Fig.~\ref{fig-phasediagsub}):
\begin{itemize}
\item  
The $\cD_1$-phase corresponds to small values of $\beta$ and small to moderate values 
of $\alpha$. In this phase there is no $AB$-localization and no maximal $A$-saturation. For 
the variational formula in \eqref{genevar} this corresponds to the restriction where the 
$AB$-localization term disappears while the $A$-block term and the $B$-block term 
contribute, i.e.,
\be{genevarD1*}
\cD_1 = \big\{(\alpha,\beta)\in \CONE\colon\,f(\alpha,\beta;\,p) 
= f_{\cD}(\alpha,\beta;\,p)>f_{\cD_2}(p)\big\}.
\ee
\item 
The $\cD_2$-phase corresponds to small values of $\beta$ and large values of $\alpha$. 
In this phase there is no $AB$-localization and maximal $A$-saturation. For the variational 
formula in \eqref{genevar} this corresponds to the restriction where the $AB$-localization 
term disappears and the $B$-block term is minimal, i.e.,
\be{genevarD21*}
\cD_2 = \big\{(\alpha,\beta)\in \CONE\colon\,f(\alpha,\beta;\,p) = f_{\cD_2}(p)\}.
\ee
\end{itemize}

Let 
\be{}
\cT_p = \big\{\bar\rho \in \bar\cR_p\colon\,K_B(\bar\rho) > K_p,\,
K_A(\bar\rho)+K_B(\bar\rho)<\infty\big\}.
\ee
To state our main result for the delocalized part of the phase diagram we need the 
following hypothesis:

\bh{hyp3} 
For all $p>p_c$,
\be{impo-alt}
\sup_{\bar\rho\in \cT_p}
\frac{\int_0^\infty g_{A,\alpha-\beta}(l)\,[\bar\rho_A + \bar\rho_\cI\,\delta_0](dl)}
{K_B(\bar\rho)}<\infty,
\ee
where 
\be{gdef-alt}
g_{A,\alpha-\beta}(l) = v_{A,l}(c)\,\big[\tilde\kappa(v_{A,l}(c),l)-c\big]
\Big|_{c\,=\,f_{\cD_2}(\alpha-\beta)}
\ee
with $v_{A,l}(c)$ as defined in \eqref{defucp1}.
\eh

\br{hyp33}
{\rm Hypothesis \ref{hyp3} will allow us to show that $\cD_1$ and $\cD_2$ are non-empty. 
It is close in spirit to Hypothesis~\ref{hyp2} and will be discussed further in Section \ref{hypo}.}
\er

\noindent 
Let
\be{baralpha}
\bar\alpha^* = \inf\big\{\alpha \geq 0\colon\,\forall\,\alpha' \geq \alpha
\,\,\exists\,\bar\rho \in \cO_{p,\alpha',0}\colon\,K_B(\bar\rho)=K_p\big\}.
\ee

\bt{phdgsub}
Assume Hypotheses {\rm \ref{hyp1}} and {\rm \ref{hyp3}} hold. Then the following hold:\\
(a)  $\bar\alpha^* \in (0,\infty)$.\\  
(b) Theorems~{\rm \ref{phdgsup}}(b,c,d) hold with $\alpha^*$ replaced by $\bar\alpha^*$.\\
(c) Theorem~{\rm \ref{phdgsup}}(f) holds on the whole $\cD$ whereas Theorem~{\rm \ref{phdgsup}}(g)
does not hold.
\et

\paragraph{Splitting of the $\cL$-phase.} 

We again partition $\cL$ into two phases: $\cL=\cL_1\cup\cL_2$. To that end we introduce the 
\emph{localized maximally $A$-saturated free energy}, denoted by $f_{\cL_2}$, which is obtained 
by restricting the supremum in \eqref{genevar} to those $\bar\rho \in\bar{\cR}_{p}$ achieving $K_p$. 
Thus, $f_{\cL_2}(\alpha,\beta;\,p)$ is defined as 
\be{genevarL2*}
f_{\cL_2}(\alpha,\beta;\,p) 
= \sup_{ {\bar{\rho}\in\bar{\cR}_{p}} \atop {K_B(\bar\rho)=K_p} }\,
\sup_{v\in\,\bar{\cB}}\,\, \frac{\bar{N}(\bar{\rho},v)}{\bar{D}(\bar{\rho},v)},
\ee
provided $D(\bar{\rho},v)<\infty$.

With the help of this definition, we can split the $\cL$-phase defined in \eqref{defD} 
into two parts (see Fig.~\ref{fig-phasediagsub}):
\begin{itemize}
\item 
The $\cL_1$-phase corresponds to small to moderate values of $\alpha$ and large values 
of $\beta$. In this phase $AB$-localization occurs, but maximal $A$-saturation does not, so 
that the free energy is given by the variational formula in \eqref{genevar} without 
restrictions, i.e.,
\be{genevarL1*}
\cL_1 = \big\{(\alpha,\beta)\in \CONE\colon\,f(\alpha,\beta;p) 
>\max \{f_{\cL_2}(\alpha,\beta;p),\,f_{\cD}(\alpha,\beta;p)\}\big\}.
\ee
\item 
The $\cL_2$-phase corresponds to large values of $\alpha$ and $\beta$. In this phase 
both $AB$-localization and maximal $A$-saturation occur. For the variational formula in
\eqref{genevar} this corresponds to the restriction where the contribution of 
$B$-blocks is minimal, i.e.,
\be{genevarL21*}
\cL_2 = \big\{(\alpha,\beta)\in\CONE\colon\,
f(\alpha,\beta;p) = f_{\cL_2}(\alpha,\beta;p)>f_{\cD}(\alpha,\beta;\,p)\big\}.
\ee
\end{itemize}

\bconj{phdgsub*} 
Conjecture {\rm \ref{phdgsup*}} holds with $\bar\alpha^{*}$ instead of $\alpha^*$.
\econj


\subsection{Heuristics in support of the hypotheses}
\label{hypo}

{\bf Hypothesis~\ref{hyp1}.} 
At $(\alpha,0)\in \CONE$, the $BB$-interaction vanishes while the $AA$-interaction does 
not, and we have seen earlier that there is no localization of the copolymer along $AB$-interfaces 
when $\beta=0$. Consequently, when the copolymer moves at a non-zero slope $l\in \R
\setminus\{0\}$ it necessarily reduces the time it spends in the $B$-solvent. To be more specific, 
let $\bar\rho\in \bar \cR_p$ be a maximizer of the variational formula in \eqref{genevar}, and assume 
that the copolymer moves in the emulsion by following the strategy of  displacement associated 
with $\bar\rho$. Consider the situation in which the copolymer moves upwards for awhile at slope 
$l>0$ and over a horizontal distance $h>0$, and subsequently changes direction to move 
downward at slope $l'<0$ and over a horizontal distance $h'>0$. This change of vertical direction 
is necessary to pass over a $B$-block, otherwise it would be entropically more advantageous to 
move at slope $(hl+h'l')/(h+h')$ over a horizontal distance $h+h'$ (by the strict concavity of $\tilde
\kappa$ in Lemma~\ref{l:lemconv2}(i)). Next, we observe (see Fig.~\ref{fig-optim}) that when the 
copolymer passes over a $B$-block, the best strategy in terms of entropy is to follow the 
$AB$-interface (consisting of this $B$-block and the $A$-solvent above it) without being localized, 
i.e., the copolymer performs a long excursion into the $A$-solvent but the two ends of this 
excursion are located on the $AB$-interface. This long excursion is counted in $\bar\rho_\cI$. Consequently, Hypothesis~\ref{hyp1} ($\bar\rho_\cI>0$) will be satisfied if we can show that 
the copolymer necessarily spends a strictly positive fraction of its time performing such changes 
of vertical direction. But, by the ergodicity of $\omega$ and $\Omega$, this has to be the case.

\begin{figure}[htbp]
\begin{center}
\includegraphics[width=.55\textwidth]{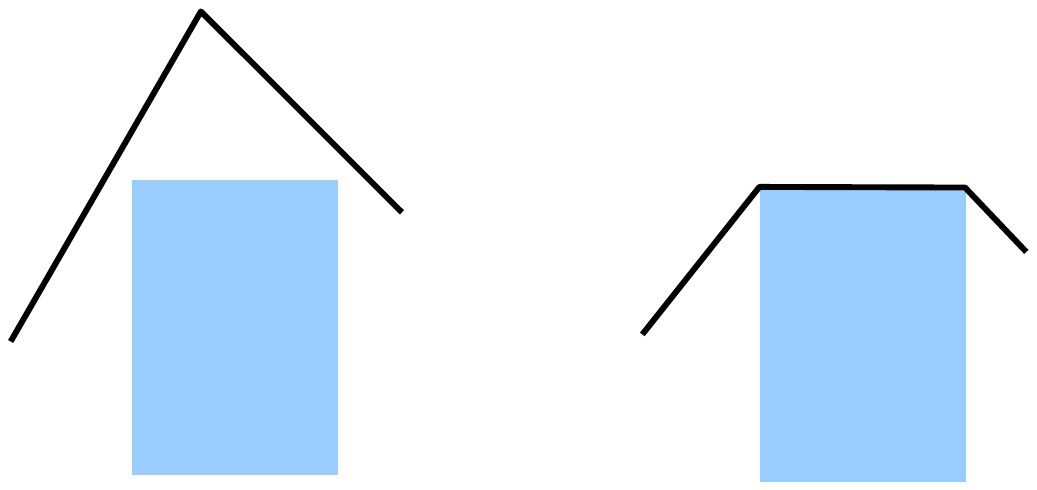}
\vspace{-8cm}
\caption{Entropic optimization when the copolymer passes over a $B$-block.} 
\label{fig-optim}
\end{center}
\end{figure} 

\medskip\noindent
{\bf Hypothesis~\ref{hyp2}.} 
The hypothesis can be rephrased in a simpler way. Recall Remark~\ref{rr2} and note that there 
is an $l_0\in (0,\infty)$ such that $g_A>0$ on $[0,l_0)$ and $g_A<0$ on $(l_0,\infty)$. Assume 
by contradiction that Hypothesis~\ref{hyp2} fails, so that the ratio in \eqref{impo} is unbounded. 
Then, by spending an arbitrarily small amount of time in the $B$-solvent, the copolymer can 
improve the best saturated strategies by moving some of the mass of $\bar\rho_A(l_0,\infty)$ 
to $\bar \rho_A (0,l_0)$, such that the entropic gain is arbitrarily larger than the time spent in 
the $B$-solvent. In other words, failure of Hypothesis~\ref{hyp2} means that spending an 
arbitrarily small fraction of time in the $B$-solvent allows the copolymer to travel flatter when 
it is in the $A$-solvent during a fraction of the time that is arbitrarily larger than the fraction of 
the time it spends in the $B$-solvent. This means that, instead of going around some large 
cluster of the $B$-solvent, the copolymer simply crosses it straight to travel flatter. However, 
the fact that large subcritical clusters scale are shaped like large balls contradicts this scenario, 
because it means that the time needed to go around the cluster is of the same order as the 
time required to cross the cluster, which makes the unboundedness of the ratio in \eqref{impo} 
impossible.

\medskip\noindent
{\bf Hypothesis~\ref{hyp3}.}
Hypothesis~\ref{hyp3} is similar to Hypothesis~\ref{hyp2}, except that in the subcritical 
regime the copolymer spends a strictly positive fraction of time in the $B$-solvent. Failure 
of  Hypothesis~\ref{hyp3} would lead to the same type of contradiction. Indeed, the 
unboundedness of the ratio in \eqref{impo-alt} would mean that there are optimal paths 
that spend an arbitrarily small additional fraction of time in the $B$-solvent in such a way 
that the path can travel flatter in the $A$-solvent during a fraction of the time that is arbitrarily 
larger than the fraction of the time it spends in the $B$-solvent. Again, the fact that large 
subcritical clusters adopt round shapes rules out such a scenario.


\section{Key ingredients}
\label{keyingr}

In Section~\ref{pathentr}, we define the entropy per step $\tilde{\kappa}(u,l)$ carried 
by trajectories moving at slope $l\in\R_+$ with the constraint that the total number of 
steps divided by the total number of horizontal steps is equal to $u \in [1+l,\infty)$ 
(Proposition~\ref{lementr} below). In Section~\ref{interf}, we define the free energy 
per step $\phi_{\cI}(\mu)$ of a copolymer in the vicinity of an $AB$-interface with the 
constraint that the total number of steps divided by the total number of horizontal steps 
is equal to $\mu \in [1,\infty)$ (Proposition~\ref{l:feinflim} below). In Section~\ref{Mtruncation}, 
we introduce a truncated version of the model in which we bound the vertical displacement 
on the block scale in each column of blocks by $M$, with $M\in\N$ arbitrary but fixed. 
(This restriction will be removed in Section~\ref{Minfinity} by letting $M\to\infty$.) In 
Section~\ref{freeenp}, we combine the definitions in Sections~\ref{pathentr}--\ref{interf} 
to obtain a variational formula for the free energy per step in single columns of different 
types (Proposition~\ref{energ} below). In Section~\ref{Percofreq} we define the set of 
probability laws introduced in \eqref{susub}, which is a key ingredient of the slope-based 
variational formula in Theorem~\ref{varformula2}. Finally, in Section~\ref{Positivity of the
free energy}, we prove that the quenched free energy per step $f(\alpha,\beta; p)$ is strictly 
positive on $\CONE$.


\subsection{Path entropies at given slope}
\label{pathentr}

{\bf Path entropies.} 
We define the entropy of a path crossing a single column. To that aim, we set 
\begin{align}
\label{add4}
\nonumber 
\cH &= \{(u,l)\in [0,\infty) \times \R \colon\, u\geq 1+|l|\},\\
\cH_L &= \big\{(u,l)\in \cH \colon\,l \in \tfrac{\Z}{L},\, 
u\in 1+|l|+\tfrac{2\N}{L}\big\}, \qquad  L \in \N,
\end{align}
and note that $\cH\cap \mathbb{Q}^2=\cup_{L\in \N} \cH_L$. For $(u,l) \in \cH$, we
denote by $\cW_L(u,l)$ the set containing those paths $\pi=(0,-1)+\widetilde \pi$ with 
$\widetilde\pi\in\cW_{uL}$ (recall \eqref{defw}) for which $\pi_{uL}= (L,lL)$ (see 
Fig.~\ref{figtra}). The entropy per step associated with the paths in $\cW_L(u,l)$ is 
given by
\be{ttrajblock}
\tilde{\kappa}_L(u,l)=\tfrac{1}{uL} \log |\cW_L(u,l)|.
\ee

\begin{figure}[htbp]
\vspace{-.5cm}
\begin{center}
\includegraphics[width=.44\textwidth]{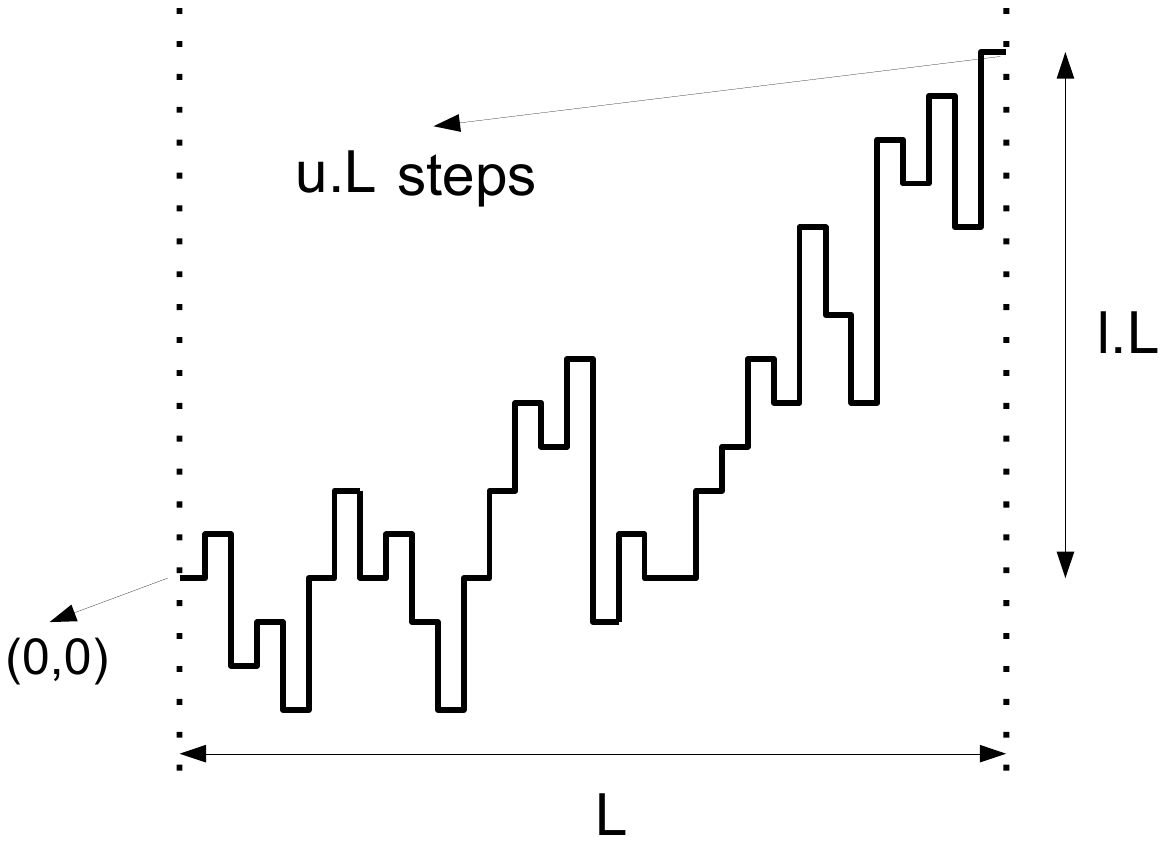}
\end{center}
\vspace{-4.7cm}
\caption{A trajectory in $\cW_L(u,l)$.}
\label{figtra}
\end{figure}

\noindent 
The following propositions will be proven in Appendix \ref{Path entropies}.
\bp{lementr}
For all $(u,l)\in\cH\cap\mathbb{Q}^2$ there exists a $\tilde{\kappa}(u,l) \in 
[0,\log 3]$ such that 
\be{conventr}
\lim_{ {L\to \infty} \atop {(u,l)\in \cH_L} } \tilde{\kappa}_L(u,l)
= \sup_{  {L\in \N} \atop {(u,l)\in \cH_L}} \tilde{\kappa}_L(u,l)
= \tilde{\kappa}(u,l). 
\ee  
\ep

\noindent
An explicit formula is available for $\tilde{\kappa}(u,l)$, namely, 
\be{kapexplform}
\tilde{\kappa}(u,l) = \left\{\begin{array}{ll}
\kappa(u/|l|,1/|l|), &\l \neq 0,\\
\hat{\kappa}(u), &l = 0,
\end{array}
\right.
\ee
where $\kappa(a,b)$, $a\geq 1+b$, $b\geq 0$, and $\hat{\kappa}(\mu)$, $\mu \geq 1$, 
are given in \cite{dHW06}, Section 2.1, in terms of elementary variational formulas
involving entropies (see \cite{dHW06}, proof of Lemmas 2.1.1--2.1.2). The two formulas 
in \eqref{kapexplform} allow us to extend $(u,l)\mapsto \tilde{\kappa}(u,l)$ to a 
continuous and strictly concave function on $\cH$ (see  Lemma \ref{l:lemconv2} ).


\subsection{Free energy for a linear interface}
\label{interf}

{\bf Free energy along a single linear interface.} 
To analyze the free energy per monomer in a single column we need to first analyze the
free energy per monomer when the path moves in the vicinity of an $AB$-interface. To 
that end we consider a \emph{single linear interface} $\cI$ separating a solvent $B$ 
in the lower halfplane from a solvent $A$ in the upper halfplane (the latter is assumed 
to include the interface itself).

For $L\in \N$ and $\mu\in 1+\frac{2\N}{L}$, let $\cW^\cI_L(\mu)=\cW_L(\mu,0)$ denote 
the set of $\mu L$-step directed self-avoiding paths starting at $(0,0)$ and ending at 
$(L,0)$. Recall \eqref{bondlabel} and define
\be{feinf}
\phi^{\omega,\cI}_L(\mu) = \frac{1}{\mu L} \log Z^{\omega,\cI}_{L,\mu} 
\quad \text{ and } \quad \phi^\cI_L(\mu)=\E[\phi^{\omega,\cI}_L(\mu)] ,
\ee
with
\be{Zinf}
\begin{aligned}
Z^{\omega,\cI}_{L,\mu}
&= \sum_{\pi\in\cW_L^\cI(\mu)} \exp\left[H^{\omega,\cI}_{L}(\pi)\right],\\
H^{\omega,\cI}_{L}(\pi)
&= \sum_{i=1}^{\mu L}\big(\beta\, 1\{\omega_i=B\}-\alpha\,
1\{\omega_i=A\}\big)\  1\{(\pi_{i-1},\pi_i) < 0\},
\end{aligned}
\ee
where $(\pi_{i-1},\pi_i) < 0$ means that the $i$-th step lies in the lower halfplane, 
strictly below the interface (see Fig.~\ref{fig7}).

\bp{l:feinflim} {\rm (\cite{dHW06}, Section 2.2.2)}\\
For all $(\alpha,\beta)\in\CONE$ and $\mu \in \mathbb{Q}\cap[1,\infty)$ there exists 
a $\phi_\cI(\mu)=\phi_\cI(\mu;\alpha,\beta)$ $\in\R$ such that
\be{fesainf}
\lim_{ {L\to \infty} \atop {\mu\in 1+\frac{2\N}{L}} } 
\phi^{\omega,\cI}_L(\mu) = \phi_\cI(\mu) 
\quad \text{for $\P$-a.e. $\omega$ and in $L^1(\P)$.}
\ee
\ep

It is easy to check (with the help of concatenation of trajectories) that $\mu \mapsto \mu 
\phi_{\cI}(\mu;\alpha,\beta)$ is concave. For later use we need strict concavity:

\begin{assumption}
\label{assu}
For all $(\alpha,\beta)\in\CONE$ the function $\mu\mapsto\mu \phi_{\cI}(\mu;\alpha,\beta)$ 
is strictly concave on $[1,\infty)$.
\end{assumption}

\noindent
This property is plausible, but hard to prove. There is to date no model of a polymer near a
linear interface with disorder for which a property of this type has been established. A proof
would require an explicit representation for the free energy, which for models with disorder 
typically is not available. 

\begin{figure}[htbp]
\vspace{-1cm}
\begin{center}
\includegraphics[width=.58\textwidth]{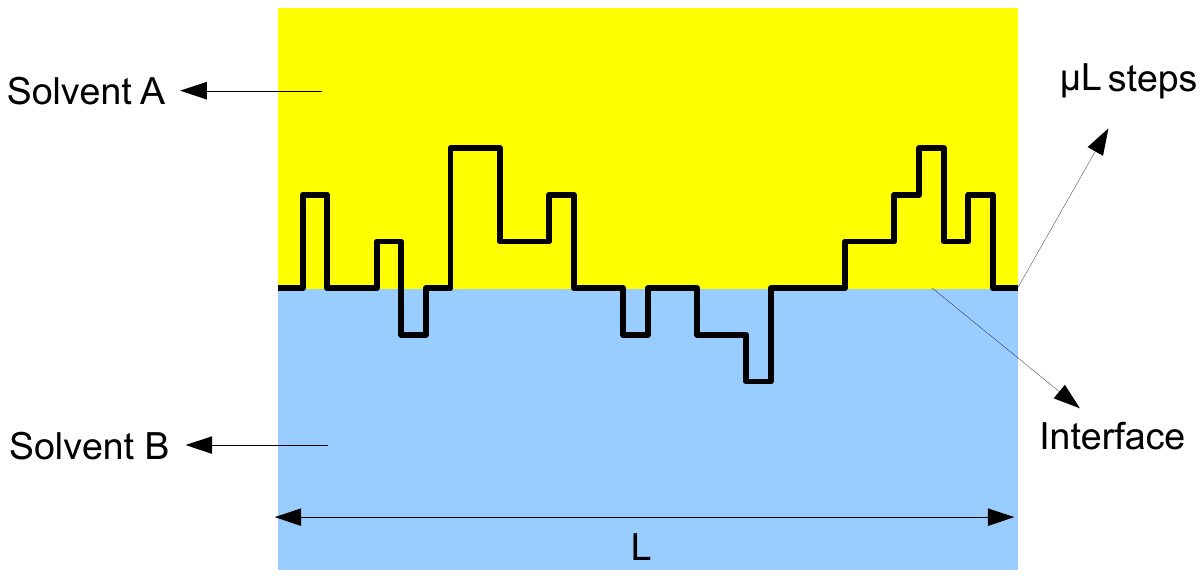}
\end{center}
\vspace{-8cm}
\caption{Copolymer near a single linear interface.}
\label{fig7}
\end{figure}


\subsection{Path restriction}
\label{Mtruncation}
 
In the remainder of this section, as well as in Sections~\ref{proofprop1}--\ref{varfo2}, we will 
work with a \emph{truncation} of the model in which we bound the vertical displacement on 
the block scale in each column of blocks by $M\in\N$. The value of $M$ will be \emph{arbitrary 
but fixed}. In other words, instead of considering the full set of trajectories $\cW_n$, we 
consider only trajectories that exit a column through a block at most $M$ above or $M$ below 
the block where the column was entered (see Fig.~\ref{fig-const}). The reason for doing the
truncation is that is simplifies our proof of the column-based variational formula. In 
Section~\ref{Minfinity} we will remove the truncation by showing that the free energy of
the untruncated model is the $M\to\infty$ limit of the free energy of the $M$-truncated 
model, and that the variational formulas match up as well.   

We recall \eqref{blocks} and, formally, we partition $(0,\infty) \times \R$ into columns of blocks of width $L_n$, i.e.,  
\be{blockcol}
(0,\infty)\times\R = \cup_{j\in \N_0} \cC_{j,L_n}, \qquad 
\cC_{j,L_n}=\cup_{k\in \Z} \Lambda_{L_n}(j,k),
\ee
where $C_{j,L_n}$ is the $j$-th column. For each $\pi \in \cW_n$, we let $\tau_j$ 
be the time at which $\pi$ leaves the $(j-1)$-th column and enters the $j$-th column, 
i.e.,
\be{deftau}
\tau_j = \sup\{i\in \N_{0} \colon \,\pi_i\in \cC_{j-1,n}\}
= \inf \{i\in \N_0 \colon \, \pi_i\in \cC_{j,n} \}-1,
\qquad j = 1,\dots,N_\pi-1,
\ee
where $N_\pi$ indicates how many columns have been visited by $\pi$. Finally, we let 
$v_{-1}(\pi)=0$ and, for $j \in \{0,\dots,N_\pi-1\}$, we let $v_{j}(\pi) \in \Z$ be 
such that the block containing the last step of the copolymer in $\cC_{j,n}$ is 
labelled by $(j,v_j(\pi))$, i.e., $(\pi_{\tau_{j+1}-1},\pi_{\tau_{j+1}}) \in \Lambda_{L_N}
(j,v_j(\pi))$. Thus, we restrict $\cW_{n}$ to the subset $\cW_{n,M}$ defined as 
\be{defwalt}
\cW_{n,M} = \big\{\pi\in \cW_n \colon\, |v_j(\pi)-v_{j-1}(\pi)|\leq M
\,\,\forall\,j\in \{0,\dots,N_\pi-1\} \big\}.
\ee

\begin{figure}[htbp]
\begin{center}
\includegraphics[width=.3\textwidth]{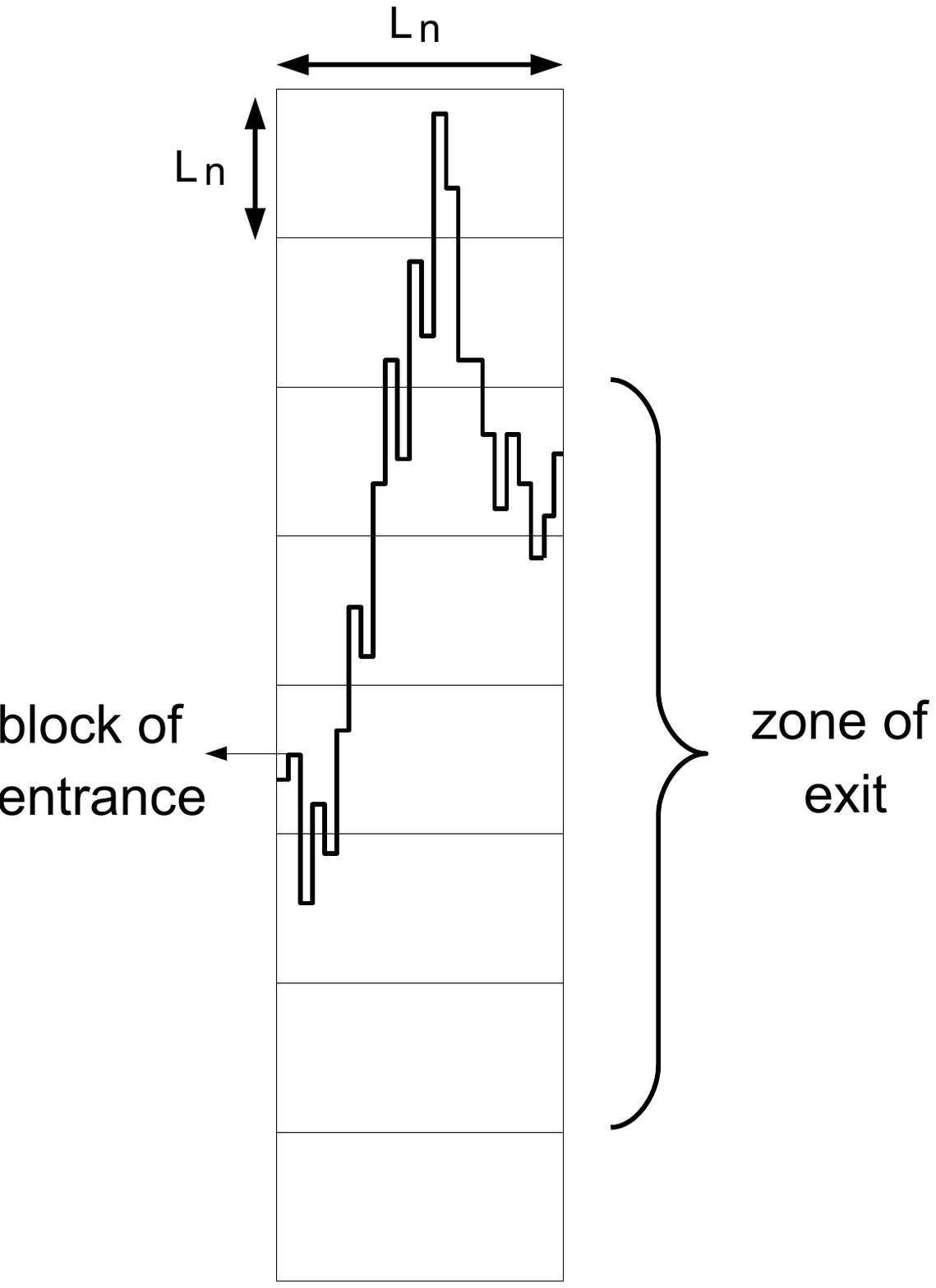}
\caption{Example of a trajectory $\pi \in\cW_{n,M}$ with $M=2$ crossing the column 
$\cC_{0,L_n}$ with $v_0(\pi)=2$.} 
\label{fig-const}
\end{center}
\end{figure} 

\noindent 
We recall \eqref{partfunc} and we define $Z_{n,L_n}^{\omega,\Omega}(M;\alpha,\beta)$
and $f_{n}^{\omega,\Omega}(M;\alpha,\beta)$ the partition function and the quenched 
free energy restricted to those trajectories in $\cW_{n,M}$, i.e., 
\be{partfunc11} 
f_{n}^{\omega,\Omega}(M;\alpha,\beta)
=\tfrac{1}{n}\log Z_{n,L_n}^{\omega,\Omega}(M;\alpha,\beta) 
\quad \text{with}\quad Z_{n,L_n}^{\omega,\Omega}(M;\alpha,\beta)
=\sum_{\pi \in \cW_{n,M}} e^{H_{n,L_n}^{\omega,\Omega}(\pi;\alpha,\beta)},
\ee 
and,  as $n\to\infty$, the free energy per monomer is given by
\be{felimdef}
f(M;\alpha,\beta) = \lim_{n\to \infty}f_{n}^{\omega,\Omega}(M;\alpha,\beta)
\ee
provided this limit exists $\omega,\Omega$-a.s. 

\noindent
In Remark~\ref{reM} below we discuss how the mesoscopic vertical restriction can be 
relaxed by letting $M\to \infty$.

\begin{remark}
\label{reM}
{\rm 
As mentioned in Section~\ref{discus}, the \emph{slope-based variational formula} in 
Theorem~\ref{varformula2} will be deduced from a \emph{column-based variational 
formula} stated in Proposition~\ref{varformula}. In this framework, the truncated model 
is used as follows. First, we prove the column-based variational formula for the truncated 
model: this will be the object of Propositions~\ref{pr:formimp}--\ref{pr:varar} in 
Section~\ref{Key Propositions}. Next, we show with Proposition~\ref{pr:formimppp} 
that, as the truncation levl $M$ diverges, the truncated free energy converges to the 
non-truncated free energy. This will complete the proof of the column-based variational 
formula for the non-truncated model. Finally, in Section~\ref{varfo2}, we transform the 
column-based variational formula into the slope-based variational formula for the 
non-truncated model.}
\end{remark}


\subsection{Free energy in a single column and variational formulas}
\label{freeenp}

In this section, we prove the convergence of  the free energy per step in a single column 
(Proposition~\ref{convfree1}) and derive a variational formula for this free energy with the 
help of Propositions~\ref{lementr}--\ref{l:feinflim}. The variational formula takes different 
forms (Propositions~\ref{energ}), depending on \emph{whether there is or is not an $AB$-interface 
between the heights where the copolymer enters and exits the column, and in the latter case 
whether an $AB$-interface is reached or not}. 

In what follows we need to consider the randomness in a single column. To that 
aim, we recall \eqref{blockcol}, we pick $L\in \N$ and once $\Omega$ is chosen, 
we can record the randomness of $C_{j,L}$ as
\be{add1}
\Omega_{(j,~\cdot~)}=\{\Omega_{(j,l)}\colon\, l\in \Z\}.
\ee
We will also need to consider the randomness of the $j$-th column seen by a 
trajectory that enters $\cC_{j,L}$ through the block $\Lambda_{j,k}$ with 
$k\neq 0$ instead of $k=0$. In this case, the randomness of $\cC_{j,L}$ is 
recorded as
\be{add2}
\Omega_{(j,k+~\cdot~)}=\{\Omega_{(j,k+l)}\colon\, l\in \Z\}.
\ee

Pick $L\in \N$, $\chi\in \{A,B\}^\Z$ and consider $\cC_{0,L}$ endowed with the disorder 
$\chi$, i.e., $\Omega(0,\cdot)=\chi$.  Let $(n_i)_{i\in \Z}\in \Z^\Z$ be the successive 
heights of the $AB$-interfaces in $\cC_{0,L}$ divided by $L$, i.e.,  
\be{efn}
\dots<n_{-1}<n_0\leq 0< n_1<n_2<\dots.
\ee
and the $j$-th interface of $\cC_{0,L}$ 
is $\cI_j=\{0,\dots,L\}\times \{n_j L\}$ (see 
Fig.~\ref{fig5bis}). Next, for $r\in \N_0$ we set
\begin{align}\label{defchi1}
k_{r,\chi}=&\,0 \text{ if } n_1>r \text{ and } 
k_{r,\chi}=\max\{i\geq 1\colon n_i\leq r\} \text{ otherwise},
\end{align}
while for $r\in -\N$ we set
\begin{align}\label{defchi2}
\quad \quad \quad \quad  k_{r,\chi}=&\, 0 \text{ if }
n_0\leq r\text{ and } k_{r,\chi}
=\min\{i\leq 0\colon n_{i}\geq r+1\}-1 \text{ otherwise}.
\end{align}
Thus, $|k_{r,\chi}|$ is the number of $AB$-interfaces between heigths $1$ and $r L$ 
in $\cC_{0,L}$. 

\begin{figure}[htbp]
\begin{center}
\includegraphics[width=.38\textwidth]{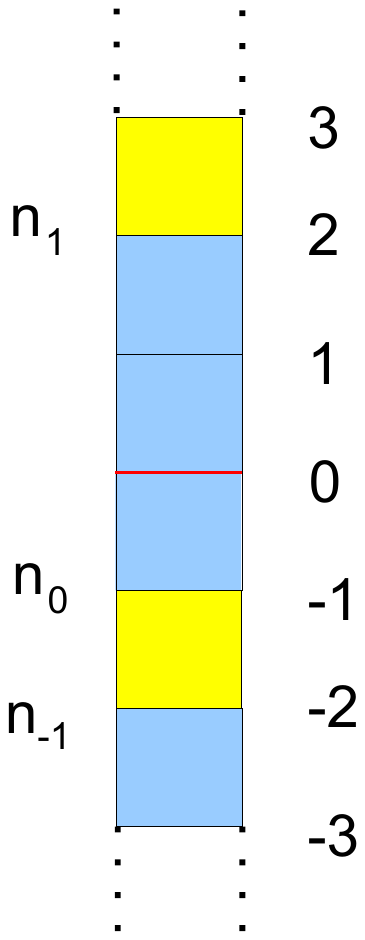}
\end{center}
\vspace{-3.5cm}
\caption{Example of a column with disorder $\chi=(\dots,\chi(-3),\chi(-2),\chi(-1),
\chi(0),\chi(1),\chi(2),$ $\dots)=(\dots,B,A,B,B,B,A,,\dots)$. In this example, for 
instance, $k_{-2,\chi}=-1$ and $k_{1,\chi}=0$.}
\label{fig5bis}
\end{figure}


\subsubsection{Free energy in a single column}
\label{frenp1}

{\bf Column crossing characteristics.} 
Pick $L,M\in \N$, and consider the first column $\cC_{0,L}$. The type of $\cC_{0,L}$ is 
determined by $\Theta=(\chi,\Xi,x)$, where $\chi=(\chi_j)_{j\in \Z}$ encodes the type 
of each block in $\cC_{0,L}$, i.e., $\chi_j=\Omega_{(0,j)}$ for $j\in \Z$, and $(\Xi,x)$ 
indicates which trajectories $\pi$ are taken into account. In the latter, $\Xi$ is 
given by $(\Delta \Pi,b_0,b_1)$ such that the vertical increment in $\cC_{0,L}$ on 
the block scale is $\Delta \Pi$ and satisfies $|\Delta \Pi|\leq M$ , i.e., $\pi$ 
enters $\cC_{0,L}$ at $(0,b_0 L)$ and exits $\cC_{0,L}$ at $(L,(\Delta \Pi+b_1)L)$. 
As in \eqref{defchi1} and \eqref{defchi2}, we set $k_\Theta=k_{\Delta\Pi,\chi}$ and 
we let $\cV_\AB$ be the set containing those $\Theta$ satisfying $k_\Theta\neq 0$. 
Thus, $\Theta\in \cV_\AB$ means that the trajectories crossing $\cC_{0,L}$ from 
$(0,b_0 L)$ to $(L,(\Delta \Pi+b_1)L)$ necessarily hit an $AB$-interface, and in 
this case we set $x=1$. If, on the other hand, $\Theta\in \cV_{\nAB}=\cV\setminus
\cV_{\AB}$, then we have $k_\Theta= 0$ and we set $x=1$ when the set of trajectories 
crossing $\cC_{0,L}$ from $(0,b_0 L)$ to $(L,(\Delta \Pi+b_1)L)$ is restricted to 
those that do not reach an $AB$-interface before exiting $\cC_{0,L}$, while we set 
$x=2$ when it is restricted to those trajectories that reach at least one $AB$-interface 
before exiting $\cC_{0,L}$. To fix the possible values taken by $\Theta=(\chi,\Xi,x)$ 
in a column of width $L$, we put $\cV_{L,M}=\cV_{\AB,L,M}\cup\cV_{\nAB,L,M}$ with 
\begin{align}
\label{set2}
\nonumber \cV_{\AB,L,M}&=\big\{(\chi,\Delta\Pi,b_0,b_1,x)\in 
\{A,B\}^\Z\times \Z\times \big\{\tfrac{1}{L},\tfrac{2}{L},\dots,1\big\}^2
\times \{1\}\colon\\
\nonumber 
&\hspace{6cm}\qquad \qquad\qquad\qquad 
|\Delta\Pi|\leq M,\,k_{\Delta\Pi,\chi}\neq 0\big\},\\
\nonumber \cV_{\nAB,L,M}&=\big\{(\chi,\Delta\Pi,b_0,b_1,x)\in 
\{A,B\}^\Z\times \Z\times\big\{\tfrac{1}{L},\tfrac{2}{L},\dots,1\big\}^2
\times \{1,2\}\colon\\
&\hspace{6cm}
\qquad \qquad\qquad\qquad |\Delta\Pi|\leq M,\,  k_{\Delta\Pi,\chi}= 0\big\}.
\end{align}
Thus, the set of all possible values of $\Theta$ is $\cV_M=\cup_{L\geq 1} \cV_{L,M}$, 
which we partition into $\cV_M=\cV_{\AB,M}\cup\cV_{\nAB,M}$ (see Fig.~\ref{fig6}) with
\begin{align}
\label{sset2}
\nonumber \cV_{\AB,M}
&=\cup_{L\in \N}\  \cV_{\AB,L,M}\\
\nonumber
&= \big\{(\chi,\Delta\Pi,b_0,b_1,x)\in \{A,B\}^\Z\times \Z\times 
(\mathbb{Q}_{(0,1]})^2\times\{1\}\colon\,\,
|\Delta\Pi|\leq M,\, k_{\Delta\Pi,\chi}\neq 0\big\},\\
\nonumber \cV_{\nAB,M}&=\cup_{L\in \N}\  \cV_{\nAB,L,M}\\
&= \big\{(\chi,\Delta\Pi,b_0,b_1,x)\in \{A,B\}^\Z\times \Z\times 
(\mathbb{Q}_{(0,1]})^2\times\{1,2\}\colon\,
|\Delta\Pi|\leq M,\, k_{\Delta\Pi,\chi}= 0\big\},
\end{align}
where, for all $I\subset \R$, we set $\mathbb{Q}_{I}=I\cap \mathbb{Q}$. We define the 
closure of $\cV_M$ as $\overline\cV_M=\overline\cV_{\AB,M}\cup\overline\cV_{\nAB,M}$ with
\begin{align}
\label{sset3}
\nonumber &\overline\cV_{\AB,M}
= \big\{(\chi,\Delta\Pi,b_0,b_1,x)\in \{A,B\}^\Z\times \Z\times 
[0,1]^2\times\{1\}\colon\,|\Delta\Pi|\leq M,\, k_{\Delta\Pi,\chi}\neq 0\big\},\\
&\overline\cV_{\nAB,M}=\big\{(\chi,\Delta\Pi,b_0,b_1,x)\in \{A,B\}^\Z\times\Z
\times [0,1]^2\times\{1,2\}\colon\,|\Delta\Pi|\leq M,\, k_{\Delta\Pi,\chi}= 0\big\}.
\end{align}

\begin{figure}[htbp]
\begin{center}
\includegraphics[width=.42\textwidth]{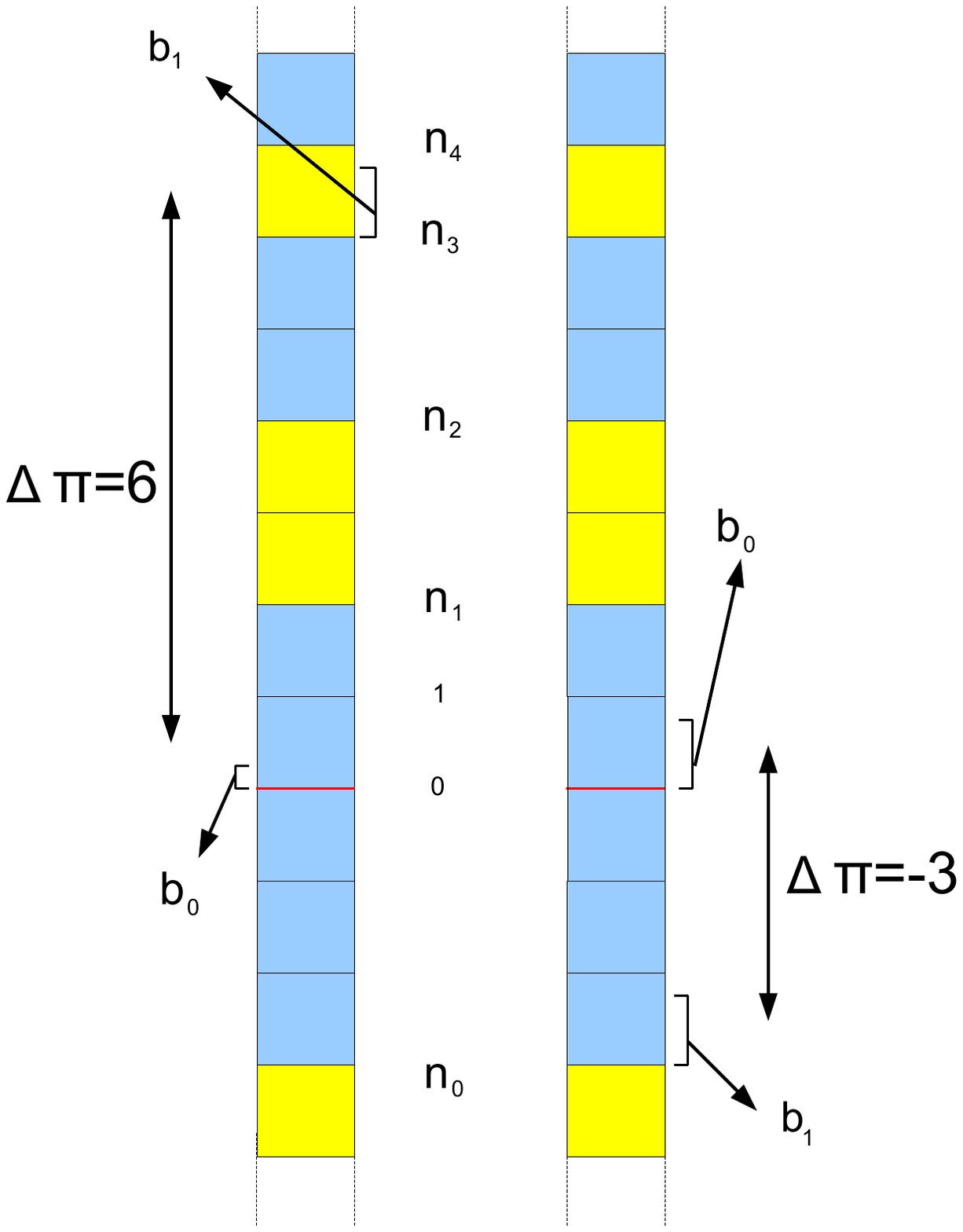}
\end{center}
\vspace{-1cm}
\caption{Labelling of coarse-grained paths and columns. On the left the type of the 
column is in $\cV_{\AB,M}$, on the right it is in $\cV_{\nAB,M}$ (with $M\geq 6$).}
\label{fig6}
\end{figure}

\medskip\noindent 
{\bf Time spent in columns.} 
We pick $L,M\in \N$, $\Theta=(\chi,\Delta\Pi,b_0,b_1,x)\in \cV_{L,M}$ and we specify the 
total number of steps that a trajectory crossing the column $\cC_{0,L}$ of type 
$\Theta$ is allowed to make. For $\Theta=(\chi,\Delta\Pi,b_0,b_1,1)$, set
\be{deft}
t_{\Theta}= 1+\mathrm{sign} (\Delta \Pi)\,
(\Delta \Pi+ b_1-b_0) \,\ind_{\{\Delta \Pi\neq 0\}} + |b_1-b_0|
\,\ind_{\{\Delta \Pi=0\}},
\ee
so that a trajectory $\pi$ crossing a column of width $L$ from $(0,b_0 L)$ to 
$(L,(\Delta \Pi+b_1)L)$ makes a total of $uL$ steps with $u\in t_{\Theta}
+\tfrac{2\N}{L}$.
For $\Theta=(\chi,\Delta\Pi,b_0,b_1,2)$ in turn, recall \eqref{efn} and let
\be{deftt}
t_{\Theta}= 1+\mathrm{min}\{2 n_1-b_0-b_1-\Delta \Pi, 2 |n_0| +b_0+b_1+\Delta \Pi\},
\ee
so that  a trajectory $\pi$ crossing a column of width $L$ and type $\Theta\in 
\cV_{\nAB,L,M}$ from $(0,b_0 L)$ to $(L,(\Delta \Pi+b_1)L)$ and reaching an $AB$-interface 
makes a total of $uL$ steps with $u\in t_{\Theta}+\tfrac{2\N}{L}$.

\begin{figure}[htbp]
\begin{center}
\includegraphics[width=.3\textwidth]{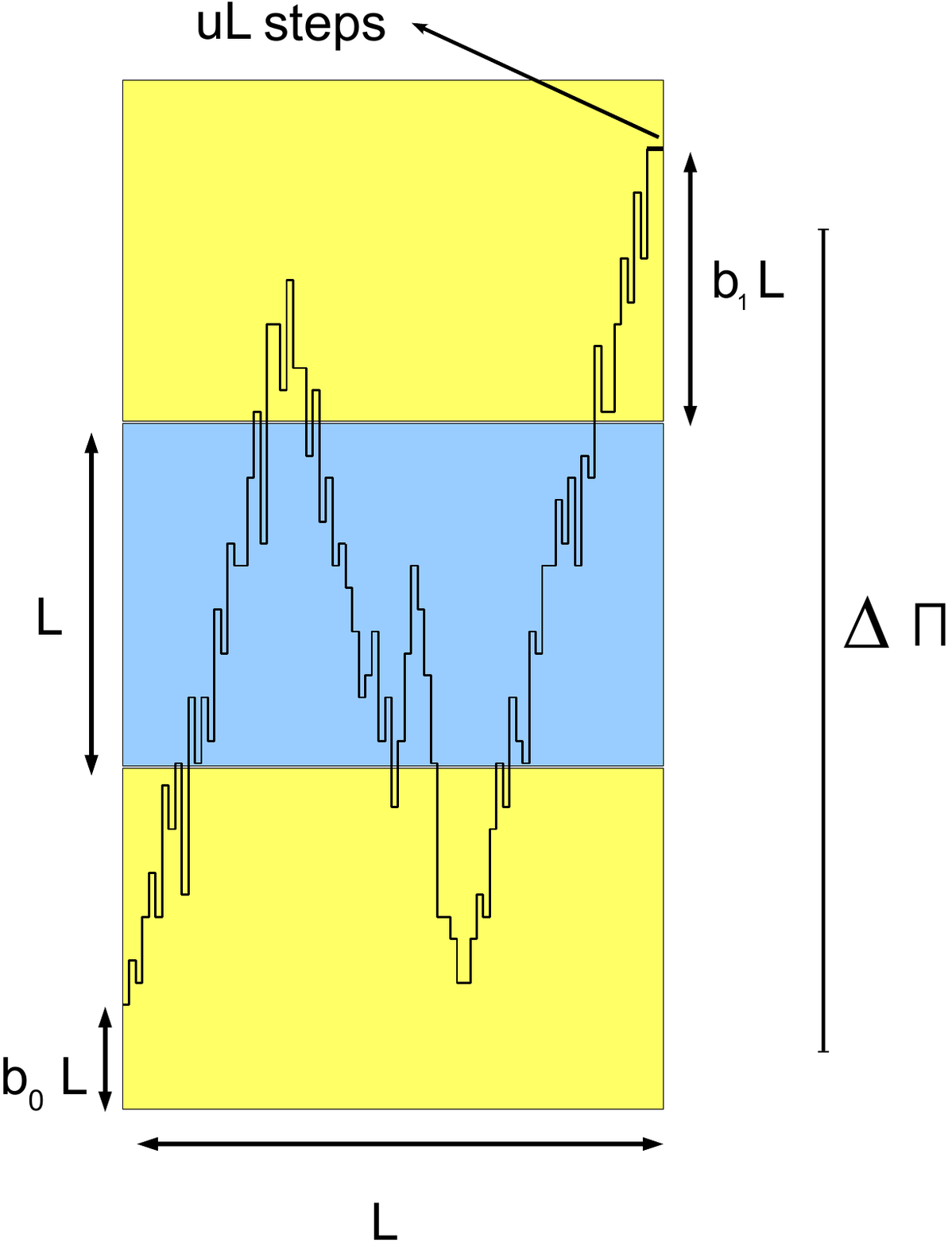}
\end{center}
\vspace{-1cm}
\caption{Example of a $uL$-step path inside a column of type $(\chi,\Delta\Pi,b_0,b_1,1)
\in \cV_{\AB,L}$ with disorder $\chi=(\dots,\chi(0),\chi(1),\chi(2),\dots)=(\dots,A,B,
A,\dots)$, vertical displacement $\Delta\Pi=2$, entrance height $b_0$ and exit height 
$b_1$.}
\label{fig5}
\end{figure}

\begin{figure}[htbp]
\begin{center}
\includegraphics[width=.3\textwidth]{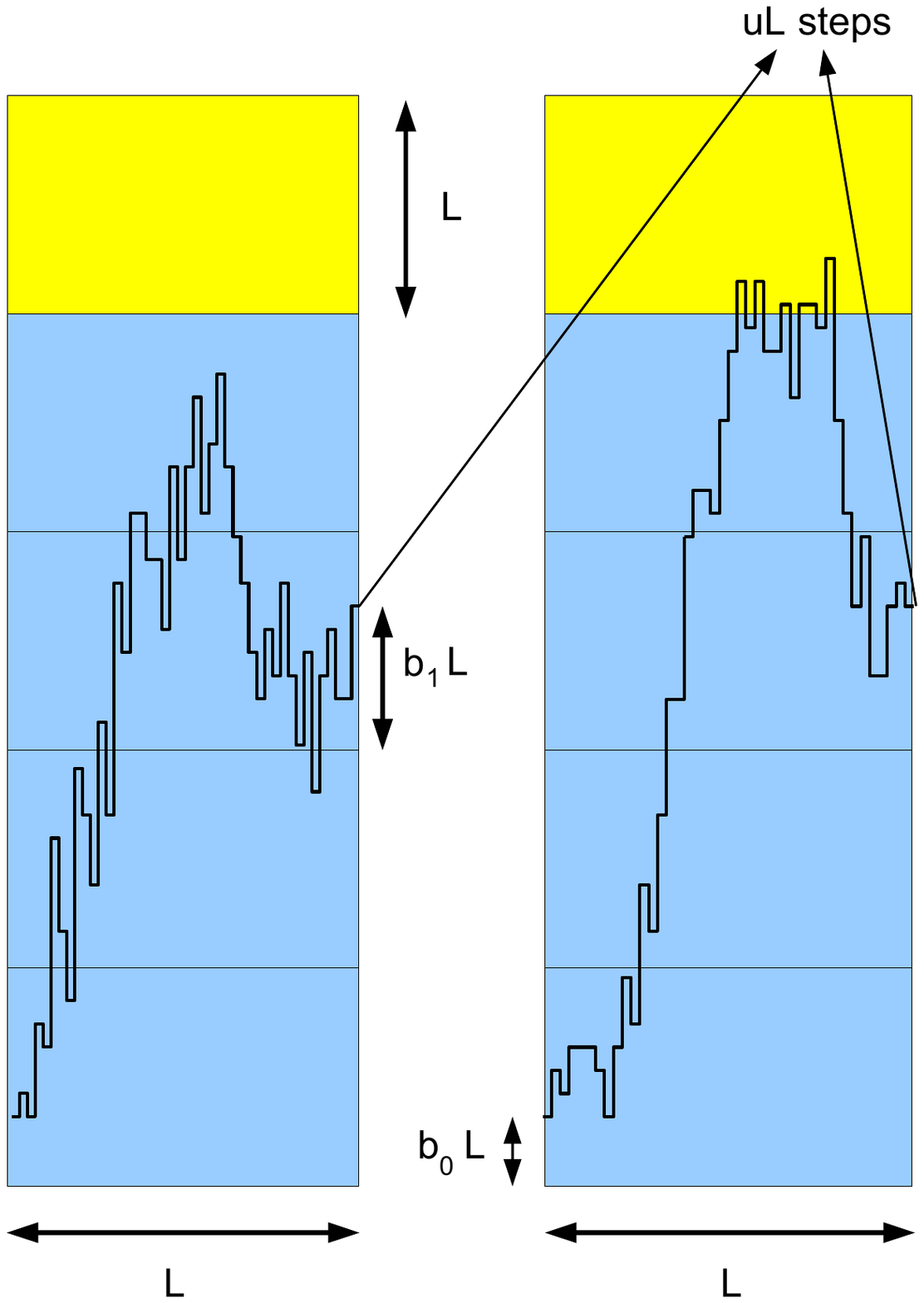}
\end{center}
\vspace{-1cm}
\caption{Two examples of a $uL$-step path inside a column of type $(\chi,\Delta\Pi,b_0,b_1,1)
\in \cV_{\nAB,L}$ (left picture) and $(\chi,\Delta\Pi,b_0,b_1,2)\in \cV_{\nAB,L}$ 
(right picture) with disorder $\chi=(\dots,\chi(0),\chi(1),\chi(2),\chi(3),\chi(4),
\dots)=(\dots,B,B,B,B,A,\dots)$, vertical displacement $\Delta\Pi=2$, entrance
height $b_0$ and exit height $b_1$.}
\label{block2}
\end{figure}

At this stage, we can fully determine the set $\cW_{\Theta,u,L}$ consisting of the 
$uL$-step trajectories $\pi$ that are considered in a column of width $L$ and type 
$\Theta$. To that end, for $\Theta\in \cV_{\AB,L,M}$ we map the trajectories $\pi\in
\cW_{L}(u,\Delta\Pi+b_1-b_0)$ onto $\cC_{0,L}$ such that $\pi$ enters $\cC_{0,L}$ 
at $(0,b_0 L)$ and exits $\cC_{0,L}$ at $(L,(\Delta\Pi+b_1)L)$ (see Fig.~\ref{fig5}), 
and for $\Theta\in\cV_{\nAB,L,M}$ we remove, dependencing on $x\in\{1,2\}$, those 
trajectories that reach or do not reach an $AB$-interface in the column (see 
Fig.~\ref{block2}). Thus, for  $\Theta\in \cV_{\AB,L,M}$ and $u\in t_{\Theta}+
\tfrac{2\N}{L}$, we let
\be{ww}
\cW_{\Theta,u,L}=\big\{\pi=(0,b_0 L)+\widetilde\pi\colon\,
\widetilde\pi\in \cW_{L}(u,\Delta\Pi+b_1-b_0) \big\},
\ee
and, for $\Theta\in \cV_{\nAB,L,M}$ and $u\in t_{\Theta}+\tfrac{2\N}{L}$,
\begin{align}
\label{ww2}
\nonumber \cW_{\Theta,u,L}
&=\big\{\pi\in(0,b_0 L)+ \cW_{L}(u,\Delta\Pi+b_1-b_0)\colon\,
\text{$\pi$ reaches no $AB$-interface} \big\}\ \text{if}\ x_\Theta=1,\\
\cW_{\Theta,u,L}&=\big\{\pi\in (0,b_0 L)+\cW_{L}(u,\Delta\Pi+b_1-b_0)\colon\, 
\text{$\pi$ reaches an $AB$-interface} \big\}\ \text{if}\  x_\Theta=2,
\end{align}
with $x_\Theta$ the last coordinate of $\Theta \in \cV_M$. Next, we set
\begin{align}
\label{setp}
\nonumber \cV_{L,M}^* &= \Big\{(\Theta,u)\in \cV_{L,M} \times [0,\infty)\colon\, 
u\in t_{\Theta}+\tfrac{2\N}{L}\Big\},\\
\nonumber \cV_M^* 
&=\big\{(\Theta,u)\in \cV_M \times \mathbb{Q}_{[1,\infty)}\colon\,
u\geq  t_{\Theta}\big\},\\
\overline \cV_M^{*}
&=\big\{(\Theta,u)\in \overline\cV_M\times [1,\infty) \colon\,
u\geq  t_{\Theta}\big\},
\end{align}
which we partition into $\cV^{*}_{\AB,L,M}\cup\cV^{*}_{\nAB,L,M}$, $\cV^*_{\AB,M}\cup
\cV^*_{\nAB,M}$ and $\overline\cV^{*}_{\AB,M}\cup\overline\cV^{*}_{\nAB,M}$. Note that 
for every $(\Theta,u)\in\cV^*_M$ there are infinitely many $L\in \N$ such that 
$(\Theta,u)\in \cV_{L,M}^*$, because $(\Theta,u)\in \cV_{qL,M}^*$ for all $q\in \N$ 
as soon as $(\Theta,u)\in \cV_{L,M}^*$.

\medskip\noindent 
{\bf Restriction on the number of steps per column.} 
In what follows we abbreviate
\be{defhe}
\EIGH = \{(M,m)\in\N\times \N\colon\,m\geq M+2\},
\ee
and, for $(M,m)\in \EIGH$, we consider the situation where the number of steps 
$uL$ made by a trajectory $\pi$ in a column of width $L\in \N$ is bounded by $m L$. 
Thus, we restrict the set $\cV_{L,M}$ to the subset $\cV_{L,M}^{\,m}$ containing only 
those types of columns $\Theta$ that can be crossed in less than $m L$ steps, i.e., 
\be{}
\cV_{L,M}^{\,m}=\{\Theta\in \cV_{L,M}\colon\, t_\Theta\leq m\}.
\ee
Note that the latter restriction only conconcerns those $\Theta$ satisfying $x_\Theta=2$. 
When $x_\Theta=1$ a quick look at \eqref{deft} suffices to state that $t_\Theta\leq M+2
\leq m$. Thus, we set $\cV_{L,M}^{\,m}=\cV_{\AB,L,M}^{\,m}\cup\cV_{\nAB,L,M}^{\,m}$ with 
$\cV_{\AB,L,M}^{\,m}=\cV_{\AB,L,M}$ and with
\begin{align}
\label{set2alt1}
\nonumber \cV_{\nAB,L,M}^{\,m}
= \Big\{\Theta\in \{A,B\}^{\Z}\times \Z \times
\big\{\tfrac{1}{L},\tfrac{2}{L},\dots,1\big\}^2&\times \{1,2\}\colon\\
& |\Delta\Pi| \leq M,\  k_{\Theta}= 0\ \ \text{and}\  t_\Theta\leq m\Big\}.
\end{align}
The sets $\cV_{M}^{\,m}=\cV_{\AB,M}^{\,m}\cup\cV_{\nAB,M}^{\,m}$ and $\overline\cV_{M}^{\,m}
=\overline\cV_{\AB,M}^{\,m}\cup\overline\cV_{\nAB,M}^{\,m}$ are obtained by mimicking 
(\ref{sset2}--\ref{sset3}). In the same spirit, we restrict $\cV_{L,M}^{*}$ to 
\be{setpalt}
\cV_{L,M}^{*,\,m}=\{(\Theta,u)\in \cV_{L,M}^{*}\colon\, \Theta\in \cV_{L,M}^{\,m}, u\leq m\}
\ee 
and $\cV_{L,M}^*=\cV_{\AB,L,M}^*\cup \cV_{\nAB,L,M}^*$ with
\be{setpaltalt}
\begin{aligned}
\cV_{\AB,L,M}^{*,\,m} &= \Big\{(\Theta,u)&\in \cV_{\AB,L,M}^{\,m} \times [1,m]\colon\, 
u\in t_{\Theta}+\tfrac{2\N}{L}\Big\},\\
\cV_{\nAB,L,M}^{*\,m} &= \Big\{(\Theta,u)&\in \cV_{\nAB,L,M}^{\,m}\times [1,m]\colon\, 
u\in t_{\Theta}+\tfrac{2\N}{L} \Big\}.
\end{aligned}
\ee
We set also  $\cV_M^{*,\,m}=\cV^{*,\,m}_{\AB,M}\cup \cV^{*,\,m}_{\nAB,M}$ with 
$\cV_{\AB,M}^{*,\,m}=\cup_{L\in \N}\cV^{*,\,m}_{\AB, L,M}$ and $\cV^{*,\,m}_{\nAB,M}
=\cup_{L\in \N} \cV^{*,\,m}_{\nAB,L,M}$, and rewrite these as
\begin{align}
\label{set2alt2}
\nonumber \cV^{*,\,m}_{\AB,M} 
=\big\{(\Theta,u)&\in \cV_{\AB,M}^{\,m} \times \mathbb{Q}_{[1,m]}\colon\,
u\geq  t_{\Theta}\big\},\\
\cV_{\nAB,M}^{*,\,m}= \big\{(\Theta,u)&\in \cV_{\nAB,M}^{\,m} \times  \mathbb{Q}_{[1,m]}\colon\, 
u\geq t_{\Theta}\big\}.
\end{align}
We further set  $\overline{\cV}^{\,*}_M = \overline{\cV}^{\,*,\,m}_{\AB,M}\cup
\overline{\cV}^{\,*,\,m}_{\nAB,M}$ with 
\be{set2alt3}
\begin{aligned}
\overline\cV^{\,*,\,m}_{\AB,M} 
&=\big\{(\Theta,u)\in \overline\cV_{\AB,M}^{\,m} \times [1,m]\colon\,
u\geq  t_{\Theta}\big\},\\
\overline\cV_{\nAB,M}^{\,*,\,m} 
&= \Big\{(\Theta,u) \in \overline\cV_{\nAB,M}^{\,m} \times [1,m]\colon\, 
u\geq  t_{\Theta}\Big\}.
\end{aligned}
\ee

\medskip\noindent 
{\bf Existence and uniform convergence of free energy per column.} 
Recall \eqref{ww}, \eqref{ww2} and, for $L\in \N$, $\omega\in \{A,B\}^\N$ and $(\Theta,u)
\in\cV^{\,*}_{L,M}$, we associate with each $\pi \in \cW_{\Theta,u,L}$ the energy
\be{Hamiltonian1}
H_{uL,L}^{\omega,\chi}(\pi)
=  \sum_{i=1}^{uL} \big(\beta\, 1\left\{\omega_i=B\right\}
-\alpha\, 1\left\{\omega_i=A\right\}\big)\,
1\Big\{\chi^{L}_{(\pi_{i-1},\pi_i)}=B\Big\},
\ee
where $\chi^{L}_{(\pi_{i-1},\pi_i)}$ indicates the label of the block containing 
$(\pi_{i-1},\pi_i)$ in a column with disorder $\chi$ of width $L$. (Recall that
the disorder in the block is part of the type of the block.) The latter allows us 
to define the quenched free energy per monomer in a column of type $\Theta$ and size 
$L$ as
\be{partfunc2}
\psi^{\omega}_L(\Theta,u)
= \frac{1}{u L} \log Z^{\omega}_L(\Theta,u) 
\quad \text{with} \quad Z^{\omega}_L(\Theta,u)
=\sum_{\pi \in \cW_{\Theta,u,L}} e^{\,H_{uL,L}^{\omega,\chi}(\pi)}.
\ee
Abbreviate $\psi_L(\Theta,u)=\E[\psi^\omega_L(\Theta,u)]$, and note that for $M\in \N$,
$m\geq M+2$ and $(\Theta,u)\in \cV_{L,M}^{\,*,\,m}$ all $\pi\in \cW_{\Theta,u,L}$ 
necessarily remain in the blocks $\Lambda_L(0,i)$ with $i\in \{-m+1,\dots,m-1\}$. 
Consequently, the dependence on $\chi$ of $\psi^{\omega}_L(\Theta,u)$ is restricted 
to those coordinates of $\chi$ indexed by $\{-m+1,\dots,m-1\}$. The following proposition 
will be proven in Section \ref{proofprop1}.

\begin{proposition}
\label{convfree1}
For every $M\in\N$ and $(\Theta,u)\in\cV_M^*$ there exists a $\psi(\Theta,u)\in\R$ such that 
\be{conve}
\lim_{ {L\to \infty} \atop {(\Theta,u)\in\cV_{L,M}^*} } \psi^{\omega}_L(\Theta,u)
= \psi(\Theta,u) = \psi(\Theta,u;\alpha,\beta) \quad \omega-a.s.
\ee
Moreover, for every $(M,m)\in \EIGH$ the convergence is uniform in $(\Theta,u)\in
\cV^{*,\,m}_M$.
\ep

\medskip\noindent
{\bf Uniform bound on the free energies.}
Pick $(\alpha,\beta)\in \CONE$, $n\in \N$, $\omega\in \{A,B\}^\N$, $\Omega\in
\{A,B\}^{\N_0\times\Z}$, and let $\bar\cW_n$ be any non-empty subset of $\cW_n$ 
(recall \eqref{defw}). Note that the quenched free energies per monomer introduced 
until now are all of the form
\be{unifb}
\psi_n=\tfrac1n\log \sum_{\pi\in\bar\cW_n} e^{\,H_n(\pi)},
\ee
where $H_n(\pi)$ may depend on $\omega$ and $\Omega$ and satisfies $-\alpha n\leq 
H_n(\pi)\leq \alpha n$ for all $\pi\in \bar\cW_n$ (recall that $|\beta|\leq \alpha$ 
in $\CONE$). Since $1\leq|\bar\cW_n|\leq |\cW_n|\leq 3^n$, we have
\be{boundel}
|\psi_n|\leq \log 3+\alpha =^\mathrm{def} C_{\text{uf}}(\alpha).
\ee
The uniformity of this bound in $n$, $\omega$ and $\Omega$ allows us to average over 
$\omega$ and/or $\Omega$ or to let $n\to \infty$.


\subsubsection{Variational formulas for the free energy in a single column}
\label{newsec}

We next show how the free energies per column can be expressed in terms of a 
variational formula involving the path entropy and the single interface free energy 
defined in Sections ~\ref{pathentr} and  \ref{interf}. Throughout this section $M \in \N$ 
is fixed. 

For $\Theta\in  \overline\cV_{M}$ we need to specify $l_{A,\Theta}$ and $l_{B,\Theta}$, 
the minimal vertical distances the copolymer must cross in blocks of type $A$ and $B$, 
respectively, when crossing a column of type $\Theta$.

\medskip\noindent 
{\bf Vertical distance to be crossed  in columns of class $\AB$.}
Pick $\Theta\in \overline\cV_{\AB,M}$ and put
\begin{align}
\label{defl1}
\nonumber l_1&=1_{\{\Delta \Pi>0\}} 
(n_1-b_0)+1_{\{\Delta \Pi<0\}} (b_0-n_0),\\
\nonumber l_j&=1_{\{\Delta \Pi>0\}} 
(n_j-n_{j-1})+1_{\{\Delta \Pi<0\}} (n_{-j+2}-n_{-j+1}) \quad \text{for} \quad 
j\in \{2,\dots,|k_\Theta|\},\\
l_{|k_\Theta|+1}&=1_{\{\Delta \Pi>0\}} 
(\Delta\Pi+b_1-n_{k_\Theta})+1_{\{\Delta \Pi<0\}} (n_{k_\Theta+1}-\Delta\Pi-b_1),
\end{align}
i.e., $l_1$ is the vertical distance between the entrance point and the first 
interface, $l_{i}$ is the vertical distance between the $i$-th interface and 
the $(i+1)$-th interface, and $l_{|k_\Theta|+1}$ is the vertical distance between 
the last interface and the exit point.

Recall that $\Theta=(\chi,\Delta \Pi,b_0,b_1,x)$, and let $l_{A,\Theta}$ and $l_{B,\Theta}$ 
correspond  to the minimal vertical distance the copolymer must cross in blocks of type $A$ 
and $B$, respectively, in a column with disorder $\chi$ when going from $(0,b_0)$ 
to $(1,\Delta\Pi+b_1)$, i.e., 
\begin{align}
\label{bl}
\nonumber
l_{A,\Theta}
&=1_{\{\Delta \Pi>0\}} \sum_{j=1}^{|k_\Theta|+1} l_j 1_{\{\chi(n_{j-1})=A\}} 
+1_{\{\Delta \Pi<0\}} \sum_{j=1}^{|k_\Theta|+1} l_j 1_{\{\chi(n_{-j+1})=A\}}, \\
l_{B,\Theta}
&=1_{\{\Delta \Pi>0\}} \sum_{j=1}^{|k_\Theta|+1} l_j 1_{\{\chi(n_{j-1})=B\}} 
+1_{\{\Delta \Pi<0\}} \sum_{j=1}^{|k_\Theta|+1} l_j 1_{\{\chi(n_{-j+1})=B\}}.
\end{align}

\medskip\noindent 
{\bf Vertical distance to be crossed  in columns of class $\nAB$.} 
Depending on $\chi$ and $\Delta\Pi$, we further partition 
$\overline\cV_{\nAB,M}$ into four parts
\be{partcv}
\overline \cV_{\nAB,A,1,M} \cup \overline \cV_{\nAB,A,2,M}
\cup \overline \cV_{\nAB,B,1,M} \cup \overline \cV_{\nAB,B,2,M},
\ee
where $\overline \cV_{\nAB,A,x,M}$ and $\overline \cV_{\nAB,B,x,M}$ contain those 
columns with label $x$ for which all the blocks between the entrance and the exit block 
are of type $A$ and $B$, respectively. Pick $\Theta\in \cV_{\nAB,M}$. In this case, there 
is no $AB$-interface between $b_0$ and $\Delta\Pi+b_1$, which means that $\Delta\Pi
< n_1$ if $\Delta\Pi\geq 0$ and $\Delta\Pi\geq n_0$ if $\Delta\Pi<0$ ($n_0$ and $n_1$ 
being defined in \eqref{efn}).

For $\Theta\in \overline \cV_{\nAB,A,1,M}$ we have $l_{B,\Theta}=0$, whereas $l_{A,\Theta}$ 
is the vertical distance between the entrance point $(0,b_0)$ and the exit point $(1,\Delta\Pi+b_1)$, 
i.e.,
\begin{align}
l_{A,\Theta} &= 1_{\{\Delta \Pi\geq 0\}} (\Delta \Pi-b_0+b_1)
+ 1_{\{\Delta \Pi<0\}} (|\Delta \Pi|+b_0-b_1)+ 1_{\{\Delta \Pi=0\}} |b_1-b_0|,
\end{align} 
and similarly for $\Theta\in \overline \cV_{\nAB,B,1,M}$ we have obviously $l_{A,\Theta}=0$ and 
\begin{align}
l_{B,\Theta} &= 1_{\{\Delta \Pi\geq 0\}} (\Delta \Pi-b_0+b_1)
+ 1_{\{\Delta \Pi<0\}} (|\Delta \Pi|+b_0-b_1)+ 1_{\{\Delta \Pi=0\}} |b_1-b_0|.
\end{align} 

For $\Theta\in \overline \cV_{\nAB,A,2,M}$, in turn, we have  $l_{B,\Theta}=0$ and $l_{A,\Theta}$ 
is the minimal vertical distance a trajectory has to cross in a column with disorder $\chi$, starting 
from $(0,b_0)$, to reach the closest $AB$-interface before exiting at $(1,\Delta\Pi+b_1)$, i.e., 
\begin{align}
l_{A,\Theta} &= 1_{\{\Delta \Pi\geq 0\}} (\Delta \Pi-b_0+b_1)
+ 1_{\{\Delta \Pi<0\}} (|\Delta \Pi|+b_0-b_1)+ 1_{\{\Delta \Pi=0\}} |b_1-b_0|,
\end{align} 
and similarly for $\Theta\in \overline \cV_{\nAB,B,2,M}$ we have  $l_{A,\Theta}=0$ and 
\begin{align}
l_{B,\Theta} &= 1_{\{\Delta \Pi\geq 0\}} (\Delta \Pi-b_0+b_1)
+ 1_{\{\Delta \Pi<0\}} (|\Delta \Pi|+b_0-b_1)+ 1_{\{\Delta \Pi=0\}} |b_1-b_0|.
\end{align}

\medskip\noindent 
{\bf Variational formula for the free energy in a column.} 
We abbreviate $(h)=(h_A,h_B,h_\cI)$ and $(a)=(a_A,a_B,a_\cI)$. Note that the quantity 
$h_x$ indicates the fraction of horizontal steps made by the copolymer in solvent $x$ 
for $x\in \{A,B\}$ and along $AB$-interfaces for $x=\cI$. Similarly, $a_x$ indicates 
the total number of steps made by the copolymer in solvent $x$ for $x\in\{A,B\}$ and 
along $AB$-interfaces for $x=\cI$. For $(l_A,l_B)\in [0,\infty)^2$ and $u\geq l_A+l_B+1$, 
we put 
\begin{align}
\label{defnu}
\nonumber
\cL(l_A,l_B; u) = \big\{(h),(a)\in [0,1]^3 \times [0,\infty)^3\colon\, 
&h_A+h_B+h_\cI=1,\, a_A+a_B+a_\cI=u\\
&a_A\geq h_A+l_A,\,a_B\geq h_B+l_B,\,a_\cI\geq h_\cI\big\}.
\end{align}
For $l_A\in[0,\infty)$ and $u\geq 1+l_A$, we set
\be{defnu2}
\begin{aligned}
\cL_{\nAB,A,2}(l_A;u) 
&=\big\{(h),(a)\in \cL(l_A,0; u)\colon\,h_B=a_B=0\big\},\\
\cL_{\nAB,A,1}(l_A;u)
&=\big\{(h),(a)\in \cL(l_A,0; u)\colon\,h_B=a_B=h_\cI=a_\cI=0\big\},
\end{aligned}
\ee
and, for $l_B\in[0,\infty)$ and $u\geq 1+l_B$, we set
\be{defnu3}
\begin{aligned}
\cL_{\nAB,B,2}(l_B;u)
&=\big\{(h),(a)\in \cL(0,l_B; u)\colon\,h_A=a_A=0\big\},\\
\cL_{\nAB,B,1}(l_B;u)
&=\big\{(h),(a)\in \cL(0,l_B; u)\colon\,h_A=a_A=h_\cI=a_\cI=0\big\}.
\end{aligned}
\ee

The following proposition will be proved in Section~\ref{proofprop1}. The free energy 
per step in a single column is given by the following variational formula.

\begin{proposition}
\label{energ}
For all $\Theta\in \overline \cV_M$ and $u\geq t_\Theta$,
\begin{align}
\label{Bloc of type I}
\psi(\Theta,u;\alpha,\beta)
&= \sup_{(h),(a) \in \cL(\Theta;\,u)}
\frac{a_A\,\tilde{\kappa}\big(\tfrac{a_A}{h_A},
\tfrac{l_A}{h_A}\big)+a_B\,\big[\tilde{\kappa}\big(\tfrac{a_B}{h_B},
\tfrac{l_B}{h_B}\big)+\tfrac{\beta-\alpha}{2}\big]
+a_\cI\,\phi_\cI(\tfrac{a_\cI}{h_\cI})}{u},
\end{align}
with
\be{setof}
\begin{array}{lll}
\cL_{\Theta,u}
&= \cL(l_A,l_B;u) &\text{if } \Theta\in \overline\cV_{\AB,M},\\
\cL_{\Theta,u}
&= \cL_{\nAB,k,x}(l_k;u) &\text{if } \Theta\in \overline\cV_{\nAB,k,x,M},
\,k\in \{A,B\} \text{ and } x\in \{1,2\}.
\end{array}
\ee
\end{proposition}

The importance of Proposition~\ref{energ} lies in the fact that it \emph{expresses the 
free energy in a single column in terms of the path entropy in a single column $\tilde\kappa$ 
and the free energy along a single linear interface $\phi_\cI$}, which were defined in 
Sections~\ref{pathentr}--\ref{interf} and are well understood.


\subsection{Mesoscopic percolation frequencies}
\label{Percofreq}

In Section~\ref{cgp}, we associate with each path $\pi\in \cW_L$ a coarse-grained path 
that records the mesoscopic displacement of $\pi$ in each column. In Section \ref{cgp1}, 
we define a set of probability laws providing the frequencies with which each type of column 
can be crossed by the copolymer. This set will be used in Section~\ref{proofofgene} to state 
and prove the column-based variational formula. Finally, in Section~\ref{cgpl}, we introduce 
a set of probability laws providing the fractions of horizontal steps that the copolymer can 
make when travelling inside each solvent with a given slope or along an $AB$ interface. 
This latter subset appears in the slope-based variational formula.


\subsubsection{Coarse-grained paths}
\label{cgp}

For $x\in \N_0\times\Z$ and $n\in \N$, let $c_{x,n}$ denote the center of the block
$\Lambda_{L_n}(x)$ defined in \eqref{blocks}, i.e., 
\be{defc}
c_{x,n}=x L_n+(\tfrac12,\tfrac12) L_n,
\ee
and abbreviate
\be{defc2}
(\N_0\times\Z)_n=\{c_{x,n}\colon\, x\in \N_0\times\Z\}.
\ee
Let $\widehat{\cW}$ be the set of \emph{coarse-grained paths} on $(\N_0\times\Z)_n$ 
that start at $c_{0,n}$, are self-avoiding and are allowed to jump up, down and 
to the right between neighboring sites of $(\N_0\times\Z)_n$, i.e., the increments 
of $\widehat{\Pi} = (\widehat{\Pi}_j)_{j\in\N_0} \in \widehat{\cW}$ are $(0,L_n), 
(0,-L_n)$ and $(L_n,0)$. (These paths are the coarse-grained counterparts of the 
paths $\pi$ introduced in \eqref{defw}.) For $l\in \N\cup \{\infty\}$, let 
$\widehat{\cW}_l$ be the set of $l$-step coarse-grained paths.

Recall, for $\pi\in \cW_{n}$, the definitions of $N_\pi$ and $(v_j(\pi))_{j\leq N_\pi-1}$ 
given below \eqref{deftau}.  With $\pi$ we associate a coarse-grained path 
$\widehat{\Pi}\in \widehat{\cW}_{N_\pi}$ that describes how $\pi$ moves with 
respect to the blocks. The construction of $\widehat{\Pi}$ is done as follows: 
$\widehat{\Pi}_0=c_{(0,0)}$, $\widehat{\Pi}$ moves vertically until it reaches 
$c_{(0,v_0)}$, moves one step to the right to $c_{(1,v_0)}$, moves vertically 
until it reaches $c_{(1,v_1)}$, moves one step to the right to $c_{(2,v_1)}$, and 
so on. The vertical increment of $\widehat{\Pi}$ in the $j$-th column is 
$\Delta\widehat{\Pi}_j=(v_j-v_{j-1}) L_n$ (see Figs.~\ref{fig6}--\ref{block2}).

\begin{figure}[htbp]
\begin{center}
\includegraphics[width=.42\textwidth]{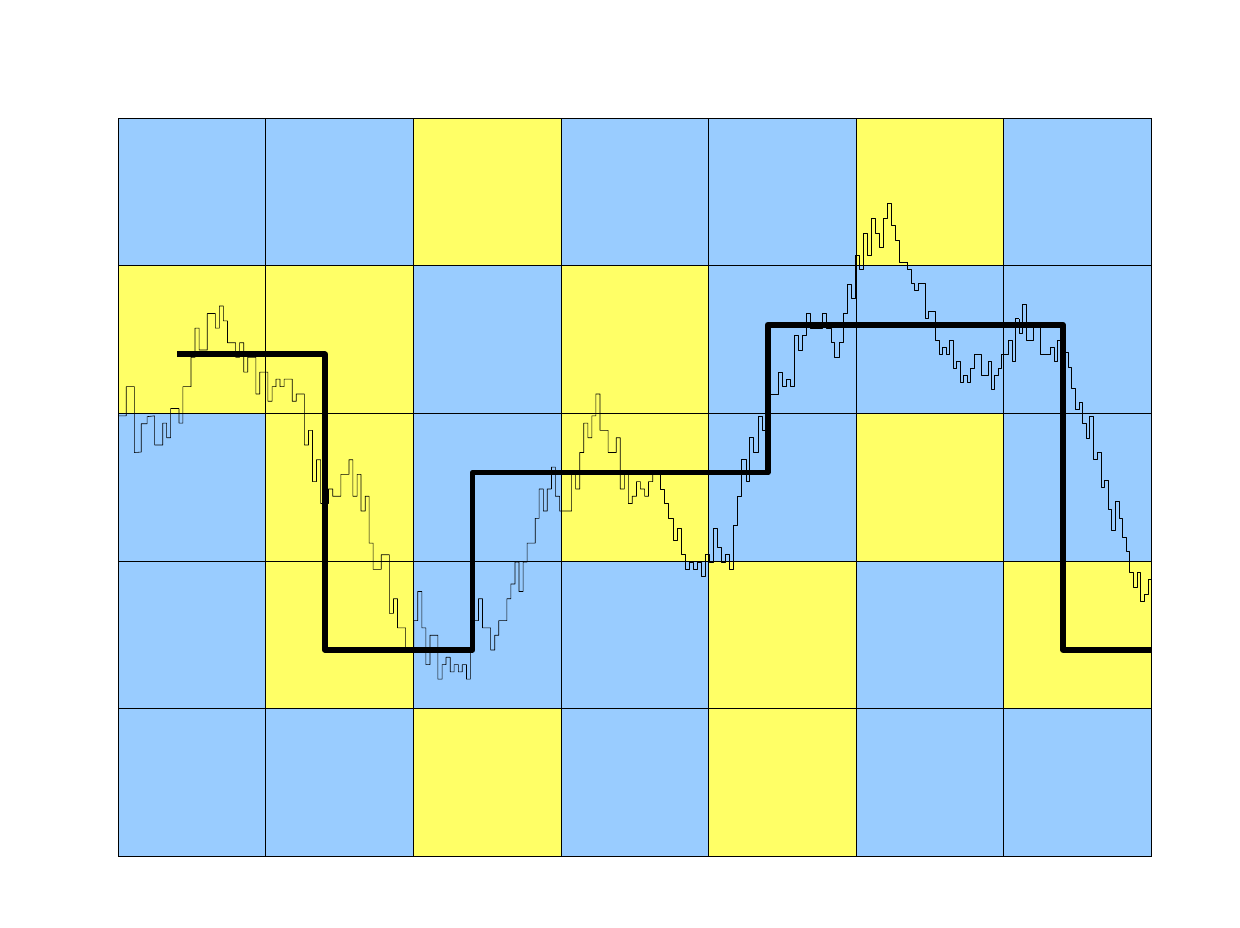}
\end{center}
\vspace{-.5cm}
\caption{Example of a coarse-grained path.}
\label{fig4}
\end{figure}

To characterize a path $\pi$, we will often use the sequence of vertical 
increments of its associated coarse-grained path $\widehat{\Pi}$, modified in 
such a way that it does not depend on $L_n$ anymore. To that end, with every 
$\pi\in \cW_n$ we associate $\Pi=(\Pi_k)_{k=0}^{N_\pi-1}$ such that 
$\Pi_0=0$ and,
\begin{align}
\label{defpi}
\Pi_k=\sum_{j=0}^{k-1} \Delta\Pi_j\quad \text{with} \quad  
\Delta\Pi_j=\frac{1}{L_n} \Delta \widehat{\Pi}_j, \qquad
j = 0,\dots,N_\pi-1.
\end{align}
Pick $M\in \N$ and note that $\pi\in\cW_{n,M}$ if and only if $|\Delta \Pi_j|\leq M$ for 
all $j\in \{0,\dots, N_\pi-1\}$.


\subsubsection{Percolation frequencies along coarse-grained paths.}
\label{cgp1}

Given $M\in \N$, we denote by $\cM_1(\overline\cV_M)$ the set of probability measures 
on $\overline\cV_M$. Pick $\Omega \in \{A,B\}^{\N_0\times\Z}$, $\Pi\in \Z^{\N_0}$ such 
that $\Pi_0=0$ and $|\Delta\Pi_i|\leq M$ for all $i\geq 0$ and $b=(b_j)_{j\in \N_0}\in 
(\mathbb{Q}_{(0,1]})^{\N_0}$. Set $\Theta_{\text{traj}}=(\Xi_j)_{j\in \N_0}$ with
\be{defxii}
\Xi_j=\big(\Delta \Pi_j,b_j,b_{j+1}\big), \qquad j\in \N_0,
\ee 
let
\be{defX}
\cX_{\,\Pi,\Omega}=\big\{x\in \{1,2\}^{\N_0}\colon\,
(\Omega(i,\Pi_i+\cdot),\Xi_i,x_i)\in  \cV_M \,\,\, \forall\,i\in \N_0\big\},
\ee
and for $x\in \cX_{\,\Pi,\Omega}$ set
\be{defth}
\Theta_j=\big(\Omega(j,\Pi_j+\cdot),
\Delta \Pi_j,b_j,b_{j+1},x_j\big), \qquad j\in \N_0.
\ee
With the help of \eqref{defth}, we can define the empirical distribution
\begin{equation}
\label{defrho}
\rho_N(\Omega,\Pi,b,x)(\Theta) = \frac{1}{N} \sum_{j=0}^{N-1} 
1_{\{\Theta_j=\Theta\}},
\quad N\in\N,\,\Theta\in \overline\cV_M.
\end{equation}

In Appendix~\ref{B.2}, we define in \eqref{dist} a distance $d$ that turns 
$\overline\cV_M$ into a Polish space. Thus, the weak convergence in 
$\cM_1(\overline\cV_M)$ is metrizable and $\cM_1(\overline\cV_M)$ is 
Polish as well.

\begin{definition}
\label{RMNdef}
For $\Omega \in \{A,B\}^{\N_0\times\Z}$ and $M\in \N$, let 
\be{RMNdefalt}
\begin{aligned}
\cR^\Omega_{M,N} &= \big\{\rho_N(\Omega,\Pi,b,x)\ \text{with}\  
b=(b_j)_{j\in \N_0} \in (\mathbb{Q}_{(0,1]})^{\N_0},\\
&\qquad \Pi=(\Pi_j)_{j\in\N_0} \in \{0\}\times \Z^{\N} \colon\,
|\Delta\Pi_j|\leq M\,\ \ \forall\,j\in\N_0,\\
&\qquad x=(x_j)_{j\in \N_0} \in \{1,2\}^{N_0}\colon\, \big(\Omega(j,\Pi_j+\cdot),
\Delta \Pi_j,b_j,b_{j+1},x_j\big)\in \cV_M\big\}
\end{aligned} 
\ee
and let $\cR^\Omega_M$ be the set of all accumulation points of those sequences  
$(\rho_N)_{N\in\N}$ satisfying $\rho_N\in  \cR^\Omega_{M,N}$ for all $N\in\N$, 
i.e., 
\be{RMdef}
\cR^\Omega_M = \bigcap_{N'\in\N} \mathrm{closure}\Big( \bigcup_{N \geq N'}\, 
\cR^\Omega_{M,N}\Big),
\ee
both of which are subsets of $\cM_1(\overline\cV_M)$.
\end{definition}

\begin{proposition}
\label{properc}
For every $p \in (0,1)$ and $M\in \N$ there exists a closed set $\cR_{p,M}\subsetneq
\cM_1(\overline\cV_M)$ such that 
\be{mesoperc}
\cR_{M}^\Omega=\cR_{p,M} \text{ for } \P\text{-a.e.}\,\Omega.
\ee
\end{proposition}

\bpr
Note that, for every $\Omega\in \{A,B\}^{\N_0\times\Z}$, the set $\cR_{M}^\Omega$ 
does not change when finitely many variables in $\Omega$ are changed. Therefore 
$\cR_{M}^\Omega$ is measurable with respect to the tail $\sigma$-algebra of 
$\Omega$. Since $\Omega$ is an i.i.d.\ random field, the claim follows from 
Kolmogorov's zero-one law. Because of the constraint on the vertical displacement, 
$\cR_{p,M}$ does not coincide with $\cM_1(\overline\cV_M)$.
\epr

Each probability measure $\rho\in \cR_{p,M}$ is associated with a strategy of displacement 
of the copolymer on the mesoscopic scale. As mentioned above, the growth rate of the 
square blocks in \eqref{speed} ensures that no entropy is carried by the mesoscopic 
displacement, and this  justifies the optimization over $\cR_{p,M}$ in the column-based 
variational formula.


\subsubsection{Fractions of horizontal steps per slope}
\label{cgpl}

In this section, we introduce $\bar{\cR}_{p,M}$ as the counterpart of $\cR_{p,M}$ for the 
slope-based variational formula. To that aim, we define 
\begin{align}
\label{setE}
\cE=\big\{(h_{A,\Theta},h_{B,\Theta},
h_{\cI,\Theta})_{\Theta\in \overline{\cV}_M}\in ([0,1]^3)^{\overline\cV_M}\colon
&\  h_{A,\Theta}+h_{B,\Theta}+h_{\cI,\Theta}=1\,\,\forall\,\Theta,\\
\nonumber 
&\ \Theta \mapsto h_{k,\Theta}\ \text{Borel}\,\,\forall\,k\in \{A,B,\cI\},\\
\nonumber &\ h_{k,\Theta}>0\,\,\text{if}\,\,l_{k,\Theta}>0\,\,\forall\,k\in \{A,B\},\\
\nonumber &h_{k,\Theta}=1\ \, \text{if} \ \, \Theta\in \cV_{\text{nint},k,1,M},\\
\nonumber &h_{\cI,\Theta}+h_{k,\Theta}=1\ \, \text{if} \ \, 
\Theta\in \cV_{\text{nint},k,2,M}\big\}.
\end{align}
With each $\rho\in \cR_{p,M}$ and $h\in \cE$ associate $G_{\rho,h}\in\cM_1\big(\R_+
\cup\R_+\cup\{\cI\}\big)$, defined by 
\begin{align}
\label{Gdef}
G_{\rho,h,A}(dl)
&=\int_{\bar\cV_M} h_{A,\Theta}\,1\Big\{\tfrac{l_{A,\Theta}}{h_{A,\Theta}}\in dl\Big\}\,
\rho(d\Theta),\\
\nonumber 
G_{\rho,h,B}(dl)
&=\int_{\bar\cV_M} h_{B,\Theta}\,1\Big\{\tfrac{l_{B,\Theta}}{h_{B,\Theta}}\in dl\Big\}\,
\rho(d\Theta),\\
\nonumber 
G_{\rho,h,\cI}
&=\int_{\bar\cV_M} h_{\cI,\Theta}\,\rho(d\Theta),
\end{align}
where $l_{k,\Theta}/h_{k,\Theta}=0$ by convention if $h_{k,\Theta}=0$ for $\Theta\in 
\overline \cV_M$ and $k\in \{A,B\}$. The set $\bar{\cR}_{p,M}$ in \eqref{genevar} is 
defined as
\begin{align}
\label{setbar}
\bar{\cR}_{p,M}=\text{Closure}\  \Big\{\bar\rho\in \cM_1\big(\R_+\cup \R_+\cup\{\cI\}\big) 
\colon\, \exists\,\rho&\in\cR_{p,M}, h \in \cE \colon\,\bar\rho=G_{\rho,h}\Big\},
\end{align}
and as the $M$-restriction is relaxed the set $\bar{\cR}_p$ in \eqref{genevar} is defined as 
 \begin{align}
\label{setbarp}
\bar{\cR}_{p}=\cup_{M\geq 1} \bar{\cR}_{p,M}.
\end{align}
For 
$\bar\rho\in\bar{\cR}_{p}$, let $\bar\rho_A$, $\bar\rho_B$ and $\bar\rho_\cI$ denote the 
restriction of $\bar\rho$ to $\R_+$, $\R_+$ and $\{\cI\}$, respectively, as in \eqref{varformod}.
The measures  $\bar\rho_A(dl)$, $\bar{\rho}_B(dl)$ represent the fraction of horizontal steps 
made by the copolymer when it moves at slope $l$ in solvent $A$, respectively, $B$. The 
number $\bar\rho_\cI$ represents the fraction of horizontal steps made by the copolymer
when it moves along the $AB$-interface.


\subsection{Positivity of the free energy}
\label{Positivity of the free energy}

It is easy to prove that for all $p\in (0,1)$, $M\in \N$ and $(\alpha,\beta)\in \CONE$ the 
two variational formulas (that is the slope-based variational formula stated in \eqref{genevar} but with 
the supremum taken over  $\bar \cR_{p,M}$ instead of $\bar\cR_{p}$
and the column-based variational formula stated in \eqref{genevarold} with the supremum taken over  $\cR_{p,M}$ instead of $\cR_{p}$) are strictly positive, i.e., 
\be{posi}
f(\alpha,\beta;M,p)>0.
\ee


To prove that the variational formula in \eqref{genevar} is strictly positive, we define 
$\bar\rho_{\text{hor}}\in \cM_1\big(\R_+\cup \R_+\cup\{\cI\}\big)$ as
\be{defrhoh}
\bar\rho_{\text{hor}} = p^2 \delta_{A,0}(dl)+ (1-p)^2 \delta_{B,0}(dl)
+ 2 p(1-p) \delta_\cI. 
\ee
When moving along the $x$-axis, the pairs of blocks appearing above and below the $x$-axis 
have density $p^2$ for type $AA$, density $(1-p)^2$ for type $BB$, and density $2p(1-p)$ 
for types $AB$ and $BA$. Consequently, $\bar\rho_{\text{hor}}$ belongs to $\bar \cR_{p}$ 
and \eqref{genevar} implies that, for any choice of $v_A,v_B\geq 1$, the variational formula 
in \eqref{genevar} is at least
\be{bouninf}
\frac{[p^2+2 p (1-p)]\,v_{A}\,\tilde{\kappa}(v_{A},0)
+(1-p)^2\,v_{B}\,[\tilde{\kappa}(v_{B},0)+\tfrac{\beta-\alpha}{2}]}{[p^2+2 p (1-p)]\,v_{A}
+(1-p)^2\,v_{B}}.
\ee
Thus, it suffices to pick $v_B=1$, to recall that $\lim_{u\to\infty} u\tilde\kappa(u,0)
=\infty$ (Lemma~\ref{l:lemconv2}(iv)), and to choose $v_A$ large enough so that
\eqref{bouninf} becomes strictly positive.

To prove that the variational formula in \eqref{genevarold} is strictly positive, we can 
argue similarly, taking both sequences $(\Pi_i)_{i\in \N_0}$ and $(b_i)_{i\in \N_0}$
constant and equal to $0$.


\section{Proof of Propositions~\ref{convfree1}--\ref{energ}}
\label{proofprop1}

In this section we prove Propositions~\ref{convfree1} and \ref{energ}, 
which were stated in Sections~\ref{frenp1} and \ref{newsec} and contain the 
precise definition of the key ingredients of the variational formula in 
Theorem~\ref{varformula}. In Section~\ref{proofofgene} we will use these 
propositions to prove Theorem~\ref{varformula}. 

In Section \ref{colcross} we associate with each trajectory $\pi$ in a column 
a sequence recording the indices of the $AB$-interfaces successively visited 
by $\pi$. The latter allows us to state a key proposition, Proposition~\ref{convunifr} 
below, from which Propositions~\ref{convfree1} and \ref{energ}  are 
straightforward consequences. In Section~\ref{convunifra} we give an outline of 
the proof of Proposition~\ref{convunifr}, in Sections~\ref{s1}--\ref{s3} we 
provide the details.


\subsection{Column crossing characteristic}
\label{colcross}


\subsubsection{The order of the visits to the interfaces}

Pick $(M,m)\in \EIGH$. To prove  Propositions~\ref{convfree1} and \ref{energ}, instead of considering 
$(\Theta,u)\in\cV^{*,\,m}_M$, we will restrict to $(\Theta,u)\in\cV^{*,\,m}_{\text{int},M}$. 
Our proof can be easily extended to $(\Theta,u)\in \cV^{*,\,m}_{\text{nint},M}$.

Pick $(\Theta,u)\in \cV^{*,\,m}_{\AB,M}$, recall \eqref{efn} and set $\cJ_{\Theta,u}
=\{\cN_{\Theta,u}^{\downarrow},\dots,\cN_{\Theta,u}^{\uparrow}\},$ with
\begin{align}
\label{defntheta}
\cN^{\uparrow}_{\Theta,u}
= &\max\{i\geq 1\colon n_i\leq u\}\quad \text{and}
\quad \cN^{\uparrow}_{\Theta,u}=0\quad  \text{if}\quad  n_1>u.  \\
\nonumber\cN^{\downarrow}_{\Theta,u}
=&\min \{i\leq 0\colon |n_i|\leq u\} \quad \text{and}
\quad \cN^{\downarrow}_{\Theta,u}=1\quad  \text{if}\quad  |n_0|>u.
\end{align}
Next pick $L\in \N$ so that $(\Theta,u)\in \cV^{*}_{\AB,L,M}$ and recall that for 
$j\in \cJ_{\Theta,u}$ the $j$-th interface of the $\Theta$-column is $\cI_j=\{0,\dots,L\}
\times \{n_j L\}$. Note also that $\pi \in \cW_{\Theta,u,L}$ makes $uL$ steps inside 
the column and therefore can not reach the $AB$-interfaces labelled outside 
$\{\cN^{\downarrow}_{\Theta,u},\dots,\cN^{\uparrow}_{\Theta,u}\}$.

First, we associate with each trajectory $\pi\in \cW_{\Theta,u,L}$ the sequence 
$J(\pi)$ that records the indices of the interfaces that are successively visited by 
$\pi$. Next, we pick $\pi\in \cW_{\Theta,u,L}$, and define $\tau_1, J_1$ as
\begin{equation}
\tau_1=\inf\{i\in\N\colon\, \exists j\in \cJ_{\Theta,u} 
\colon\,\pi_i\in \cI_j \}, \qquad \pi_{\tau_1}\in \cI_{J_1},
\end{equation} 
so that $J_1=0$ (respectively, $J_1=1$) if the first interface reached by $\pi$ 
is $\cI_0$ (respectively, $\cI_1$). For $i\in\N\setminus\{1\}$, we define $\tau_i,J_i$ 
as 
\begin{equation}
\tau_i=\inf\big\{t>\tau_{i-1}\colon\,\exists 
j\in \cJ_{\Theta,u}\setminus\{ J_{i-1}\}, \pi_i\in \cI_j\big\}, 
\qquad \pi_{\tau_i}\in \cI_{J_i},
\end{equation} 
so that the increments of $J(\pi)$ are restricted to $-1$ or $1$. The length of 
$J(\pi)$ is denoted by $m(\pi)$ and corresponds to the number of jumps made by 
$\pi$ between neighboring interfaces before time $uL$, i.e., 
$J(\pi)=(J_i)_{i=1}^{m(\pi)}$ with 
\be{add6}
m(\pi)=\max\{i\in\N\colon\,\tau_i\leq uL\}.
\ee
Note that $(\Theta,u)\in \cV_{\AB,M}^{*,\,m}$ necessarily implies $k_{\Theta}\leq 
m(\pi)\leq u\leq m$. Set
\be{add7}
\cS_r=\{j=(j_i)_{i=1}^r\in \Z^\N\colon\, j_1\in \{0,1\},\,  
j_{i+1}-j_i\in \{-1,1\}\,\,\forall\,1\leq i\leq r-1\}, \qquad r\in \N, 
\ee
and, for $\Theta\in \cV$, $r\in \{1,\dots, m\}$ and $j\in \cS_r$, define
\begin{align}
\label{defl}
\nonumber l_1&=1_{\{j_1=1\}} (n_{1}-b_0)+1_{\{j_1= 0\}} (b_0-n_{0}),\\
\nonumber l_i&=|n_{j_i}-n_{j_{i-1}}| \text{ for } i\in \{2,\dots,r\},\\
l_{r+1}&=1_{\{j_r=k_\Theta+1\}} (n_{k_\Theta+1}-\Delta \Pi-b_1)
+1_{\{j_r= k_\Theta\}} (\Delta \Pi +b_1-n_{k_\Theta}), 
\end{align}
so that $(l_i)_{i\in \{1,\dots,r+1\}}$ depends on $\Theta$ and $j$. Set
\begin{align}
\label{setofA}
\cA_{\Theta,j}
&=\{i\in \{1,\dots,r+1\}\colon\,\text{$A$ between}\,\cI_{j_{i-1}}\,
\text{and}\,\cI_{j_{i}}\},\\
\nonumber \cB_{\Theta,j}
&=\{i\in \{1,\dots,r+1\}\colon\,\text{$B$ between}\,\cI_{j_{i-1}}\, 
\text{and} \ \cI_{j_{i}}\},
\end{align}
and set $l_{\Theta,j}=(l_{A,\Theta,j},l_{B,\Theta,j})$ with
\begin{align}
\label{bl1}
l_{A,\Theta,j}
&={\textstyle \sum_{i\in \cA_{\Theta,j}}} l_i,\,\,l_{B,\Theta,j}
={\textstyle \sum_{i\in \cB_{\Theta,j}}} l_i.
\end{align}
For $L\in \N$ and  $(\Theta,u)\in \cV^{*,\,m}_{\AB,L,M}$, we denote by $\cS_{\Theta,u,L}$ 
the set $\{J(\pi), \pi\in \cW_{\Theta,u,L}\}$. It is not difficult to see that a sequence 
$j\in \cS_r$ belongs to $\cS_{\Theta,u,L}$ if and only if it satisfies the two following 
conditions. First, $j_r \in \{k_\Theta,k_\Theta +1\}$, since $j_r$ is the index of the 
interface last visited before the $\Theta$-column is exited. Second, $u\geq 1+l_{A,\Theta,j}
+l_{B,\Theta,j}$ because the number of steps taken by a trajectory $\pi\in \cW_{\Theta,u,L}$
satisfying $J(\pi)=j$ must be large enough to ensures that all interfaces $\cI_{j_s}$, 
$s\in \{1,\dots,r\}$, can be visited by $\pi$ before time $uL$. Consequently, 
$\cS_{\Theta,u,L}$ does not depend on $L$ and can be written as $\cS_{\Theta,u}
= \cup_{r=1}^{m}\cS_{\Theta,u,r}$, where 
\be{redefS}
\cS_{\Theta,u,r}=\{j\in \cS_r\colon j_r\in \{k_\Theta,k_\Theta+1\}, 
u \geq 1+l_{A,\Theta,j}+l_{B,\Theta,j}\}.
\ee
Thus, we partition $\cW_{\Theta,u,L}$ according to the value taken by $J(\pi)$, i.e.,
\be{zto}
\cW_{\Theta,u,L}=\bigcup_{r=1}^{m} \ \bigcup_{j\in \cS_{\Theta,u,r}} \ \cW_{\Theta,u,L,j}, 
\ee
where $\cW_{\Theta,u,L,j}$ contains those trajectories $\pi\in \cW_{\Theta,u,L}$  
for which $J(\pi)=j$. 

Next, for $j \in \cS_{\Theta,u}$, we define (recall \eqref{Hamiltonian1})
\be{partfunc3}
\psi^{\omega}_L(\Theta,u,j)
= \frac{1}{u L} \log Z^{\omega}_L(\Theta,u,j), \qquad 
\psi_L(\Theta,u,j)=\E\big[\psi^{\omega}_L(\Theta,u,j)\big],
\ee
with
\be{partfunc3alt}
Z^{\omega}_L(\Theta,u,j)
=\sum_{\pi \in \cW_{\Theta,u,L,j}} e^{H_{uL,L}^{\omega,\chi}(\pi)}.
\ee
For each $L\in \N$ satisfying $(\Theta,u)\in \cV^{*,\,m}_{\AB,L,M}$ and each $j\in 
\cS_{\Theta,u}$, the quantity $l_{A,\Theta,j} L$ (respectively, $l_{B,\Theta,j} L$) 
corresponds to the minimal vertical distance a trajectory $\pi\in \cW_{\Theta,u,L,j}$ 
has to cross in solvent $A$ (respectively, $B$).


\subsubsection{Key proposition}

For simplicity we give the proof for the case $(\Theta,u) \in \cV_{\AB,M}^{*,\,m}$. 
The extension to $(\Theta,u) \in \cV_{\nAB,M}^{*,\,m}$ is straightforward.

Recalling \eqref{Bloc of type I} and \eqref{bl1}, we define the free energy associated 
with $\Theta,u,j$ as 
\begin{align}
\label{Bloc of type I1}
\psi(\Theta,u,j)
&=\psi_{\AB}(u,l_{\Theta,j})\\ \nonumber 
&=\sup_{(h),(u) \in \cL(l_{\Theta,j};\, u)}
\frac{u_A\, \tilde{\kappa}\big(\tfrac{u_A}{h_A},
\tfrac{l_{A,\Theta,j}}{h_A}\big)+u_B\,
\big[\tilde{\kappa}\big(\tfrac{u_B}{h_B},
\tfrac{l_{B,\Theta,j}}{h_B}\big)
+\tfrac{\beta-\alpha}{2}\big]+u_I\, \phi(\tfrac{u^I}{h^I})}{u}.
\end{align}
Proposition~\ref{convunifr} below states that $\lim_{L\to\infty} \psi_L(\Theta,u,j) 
=\psi(\Theta,u,j)$ uniformly in $(\Theta,u) \in \cV_{\AB,M}^{*,\,m}$ and $j\in \cS_{\Theta,u}$.  

\begin{proposition}
\label{convunifr}
For every $M,m\in \N$ such that $m\geq M+2$ and every $\gep>0$ there exists an 
$L_\gep\in \N$ such that 
\be{une int}
\big|\psi_L(\Theta,u,j)-\psi(\Theta,u,j)\big|
\leq \gep \quad \forall\,(\Theta,u)\in \cV^{*,\,m}_{\AB,L,M},\  \,
j\in \cS_{\Theta,u},\ \, L\geq L_\gep.
\ee
\end{proposition}

\medskip\noindent 
{\bf Proof of Propositions~\ref{convfree1} and \ref{energ}  subject 
to Proposition~\ref{convunifr}}. Pick $\gep>0$, $L\in \N$ and $(\Theta,u)\in
\cV^{*,\,m}_{\AB,L,M}$. Recall \eqref{bl} and note that $l_A(\Theta)L$ and $l_B(\Theta)L$ 
are the minimal vertical distances the trajectories of $\cW_{\Theta,u,L}$ have to 
cross in blocks of type $A$, respectively, $B$. For simplicity, in what follows 
the $\Theta$-dependence of $l_A$ and $l_B$ will be suppressed. In other words, 
$l_A$ and $l_B$ are the two coordinates of $l_{\Theta,f}$ (recall \eqref{bl1}) 
with $f=(1,2,\dots,|k_{\Theta}|)$ when $\Delta\Pi\geq 0$ and $f=(0,-1,\dots,
-|k_{\Theta}|+1)$ when $\Delta\Pi< 0$, so \eqref{Bloc of type I} and 
\eqref{Bloc of type I1} imply
\be{pisdefalt}
\psi_{\AB}(u,l_A,l_B)= \psi(\Theta,u,f).
\ee
Hence Propositions~\ref{convfree1} and \ref{energ} will be proven once we show 
that $\lim_{L\to\infty} \psi_L(\Theta,u)=\psi(\Theta,u,f)$ uniformly in $(\Theta,u)
\in\cV^{*,\,m}_{\AB,L,M}$. Moreover, a look at \eqref{Bloc of type I1}, \eqref{pisdefalt}
and \eqref{Bloc of type I} allows us to assert that for every $j\in \cS_{\Theta,u}$ 
we have $\psi(\Theta,u,j)\leq \psi(\Theta,u,f)$. The latter is a consequence of 
the fact that $l\mapsto \tilde{\kappa}(u,l)$ decreases on $[0,u-1]$ (see 
Lemma~\ref{l:lemconv2}(ii) in Appendix~\ref{Path entropies}) and that 
\begin{align}
\label{add8}
\nonumber 
l_{A}&=l_{A,\Theta,f}=\min\{l_{A,\Theta,j} \colon\, j\in \cS_{\Theta,u}\},\\
l_{B}&=l_{B,\Theta,f}=\min\{l_{B,\Theta,j} \colon\, j\in \cS_{\Theta,u}\}.
\end{align}

By applying Proposition~\ref{convunifr} we have, for $L\geq L_\gep$, 
\begin{align}
\label{eq:restore}
\nonumber 
\psi_L(\Theta,u,j)
&\leq \psi(\Theta,u,f)+\gep \qquad \forall\,(\Theta,u)\in \cV^{*,\,m}_{\AB,L,M},\ 
\forall\,j\in \cS_{\Theta,u},\\
\psi_L(\Theta,u,f)
&\geq \psi(\Theta,u,f)-\gep \qquad \forall\,(\Theta,u)\in \cV^{*,\,m}_{\AB,L,M}.
\end{align} 
The second inequality in \eqref{eq:restore} allows us to write, for $L\geq L_\gep$, 
\be{inegleft}
\psi(\Theta,u,f)-\gep\leq \psi_L(\Theta,u,f)\leq \psi_L(\Theta,u) \qquad \forall\,
(\Theta,u)\in \cV^{*,\,m}_{\AB,L,M}.
\ee
To obtain the upper bound we introduce
\be{cAdef}
\cA_{L,\gep}=\Big\{\omega\colon\, |\psi^\omega_L(\Theta,u,j)-\psi_L(\Theta,u,j)|
\leq \gep \ \quad \forall\, 
(\Theta,u)\in \cV^{*,\,m}_{\AB,L,M},\, \forall\,j\in \cS_{\Theta,u}\Big\},
\ee
so that 
\begin{align}
\label{refre}
\psi_L(\Theta,u)
&\leq \E\big[ 1_{\cA^c_{L,\gep}}\, \psi^\omega_L(\Theta,u)\big]
+ \E\big[ 1_{\cA_{L,\gep}}\, \psi^\omega_L(\Theta,u)\big]\\
\nonumber 
&\leq C_{\text{uf}}(\alpha)\,\P(\cA_{L,\gep}^c)+\tfrac{1}{uL}
\E\Big[ 1_{\cA_{L,\gep}}\,\log {\textstyle \sum_{j\in  \cS_{\Theta,u}}}\, 
e^{uL (\psi_L(\Theta,u,j)+\gep)}\Big],
\end{align}
where we use \eqref{boundel} to bound the first term in the right-hand side, 
and the definition of $\cA_{L,\gep}$ to bound the second term. Next, with the help 
of the first inequality in \eqref{eq:restore} we can rewrite \eqref{refre} for 
$L\geq L_\gep$ and $(\Theta,u)\in \cV^{*,\,m}_{\AB,L,M}$ in the form
\begin{align}
\label{restora}
\psi_L(\Theta,u)&\leq  C_{\text{uf}}(\alpha)\, \P(\cA_{L,\gep}^c) 
+\tfrac{1}{uL}\log |\cup_{r=1}^{m} \cS_{r}|+ \psi(\Theta,u,f)+2\gep.
\end{align} 
At this stage we want to prove that $\lim_{L\to\infty} \P(\cA^c_{L,\gep})=0$. To 
that end, we use the concentration of measure property in \eqref{concmesut} in
Appendix~\ref{Ann2} with $l=uL$, $\Gamma=\cW_{\Theta,u,L,j}$, $\eta=\gep uL$, 
$\xi_i=-\alpha 1\{\omega_i=A\} +\beta 1\{\omega_i=B\}$ for all $i\in\N$ and 
$T(x,y)=1\{\chi^{L_n}_{(x,y)}=B\}$. We then obtain that there exist $C_{1},C_{2}>0$ 
such that, for all $L\in \N$,  $(\Theta,u)\in\cV^{*,\,m}_{\AB,L,M}$ and 
$j\in\cS_{\Theta,u}$, 
\be{rec222}
\P\big( |\psi^\omega_L(\Theta,u,j)-\psi_L(\Theta,u,j)|> \gep\big)
\leq C_{1} \, e^{-C_{2}\, \gep^2\, uL}.
\ee
The latter inequality, combined with the fact that $|\cV^{*,\,m}_{\AB,L,M}|$ grows 
polynomialy in $L$, allows us to assert that $\lim_{L\to\infty} \P(\cA^c_{L,\gep})=0$. 
Next, we note that $|\cup_{r=1}^{m} \cS_{r}|<\infty$, so that for $L_\gep$ large enough 
we obtain from \eqref{restora} that, for $L\geq L_\gep$,
\be{restora1}
\psi_L(\Theta,u) \leq \psi(\Theta,u,f)+3\gep \qquad \forall\,
(\Theta,u)\in \cV^{*,\,m}_{\AB,L,M}.
\ee 
Now \eqref{inegleft} and \eqref{restora1} are sufficient to complete the proof of
Propositions~\ref{convfree1}--\ref{energ} in the case $(\Theta,u) \in \cV_{\AB,M}^{*,\,m}$. 
As mentioned earlier, the proof for the case $(\Theta,u) \in \cV_{\nAB,M}^{*,\,m}$ is 
entirely similar.
\hspace*{\fill}$\square$


\subsection{Structure of the proof of Proposition \ref{convunifr}}
\label{convunifra}

{\bf Intermediate column free energies.}
Let 
\be{add9}
G_M^{\,m}=\big\{(L,\Theta,u,j)\colon\,
(\Theta,u)\in \cV^{*,\,m}_{\AB,L,M},\,  j\in \cS_{\Theta,u}\big\},
\ee 
and define the following order relation. 

\begin{definition}
For $g,\widetilde{g}\colon\,G_M^{\,m}\mapsto \R$, write $g\prec \widetilde{g}$ when  
for every $\gep>0$ there exists an $L_\gep\in \N$ such that
\be{reparco}
g(L,\Theta,u,j)\leq \widetilde{g}(L,\Theta,u,j)
+\gep \qquad \forall\,(L,\Theta,u,j)\in G_M^{\,m}\colon\, L\geq L_\gep.
\ee
\end{definition}

Recall \eqref{partfunc3} and  \eqref{Bloc of type I1}, set
\be{psidefvar}
\psi_1(L,\Theta,u,j) = \psi_L(\Theta,u,j), \qquad
\psi_4(L,\Theta,u,j)=\psi(\Theta,u,j),
\ee 
and note that the proof of Proposition~\ref{convunifr} will be complete 
once we show that $\psi_1\prec \psi_4$ and $\psi_4\prec \psi_1$. In what 
follows, we will focus on $\psi_1\prec\psi_4$. Each step of the proof can 
be adapted to obtain $\psi_4\prec\psi_1$ without additional difficulty. 

In the proof we need to define two intermediate free energies $\psi_2$ 
and $\psi_3$, in addition to $\psi_1$ and $\psi_4$ above. Our proof is divided 
into 3 steps, organized in Sections \ref{s1}--\ref{s3}, and consists of showing 
that $\psi_1\prec \psi_2\prec \psi_3\prec \psi_4$. 

\medskip\noindent 
{\bf Additional notation.} 
Before stating Step 1, we need some further notation. First, we partition 
$\cW_{\Theta,u,L,j}$ according to the total number of steps and the number 
of horizontal steps made by a trajectory along and in between $AB$-interfaces. 
To that end, we assume that $j\in \cS_{\Theta,u,r}$ with $r\in \{1,\dots,m\}$, we 
recall \eqref{defl} and we let 
\begin{align}
\label{setindice}
\nonumber 
\cD_{\Theta,L,j}
&=\big\{(d_i,t_i)_{i=1}^{r+1}\colon\, d_i\in\N\,\,\text{and}\,\,
t_i\in d_i+l_i L+2\N_0\,\,\forall\,1\leq i\leq r+1\big\},\\
\cD_{r}^\cI&=\big\{(d_i^\cI,t_i^\cI)_{i=1}^r\colon\, 
d_i^\cI \in\N\,\,\text{and}\,\,t_i^\cI\in d_i^\cI+2\N_0\,\,\forall\, 1\leq i\leq r\big\},
\end{align}
where $d_i,t_i$ denote the number of horizontal steps and the total number of 
steps made by the trajectory between the $(i-1)$-th and $i$-th interfaces, and
$d_i^\cI,t_i^\cI$ denote the number of horizontal steps and the total number of 
steps made by the trajectory along the $i$-th interface. For $(d,t)\in 
\cD_{\Theta,L,j}$, $(d^\cI,t^\cI)\in \cD_{r}^\cI$ and $1\leq i\leq r$, we set 
$T_0=0$ and
\begin{align}
\label{defU}
\nonumber 
V_{i}
&=\sum_{j=1}^i t_j+\sum_{j=1}^{i-1} t_j^{\cI}, \qquad i = 1,\dots,r,\\
T_{i}
&=\sum_{j=1}^{i} t_j+\sum_{j=1}^{i} t_j^{\cI}, \qquad i = 1,\dots,r,
\end{align}
so that $V_i$, respectively, $T_i$ indicates the number of steps made by the trajectory 
when reaching, respectively, leaving the $i$-th interface. 

Next, we let $\theta\colon\,\R^\N\mapsto\R^\N$ be the left-shift acting on infinite 
sequences of real numbers and, for $u\in \N$ and $\omega\in \{A,B\}^\N$, we put
\be{defB}
H_u^\omega(B)=\sum_{i=1}^u \big[\beta\,
1_{\{\omega_i=B\}}-\alpha\,1_{\{\omega_i=A\}}\big].
\ee
Finally, we recall that 
\begin{equation}
\label{psi1}
\psi_1(L,\Theta,u,j)=\tfrac{1}{ uL}\,\mathbb{E}[\log Z_1^{\omega}(L,\Theta,u,j)],
\end{equation}
where the partition function defined in \eqref{partfunc2} has been renamed $Z_1$ and
can be written in the form
\be{defZ}
Z^{\omega}_1(L,\Theta,u,j)=\sum_{(d,t)\in \cD_{\Theta,L,j}}\, 
\sum_{(d^\cI,t^\cI)\in \cD_{r}^\cI} A_1\,B_1\,C_1,
\ee
where (recall \eqref{setofA} and \eqref{feinf})
\begin{align}
\label{defA}
A_1 &=\prod_{i\in \cA_{\Theta,j}}\, 
e^{t_{i}\, \tilde{\kappa}_{d_i}\big(\tfrac{t_{i}}{d_i},\,
\tfrac{l_{i} L}{d_i}\big)}\,  
\prod_{i\in \cB_{\Theta,j}}\,e^{t_{i}\,
\tilde{\kappa}_{d_i}\big(\tfrac{t_{i}}{d_i},\,
\tfrac{ l_{i} L}{d_i}\big)}\,e^{H^{\theta^{T_{i-1}}(w)}_{t_{i}}(B)},\\
\nonumber 
B _1&=\prod_{i=1}^{r}\,e^{t_{i}^\cI\,\phi^{\theta^{V_{i}}(w)}_{d^{\cI}_{i}}
\big(\tfrac{t^{\cI}_{i}} {d_{i}^\cI}\big)},\\
\nonumber 
C _1&= \ind_{\big\{\sum_{i=1}^{r+1} d_i+\sum_{i=1}^{r} d_i^\cI=L\big\}}\, 
\ind_{\big\{\sum_{i=1}^{r+1} t_i+\sum_{i=1}^{r} t_i^\cI=u L\big\}}.
\end{align}
It is important to note that a simplification has been made in the term $A_1$ in 
\eqref{defA}. Indeed, this term is not $\tilde{\kappa}_{d_i}(\cdot,\cdot)$ 
defined in \eqref{ttrajblock}, since the latter does not take into account 
the vertical restrictions on the path when it moves from one interface to 
the next. However, the fact that two neighboring $AB$-interfaces are necessarily 
separated by a distance at least $L$ allows us to apply Lemma~\ref{conunifalt1} 
in Appendix~\ref{A.3}, which ensures that these vertical restrictions can be
removed at the cost of a negligible error.

To show that $\psi_1\prec\psi_2\prec\psi_3\prec\psi_4$, we fix $(M,m)\in \EIGH$ and $\gep>0$, 
and we show that there exists an $L_\gep\in \N$s such that $\psi_k(L,\Theta,u,j)
\leq \psi_{k+1}(L,\Theta,u,j)+\gep$ for all $(L,\Theta,u,j)\in G_M^{\,m}$ and 
$L\geq L_\gep$. The latter will complete the proof of Proposition~\ref{convunifr}.


\subsubsection{Step 1}
\label{s1}

In this step, we remove the $\omega$-dependence from $Z_1^{\,\omega}(L,\Theta,u,j)$. 
To that aim,  we put
\be{psiLTdef}
\psi_2(L,\Theta,u,j)=\frac{1}{ uL} \log Z_2(L,\Theta,u,j)
\ee 
with
\be{Ztilde}
Z_2(L,\Theta,u,j) = \sum_{(d,t)\in \cD_{\Theta,L,j}}\, 
\sum_{(d^\cI,t^\cI)\in \cD_{r}^\cI} A_2 \  B_2 \  C_2,
\ee
where
\begin{align}
\label{defA1}
A_2 &= \prod_{i\in \cA_{\Theta,j}}\,e^{t_{i}\, \tilde{\kappa}_{d_i}
\big(\tfrac{t_{i}}{d_i},\,\tfrac{ l_{i} L}{d_i}\big)}
\prod_{i\in \cB_{\Theta,j}}\,e^{t_{i}\,\tilde{\kappa}_{d_i}
\big(\tfrac{t_{i}}{d_i},
\,\tfrac{ l_{i} L}{d_i}\big)}\,e^{\tfrac{\beta-\alpha}{2}\, t_{i}},\\
\nonumber 
B_2 &= \prod_{i=1}^{r}\,e^{t_{i}^\cI\,\phi_{d^\cI_{i}}
\Big(\tfrac{t^\cI_{i}}{d^\cI_{i}}\Big)},\\ 
\nonumber 
C_2 &= C_1.
\end{align}
Next, for $n\in \N$ we define 
\begin{align}
\label{subset1}
\nonumber 
\cA_{\gep,n} &= \Big\{\exists\,0\leq t,s\leq n\colon\, t\geq \gep n, 
\,\big|H_t^{\theta^s(\omega)}(B)-\tfrac{\beta-\alpha}{2} t\big|>\gep t\Big\},\\
\cB_{\gep,n} &= \Big\{\exists\,0\leq t,d,s\leq n\colon\,t\in d+2\N_0,\, 
t\geq \gep n,\,\big|\phi^{\theta^s(w)}_d(\tfrac{t}{d})-\phi_d 
(\tfrac{t}{d})\big|>\gep \Big\}.
\end{align}
By applying Cram\'er's theorem for i.i.d.\ random variables (see e.g.\ den 
Hollander~\cite{dH00}, Chapter 1), we obtain that there exist $C_{1}(\gep),C_{2}(\gep)>0$ 
such that
\be{rec}
\P\big(\big|H_t^{\theta^s(w)}(B)-\tfrac{\beta-\alpha}{2} t \big|>\gep t\big)
\leq C_{1}(\gep)\, e^{-C_2(\gep) t}, \qquad t,s\in \N.
\ee
By using the concentration of measure property in \eqref{concmesut} in Appendix~\ref{Ann2} 
with $l=t$, $\Gamma=\cW^\cI_{d}(\tfrac{t}{d})$, $T(x,y)=1\{(x,y)< 0\}$, 
$\eta=\gep t$ and $\xi_i=-\alpha 1\{\omega_i=A\} +\beta 1\{\omega_i=B\}$ for 
all $i\in\N$, we find that there exist $C_{1},C_{2}>0$ such that
\be{rec22}
\P\big(\big|\phi^{\theta^s(w)}_d(\tfrac{t}{d})-\phi_d(\tfrac{t}{d})\big|>\gep \big|\big)
\leq C_{1} \, e^{-C_{2}\, \gep^2 t}, \qquad t,d,s \in\N,\,t\in d+2\N_0. 
\ee
With the help of \eqref{boundel} and \eqref{psi1} we may write, for $(L,\Theta,u,j)
\in G_M^{\,m}$,
\be{boundZ}
\psi_1(L,\Theta,u,j)\leq C_{\text{uf}}(\alpha)\,
\P\big(\cA_{\gep,m L}\cup \cB_{\gep,m L}\big)
+\tfrac{1}{uL}\,\E\big[1_{\{\cA^c_{\gep,m L}\cap\cB^c_{\gep,m L}\}}\,
\log Z_1^\omega(L,\Theta,u,j)\big].
\ee
With the help of \eqref{rec} and \eqref{rec22}, we get that $\P(\cA_{\gep,m L})\to 0$ 
and $\P(\cB_{\gep,m L})\to 0$ as $L\to \infty$. Moreover, from (\eqref{defZ}-\eqref{subset1}) 
it follows that, for $(L,\Theta,u,j)\in G_M^{\,m}$ and $\omega\in\cA^c_{\gep,m L}\cap
\cB^c_{\gep,ML}$, 
\be{boundZalt}
Z^\omega_1(L,\Theta,u,j)\leq Z_2(L,\Theta,u,j)\,e^{\gep u L}.
\ee
The latter completes the proof of $\psi_1\prec\psi_2$. 


\subsubsection{Step 2}

In this step, we concatenate the pieces of trajectories that travel in $A$-blocks, 
respectively, $B$-blocks, respectively, along the $AB$-interfaces and replace the 
finite-size entropies and free energies by their infinite-size counterparts. Recall 
the definition of $l_{A,\Theta,j}$ and $l_{B,\Theta,j}$ in \eqref{bl1} and define, for 
$(L,\Theta,u,j)\in G_M^{\,m}$, the sets  
\begin{align}
\cJ_{\Theta,L,j}
&=\Big\{\big(a_A,h_A,a_B,h_B\big)\in \N^4\colon\, a_A\in l_{A,\Theta,j} L+h_A+2\N_0,\, 
a_B\in l_{B,\Theta,j} L+h_B+2 \N_0\Big\},\\
\nonumber 
\cJ^\cI
&=\Big\{\big(a^\cI,h^\cI\big)\in \N^2\colon\, a^\cI\in h^\cI+ 2\N_0\Big\},
\end{align}
and put $\psi_3(L,\Theta,u,j)=\frac{1}{ uL} \log Z_3(L,\Theta,u,j)$ with
\begin{align}
\label{fatZ}
Z_3(L,\Theta,u,j)=\sum_{(a,h)\in \cJ_{\Theta,L,j}} 
&\sum_{(a^\cI,h^\cI)\in \cJ^\cI} A_3\,B_3\,C_3,
\end{align}
where
\begin{align}
\label{defAhat1}
\nonumber 
A_3 &= e^{a_A\,\tilde{\kappa}\Big(\tfrac{a_A}{h_A},\,\tfrac{l_{A,\Theta,j} L}{h_A}\Big)}\, 
e^{a_B\, \tilde{\kappa}\Big(\tfrac{a_B}{h_B},\,\tfrac{l_{B,\Theta,j} L}{h_B}\Big)}\, 
e^{\tfrac{\beta-\alpha}{2}\, a_B},\\
\nonumber 
B_3 &= e^{a^\cI\, \phi \big(\tfrac{a^\cI}{h^\cI}\big)},\\
C_3 &= 1_{\{a_A+a_B+a^\cI=uL\}}\,1_{\{h_A+h_B+h^\cI=L\}}.
\end{align}
In order to establish a link between $\psi_2$ and $\psi_3$ we define, for $(a,h)\in
\cJ_{\Theta,L,j}$ and $(a^\cI,h^\cI)\in \cJ^\cI$,
\begin{align}
\label{recod}
\nonumber 
\cP_{(a,h)} &= \big\{(t,d)\in \cD_{\Theta,L,j}\colon\,
\textstyle\sum_{i\in \cA_{\Theta,j}} (t_{i},d_{i})=(a_A,h_A),\, 
\sum_{i\in \cB_{\Theta,j}} (t_{i},d_{i})=(a_B,h_B)\big\},\\
\cQ_{(a^\cI,h^\cI)}&=\big\{(t^\cI,d^\cI)\in \cD_{r}^\cI\colon\,
\textstyle \sum_{i=1}^{r}(t^\cI_{i},d^\cI_{i})=(a^\cI,h^\cI)\big\}.
\end{align}
Then we can rewrite $Z_2$ as 
\begin{align}
\label{fatZT}
Z_2(L,\Theta,u,j)=\sum_{(a,h)\in \cJ_{\Theta,L,j}} 
&\sum_{(a^\cI,h^\cI)\in \cJ^\cI} C_3
\sum_{(t,d)\in \cP_{(a,h)}}
\sum_{(t^\cI,d^\cI)\in \cQ_{(a^\cI,h^\cI)}} A_2\,B_2.
\end{align}

To prove that $\psi_2\prec \psi_3$, we need the following lemma.

\begin{lemma}
\label{tesla3}
For every $\eta>0$ there exists an $L_\eta\in \N$ such that, for every $(L,\Theta,u,j)
\in G_M^{\,m}$ with $L\geq L_\eta$ and every $(d,t)\in \cD_{\Theta,L,j}$ and 
$(d^\cI,t^\cI)\in \cD_{r}^\cI$ satisfying $\sum_{i=1}^{r+1} d_i+\sum_{i=1}^{r} 
d_i^\cI=L$ and $\sum_{i=1}^{r+1} t_i+\sum_{i=1}^{r} t_i^\cI=u L$,
\begin{align}
\label{necbound}
t_i\,\tilde{\kappa}\big(\tfrac{t_i}{d_i},\,\tfrac{l_i L}{d_i}\big)-\eta uL 
&\leq t_i\,\tilde{\kappa}_{d_i}\big(\tfrac{t_i}{d_i},\,\tfrac{l_i L}{d_i}\big)
\leq  t_i\,\tilde{\kappa}\big(\tfrac{t_i}{d_i},\,\tfrac{l_i L}{d_i}\big)
+\eta uL \quad i = 1,\dots,r+1,\\
\nonumber
t_i^\cI \phi(\tfrac{t_i^\cI}{d_i^\cI})-\eta uL 
&\leq t_i^\cI \phi_{d^\cI_i}(\tfrac{t_i^\cI}{d_i^\cI})
\leq t_i^\cI \phi(\tfrac{t_i^\cI}{d_i^\cI})+\eta uL \quad i = 1,\dots,r.
\end{align}
\end{lemma}

\begin{proof}  
By using Lemmas~\ref{conunif1} and \ref{l:feinflim1} in Appendix~\ref{Path entropies},  
we have that there exists a $\tilde{L}_\eta\in \N$ such that, for $L\geq \tilde{L}_\eta$, 
$(u,l)\in \cH_L$ and $\mu\in 1+\frac{2\N}{L}$, 
\be{tesla}
|\tilde{\kappa}_L(u,l)-\tilde{\kappa}(u,l)| 
\leq \eta, \qquad |\phi^\cI_L(\mu)- \phi^\cI(\mu)|\leq \eta.
\ee 
Moreover, Lemmas~\ref{lementr}, \ref{l:lemconv2}(ii--iii), \ref{l:lemconv}(ii) and
\ref{l:feinflim1} ensure that there exists a $v_\eta>1$ such that, for $L\geq 1$, 
$(u,l)\in \cH_L$ with $u\geq v_\eta$ and $\mu\in 1+\frac{2\N}{L}$ with $\mu\geq v_\eta$,
\be{tesla2}
0\leq \tilde{\kappa}_L(u,l)\leq \eta, \qquad 
0\leq  \phi_L(\mu)\leq \eta.
\ee 
Note that the two inequalities in \eqref{tesla2} remain valid when $L=\infty$. Next, 
we set $r_\eta=\eta/(2 v_\eta C_{\text{uf}})$ and $L_\eta=\tilde{L}_\eta/r_\eta$, and 
we consider $L\geq L_\eta$. Because of the left-hand side of \eqref{tesla}, the two 
inequalities in the first line of \eqref{necbound} hold when $d_i \geq r_\eta L\geq 
\tilde{L}_\eta$. We deal with the case $d_i \leq r_\eta L$ by considering first the 
case $t_i\leq \eta u L/2 C_{\text{uf}}$, which is easy because $\tilde{\kappa}_{d_i}$ 
and $\tilde{\kappa}$ are uniformly bounded by $C_{\text{uf}}$ (see \eqref{boundel}). 
The case $t_i\geq \eta u L/2 C_{\text{uf}}$ gives $t_i/d_i\geq u v_\eta\geq v_\eta$, 
which by the left-hand side of \eqref{tesla2} completes the proof of the first line 
in \eqref{necbound}. The same observations applied to $t_i^\cI, d_i^\cI$ combined 
with the right-hand side of \eqref{tesla} and \eqref{tesla2} provide the two inequalities 
in the second line in \eqref{necbound}.
\end{proof}

To prove that $\psi_2\prec\psi_3$, we apply Lemma~\ref{tesla3} with $\eta=\gep/(2m+1)$ 
and we use \eqref{defA1} to obtain, for $L\geq L_{\gep/(2m+1)}$, $(d,t)\in\cD_{\Theta,L,j}$ 
and $(d^\cI,t^\cI)\in \cD_{r}^\cI$, 
\begin{align}
\label{defAA1}
A_2 &\leq \prod_{i\in \cA_{\Theta,j}} e^{ t_{i}\, \tilde{\kappa}
\big(\tfrac{t_{i}}{d_i}, \,\tfrac{ l_{i} L}{d_i}\big) 
+ \tfrac{\gep uL}{2m+1}} \prod_{i\in \cB_{\Theta,j}} 
\,e^{t_{i}\,\tilde{\kappa}\big(\tfrac{t_{i}}{d_i}, 
\,\tfrac{ l_{i} L}{d_i}\big) + t_i\,\tfrac{\beta-\alpha}{2} 
+ \tfrac{\gep uL}{2m+1}},\\
\nonumber 
B_2 &\leq \prod_{i=1}^{r}\, e^{t_{i}^\cI\,\phi\Big(\tfrac{t^\cI_{i}}{d^\cI_{i}}\Big)
+ \tfrac{\gep uL}{2m+1}}.
\end{align}
Next, we pick $(a,h)\in \cJ_{\Theta,L,j}$, $(a^\cI,h^\cI)\in \cJ^\cI$, $(t,d)\in
\cP_{(a,h)}$ and $(t^\cI,d^\cI)\in \cQ_{(a^\cI,h^\cI)}$, and we use the concavity 
of $(a,b)\mapsto a \tilde{\kappa}(a,b)$ and $\mu\mapsto \phi^\cI(\mu)$ (see 
Lemma~\ref{l:lemconv2} in Appendix~\ref{Path entropies} and Lemma~\ref{l:lemconv}
in Appendix~\ref{B}) to rewrite \eqref{defAA1} as
\begin{align}
\label{defAA2}
A_2 &\leq e^{a_A\,\tilde{\kappa}\big(\tfrac{a_A}{h_A},\,
\tfrac{l_{A,\Theta,j} L}{h_A}\big) + a_B \, \tilde{\kappa}\big(\tfrac{a_B}{h_B},\,
\tfrac{l_{B,\Theta,j} L}{h_B}\big) + \tfrac{\beta-\alpha}{2} a_B 
+ \tfrac{\gep (r+1)uL}{2m+1}} 
= A_3\,e^{\tfrac{\gep (r+1) u L}{2m+1}},\\
\nonumber
B_2 &\leq e^{a^\cI\, \phi^\cI\big(\tfrac{a^\cI}{h^\cI}\big)
+ \tfrac{\gep r uL}{2m+1}} = B_3\, e^{\tfrac{\gep r uL}{2m+1}}.
\end{align}
Moreover, $r$, which is the number of $AB$ interfaces crossed by the trajectories in 
$\cW_{\Theta,u,j,L}$, is at most $m$ (see \eqref{zto}), so that \eqref{defAA2} 
allows us to rewrite \eqref{fatZT} as
\begin{align}
\label{fatZTU}
Z_2(L,\Theta,u,j) \leq e^{\gep u L} \sum_{(a,h)\in \cJ_{\Theta,L,j}} 
&\sum_{(a^\cI,h^\cI)\in \cJ^\cI} C_3\,|\cP_{(a,h)}|\, |\cQ_{(a^\cI,h^\cI)}|\, A_3\,B_3.
\end{align}
Finally, it turns out that $|\cP_{(a,h)}|\leq (uL)^{8r}$ and $|\cQ_{(a^\cI,h^\cI)}|
\leq (uL)^{8r}$. Therefore, since $r\leq m$, \eqref{fatZ} and \eqref{fatZTU} allow 
us to write, for $(L,\Theta,u,j)\in G_M^{\,m}$ and $L\geq L_{\gep/2m+1}$,
\be{add10}
Z_2(L,\Theta,u,j)\leq (mL)^{16 m} Z_3(L,\Theta,u,j).
\ee
The latter is sufficient to conclude that $\psi_2\prec\psi_3$.


\subsubsection{Step 3}
\label{s3}

For every $(L,\Theta,u,j)\in G_M^{\,m}$ we have, by the definition of $\cL(l_{A,\Theta,j},
l_{B,\Theta,j};u)$ in \eqref{defnu}, that $(a,h)\in \cJ_{\Theta,L,j}$ and $(a^\cI,h^\cI)\in \cJ^\cI$ 
satisfying $a_A+a_B+a^\cI=uL$ and $h_A+h_B+h^\cI=L$ also satisfy  
\be{innu}
\Big(\big(\tfrac{a_A}{L},\tfrac{a_B}{L}, 
\tfrac{a^\cI}{L}\big),\big(\tfrac{h_A}{L},\tfrac{h_B}{L},
\tfrac{h^\cI}{L}\big)\Big)\in \cL(l_{A,\Theta,j},l_{B,\Theta,j};u).
\ee
Hence, \eqref{innu} and the definition of $\psi_\cI$ in \eqref{Bloc of type I} ensure 
that, for this choice of $(a,h)$ and $(a^\cI,h^\cI)$,
\be{add11}
A_3 B_3\leq e^{uL \psi_\cI(u,\,l_{A,\Theta,j},\,l_{B,\Theta,j})}.
\ee
Because of $C_3$, the summation in \eqref{fatZ} is restricted to those $(a,h)\in
\cJ_{\Theta,L,j}$ and $(a^\cI,h^\cI)\in \cJ^\cI$ for which $a_A,a_B,a^\cI \leq uL$ 
and $h_A,h_B,h^\cI\leq L$. Hence, the summation is restricted to a set of cardinality 
at most $(uL)^3 L^3$. Consequently, for all $(L,\Theta,u,j)\in G_M^{\,m}$ we have
\begin{align}
\label{fatZalt}
Z_3(L,\Theta,u,j)&=\sum_{(a,h)\in \cJ_{\Theta,L,j}} 
\sum_{(a^\cI,h^\cI)\in \cJ^\cI}  A_4\, B_4\, C_4\leq (m L)^3 L^3\, 
e^{u\,L\, \psi_\cI(u,\,l_{A,\Theta,j},\,l_{B,\Theta,j})}.
\end{align}
The latter implies that $\psi_3\prec \psi_4$ since $\psi_4=\psi_\cI(u,\,l_{A,\Theta,j},
\,l_{B,\Theta,j})$ by definition (recall \eqref{Bloc of type I1} and \eqref{psidefvar}).


\section{Column-based variational formula}
\label{proofofgene}

To derive the \emph{slope-based variational formula} that is the cornerstone of our analysis, 
we state and prove in this section an auxiliary variational formula for the quenched free energy 
per step that involves the fraction of the time spent by the copolymer in each type of block 
columns and the free energy per step of the copolymer in a given block column. This auxiliary 
variational formula will be used in Section~\ref{varfo2} in combination with Proposition~\ref{energ} 
to complete the proof of the \emph{slope-based variational formula}.

With each $\Theta\in \overline\cV_M$ we associate a quantity $u_\Theta \in [t_\Theta,
\infty)$ indicating how many \emph{steps on scale} $L_n$ the copolymer makes in columns 
of type $\Theta$, where $t_\Theta$ is the minimal number of steps required to cross a 
column of type $\Theta$. These numbers are gathered into the set 
\be{BVdef}
\cB_{\overline\cV_M}=\big\{(u_\Theta)_{\Theta\in \overline\cV_M}\in\R^{\overline\cV_M}
\colon\,u_\Theta\geq t_\Theta\,\,\forall\,\Theta \in \overline\cV_M,\,
\Theta \mapsto u_\Theta \mbox{ continuous}\big\},
\ee 
where the continuity in $\Theta$ is with respect to the distance $d_M$ defined in \eqref{dist} 
in Appendix~\ref{B.2}. We recall Proposition~\ref{energ}, which identifies the free energy per 
step $\psi(\Theta,u_\Theta;\alpha,\beta)$ associated with the copolymer when crossing a 
column of type $\Theta$ in $u_\Theta$ steps, and we recall that the set $\cR_{p,M}$ introduced 
in Section \ref{cgp1} gathers the \emph{frequencies with which different types of columns can 
be visited by the copolymer}.

\begin{theorem}(column-based variational formula) 
\label{varformula}
For every $(\alpha,\beta)\in\CONE$, and $p \in (0,1)$ the free energy in 
{\rm \eqref{felimdeff}} exists for $\P$-a.e.\ $(\omega,\Omega)$ and in 
$L^1(\P)$, and is given by
\be{genevarold}
f(\alpha,\beta;p) 
= \sup_{M\geq 1} \, \sup_{\rho\in \cR_{p,M}}\,\sup_{(u_\Theta)_{\Theta\in \overline\cV_M}\,
\in\,\cB_{\,\overline\cV_M}}\,\,\frac{N(\rho,u)}{D(\rho,u)},
\ee
where 
\begin{align}
\label{genevar2}
\nonumber
N(\rho,u) &= \int_{\overline\cV_M}\,u_\Theta\,\psi(\Theta,u_{\Theta};\alpha,\beta)\,
\rho(d\Theta),\\
D(\rho,u) &= \int_{\overline\cV_M}\,u_\Theta\,\rho(d\Theta),
\end{align}
with the convention that $N(\rho,u)/D(\rho,u)=-\infty$ when $D(\rho,u)=\infty$.
\end{theorem}

The present section is technically involved because it goes through a sequence of approximation
steps in which the self-averaging of the free energy with respect to $\omega$ and $\Omega$ 
in the limit as $n\to\infty$ is proven, and the various ingredients of the variational 
formula in Theorem~\ref{varformula} that were constructed in Section~\ref{keyingr} are 
put together.

In Section~\ref{strprTh} we introduce additional notation and state Propositions~\ref{pr:formimp}, 
\ref{pr:formimpp}, \ref{pr:varar} and \ref{pr:formimppp} from which Theorem~\ref{varformula} is a straightforward 
consequence. Proposition~\ref{pr:formimp}, which deals with $(M,m)\in\EIGH$, is proven in 
Section~\ref{prooffor} and the details of the proof are worked out in 
Sections~\ref{s22}--\ref{saltisgp}, organized into 5 Steps that link intermediate free 
energies. We pass to the limit $m\to\infty$ with Propositions~\ref{pr:formimpp} and 
\ref{pr:varar} which are proven in Section~\ref{sMinf} and \ref{svarar}, respectively.  
Finally, we pass to the limit $M\to \infty$ with Proposition \ref{pr:formimppp} which is proven in Section 
\ref{Minfinity}.


\subsection{Proof of Theorem~\ref{varformula}}
\label{strprTh}


\subsubsection{Additional notation}

Pick $(M,m)\in \EIGH$ and recall that $\Omega$ and $\omega$ are independent, i.e., 
$\P=\P_\omega \times \P_\Omega$. For $\Omega\in \{A,B\}^{\N_0\times\Z}$, $\omega\in 
\{A,B\}^\N$, $n\in \N$ and $(\alpha,\beta)\in \CONE$, define
\begin{equation}
\label{intermp}
f^{\omega,\Omega}_{1,n}(M,m;\alpha,\beta)
= \tfrac{1}{n} \log Z_{1,n,L_n}^{\,\omega,\Omega}(M,m)
\quad\text{with}\quad
Z_{1,n,L_n}^{\,\omega,\Omega}(M,m)
=\sum_{\pi \in \cW_{n,M,}^{\,m}} e^{\,H_{n,L_n}^{\omega,\Omega}(\pi)},
\ee 
where $\cW_{n,M}^{\,m}$ contains those paths in $\cW_{n,M}$ that, in each column, make 
at most $m L_n$ steps. We also restrict the set $\cR_{p,M}$ in \eqref{RMNdef} to those 
limiting empirical measures whose support is included in $\overline \cV_{M}^{\,m}$, i.e., 
those measures charging the types of column that can be crossed in less than $m L_n$ 
steps only. To that aim we recall \eqref{RMNdefalt} and define, for $\Omega\in \{A,B\}^{\N_0
\times\Z}$ and $N\in \N$,
\be{RMNNdefalt}
\begin{aligned}
\cR^{\Omega,m}_{M,N} &= \big\{\rho_N(\Omega,\Pi,b,x)\ \text{with}\  
b=(b_j)_{j\in \N_0} \in (\mathbb{Q}_{(0,1]})^{\N_0},\\
&\qquad \Pi=(\Pi_j)_{j\in\N_0} \in \{0\}\times \Z^{\N} \colon\,
|\Delta\Pi_j|\leq M\,\ \ \forall\,j\in\N_0,\\
&\qquad x=(x_j)_{j\in \N_0} \in \{1,2\}^{N_0}\colon\, \big(\Omega(j,\Pi_j+\cdot),
\Delta \Pi_j,b_j,b_{j+1},x_j\big)\in \cV_M^{m}\big\}
\end{aligned} 
\ee
which is a  subset of $\cR_{M,N}^{\Omega}$ and allows us to define 
\be{RMmdef}
\cR^{\Omega,m}_M = \mathrm{closure}\Big(\cap_{N'\in\N} \cup_{N \geq N'}\, 
\cR^{\Omega,m}_{M,N}\Big),
\ee
which, for $\P$-a.e. $\Omega$ is equal to $\cR_{p,M}^{m}\subsetneq \cR_{p,M}$.

At this stage, we further define, 
\be{modvaralt}
f(M,m;\alpha,\beta)
= \sup_{\rho\in \cR_{p,M}^{\,m}}\,
\sup_{(u_\Theta)_{\Theta\in \overline\cV_M^{\,m}}\in\cB_{\,\overline\cV_M^{\,m}}}
V(\rho,u),
\ee
where
\be{ddeff}
V(\rho,u)=\frac{\int_{\overline \cV_M^{\,m}}\,u_\Theta\,
\psi(\Theta,u_{\Theta};\alpha,\beta)\,\rho(d\Theta)}
{\int_{\overline \cV_M^{\,m}}\,u_\Theta\, \rho(d\Theta)},
\ee
where (recall \eqref{set2alt1})
\be{defbm}
\cB_{\overline \cV_M^{\,m}}
= \Big\{(u_\Theta)_{\Theta\in \overline \cV_M^{\,m}}\in\R^{\overline \cV_M^{\,m}}
\colon \Theta \mapsto u_\Theta\in \cC^0\big(\overline \cV_M^{\,m},\R\big),\,
t_\Theta \leq u_\Theta\leq m\,\,\forall\,\Theta\in \overline \cV_M^{\,m} \Big\},
\ee
and where $\overline{\cV}_M^{\,m}$ is endowed with the distance $d_M$ defined in 
\eqref{dist} in Appendix~\ref{B.2}.

Let $\cW^{*,m}_{n,M}\subset\cW_{n,M}^{\,m}$ be the subset consisting of those paths 
whose endpoint lies at the boundary between two columns of blocks, i.e., satisfies
$\pi_{n,1}\in \N L_n$. Recall \eqref{intermp}, and define $Z^{*,\omega,\Omega}_{n,L_n}
(M)$ and $f^{*,\omega,\Omega}_{1,n}(M,m;\alpha,\beta)$ as the counterparts of 
$Z^{\omega,\Omega}_{n,L_n}(M,m)$ and  $f^{\omega,\Omega}_{1,n}(M,m;\alpha,\beta)$ 
when $\cW_{n,M}^{\,m}$ is replaced by $\cW_{n,M}^{*,m}$. Then there exists a constant 
$c>0$, depending on $\alpha$ and $\beta$ only, such that
\be{inegW}
\begin{aligned}
&Z^{\omega,\Omega}_{1,n,L_n}(M,m) e^{-c L_n}
\leq Z^{*,\omega,\Omega}_{1,n,L_n}(M,m)
\leq Z^{\omega,\Omega}_{1,n,L_n}(M,m),\\  
&n\in \N, \,\omega \in \{A,B\}^\N,\, \Omega \in \{A,B\}^{\N_0\times \Z}.
\end{aligned}
\ee
The left-hand side of the latter inequality is obtained by changing the last $L_n$ 
steps of each trajectory in $\cW_{n,M}^{\,m}$ to make sure that the endpoint falls in 
$L_n\N$. The energetic and entropic cost of this change are obviously $O(L_n)$. 
By assumption, $\lim_{n\to\infty} L_n/n=0$, which together with \eqref{inegW} 
implies that the limits of $f^{\omega,\Omega}_{1,n}(M,m;\alpha,\beta)$ and 
$f^{*,\omega,\Omega}_{1,n}(M,m;\alpha,\beta)$ as $n\to\infty$ are the same. In the 
sequel we will therefore restrict the summation in the partition function to 
$\cW_{n,M}^{*,m}$ and drop the $*$ from the notations.

Finally, let
\be{add12}
\begin{aligned}
f_{1,n}^{\Omega}(M,m;\alpha,\beta)
&= \E_\omega\big[f^{\omega,\Omega}_{1,n}(M,m;\alpha,\beta)\big],\\
f_{1,n}(M,m;\alpha,\beta)
&= \E_{\omega,\Omega}\big[f^{\omega,\Omega}_{1,n}(M,m;\alpha,\beta)\big],
\end{aligned}
\ee
and recall \eqref{partfunc} and \eqref{partfunc11} to set 
\be{}
f_n^\Omega(\alpha,\beta)=\E_\omega[f_n^{\omega,\Omega}(\alpha,\beta)],
\qquad f_n^\Omega(M;\alpha,\beta)
=\E_\omega[f_n^{\omega,\Omega}(M;\alpha,\beta)].
\ee


\subsubsection{Key Propositions}
\label{Key Propositions}

Theorem~\ref{varformula} is a consequence of Propositions \ref{pr:formimpp},
\ref{pr:varar} and \ref{pr:formimppp}   stated below and proven in Sections~\ref{step1}--\ref{step3}, 
Section~\ref{svarar} and Sections~\ref{step11}--\ref{step77}, respectively.

Proposition \ref{pr:formimp}, that is stated first is required to prove Proposition \ref{pr:formimpp}
and will be proven in Sections~\ref{s22}--\ref{s23alt}.

\bp{pr:formimp}
For all $(M,m)\in \EIGH$, 
\be{formimp}
\lim_{n\to \infty} f^{\Omega}_{1,n}(M,m;\alpha,\beta)
=f(M,m;\alpha,\beta)\quad \text{ for } \P-a.e.\,\Omega.
\ee
\ep
\bp{pr:formimpp}
For all $M\in \N$,
\be{formimpp}
\lim_{n\to \infty} f^{\Omega}_{n}(M;\alpha,\beta)
=\sup_{m\geq M+2} f(M,m;\alpha,\beta)\quad \text{ for } \P-a.e.\,\Omega.
\ee
\ep

\bp{pr:varar}
For all $M\in \N$,
\be{varar}
\sup_{m\geq M+2} f(M,m;\alpha,\beta)=\sup_{\rho\in \cR_{p,M}}\,
\sup_{(u_\Theta)_{\Theta\in \overline\cV_M}\, 
\in\,\cB_{\,\overline\cV_M}}   V(\rho,u),
\ee
\ep
where, in the righthand side of \eqref{varar}, we recognize the variational formula of 
Theorem \ref{varformula} and with $\cB_{\overline\cV_M}$ defined in \eqref{set2}.

\bp{pr:formimppp}
\be{formimppp}
\limsup_{n\to \infty} f^{\Omega}_{n}(\alpha,\beta)
\leq  \sup_{M\geq 1}  \lim_{n\to \infty} f^{\Omega}_{n}(M;\alpha,\beta)\quad \text{ for } \P-a.e.\,\Omega.
\ee
\ep

\medskip\noindent 
{\bf Proof of Theorem~\ref{varformula} subject to Propositions 
\ref{pr:formimpp}, \ref{pr:varar} and \ref{pr:formimppp}.}
With Propositions \ref{pr:formimpp}, \ref{pr:varar} and  \ref{pr:formimppp} in hand, the proof of Theorem~\ref{varformula} 
will be complete once we show that 
\be{rstt} \lim_{n\to \infty} |f_{n}^{\omega,\Omega}(\alpha,\beta)
-f_{n}^{\Omega}(\alpha,\beta)|=0
\quad \text{ for } \P-a.e.\,\,(\omega,\Omega).
\ee 
To that aim, we note that for all $n\in \N$ the $\Omega$-dependence of $f_{n}^{\omega,\Omega}
(\alpha,\beta)$ is restricted to $\big\{\Omega_x\colon\,x\in G_n\big\}$ with $G_n
=\{0,\dots,\frac{n}{L_n}\}\times\{-\frac{n}{L_n},\dots,\frac{n}{L_n}\}$. Thus, for $n\in\N$ 
and $\gep>0$ we set 
\be{defac}
A_{\gep,n}=\{| f_{n}^{\omega,\Omega}(\alpha,\beta)- f_{n}^{\Omega}(\alpha,\beta)|>\gep)\},
\ee
and by independence of $\omega$ and $\Omega$ we can write
\begin{align}
\label{concv}
{\textstyle \nonumber\P_{\omega,\Omega}(A_{\gep,n})}
&{\textstyle =\sum_{\Upsilon\in \{A,B\}^{G_n}} 
\P_{\omega,\Omega}(A_{\gep,n}\cap \{\Omega_{G_n}=\Upsilon\})}\\
&{\textstyle =\sum_{\Upsilon\in \{A,B\}^{G_n}} 
\P_{\omega}(| f_{n}^{\omega,\Upsilon}(\alpha,\beta)
- f_{n}^{\Upsilon}(\alpha,\beta)|>\gep)\ \P_{\Omega}(\{\Omega_{G_n}=\Upsilon\})}.
\end{align}
At this stage, for each $n\in \N$ we can apply the concentration inequality $\eqref{concmesut}$ 
in Appendix \ref{Ann2} with $\Gamma=\cW_{n}$, $l=n$, $\eta=\gep n$,
\be{addd25}
\xi_i = -\alpha\, 1\{\omega_i=A\} +\beta\, 
1\{\omega_i=B\}, \qquad i\in\N,
\ee 
and with $T(x,y)$ indicating in which block step $(x,y)$ lies in. Therefore, there exist 
$C_1, C_2>0$ such that for all $n\in \N$ and all $\Upsilon\in \{A,B\}^{G_n}$ we have
\be{resconc}
\P_{\omega}(| f_{n}^{\omega,\Upsilon}(M,m;\alpha,\beta)
- f_{n}^{\Upsilon}(M,m;\alpha,\beta)|>\gep) \leq C_1 e^{-C_2 \gep^2 n},
\ee
which, together with \eqref{concv} yields $\P_{\omega,\Omega}(A_{\gep,n})\leq 
C_1 e^{-C_2 \gep^2 n}$ for all $n\in \N$. By using the Borel-Cantelli Lemma, we 
obtain \eqref{rstt}.
\hspace*{\fill}$\square$


\subsection{Proof of Proposition \ref{pr:formimp}}
\label{prooffor} 
Pick $(M,m)\in \EIGH$ and $(\alpha,\beta)\in \CONE$. In Steps 1--2 in 
Sections~\ref{s22}--\ref{s23} we introduce an intermediate free energy 
$f_{3,n}^{\Omega}(M,m;\alpha,\beta)$ and show that
\begin{align}
\label{limsup1}
\lim_{n\to \infty} |f_{1,n}^{\Omega}(M,m;\alpha,\beta)
- f_{3,n}^{\Omega}(M,m;\alpha,\beta)|=0 \qquad \forall\,
\Omega\in \{A,B\}^{\N_0\times \Z}.
\end{align}
Next, in Steps 3--4 in Sections~\ref{s24}--\ref{s24alt} we show that 
\be{limsup3}
\limsup_{n\to \infty} f_{3,n}^{\Omega}(M,m;\alpha,\beta)
= f(M,m;\alpha,\beta)
\qquad \text{for}\,\,\P-a.e.\,\, \Omega,
\ee
while in Step 5 in Section~\ref{s23alt} we prove that 
\be{convco1}
\liminf_{n\to \infty} f_{3,n}^{\Omega}(M,m;\alpha,\beta)
= \limsup_{n\to \infty} f_{3,n}^{\Omega}(M,m;\alpha,\beta) 
\qquad \text{for}\,\,\P-a.e.\,\,\Omega.
\ee
Combing (\ref{limsup1}--\ref{convco1}) we get
\be{fin}
\liminf_{n\to \infty} f_{1,n}^{\Omega}(M,m;\alpha,\beta)
= \limsup_{n\to \infty} f_{1,n}^{\Omega}(M,m;\alpha,\beta)=f(M,m;\alpha,\beta)
\qquad \text{for}\,\,\P-a.e.\,\, \Omega,
\ee
which completes the proof of Proposition~\ref{pr:formimp}.

In the proof we need the following order relation. 

\begin{definition}
For $g,\widetilde{g}\colon\,\N^3\times \CONE\mapsto\R$, write $g\prec \widetilde{g}$ 
if for all $(M,m)\in \EIGH$, $(\alpha,\beta)\in \CONE$ and $\gep>0$ there exists an 
$n_{\gep}\in \N$ such that
\be{add13}
g(n,M,m; \alpha,\beta) \leq \widetilde{g}(n,M,m; \alpha,\beta)
+\gep \qquad \forall\,n\geq n_{\gep}.
\ee
\end{definition}

\noindent
The proof of \eqref{limsup1} will be complete once we show that $f_1^\Omega \prec
f_3^\Omega$ and $f_3^\Omega \prec f_1^\Omega$ for all $\Omega\in \{A,B\}^{\N_0
\times \Z}$. We will focus on $f_1^\Omega\prec f_3^\Omega$, since the proof of 
the latter can be easily adapted to obtain $f_3^\Omega \prec f_1^\Omega$. To prove
$f_1^\Omega \prec f_3^\Omega$ we introduce another intermediate free energy 
$f_2^\Omega$, and we show that $f_1^\Omega\prec f_2^\Omega$ and $f_2^\Omega \prec 
f_3^\Omega$. 

For $L\in \N$, let
\begin{equation}
\label{defDD2}
\cD_L^M=\left\{\Xi
=(\Delta\Pi,b_0,b_1)\in \{-M,\dots,M\}\times 
\{\tfrac1L,\tfrac2L,\dots,1\}^2 \right\}.	
\end{equation}
For $L,N\in\N$, let 
\begin{align}
\label{defD2}
\widetilde{\cD}_{L,N}^M 
&= \Big\{\Theta_{\text{traj}}=(\Xi_i)_{i\in \{0,\dots, N-1\}}
\in (\cD_L^M)^N\colon\, b_{0,0}=\tfrac1L,\,
b_{0,i}=b_{1,i-1}\,\,\forall\,1\leq i\leq N-1\Big\},
\end{align}
and with each $\Theta_{\text{traj}}\in\widetilde{\cD}_{L,N}^M$ associate the sequence 
$(\Pi_i)_{i=0}^N$ defined by $\Pi_0=0$ and $\Pi_i=\sum_{j=0}^{i-1} \Delta\Pi_j$ 
for $1\leq i\leq N$. Next, for $\Omega\in \{A,B\}^{\N_0\times \Z}$ and $\Theta_{\text{traj}}
\in \widetilde{\cD}_{L,N}^M$, set
\be{defXalt}
\cX_{\Theta_{\text{traj}},\Omega}^{M,m}
=\big\{x\in \{1,2\}^{\{0,\dots,N-1\}}\colon 
(\Omega(i,\Pi_i+\cdot),\Xi_i,x_i)\in  \cV_{M}^{m} \,\,\forall\,0\leq i\leq N-1\big\},
\ee
and, for $x\in \cX_{\Theta_{\text{traj}},\Omega}^{M,m}$, set
\be{tthe}
\Theta_i=(\Omega(i,\Pi_i+\cdot),\Xi_i,x_i) \quad \text{for}\quad i\in \{0,\dots,N-1\}
\ee 
and 
\begin{equation}
\begin{aligned}
\label{defU2}
\cU_{\,\Theta_{\text{traj}},x,n}^{\,M,m,L} &= \Big\{u=(u_i)_{ i\in \{0,\dots, N-1\}}
\in [1,m]^{N}\colon u_i\in  t_{\Theta_i}+\tfrac{2\N}{L}
\ \,\,\forall\,0\leq i\leq N-1,\,\sum_{i=0}^{N-1} u_i=\tfrac{n}{L}\Big\}.
\end{aligned}
\end{equation}
Note that $\cU_{\,\Theta_{\text{traj}},x,n}^{\,M,m,L}$  is empty when $N\notin
\big[\frac{n}{m L},\frac{n}{L}\big]$.

For $\pi\in \cW_{n,M}^{\,m}$, we let $N_\pi$ be the number of columns crossed by $\pi$ 
after $n$ steps. We denote by $(u_0(\pi),\dots,u_{N_\pi-1}(\pi))$ the time spent 
by $\pi$ in each column divided by $L_n$, and we set $\widetilde{u}_0(\pi)=0$ 
and $\widetilde{u}_{j}(\pi)=\sum_{k=0}^{j-1} u_k(\pi)$ for $1\leq j\leq N_\pi$. With 
these notations, the partition function in \eqref{intermp} can be rewritten as
\begin{equation}
\label{A11}
Z_{1,n,L_n}^{\,\omega,\Omega}(M,m)=\sum_{N=n/m L_n}^{n/L_n}\,
\,\sum_{\Theta_{\text{traj}}\in \widetilde{\cD}_{L_n,N}^M }\,
\sum_{x\in \cX_{\Theta_{\text{traj}},\Omega}^{M,m} }
\,\sum_{u\in \,\cU_{\,\Theta_{\text{traj}},x,n}^{\,M,m,L_n}}   
\,A_1,
\end{equation}
with (recall \eqref{partfunc2})
\begin{align}
A_1=\prod_{i=0}^{N-1}\, Z^{\theta^{\widetilde{u}_{i} L_n}(\omega)}_{L_{n}}
(\Omega(i,\Pi_{i}+\cdot),\Xi_i,x_i, u_i).
\end{align}


\subsubsection{Step 1}
\label{s22}

In this step we average over the disorder $\omega$ in each column. To that end, we set
\be{interm3}
f_{2,n}^\Omega(M,m;\alpha,\beta) = \tfrac{1}{n} \log Z_{2,n,L_n}^{\Omega}(M,m)
\ee
with
\be{inteerm3}
Z_{2,n,L_n}^{\Omega}(M,m)=\sum_{N=n/m L_n}^{n/L_n} \ 
\,\sum_{\Theta_{\text{traj}}\in \widetilde{\cD}_{L_n,N}^M }\,
\sum_{x\in \cX_{\Theta_{\text{traj}},\Omega}^{M,m}}
\,\sum_{u\in \,\cU_{\,\Theta_{\text{traj}},x,n}^{\,M,m,L_n}}  
\,A_2,
\ee
where
\begin{align}
\label{defa24}
A_2 &=\prod_{i=0}^{N-1} e^{\E_\omega\big[\log  
Z^{\theta^{\widetilde{u}_{i}}(\omega)}_{L_{n}}
(\Omega(i,\Pi_{i}+\cdot),\Xi_i, x_i, u_i)\big]}
= \prod_{i=0}^{N-1}\, 
e^{u_i L_n \psi_{L_n}(\Omega(i,\Pi_{i}+\cdot),\Xi_i,x_i, u_i)}.
\end{align}
Note that the $\omega$-dependence has been removed from $Z^\Omega_{2,n,L_n}(M,m)$.

To prove that $f_1^\Omega\prec f_2^\Omega$, we need to show that for all $\gep>0$ 
there exists an $n_{\gep}\in \N$ such that, for $n\geq n_{\gep}$ and all $\Omega$,
\be{Eform}
\E_\omega\big[\log Z_{1,n,L_n}^{\,\omega,\Omega}(M,m)\big]
\leq \log Z_{2,n,L_n}^{\Omega}(M,m)+\gep n.
\ee
To this end, we rewrite $Z_{1,n,L_n}^{\,\omega,\Omega}(M,m)$ as
\be{rewr}
Z_{1,n,L_n}^{\,\omega,\Omega}(M,m)=\sum_{N=n/m L_n}^{n/L_n}
\sum_{\Theta_{\text{traj}}\in \widetilde{\cD}_{L_n,N}^M } 
\,\sum_{x\in \cX_{\Theta_{\text{traj}},\Omega}^{M,m} }
\,\sum_{u\in \,\cU_{\Theta_{\text{traj}},x,n}^{\,M,m,L_n}}
A_2\,\frac{A_1}{A_2},
\ee
where we note that
\be{A23}
\frac{A_1}{A_2} 
=\prod_{i=0}^{N-1}\,
e^{u_i L_n \big[\psi^{\theta^{\widetilde{u}_{i}L_n}(\omega)}_{L_n}
(\Omega(i,\Pi_{i}+\cdot),\Xi_i,x_i,u_i)
-\psi_{L_n}(\Omega(i,\Pi_{i}+\cdot),\Xi_i,x_i,u_i)\big]}.
\ee
In order to average over $\omega$, we apply a concentration of measure inequality. 
Set 
\be{add23}
\mathcal{K}_n = \bigcup_{N=n/m L_n}^{n/L_n}
\bigcup_{\Theta_{\text{traj}}\in \widetilde{\cD}_{L_n,N}^M} 
\,\bigcup_{x\in \cX_{\Theta_{\text{traj}},\Omega}^{M,m} }
\,\bigcup_{u\in \,\cU_{\Theta_{\text{traj}},x,n}^{\,M,m,L_n}}
\Big\{ |\log A_1 -\log A_2| \geq \gep n\Big\},
\ee
and note that $\omega\in \mathcal{K}_n^c$ implies that $Z_{1,n,L_n}^{\,\omega,\Omega}(M,m)
\leq e^{\gep n} Z_{2,n,L_n}^{\Omega}(M,m)$. Consequently, we can write
\begin{align}
\label{borne}
\nonumber 
\E_\omega\big[\log Z_{1,n,L_n}^{\,\omega,\Omega}(M,m)\big]
&= \E_\omega\big[\log Z_{1,n,L_n}^{\,\omega,\Omega}(M,m)
\,1_{\{\mathcal{K}_n\}} \big]+\E_\omega\big[\log Z_{1,n,L_n}^{\,\omega,\Omega}(M,m)
\,1_{\{\mathcal{K}_n^c\}}\big]\\
&\leq \E_\omega\big[\log Z_{1,n,L_n}^{\,\omega,\Omega}(M,m)
\,1_{\{\mathcal{K}_n\}} \big]+ \log Z_{2,n,L_n}^{\,\Omega}(M,m)\,
+ \gep n.
\end{align}
We can now use the uniform bound in \eqref{boundel} to control the first term in 
the right-hand side of \eqref{borne}, to obtain  
\begin{align}
\E_\omega\big[\log Z_{1,n,L_n}^{\,\omega,\Omega}(M,m)\big]
\leq &  \log Z_{2,n,L_n}^{\,\Omega}(M,m) +\gep n
+ C_{\text{uf}}(\alpha)\,n\,\P_\omega(\mathcal{K}_n).
\end{align}
Therefore the proof of this step will be complete once we show that 
$\P_\omega(\mathcal{K}_n)$ vanishes as $n\to\infty$.

\begin{lemma}
There exist $C_1,C_2>0$ such that, for all $\gep>0$, $n\in\N$, $N\in\big\{\tfrac{n}{m L_n},
\dots,\frac{n}{L_n}\big\}$, $\Omega\in \{A,B\}^{\N_0\times\Z}$, $\Theta_{\mathrm{traj}}
\in \widetilde{\cD}_{L_n,N}^M$, $x\in \cX_{\Theta_{\mathrm{traj}},\Omega}^{M,m}$ and 
$u\in \,\cU_{\Theta_{\mathrm{traj}},x,n}^{\,M,m,L_n}$,
\begin{equation}
\P_{\omega}(|\log  A_1-\log A_2|
\geq \gep n)\leq C_1 e^{-C_2\gep^2 n}.
\end{equation} 
\end{lemma}

\begin{proof}
Pick $\Theta_{\text{traj}}\in \widetilde{\cD}_{L_n,N}^M$, $x\in \cX_{\Theta_{\text{traj},
\Omega}}^{M,m}$ and $u\in\,\cU_{\Theta_{\text{traj}},x,n}^{\,M,m,L_n}$, and consider the 
subset $\Gamma$ of $\cW_{n,M}^{\,m}$ consisting of those paths of length $n$ that first 
cross the $(\Omega(0,\cdot),\Xi_{0},x_0)$ column such that $\pi_0=(0,1)$ and 
$\pi_{\widetilde{u}_{1}L_n}=(1,\Pi_{1}+b_{1,0}) L_n$, then cross the $(\Omega(1,\cdot),
\Xi_{1},x_1)$ column such that $\pi_{\widetilde{u}_{1}L_n+1}=(1+1/L_n,\Pi_{1}+b_{1,0})
L_n$ and $\pi_{\widetilde{u}_{2}L_n}=(2,\Pi_{2}+b_{1,1}) L_n$, and so on. We can apply 
the concentration of measure inequality stated in \eqref{concmesut} to the set 
$\Gamma$ defined above, with $l=n$, $\eta=\gep n$,
\be{add25}
\xi_i = -\alpha\, 1\{\omega_i=A\} +\beta\, 
1\{\omega_i=B\}, \qquad i\in\N,
\ee  
and with $T(x,y)$ indicating in which block step $(x,y)$ lies in. After noting that 
$\E_\omega(\log A_1)= \log A_2$, we  obtain that there exist $C_{1},C_{2}>0$ such 
that, for all $n\in\N$, $N\in\big\{\tfrac{n}{m L_n},\dots,\frac{n}{L_n}\big\}$, 
$\Omega\in \{A,B\}^{\N_0\times\Z}$, $\Theta_{\mathrm{traj}}\in \widetilde{\cD}_{L_n,N}^M$, 
$x\in \cX_{\Theta_{\text{traj},\Omega}}^{M,m}$ and $u\in \,\cU_{\Theta_{\mathrm{traj}},
x,n}^{\,M,m,L_n}$,
\be{rec2}
\P\big(|\log  A_1 -\log A_2| \geq \gep\,n \big) \leq C_{1}\,e^{-C_{2}\,\gep^3\, n}.
\ee
\end{proof}

It now suffices to remark that 
\be{setgr}
\big|\{(N,\Theta_{\text{traj}},x,u)\colon\,N\in \{\tfrac{n}{m L_n},
\dots,\tfrac{n}{L_n}\}, \Theta_{\text{traj}}\in
\widetilde{\cD}_{L_n,N}^M,\, x\in \cX_{\Theta_{\text{traj},\Omega}}^{M,m}, 
u\in \,\cU_{\Theta_{\mathrm{traj}},x,n}^{\,M,m,L_n}\}\big|
\ee 
grows subexponentially in $n$ to obtain that $f_1^\Omega\prec f_2^\Omega$ for all $\Omega$.


\subsubsection{Step 2}
\label{s23}

In this step we replace the finite-size free energy $\psi_{L_n}$ by its limit $\psi$. 
To do so we introduce a third intermediate free energy, 
\be{interm1}
f_{3,n}^\Omega(M,m;\alpha,\beta) =
\E\big[\tfrac{1}{n} \log Z_{3,n,L_n}^{\Omega}(M,m)\big],
\ee
where
\be{inteerm1}
Z_{3,n,L_n}^{\Omega}(M,m)=\sum_{N=n/m L_n}^{n/L_n}\  
\sum_{\Theta_{\text{traj}}\in \widetilde{\cD}_{L_n,N}^M } \,
\sum_{x\in \cX_{\Theta_{\text{traj}},\Omega}^{M,m} }
\,\sum_{u\in \,\cU_{\Theta_{\text{traj}},x,n}^{\,M,m,L_n}}A_3
\ee
with
\begin{align}
\label{defa23}
A_{3} &=\prod_{i=0}^{N-1}\, 
e^{u_i L_n \psi(\Omega(i,\Pi_{i}+\cdot),\Xi_i,x_i,u_i)}.
\end{align}
For all $\Omega$,
\be{a3}
\frac{A_2}{A_3} = \prod_{i=0}^{N-1}\,
e^{u_i L_n \big[\psi_{L_n}(\Omega(i,\Pi_{i}+\cdot),\Xi_i,x_i,u_i)
-\psi(\Omega(i,\Pi_{i}+\cdot),\Xi_i,x_i,u_i)\big]},
\ee 
and, for all $i\in \{0,\dots,N-1\}$, we have 
$(\Omega(i,\Pi_{i}+\cdot),\Xi_i,x_i,u_i)\in \cV^{*,\,m}_M$, so that 
Proposition~\ref{convfree1} can be applied.


\subsubsection{Step 3}
\label{s24}

In this step we want the variational formula \eqref{modvaralt} to appear. Recall \eqref{defrho} 
and define, for $n\in \N$, $(M,m)\in \EIGH$, $N\in \{\frac{n}{m L_n},\dots,
\frac{n}{L_n}\}$, $\Theta_{\text{traj}}\in \widetilde{\cD}_{L_n,N}^M$ and $x\in
\cX_{\Theta_{\text{traj}},\Omega}^{M,m}$,
\be{defthalt}
\Theta_j=(\Omega(j,\Pi_{j}+\cdot),\Xi_j,x_j),
\qquad j = 0,\dots,N-1,
\ee
and
\begin{equation}
\label{defrho2}
\rho^\Omega_{\Theta_{\text{traj}},x}\big(\Theta,\Theta^{'}\big)
=\frac{1}{N} \sum_{j=1}^{N} 
1_{\big\{(\Theta_{j-1},\Theta_{j})=(\Theta,\Theta^{'})\big\}},
\end{equation}
and, for $u\in \,\cU_{\Theta_{\text{traj}},x,n}^{\,M,m,L_n}$,
\be{defHam}
H^\Omega(\Theta_{\text{traj}},x,u)=\sum_{j=0}^{N-1} u_j\,\psi(\Theta_j, u_{j}). 
\ee
In terms of these quantities we can rewrite $Z_{3,n,L_n}^\Omega(M,m)$ in \eqref{inteerm1} 
as
\begin{align}
\label{variaf2}
Z_{3,n,L_n}^{\Omega}(M,m)
&=\sum_{N=n/m L_n}^{n/L_n} \sum_{\Theta_{\text{traj}}\in \widetilde{\cD}_{L_n,N}^M } 
 \,\sum_{x\in \cX_{\Theta_{\text{traj}},\Omega}^{M,m} }
\,\sum_{u\in \,\cU_{\Theta_{\text{traj}},x,n}^{\,M,m,L_n}}
e^{L_n\,  H^\Omega(\Theta_{\text{traj}},x,u) }.
\end{align}
For $n\in \N$, denote by
\be{}
N_n^{\Omega}, \qquad \Theta_{\text{traj},n}^{\Omega}
\in \widetilde{\cD}_{L_n,N^{\Omega}_n}^M, 
\qquad x_n^\Omega\in \cX_{\Theta_{\text{traj},n}^\Omega,\Omega}^{M,m},
\qquad u^\Omega_n \in\cU_{\Theta_{\text{traj},n}^\Omega,x_n^\Omega,n}^{\,M,m,L_n},
\ee 
the indices in the summation set of \eqref{variaf2} that maximize $H^\Omega(\Theta_{
\text{traj}},x,u)$. For ease of notation we put  
\be{redef}
\Theta_{\text{traj},n}^{\Omega}
=(\Xi_{j}^n)_{j=0}^{N_n^{\Omega}-1},\quad x^\Omega_n
=(x_j^n)_{j=0}^{N_n^{\Omega}-1}, \quad u^\Omega_n
=(u_j^n)_{j=0}^{N_n^{\Omega}-1},
\ee
and
\be{carsum}
c_n =\big|\{(N,\Theta_{\text{traj}},x,u)\colon\, 
\tfrac{n}{m L_n} \leq N\leq \tfrac{n}{L_n},\, \Theta_{\text{traj}}\in
\widetilde{\cD}^M_{L_n,N},\, x\in \cX_{\Theta_{\text{traj}},\Omega}^{M,m},\,
u\in \,\cU_{\Theta_{\text{traj}},x,n}^{\,M,m,L_n}\}\big|.
\ee
Then we can estimate
\be{impo1alt}
\frac1n \log Z_{3,n,L_n}^{\Omega}(M,m)\leq \frac1n \log c_n
+ \tfrac{L_n}{n}\sum_{j=0}^{N_n^{\Omega}-1} u_j^n\, 
\psi(\Theta_j^n, u_{j}^n). 
\ee

We next note that $u\mapsto u \psi(\Theta,u)$ is concave for all $\Theta\in 
\overline{\cV}_M$ (see Lemma~\ref{concav}). Hence, after setting 
\begin{align}
v_\Theta^n=\sum_{j=0}^{N_n^{\Omega}-1} 1_{\{\Theta_j^n=\Theta\}}\, 
u_j^n, \qquad  d_\Theta^n=\sum_{j=0}^{N_n^{\Omega}-1} 
1_{\{\Theta_j^n=\Theta\}}, \qquad \Theta\in \overline\cV_M^{\,m},
\end{align}
we can estimate 
\be{utilconcav}
\sum_{j=0}^{N_n^{\Omega}-1} 1_{\{\Theta_j^n=\Theta\}}\, 
u_{j}^n\, \psi(\Theta_j^n, u_{j}^n)
\leq v_\Theta^n\,\psi\big(\Theta, \tfrac{ v_\Theta^n}{d_\Theta^n}\big) 
\quad \text{for} \quad \Theta\in \overline\cV_M^{\,m}\colon d_\Theta^n\geq 1.
\ee 
Next, we recall \eqref{defrho2} and we set $\rho_{n}=\rho_{\Theta_{\text{traj},
n}^{\Omega},x_n^\Omega}^\Omega$, so that $\rho_{n,1}(\Theta)=d_\Theta^n/N_n^\Omega$ 
for all $\Theta\in \overline{\cV}_M^{\,m}$. Since $\{\Theta\in \overline{\cV}_M^{\,m}\colon\,
d_\Theta^n\geq 1\}$ is a finite subset of $\overline{\cV}_M^{\,m}$, we can easily extend 
$\Theta\mapsto v_\Theta^n/d_\Theta^n$ from $\{\Theta\in \overline{\cV}_M\colon\,
d_\Theta^n\geq 1\}$ to $\overline{\cV}_M^{\,m}$ as a continuous function. Moreover, 
$\sum_{j=0}^{N_n^\Omega-1} u_j^n=n/L_n$ implies that $N_n^\Omega \int_{\overline\cV_M^{\,m}}
v_\Theta^n/d_\Theta^n\,\rho_{n,1} (d\Theta)=n/L_n,$ which, together with
\eqref{impo1alt} and \eqref{utilconcav} gives 
\begin{align}\label{immp}
\tfrac1n \log Z_{3,n,L_n}^{\Omega}(M,m)
&\leq \sup_{u \in \cB_{\,\overline \cV_M^{\,m}}}\frac{\int_{\overline \cV_M^{\,m}} 
u_{\Theta}\, \psi(\Theta,u_\Theta)\,
\rho_{n}(d\Theta)}{\int_{\overline\cV_M^{\,m}} u_{\Theta}\,
\rho_{n}(d\Theta)}+ o(1), \qquad n\to\infty,
\end{align}
where we use that $\lim_{n\to\infty} \frac1n \log c_n=0$. In what follows, we abbreviate 
the first term in the right-hand side of the last display by $l_n$. We want to show 
that $\limsup_{n\to \infty} \tfrac1n \log Z_{3,n,L_n}^{\Omega}(M,m)$ $\leq f(M,m;\alpha,\beta)$. 
To that end, we assume that $\tfrac1n \log Z_{3,n,L_n}^{\Omega}(M,m)$ converges to some
$t\in \R$ and we prove that $t\leq f(M,m;\alpha,\beta)$. Since $(l_n)_{n\in\N}$ is bounded 
and $\overline \cV_M^{\,m}$ is compact, it follows from the definition of $l_n$ that along 
an appropriate subsequence both $l_n \to l_\infty \geq t$ and $\rho_n\to\rho_\infty\in
\cR_{p,M}^{\,m}$ as $n\to\infty$.  Hence, the proof will be complete once we show that 
\be{eqfin}
l_\infty\leq \sup_{u\in \cB_{\,\overline \cV_M^{\,m}}} V(\rho_{\infty},u),
\ee
because the right-hand side in \eqref{eqfin} is bounded from above by $f(M,m;\alpha,\beta)$. 

Recall \eqref{deft} and, for $\Theta\in \overline\cV_M^{\,m}$ and $y\in \R$, define
\be{defdef}
u_\Theta^{M,m}(y)= \left\{
\begin{array}{ll}
\vspace{.1cm}
t_\Theta
& \mbox{if } \partial^+_u (u\,\psi(\Theta,u))(t_\Theta) \leq y, \\
\vspace{.1cm}
m   
& \mbox{if } \partial^-_u (u\,\psi(\Theta,u))(m) \geq y,\\
z
&\mbox{otherwise, with } z \mbox{ such that } \partial^-_u (u\,\psi(\Theta,u))(z) 
\geq y \geq \partial^+_u (u\,\psi(\Theta,u))(z), 
\end{array}
\right.
\ee
where $z$ is unique by strict concavity of $u\to u\psi(\Theta,u)$ (see Lemma~\ref{B.2}). 

\begin{lemma}
\label{conti}
(i) For all $y\in \R$ and $(M,m)\in \EIGH$, $\Theta \mapsto u_\Theta^{M,m}(y)$ is 
continuous on $(\overline \cV_M^{\,m},d_M)$, where $d_M$ is defined in \eqref{dist} in 
Appendix~{\rm \ref{B}}.\\
(ii) For all $(M,m)\in\EIGH$ and $\Theta\in\overline{\cV}_M^{\,m}$, $y \mapsto
u_\Theta^{M,m}(y)$ is continuous on $\R$.
\end{lemma}

\begin{proof}
The proof uses the strict concavity of $u\to u\psi(\Theta,u)$ (see Lemma~\ref{B.2}). 

\medskip\noindent
(i) The proof is by contradiction. Pick $y\in \R$, and pick a sequence $(\Theta_n)_{n
\in \N}$ in $\overline\cV_M^{\,m}$ such that $\lim_{n\to\infty} \Theta_n=\Theta_\infty
\in \overline\cV_M^{\,m}$. Suppose that $u_{\Theta_n}^{M,m}(y)$ does not tend to 
$u_{\Theta_\infty}^{M,m}(y)$ as $n\to \infty$. Then, by choosing an appropriate 
subsequence, we may assume that $\lim_{n\to \infty} u_{\Theta_n}^{M,m}(y) = u_1 
\in [t_{\Theta_\infty},m]$ with $u_1<u_{\Theta_\infty}^{M,m}(y)$. The case 
$u_1>u_{\Theta_\infty}^{M,m}(y)$ can be handled similarly.

Pick $u_2\in (u_1,u_{\Theta_\infty}^{M,m}(y))$. For $n$ large enough, we have 
$u_{\Theta_n}^{M,m}(y)< u_2<u_{\Theta_\infty}^{M,m}(y)$. By the definition of 
$u_{\Theta_n}^{M,m}(y)$ in \eqref{defdef} and the strict concavity of $u\mapsto
u \psi(\Theta_n,u)$ we have, for $n$ large enough,
\be{inegstrict}
\partial^+_u (u\,\psi(\Theta_n,u))(u_{\Theta_n}^{M,m}(y))
> \frac{u_{\Theta_\infty}^{M,m}(y)\psi(\Theta_n,u_{\Theta_\infty}^{M,m}(y))
-u_2\psi(\Theta_n,u_2)}{u_{\Theta_\infty}^{M,m}(y)-u_2}.
\ee
Let $n\to \infty$ in \eqref{inegstrict} and use the strict concavity once again, 
to get
\begin{align}
\label{inegstrict1}
\liminf_{n\to \infty} \partial^+_u (u\,\psi(\Theta_n,u))
(u_{\Theta_n}^{M,m}(y))&>\partial^-_u (u\,\psi(\Theta_\infty,u))
(u_{\Theta_\infty}^{M,m}(y)).
\end{align}
If $u_{\Theta_\infty}^{M,m}(y)\in (t_{\Theta_\infty},m]$, then \eqref{defdef} implies 
that the right-hand side of \eqref{inegstrict1} is not smaller than $y$. Hence 
\eqref{inegstrict1} yields that $\partial^+_u (u\,\psi(\Theta_n,u))(u_{\Theta_n}^{M,m}(y))
>y$ for $n$ large enough, which implies that $u_{\Theta_n}^{M,m}(y)=m$ by \eqref{defdef}. 
However, the latter inequality contradicts the fact that $u_{\Theta_n}^{M,m}(y)<u_2
<u_{\Theta_\infty}^{M,m}(y)$ for $n$ large enough. If $u_{\Theta_\infty}^{M,m}(y)
= t_{\Theta_\infty}$, then we note that $\lim_{n\to \infty} t_{\Theta_n}
=t_{\Theta_\infty}$, which again contradicts that $t_{\Theta_n}\leq u_{\Theta_n}^{M,m}(y) 
< u_2<u_{\Theta_\infty}^{M,m}(y)$ for $n$ large enough.

\medskip\noindent
(ii) The proof is again by contradiction. Pick $\Theta \in \overline{\cV}_M^{\,m}$, and pick
an infinite sequence $(y_n)_{n\in \N}$ such that $\lim_{n\to\infty} y_n=y_\infty\in \R$ 
and such that $u_{\Theta}^{M,m}(y_n)$ does not converge to $u_{\Theta}^{M,m}(y_\infty)$. Then, 
by choosing an appropriate subsequence, we may assume that there exists a $u_1 < 
u_{\Theta}^{M,m}(y_\infty)$ such that $\lim_{n\to\infty}u_{\Theta}^{M,m}(y_n)=u_1$. The case
$u_1>u_{\Theta}^{M,m}(y_\infty)$ can be treated similarly. 

Pick $u_2,u_3\in (u_1, u_{\Theta}^{M,m}(y_\infty))$ such that $u_2<u_3$. Then, for $n$ large 
enough, we have 
\be{relat}
t_\Theta\leq u_{\Theta}^{M,m}(y_n)<u_2<u_3<u_{\Theta}^{M,m}(y_\infty)\leq m.
\ee
Combining  \eqref{defdef} and \eqref{relat} with the strict concavity of $u\mapsto 
u\psi(\Theta,u)$ we get, for $n$ large enough, 
\be{relat1}
y_n>\partial^+_u (u\,\psi(\Theta,u))(u_2) >\partial^-_u (u\,\psi(\Theta,u))(u_3)
> y_\infty,
\ee
which contradicts $\lim_{n\to \infty} y_n=y_\infty$.
\end{proof}

We resume the line of proof. Recall that $\rho_{n,1}$, $n\in \N$, charges finitely 
many $\Theta\in\overline{\cV}_M^{\,m}$. Therefore the continuity and the strict concavity 
of $u\mapsto u\psi(\Theta,u)$ on $[t_\Theta,m]$ for all $\Theta\in \overline{\cV}_M^{\,m}$ 
(see Lemma \ref{concav}) imply that the supremum in \eqref{immp} is attained at some
$u_{n}^{M,m}\in \cB_{\,\overline{\cV}_M^{\,m}}$ that satisfies $u_{n}^{M,m}(\Theta) 
= u_\Theta^{M,m}(l_n)$ for $\Theta\in \overline{\cV}_M^{\,m}$. Set $u^{M,m}_{\infty}
(\Theta) = u^{M,m}_{\Theta}(l_\infty)$ for $\Theta\in \overline{\cV}_M^{\,m}$ and note 
that $(l_n)_{n\in \N}$ may be assumed to be monotone, say, non-decreasing. Then the 
concavity of $u\mapsto u\psi(\Theta,u)$ for $\Theta\in \overline{\cV}_M^{\,m}$ implies 
that $(u_n^{M,m})_{n\in\N}$ is a non-increasing sequence of functions on $\overline
\cV_M^{\,m}$. Moreover, $\overline{\cV}_M^{\,m}$ is a compact set and, by Lemma
\ref{conti}(ii), $\lim_{n\to \infty} u_n^{M,m}(\Theta)=u_\infty^{M,m}(\Theta)$ 
for $\Theta\in\overline{\cV}_M^{\,m}$. Therefore Dini's theorem implies that 
$\lim_{n\to\infty} u_n^{M,m}=u_\infty^{M,m}$ uniformly on $\overline{\cV}_M^{\,m}$. 
We estimate
\begin{align}
\label{inegd}
\nonumber&\left|l_n-\int_{\overline\cV_M^{\,m}}  
u^{M,m}_{\infty}(\Theta)\,\psi(\Theta,u^{M,m}_{\infty}(\Theta))
\rho_{\infty}(d\Theta)\right|\\
 &\qquad\leq \int_{\overline\cV_M^{\,m}} \Big|u^{M,m}_{n}(\Theta)\,
\psi(\Theta,u^{M,m}_{n}(\Theta))-u^{M,m}_{\infty}(\Theta)\,
\psi(\Theta,u^{M,m}_{\infty}(\Theta))\Big|\,\rho_{n}(d\Theta)\\
\nonumber&\qquad +\Big|\int_{\overline\cV_M^{\,m}} u^{M,m}_{\infty}(\Theta)\, 
\psi(\Theta,u^{M,m}_{\infty}(\Theta))\,\rho_{n}(d\Theta)
-\int_{\overline\cV_M^{m}} u^{M,m}_{\infty}(\Theta)\,
\psi(\Theta,u^{M,m}_{\infty}(\Theta))\, \rho_{\infty}(d\Theta)\Big|.
\end{align}
The second term in the right-hand side of \eqref{inegd} tends to zero as $n\to\infty$ 
because, by Lemma \ref{conti}(i), $\Theta \mapsto u_{\infty}^{M,m}(\Theta)$ is continuous 
on $\overline{\cV}_M^{\,m}$ and because $\rho_n$ converges in law to $\rho_\infty$ as 
$n\to\infty$. The first term in the right-hand side of \eqref{inegd} tends to zero as 
well, because $(\Theta,u)\mapsto u\psi(\Theta,u)$ is uniformly continuous on 
$\overline{\cV}^{\,*,m}_M$ (see Lemma~\ref{concavt}) and because we have proved above 
that $u^{M,m}_{n}$ converges to $u^{M,m}_{\infty}$ uniformly on $\overline\cV_M^{\,m}$. 
This proves \eqref{eqfin}, and so Step 3 is complete.


\subsubsection{Step 4}
\label{s24alt}

In this step  we prove that 
\be{limsuplar}
\limsup_{n\to \infty} f_{3,n}^{\Omega}(M,m;\alpha,\beta)\geq f(M,m;\alpha,\beta)\, 
\text{ for } \P-a.e.\ \Omega.
\ee
Note that the proof will be complete once we show that 
\be{pr}
\limsup_{n\to \infty} f_{3,n}^\Omega(M,m,
\alpha,\beta)\geq V(\rho,u) \ \text{for} \ \rho\in \cR_{p,M}^m,  
u\in \cB_{\,\overline \cV_M^{\,m}}. 
\ee
Pick $\Omega\in \{A,B\}^{\N_0\times \Z}$, $\rho\in \cR^{\Omega,m}_{p,M}$ and $u\in
\cB_{\,\overline \cV_M^{\,m}}$. By the definition of $\cR_{p,M}^{\Omega,m}$, there exists a 
strictly increasing subsequence $(n_k)_{k\in \N}\in \N^\N$ such that, for all 
$k\in \N$, there exists an 
\be{Nkexist}
N_k\in\left\{\frac{n_k}{m L_{n_k}},\dots,
\frac{n_k}{L_{n_k}}\right\},
\ee
a $\Theta_{\text{traj}}^k\in \widetilde{\cD}_{L_{n_k},N_k}^M$ and a $x^k\in
\cX_{\Theta_{\text{traj}}^k,\Omega}^{M,m}$ such that $\rho_k=^\mathrm{def}
\rho_{\Theta_{\text{traj}}^k,x^k}^\Omega$ (see \eqref{defrho2}) converges in 
law to $\rho$ as $k\to\infty$. Recall \eqref{defD2}, and note that 
\begin{align}
\Xi_j^k
=\big(\Delta\Pi^k_j,\,b^k_j,\,b_{j+1}^k\big),
\quad \text{$j=0,\dots,N_k-1$},
\end{align}
with $\Delta\Pi^k_j\in \{-M,\dots,M\}$ and $b_j^k\in (0,1]
\cap\frac{\N}{L_{n_k}}$ for $j=0,\dots,N_k$. For ease of notation we define 
\be{notat}
\Theta_{j}^k
=\big(\Omega(j,\Pi^k_j+\cdot),\Xi_j^k,x_j^k\big)\quad \text{with} \quad 
\Pi_j^k=\sum_{i=0}^{j-1} \Delta\Pi_i^k, \qquad  j=0,\dots,N_k-1,
\ee
and
\be{defv}
v_k=N_k \int_{\Theta\in \cV_M^{\,m}} u_\Theta\,\rho_{k,1}(d\Theta)
= \sum_{j=0}^{N_k-1} u_{\Theta_j^k},
\ee
where we recall that $u=(u_\Theta)_{\Theta\in \overline\cV_M^{\,m}}$ was fixed at the 
beginning of the section.

Next, we recall that $\lim_{n\to\infty} L_n/n=0$ and that $L_n$ is non-decreasing (see 
\eqref{speed}). Thus, $L_n$ is constant on intervals. On those intervals, $n/L_n$ takes 
constant increments. The latter implies that there exists an $\widetilde{n}_k\in \N$ satisfying
\be{deftil}
0\leq v_k-\tfrac{\widetilde{n}_k}{L_{\widetilde{n}_k}} 
\leq \tfrac{1}{L_{\widetilde{n}_k}}
\quad \text {and therefore}\quad 
0\leq  v_k L_{\widetilde{n}_k}-\widetilde{n}_k \leq 1.
\ee
Next, for $j=0,\dots,N_k-1$ we pick $\overline{b_j^k}\in (0,1] \cap 
\frac{N}{L_{\widetilde{n}_k}}$ such that $|\overline{b_j^k}-b_j^k|\leq 
\tfrac{1}{L_{\widetilde{n}_k}}$, define
\be{defbar}
\overline{\Xi_j^k} = \big(\Delta \Pi^k_j,
\overline{b_j^k},\overline{b_{j+1}^k}\,\big), 
\quad \overline{\Theta_j^k}
= \big(\Omega(j,\Pi^k_j+\cdot),\overline{\Xi_j^k},x_j^k\,\big),
\ee
and pick
\be{defs}
s_j^k\in t_{\,\overline{\Theta_j^k}}+\frac{2\N}{L_{\widetilde{n}_k}}\quad
\text{such that} \quad |s_j^k-u_{\Theta_j^k}|\leq 2/L_{\widetilde{n}_k}.
\ee
We use \eqref{defv} to write
\be{estis}
L_{\widetilde{n}_k} \sum_{j=0}^{N_k-1} s^k_j=L_{\widetilde{n}_k} 
\bigg(v_k +\sum_{j=0}^{N_k-1} (s_j^k-u_{\Theta_j^k})\bigg)
= L_{\widetilde{n}_k} (I+II).
\ee
Next, we note that \eqref{deftil} and \eqref{defs} imply that $|L_{\widetilde{n}_k} 
I-\widetilde{n}_k|\leq 1$ and $|L_{\widetilde{n}_k} II|\leq 2 N_k$. The latter 
in turn implies that, by adding or subtracting at most 3 steps per colum, the 
quantities $s_j^k$ for $j=0,\dots,N_k-1$ can be chosen in such a way that 
$\sum_{j=0}^{N_k-1} s_j^k=\widetilde{n}_k/L_{\widetilde{n}_k}$.

Next, set 
\be{}
\overline{\Theta_{\text{traj}}^k}=(\overline{\Xi_j^k})_{j=0}^{N_k-1}
\in \widetilde{\cD}_{L_{\widetilde{n}_k},N_k }^M,
\quad s^k=(s^k_j)_{j=0}^{N_k-1}
\in \cU_{\overline{\Theta_{\text{traj}}^k},\,
x^k,\widetilde{n}_k}^{M,m,L_{\widetilde{n}_k}},
\ee
and recall \eqref{inteerm1} to get $f_3^\Omega(\widetilde{n}_k,M)\geq R_k$ with
\be{rk}
R_k=\frac{L_{\widetilde{n}_k} \,
H^\Omega\big(\,\overline{\Theta_{\text{traj}}^k},x^k,s^k\,\big)}
{\widetilde{n}_k}= \frac{\sum_{j=0}^{N_k-1}\, s_j^k\,
\psi\Big(\overline{\Theta_j^k},\,s_j^k\Big)}{\sum_{j=0}^{N_k-1} s_j^k}
=\frac{R_{\text{nu}}^k}{R_{\text{de}}^k}.
\ee
Further set
\be{defR} 
R^{'}_k=\frac{R^{' k}_{\text{nu}}}{R^{' k}_{\text{de}}}
=\frac{\int_{\overline\cV_M^{\,m}} u_\Theta\, \psi(\Theta,u_\Theta) 
\rho_{k}(d\Theta)}{\int_{\overline\cV_M^{\,m}} u_\Theta \,\rho_{k}(d\Theta)},
\ee
and note that $\lim_{k\to\infty} R^{'}_k=V(\rho,u)$, since $\lim_{k\to\infty}
\rho_k=\rho$ by assumption and $\Theta\mapsto u_{\Theta}$ is continuous on 
$\cV_M^{\,m}$. We note that $R^{'}_k$ can be rewritten in the form 
\be{defRalt} 
R^{'}_k=\frac{R^{' k}_{\text{nu}}}{R^{' k}_{\text{de}}}
= \frac{\sum_{j=0}^{N_k-1}\, u_{\Theta_j^k}\,
\psi\big(\Theta_j^k,\,u_{\Theta_j^k}\big)}
{\sum_{j=0}^{N_k-1}\, u_{\Theta_j^k}}.
\ee
Now recall that $\lim_{k\to \infty} n_k=\infty$. Since $N_k\geq n_k/M L_{n_k}$, it 
follows that $\lim_{k\to \infty} N_k=\infty$ as well. Moreover, $N_k\leq\widetilde{n}_k
/L_{\widetilde{n}_k}$ with $\lim_{k\to\infty} \widetilde{n}_k=\infty$. Therefore 
(\ref{defv}--\ref{deftil}) allow us to conclude that $R_{\text{de}}^k
=\widetilde{n}_k/L_{\widetilde{n}_k}=R^{'k}_{\text{de}}[1+o(1)]$.

Next, note that $\cH_M$ is compact, and that $(\Theta,u)\mapsto u \psi(\Theta,u)$ is 
continuous on $\cH_M$ and therefore is uniformly continuous. Consequently, for 
all $\gep>0$ there exists an $\eta>0$ such that, for all $(\Theta,u),(\Theta^{'},u^{'})
\in \cH_M$ satisfying $|\Theta-\Theta^{'}|\leq \eta$ and $|u-u^{'}|\leq \eta$,
\be{unifcont}
|u\psi(\Theta,u)-u^{'}\psi(\Theta^{'},u^{'})|\leq \gep.
\ee
We recall \eqref{defbar}, which implies that $d_M(\overline{\Theta_j^k},\Theta_j)
\leq 2/L_{\widetilde{n}_k}$ for all $j\in \{0,\dots,N_k-1\}$, we choose $k$ large enough to 
ensure that $2/L_{\widetilde{n}_k}\leq \eta$, and we use \eqref{unifcont}, to obtain
\be{equa} 
R^{k}_{\text{nu}}=\sum_{j=0}^{N_k-1}\, 
s_j^k\, \psi\Big(\overline{\Theta_j^k},\,s_j^k\Big)
= \sum_{j=0}^{N_k-1}\, u_{\Theta_j^k}\, 
\psi\big(\Theta_j^k,\,u_{\Theta_j^k}\big)+T=R^{' k}_{\text{nu}}+T,
\ee
with $|T| \leq \gep N_k$. Since $\lim_{k\to\infty} R^{'}_k=V(\rho,u)$ and 
$\sum_{j=0}^{N_k-1} u_{\Theta_j^k}=v_k\geq \widetilde{n}_k/L_{\widetilde{n}_k}$ 
(see \eqref{deftil}), if $ V(\rho,u)\neq 0$, then $\big|R^{' k}_{\text{nu}}\big|
\geq \mathrm{Cst.}\,\,\widetilde{n}_k/L_{\widetilde{n}_k}$, whereas $|T|\leq \gep N_k
\leq  \gep \widetilde{n}_k/L_{\widetilde{n}_k} $ for $k$ large enough. Hence 
$T=o(R^{' k}_{\text{nu}})$ and 
\be{fina}
\frac{R^{k}_{\text{nu}}}{R^{k}_{\text{de}}}
= \frac{R^{' k}_{\text{nu}}\, [1+o(1)]}{R^{'k}_{\text{de}}\,
[1+o(1)]}\to V(\rho,u), \qquad k\to \infty. 
\ee
Finally, if $ V(\rho,u)= 0$, then $R^{' k}_{\text{nu}}=o(R^{'k}_{\text{de}})$ and 
$T=o(R^{'k}_{\text{de}})$, so that $R_k$ tends to $0$. This completes the proof
of Step 4.


\subsubsection{Step 5}
\label{s23alt}

In this step we prove \eqref{convco1}, suppressing the $(\alpha,\beta)$-dependence 
from the notation. For $\Omega\in \{A,B\}^{\N_0\times \Z^2}$, $n\in \N$, $N\in
\{n/m L_n,\dots,n/L_n\}$ and $r\in \{-N M,\dots,N M\}$, we recall \eqref{defD2} and 
define
\be{defDD2alt}
\widetilde{\cD}_{L,N}^{M,m,r}= \Big\{\Theta_{\text{traj}}
\in \widetilde{\cD}_{L,N}^{M,m}\colon \Pi_N=r\Big\},
\ee
where we recall that $\Pi_N=\sum_{j=0}^{N-1} \Delta\Pi_j$. We set 
\begin{equation}
\label{intermalt}
f_{3,n}^\Omega(M,m,N,r) = \tfrac{1}{n} \log Z_{3,n,L_n}^{\Omega}(N,M,m,r)
\ee
with
\be{iiiter}
Z_{3,n,L_n}^{\Omega}(N,M,m,r)
= \sum_{\Theta_{\text{traj}}\in \widetilde{\cD}_{L_n,N}^{M,m,r} }
\,\sum_{x\in \cX_{\Theta_{\text{traj}},\Omega}^{M,m} }
\,\sum_{u\in \,\cU_{\Theta_{\text{traj}},n}^{\,M,m,L_n}} A_3,
\end{equation}
where $A_3$ is defined in \eqref{defa23}. We further set $f_3(\cdot)=\E_\Omega
\big(f_3^\Omega(\cdot)\big)$.


\subsubsection{Concentration of measure} 

In the first part of this step we prove that  for all $(M,m,\alpha,\beta)\in
\EIGH\times \CONE$ there exist $c_1,c_2>0$ (depending on $(M,m,\alpha,\beta)$ only) 
such that, for all  $n\in \N$,  $N\in \{n/(m L_n),\dots n/L_n\}$ and $r\in
\{-N M,\dots,N M\}$, 
\begin{align}
\label{conco}
&\P_{\Omega}\big(\big|f_{3,n}^{\Omega}(M,m)-f_{3,n}(M,m)\big|>\gep\big)
\leq c_1\  e^{-\frac{c_2 \gep^2 n}{L_n}}, \qquad \\
\nonumber 
&\P_{\Omega}\big(\big|f_{3,n}^{\Omega}(M,m,N,r)-f_{3,n}(M,m,N,r)\big|>\gep\big)
\leq c_1\  e^{-\frac{c_2 \gep^2  n}{L_n}}.
\end{align} 
We only give the proof of the first inequality. The second inequality is proved 
in a similar manner. The proof uses Theorem~\ref{theoco}. Before we start we note 
that, for all $n\in \N$, $(M,m)\in \EIGH$ and $\Omega\in \{A,B\}^{\N_0\times \Z}$, 
$f_{3,n}^\Omega(M,m)$ only depends on 
\be{defC}
\cC_{0,L_n}^\Omega,\dots,\cC_{n/L_n,L_n}^\Omega \quad \text{with}\quad 
\cC_{j,L_n}^\Omega=(\Omega(j,i))_{i=-n/L_n}^{n/L_n}.
\ee
We apply Theorem~\ref{theoco} with $\cS=\{0,\dots,n/L_n\}$, with  $X_i=\{A,B\}^{\{
-\frac{n}{L_n},\dots,\frac{n}{L_n}\}}$ and with $\mu_i$ the uniform measure on 
$X_i$ for all $i\in \cS$. Note that $|f_{3,n}^{\Omega_1}(M,m)-f_{3,n}^{\Omega_2}(M,m)|
\leq 2 C_{\text{uf}}(\alpha) m \tfrac{L_n}{n}$ for all $i\in \cS$ and all 
$\Omega_1,\Omega_2$ satisfying $\cC_{j,n}^{\Omega_1}=\cC_{j,n}^{\Omega_2}$ 
for all $j\neq i$. After we set $c=2 C_{\text{uf}}(\alpha) m$ we can apply 
Theorem~\ref{theoco} with $D=c^2 L_n/n$ to get \eqref{conco}. 

Next, we note that the first inequality in \eqref{conco}, the Borel-Cantelli lemma 
and the fact that $\lim_{n\to \infty} n/L_n \log n = \infty$ (recall \eqref{speed}) imply that, for all 
$(M,m)\in \EIGH$,
\be{limialt}
\lim_{n\to \infty} \Big[f_{3,n}^\Omega(M,m)- f_{3,n}(M,m)\Big]=0 
\quad \text{for } \P-a.e.~\Omega.
\ee 
Therefore \eqref{convco1} will be proved once we show that 
\be{fin1}
\liminf_{n\to \infty} f_{3,n}(M,m)
= \limsup_{n\to \infty} f_{3,n}(M,m). 
\ee
To that end, we first prove that, for all $n\in \N$ and all $(M,m)\in \EIGH$, there exist 
an $N_n\in \{n/m L_n,\dots,n/L_n\}$ and an $r_n\in \{-M N_n,\dots,M N_n\}$ such that 
\be{fin2}
\lim_{n\to \infty} \Big[f_{3,n}(M,m)-f_{3,n}(M,m,N_n,r_n)\Big] = 0.
\ee
The proof of \eqref{fin2} is done as follows. Pick $\gep>0$, and for $\Omega\in
\{A,B\}^{\N_0\times \Z}$, $n\in \N$ and $(M,m)\in \EIGH$, denote by $N_n^\Omega$ 
and $r_n^{\Omega}$ the maximizers of $f^\Omega_{3,n}(M,m,N,r)$. Then 
\be{rrde}
f^\Omega_{3,n} \big(M,m,N_n^{\Omega},r_n^{\Omega}\big)\leq f^\Omega_{3,n}(M,m)
\leq  \tfrac1n \log (\tfrac{n^2}{L_n^2})
+ f^\Omega_{3,n}\big(M,m,N_n^{\Omega},r_n^{\Omega}\big),
\ee
so that, for $n$ large enough and every $\Omega$, 
\be{recty}
0\leq f^\Omega_{3,n}(M,m)-f^\Omega_{3,n}\big(M,m,N_n^{\Omega},r_n^{\Omega}\big)\leq \gep.
\ee
For $n\in \N$, $N\in \{n/m L_n,\dots, n/L_n\}$ and $r\in \{-N M,\dots, N M\},$ we set
\be{defAnNr}
A_{n,N,r}=\{\Omega\colon (N_n^\Omega,r_n^\Omega)=(N,r)\}.
\ee
Next, denote by $N_n,r_n$ the maximizers of $\P(A_{n,N,r})$. Note that \eqref{fin2} 
will be proved once we show that, for all $\gep>0$, $|f_{3,n}(M,m)-f_{3,n}(M,m,N_n,r_n)|\leq
\gep$ for $n$ large enough. Further note that $\P(A_{n,N_n,r_n})\geq L_n^2/n^2$ for 
all $n\in \N$. For every $\Omega$ we can therefore estimate
\be{iimp1}
|f_{3,n}(M,m)-f_{3,n}(M,m,N_n,r_n)|\leq I+II+III
\ee 
with
\begin{align}
\label{iimp}
&I=|f_{3,n}(M,m)-f_{3,n}^\Omega(M,m)|,\\
\nonumber & II=|f_{3,n}^\Omega(M,m)-f_{3,n}^\Omega(M,m,N_n,r_n)|, \\
\nonumber &III=|f_{3,n}^\Omega(M,m,N_n,r_n)-f_{3,n}(M,m,N_n,r_n)|.
\end{align}
Hence, the proof of \eqref{fin2} will be complete once we show that, for $n$ large 
enough, there exists an $\Omega_{\gep,n}$ for which $I,II$ and $III$ in \eqref{iimp} 
are bounded from above by $\gep/3$.

To that end, note that, because of \eqref{conco}, the probabilities $\P(\{I> \gep/3\})$ 
and $\P(\{III> \gep/3\})$ are bounded from above by $c_1 e^{-c_2 \gep^2 n/9 L_n}$, 
while 
\be{boundA}
\P(\{II>\gep\})\leq \P(A_{n,N_n,r_n}^c)\leq 1-(L_n^2/n^2), \qquad n\in \N.
\ee
Since $\lim_{n\to \infty} n/L_n \log n=\infty$, we have $\P(\{I,II,III \leq \gep/3\})>0$ 
for $n$ large enough. Consequently, the set $\{I,II,III\leq \gep/3\}$ is non-empty 
and \eqref{fin2} is proven.


\subsubsection{Convergence}\label{saltisgp}

It remains to prove \eqref{fin1}. Assume that there exist two strictly increasing subsequences 
$(n_k)_{k\in\N}$ and $(t_k)_{k\in\N}$ and two limits $l_2>l_1$ such that $\lim_{k\to\infty} 
f_{3,n_k}(M,m)=l_2$ and $\lim_{k\to\infty}$ $f_{3,t_k}(M,m) =l_1$. By using \eqref{fin2}, we 
have that for every $k\in\N$ there exist $N_k\in\{n_k/m L_{n_k},$ $\dots,n_k/L_{n_k}\}$ and 
$r_k\in \{-M N_k,\dots,M N_k\}$ such that $\lim_{k\to\infty}$ $f_{3,n_k}(M,m,N_k,r_k)
=l_2$. Denote by 
\be{tdt}
(\Theta_{\text{traj,max}}^{k,\Omega},
x_{\text{max}}^{k,\Omega},u_{\text{max}}^{k,\Omega})\in
\widetilde{\cD}_{L_{n_k},N_k}^{M,r_k}\times 
\cX_{\Theta_{\text{traj,max}}^{k,\Omega},\Omega}^{M,m}
\times  \cU_{\Theta_{\text{traj,max}}^{k,\Omega},
x_{\text{max}}^{k,\Omega},n_k}^{\,M,m,L_n}
\ee
the maximizer of $H^\Omega(\Theta_{\text{traj}},x,u)$. We recall that $\Theta_{\text{traj}},
x$ and $u$ take their values in sets that grow subexponentially fast in $n_k$, and therefore
\be{equan}
\lim_{k\to \infty} \tfrac{L_{n_k}}{n_k}\,
\E_\Omega\big[H^\Omega(\Theta_{\text{traj,max}}^{k,\Omega},x_{\text{max}}^{k,\Omega},
u_{\text{max}}^{k,\Omega})\big]=l_2.
\ee
Since $l_2>l_1$, we can use \eqref{equan} and the fact that $\lim_{k\to \infty} n_k/L_{n_k}
=\infty$ to obtain, for $k$ large enough, 
\be{equan1}
\E_\Omega\big[H^\Omega(\Theta_{\text{traj,max}}^{k,\Omega},
x_{\text{max}}^{k,\Omega},u_{\text{max}}^{k,\Omega})\big]
+(\beta-\alpha)\geq \tfrac{n_k}{L_{n_k}} \big(l_1+\tfrac{l_2-l_1}{2}\big).
\ee
(The term $\beta-\alpha$ in the left-hand side of \eqref{equan1} is introduced for 
later convenience only.) Next, pick $k_0\in \N$ satisfying \eqref{equan1}, whose value 
will be specified later. Similarly to what we did in \eqref{defs} and \eqref{estis}, 
for $\Omega\in \{A,B\}^{\N_0\times \Z}$ and $k\in \N$ we associate with 
\be{asso}
\Theta_{\text{traj,max}}^{k_0,\Omega}
=\big(\Delta\Pi_{\,j}^{k_0,\Omega}, b^{k_0,\Omega}_{\,0,j}, 
b_{\,1,j}^{k_0,\Omega}\big)_{j=0}^{N_{k_0}-1}\in 
\widetilde{\cD}_{L_{n_{k_0}},N_{k_0}}^{M,r_{k_0}}
\ee
and 
\be{assocc}
x_{\text{max}}^{k_0,\Omega}=\big(x_{\,j}^{k_0,\Omega}\big)_{j=0}^{N_{k_0}-1}
\in \cX_{\Theta_{\text{traj,max}}^{k_0,\Omega},\Omega}^{M,m}
\ee
and 
\be{assoc}
u_{\text{max}}^{k_0,\Omega}=\big(u_{\,j}^{k_0,\Omega}\big)_{j=0}^{N_{k_0}-1}
\in \cU_{\Theta_{\text{traj,max}}^{k_0,\Omega},
x_{\text{max}}^{k_0,\Omega},n_{k_0}}^{M,m,L_{n_{k_0}}}
\ee
the quantities
\be{defTh}
\overline\Theta^{\,k,\Omega}_{\text{traj}}
= \big(\Delta\Pi_{\,j}^{k_0,\Omega}, \overline b_{\,0,j}^{\,k,\,\Omega}, 
\overline b_{\,1,j}^{\,k,\,\Omega}\big)_{j=0}^{N_{k_0}-1}
\in \widetilde{\cD}_{L_{t_{k}},N_{k_0}}^{M,r_{k_0}}
\ee  
and 
\be{assoc1}
\overline u^{\,k,\Omega}=\big(\overline u_{\,j}^{\,k,\Omega}\big)_{j=0}^{N_{k_0}-1}
\in \cU_{\overline\Theta_{\text{traj}}^{\,k,\Omega},
x_{\text{max}}^{k_0,\Omega},*}^{M,m,L_{t_{k}}}
\ee
(where $*$ will be specified later), so that 
\be{supe}
\big|\overline b_{\,0,j}^{\,k,\,\Omega}-b_{\,0,j}^{k_0,\Omega}
\big|\leq \tfrac{1}{L_{t_k}},\ \ 
\big|\overline b_{\,1,j}^{\,k,\,\Omega}-b_{\,1,j}^{k_0,\Omega}\big|
\leq \tfrac{1}{L_{t_k}},
\ \ \ \ \big|\overline u_{\,j}^{\,k,\,\Omega}-u_{\,j}^{k_0,\Omega}\big|
\leq \tfrac{2}{L_{t_k}}, 
\ \ \ j=0,\dots,N_{k_0}-1.
\ee
Next, put $\overline s_{k}^{\,\Omega}=L_{t_k} \sum_{j=0}^{N_{k_0}-1} 
\overline u_j^{\,k,\,\Omega}$, which we substitute for $*$ above. The uniform 
continuity in Lemma~\ref{concavt} allows us to claim that, for $k$ large enough 
and for all $\Omega$,
\be{trt}
\Big|\overline u_j^{\,k,\,\Omega} \ \psi\Big(\overline\Theta_j^{\,k,\,\Omega},
\overline u_j^{\,k,\,\Omega}\Big)
-u_{\,j}^{k_0,\Omega}\  \psi\Big(\Theta_{\,j}^{k_0,\Omega},
u_{\,j}^{k_0,\Omega}\Big)\Big|\leq \tfrac{l_2-l_1}{4},
\ee
where we recall that, as in \eqref{notat}, for all $j=0,\dots,N_{k_0}-1$,  
\begin{align}
\label{restr}
\overline\Theta_j^{\,k,\,\Omega}
&=\Big(\Omega\big(j,\Pi^{k_0,\Omega}_{\,j}+\cdot\big),\,\Delta\Pi_{\,j}^{k_0,\Omega}, 
\,\overline b_{\,0,j}^{\,k,\,\Omega},\, 
\overline b_{\,1,j}^{\,k,\,\Omega},\,x_j^{k_0,\Omega}\Big),\\
\nonumber \Theta_{\,j}^{k_0,\Omega}
&=\Big(\Omega\big(j,\Pi^{k_0,\Omega}_{\,j}+\cdot\big),\,\Delta\Pi_j^{k_0,\Omega}, 
\,b_{\,0,j}^{k_0,\Omega},\, b_{\,1,j}^{k_0,\Omega},\,x_j^{k_0,\Omega}\Big).
\end{align}
Recall \eqref{defHam}. An immediate consequence of \eqref{trt} is that 
\be{reft}
\big|H^\Omega(\overline\Theta_{\text{traj}}^{\,k,\Omega},
x_{\text{max}}^{k_0,\Omega},\overline u^{\,k,\Omega})
-H^\Omega(\Theta_{\text{traj,max}}^{k_0,\Omega},
x_{\text{max}}^{k_0,\Omega},u_{\text{max}}^{k_0,\Omega})\big|
\leq N_{k_0} \tfrac{l_2-l_1}{4}.
\ee
Hence we can use \eqref{equan1}, \eqref{reft} and the fact that $N_{k_0}\leq 
n_{k_0}/L_{n_{k_0}}$, to conclude that, for $k$ large enough,
\be{equan2}
\E_\Omega\big[H^\Omega(\overline\Theta_{\text{traj}}^{\,k,\Omega},
x_{\text{max}}^{k_0,\Omega},\overline u^{\,k,\Omega})\big]+(\beta-\alpha)
\geq \tfrac{n_{k_0}}{L_{n_{k_0}}}\big(l_1+\tfrac{l_2-l_1}{4}\big).
\ee

At this stage we add a column at the end of the group of $N_{k_0}$ columns in such 
a way that the conditions $\widehat b^{k,\Omega}_{1, N_{k_0}-1}=\widehat 
b^{k,\Omega}_{0, N_{k_0}}$ and $\widehat b^{k,\Omega}_{1, N_{k_0}}=1/L_{t_k}$ 
are satisfied. We put 
\be{defsk}
\widehat{\Xi}_{N_{k_0}}^{k,\,\Omega}
=\big(\Delta \Pi_{N_{{k_0}}}^{k_0,\Omega},\widehat b_{0,N_{k_0}}^{k,\,\Omega},
\widehat b_{1,N_{k_0}}^{k,\,\Omega}\big)=\big(0,\widehat b_{1,N_{k_0}-1}^{k,\,\Omega},
\tfrac{1}{L_{t_k}}\big),
\ee
and we let $\widehat{\Theta}_{\text{traj}}^{k,\,\Omega}\in \widetilde{\cD}_{L_{t_k},
\,N_{k_0}+1}^{M,\,r_{k_0}}$ be the concatenation of $\overline \Theta_{\text{traj}}^{k,
\Omega}$ (see \eqref{defTh}) and $\widehat{\Xi}_{N_{k_0}}^{k,\Omega}$. We let 
$\widehat{x}^{k_0,\Omega}\in\cX_{\widehat\Theta_{\text{traj}}^{k,\Omega},\Omega}^{M,m}$ 
be the concatenation of $x_{\text{max}}^{k_0,\Omega}$ and $0$. We further let 
\be{defshat}
\widehat{s}_k^{\,\Omega}=\overline s_k^{\,\Omega}+\Big[1+b_{1,N_{k_0}-1}^{k,\Omega}
-\tfrac{1}{L_{t_k}}\Big] L_{t_k},
\ee 
and we let $\widehat{u}^{k,\Omega}\in\cU_{\,\widehat{\Theta}_{\text{traj}}^{k,\
\Omega},\, \widehat x^{k_0,\Omega},\,\widehat{s}_k^{\,\Omega}}^{M,m,\,L_{t_k}}$ be 
the concatenation of $\overline u^{\,k,\Omega}$ (see \eqref{assoc1}) and
\be{defunk0}
\widehat u_{N_{k_0}}^{k,\Omega}=1+(b_{1,N_{k_0}-1}^{k,\Omega}-\tfrac{1}{L_{t_k}}).
\ee
Next, we note that the right-most inequality in \eqref{supe}, together with the fact 
that
\be{}
\sum_{j=0}^{N_{k_{0}}-1} u_{\,j}^{k_0,\Omega} = n_{k_0}/L_{n_{k_0}},
\ee 
allow us to asset that $|\overline s_k^{\,\Omega}-L_{t_k} n_{k_0}/L_{n_{k_0}}|
\leq 2 N_{k_0}$. Therefore the definition of $\widehat{s}_k^{\,\Omega}$ in 
\eqref{defshat} implies that  
\be{impo*}
\widehat{s}_k^{\,\Omega}=L_{t_k} \frac{n_{k_0}}{L_{n_{k_0}}}
+\widehat{m}_k^\Omega \quad \text{with} \quad |\widehat{m}_k^{\Omega}|
\leq 2 N_{k_0}+2 L_{t_k}.
\ee
Moreover,
\be{httraj}
H^\Omega\big(\widehat{\Theta}_{\text{traj}}^{k,\Omega},
\widehat x^{k_0,\Omega},\widehat{u}^{k,\Omega}\big)
\geq H^\Omega \big(\overline\Theta_{\text{traj}}^{\,k,\Omega},
x_{\text{max}}^{k_0,\Omega}, \overline u^{\,k,\Omega}\big)
+(\beta-\alpha),
\ee
because $\widehat u_{N_{k_0}}^{k,\Omega}\leq 2$ by definition (see \eqref{defunk0}) 
and the free energies per columns are all bounded from below by $(\beta-\alpha)/2$. 
Hence, \eqref{equan2} and \eqref{httraj} give that for all $\Omega$ there exist a
\be{redefTh}
\widehat{\Theta}^{k,\Omega}_{\text{traj}}\in
\widetilde{\cD}_{L_{t_{k}},\,N_{k_0}+1}^{\,M,\,r_{k_0}}
\colon\,b_{1,N_{k_0}}=\tfrac{1}{L_{t_k}},
\ee
an $\widehat{x}^{k_0,\Omega}\in\cX_{\widehat\Theta_{\text{traj}}^{k,\Omega},\Omega}^{M,m}$
and a $\widehat{u}^{k,\Omega}\in\cU_{\,\widehat{\Theta}_{\text{traj}}^{k,\,\Omega},\,
\widehat{x}^{k_0,\Omega},\,\widehat{s}_k^{\,\Omega}}^{M,m,\,L_{t_k}}$ such that, 
for $k$ large enough,
\be{rth}
\E_{\,\Omega}\big[H(\widehat{\Theta}_{\text{traj}}^{k,\Omega},
\widehat x^{k_0,\Omega},
\widehat{u}^{k,\Omega})\big]\geq \tfrac{n_{k_0}}{L_{n_{k_0}}}
\big(l_1+\tfrac{l_2-l_1}{4}).
\ee

Next, we subdivide the disorder $\Omega$ into groups of $N_{k_0}+1$ consecutive columns 
that are successively translated by $r_{k_0}$ in the vertical direction, i.e.,
$\Omega=(\Omega_1,\Omega_2,\dots)$ with (recall \eqref{add1})
\be{defO}
\Omega_j=\big(\Omega(i,\,(j-1)\, r_{k_0}+\cdot)\big)_{i=(j-1)
(N_{k_0}+1)}^{\,j (N_{k_0}+1)-1},
\ee
and we let $q_k^\Omega$  be the unique integer satisfying 
\be{defqo}
\widehat{s}_k^{\, \Omega_1}+\widehat{s}_k^{\, \Omega_2}+ \dots +
\widehat{s}_k^{\,\Omega_{q_k}}\leq t_k
<\widehat{s}_k^{\,\Omega_1}+\dots+\widehat{s}_k^{\,\Omega_{q_k+1}},
\ee
where we suppress the $\Omega$-dependence of $q_k$. We recall that
\begin{equation}
\label{interm1alt}
f_{3,t_k}^\Omega(M,m) =
\E\Bigg[\frac{1}{t_k} \log \sum_{N=t_k/m L_{t_k}}^{t_k/L_{t_k}}\  
\sum_{\Theta_{\text{traj}}\in \widetilde{\cD}_{L_{t_k},N}^{M}}\
\,\sum_{x\in \cX_{\Theta_{\text{traj}},\Omega}^{M,m} } 
\,\sum_{u\in\,\cU_{\,\Theta_{\text{traj}},\,x,\,t_k}^{M,m,\,L_{t_k}}} e^{L_{t_k}\,
H^\Omega(\Theta_{\text{traj}},x,u)}\Bigg],
\end{equation}
set $\widetilde{t}_k^{\,\Omega}=\widehat{s}_k^{\,\Omega_1}+\widehat{s}_k^{\,\Omega_2}
+\dots+\widehat{s}_k^{\,\Omega_{q_k}}$, and concatenate 
\be{conca}
\widehat{\Theta}^{k,\Omega}_{\text{traj,tot}}
=\Big(\widehat{\Theta}^{k,\Omega_1}_{\text{traj}},
\widehat{\Theta}^{k,\Omega_2}_{\text{traj}},\dots,
\widehat{\Theta}^{k,\Omega_{q_k}}_{\text{traj}}\Big) 
\in \widetilde{\cD}_{L_{t_k},\,q_k (N_{k_0}+1),\,}^{M,}
\ee
and
\be{concacat}
\widehat{x}^{k,\Omega}_{\text{tot}}=\big(\widehat x^{k_0,\Omega_1},
\widehat x^{k_0,\Omega_2},\dots,\widehat x^{k_0,\Omega_{q_k}}\big) 
\in \cX_{\widehat{\Theta}^{k,\Omega}_{\text{traj,tot}}
\Omega}^{M,m}.
\ee
and 
\be{concaca}
\widehat{u}^{k,\Omega}_{\text{tot}}=\big(\widehat{u}^{k,\Omega_1},
\widehat{u}^{k,\Omega_2},\dots,\widehat{u}^{k,\Omega_{q_k}}\big) 
\in \cU_{\widehat{\Theta}^{k,\Omega}_{\text{traj,tot}},
\widehat{x}^{k,\Omega}_{\text{tot}},
\widetilde{t}_k^{\,\Omega}}^{M,m,L_{t_k}}.
\ee
It still remains to complete $\widehat{\Theta}^{k,\Omega}_{\text{traj,tot}}$, 
$\widehat{x}^{k,\Omega}_{\text{tot}}$ and $\widehat{u}^{k,\Omega}_{\text{tot}}$ 
such that the latter becomes an element of $\cU_{\widehat{\Theta}^{k,
\Omega}_{\text{traj,tot}},\widehat{x}^{k,\Omega}_{\text{tot}},t_k}^{M,m,L_{t_k}}$. 
To that end, we recall \eqref{defqo}, which gives $t_k-\widetilde{t}_k^{\,\Omega}
\leq \widehat{s}_k^{\Omega_{q_k+1}}$. Then, using \eqref{impo*}, we have that 
there exists a $c>0$ such that 
\be{bor}
t_k-\widetilde{t}_k^{\,\Omega}\leq c L_{t_k} \tfrac{n_{k_0}}{L_{n_{k_0}}}.
\ee
Therefore we can complete $\widehat{\Theta}^{k,\Omega}_{\text{traj,tot}}$, 
$\widehat{x}^{k,\Omega}_{\text{tot}}$ and $\widehat{u}^{k,\Omega}_{\text{tot}}$ 
with
\be{}
\Theta_{\text{rest}}\in \cD_{L_{t_k},\,g_k^\Omega}^M,
\qquad  x_{\text{rest}}\in\cX_{\Theta_{\text{rest}},\Omega}^{M,m},
\qquad u_{\text{rest}}\in\cU_{\Theta_{\text{rest}}, x_{\text{rest}}, t_k
-\widetilde{t}_k^{\,\Omega}}^{M,m,L_{t_k}},
\ee 
such that, by \eqref{bor}, the number of columns $g_{k}^\Omega$ involved in 
$\Theta_{\text{rest}}$ satisfies $g_k^\Omega\leq c n_{k_0}/L_{n_{k_0}}$. 
Henceforth $\widehat{\Theta}^{k,\Omega}_{\text{traj,tot}}$, $\widehat{x}^{k,\Omega}_{
\text{tot}}$ and $\widehat{u}^{k,\Omega}_{\text{tot}}$ stand for the quantities defined 
in \eqref{conca} and \eqref{concaca}, and concatenated with $\Theta_{\text{rest}}, 
x_{\text{rest}}$ and $u_{\text{rest}}$ so that they become elements of 
\be{}
\cD_{L_{t_k},\,q_k (N_{k_0}+1)+g_k^\Omega}^M,
\qquad \cX_{\widehat{\Theta}^{k,\Omega}_{\text{traj,tot}},\Omega}^{M,m},
\qquad\cU_{\widehat{\Theta}^{k,\Omega}_{\text{traj,tot}},
\widehat{x}^{k,\Omega}_{\text{tot}},
t_k}^{M,m,L_{t_k}},
\ee 
respectively. By restricting the summation in \eqref{interm1} to $\widehat{\Theta}^{k,
\Omega}_{\text{traj,tot}}$, $\widehat{x}^{k,\Omega}_{\text{tot}}$ and 
$\widehat{u}^{k,\Omega}_{\text{tot}}$, we get 
\be{f3s}
f_{3,t_k}(M,m)\geq \frac{L_{t_k}}{t_k} \E_{\Omega}\bigg[\sum_{j=1}^{q_k} H^{\Omega_j}
(\widehat{\Theta}^{k,\Omega_j}_{\text{traj}},
\widehat x^{k_0,\Omega_j},\widehat{u}^{k,\Omega_j})
+ H(\Theta_{\text{rest}}, x_{\text{rest}}, u_{\text{rest}})\bigg],
\ee
where the term $H(\Theta_{\text{rest}},x_{\text{rest}}, u_{\text{rest}})$ is 
negligible because, by \eqref{bor},  $(t_k-\widetilde{t}_k^{\,\Omega})/t_k$ 
vanishes as $k\to \infty$,  while all free energies per column are bounded from 
below by $(\beta-\alpha)/2$. Pick $\gep>0$ and recall \eqref{impo*}. Choose 
$k_0$ such that $2 L_{n_{k_0}}/n_{k_0} \leq \gep/2$ and note that, for $k$ 
large enough, 
\be{impo1}
\widehat{s}_k^{\,\Omega}\in \Big[L_{t_k} 
\tfrac{n_{k_0}}{L_{n_{k_0}}}(1-\gep),L_{t_k} 
\tfrac{n_{k_0}}{L_{n_{k_0}}}(1+\gep)\Big].
\ee
By \eqref{defqo}, we therefore have 
\be{impo2}
q_k\in \Big[\tfrac{t_k L_{n_{k_0}}}{L_{t_k} n_{k_0}} \tfrac{1}{1+\gep}, 
\tfrac{t_k L_{n_{k_0}}}{L_{t_k} n_{k_0}} \tfrac{1}{1-\gep}\Big]=[a,b].
\ee
Recalling \eqref{f3s}, we obtain 
\be{f3s1}
f_{3,t_k}(M,m)\geq \frac{L_{t_k}}{t_k} \E_{\Omega}\bigg[\sum_{j=1}^{a} 
H^{\Omega_j}(\widehat{\Theta}^{k,\Omega_j}_{\text{traj}}, \widehat x^{k_0,\Omega_j},
\widehat{u}^{k,\Omega_j})-\sum_{j=a}^{b} \Big| H^{\Omega_j}
(\widehat{\Theta}^{k,\Omega_j}_{\text{traj}},
\widehat x^{k_0,\Omega_j},\widehat{u}^{k,\Omega_j})\Big|\bigg],
\ee
and, consequently,
\be{f3s2}
f_{3,t_k}(M,m)\geq \tfrac{L_{n_{k_0}}}{n_{k_0} (1+\gep)}\,
\E_{\,\Omega}\Big[H^{\Omega}(\widehat{\Theta}^{k,\Omega}_{\text{traj}},
\widehat x^{k_0,\Omega},\widehat{u}^{k,\Omega})\Big]-\frac{L_{t_k}}{t_k} 
(b-a) (N_{k_0}+1) m \tfrac{\beta-\alpha}{2}, 
\ee
and, by \eqref{rth},
\be{f3s2alt}
f_{3,t_k}(M,m)\geq \tfrac{l_1+\tfrac{l_2-l_1}{4}}{1+\gep} 
-(\tfrac{1}{1-\gep}-\tfrac{1}{1+\gep}) (b-a) m \tfrac{\beta-\alpha}{2}.
\ee
After taking $\gep$ small enough, we may conclude that $\liminf_{k\to\infty}
f_{3,t_k}(M,m) >l_1$, which completes the proof.


\subsection{Proof of Proposition \ref{pr:formimpp}}
\label{sMinf}

Pick $(M,m)\in \EIGH$ and note that, for every $n\in \N$, the set $\cW_{n,M}^{\,m}$ 
is contained in $\cW_{n,M}$. Thus, by using Proposition \ref{pr:formimp} we obtain 
\begin{align}
\label{limi}
\nonumber \liminf_{n\to \infty} f^{\Omega}_{1,n}(M;\alpha,\beta)
&\geq \sup_{m\geq M+2} \liminf_{n\to \infty} f^{\Omega}_{1,n}(M,m;\alpha,\beta)\\
&= \sup_{m\geq M+2} f(M,m;\alpha,\beta)\quad \text{ for } \P-a.e.\,\Omega.
\end{align}
Therefore, the proof of Proposition \ref{pr:formimpp} will be complete once we show that
\begin{align}
\label{lims}
\limsup_{n\to \infty} f^{\Omega}_{1,n}(M;\alpha,\beta)
\leq \sup_{m\geq M+2} \limsup_{n\to \infty} 
f^{\Omega}_{1,n}(M,m;\alpha,\beta)\quad \text{ for } \P-a.e.\,\Omega.
\end{align}
We will not prove \eqref{lims} in full detail, but only give the main steps in the proof. 
The proof consists in showing that, for $m$ large enough, the pieces of the trajectory 
in a column that exeed $m L_n$ steps do not contribute substantially to the free energy.

Recall (\ref{defDD2}--\ref{A11}) and use \eqref{A11} with $m=\infty$, i.e.,
\begin{align}
\label{variaf22}
Z_{n,L_n}^{\omega,\Omega}(M)
&=\sum_{N=1}^{n/L_n} \sum_{\Theta_{\text{traj}}\in \widetilde{\cD}_{L_n,N}^{M} } 
 \,\sum_{x\in \cX_{\Theta_{\text{traj}},\Omega}^{M,\infty} }
\,\sum_{u\in \,\cU_{\Theta_{\text{traj}},x,n}^{\,M,\infty,L_n}}  A_1.
\end{align} 
With each $(N,\Theta_{\text{traj}},x,u)$ in \eqref{variaf22}, we associate the trajectories
obtained by concatenating $N$ shorter trajectories $(\pi_i)_{i\in \{0,\dots,N-1\}}$ chosen 
in $(\cW_{\Theta_i,u_i,L_n})_{i\in \{0,\dots,N-1\}}$, respectively. Thus, the quantity $A_1$ 
in \eqref{variaf22} corresponds to the restriction of the partition function to the 
trajectories associated with $(N,\Theta_{\text{traj}},x,u)$. In order to discriminate between 
the columns in which more than $m L_n$ steps are taken and those in which less are taken, 
we rewrite $A_1$ as $A_2 \widetilde{A}_2$ with
\begin{align}\label{dftn}
 A_2&=\prod_{i\in V_{u,m} }\, 
Z_{L_n}^{\omega_{I_i}}(\Theta_i,u_i),
\qquad 
\widetilde{A}_2=\prod_{i\in \widetilde V_{u,m} }\, 
Z_{L_n}^{\omega_{I_i}}(\Theta_i,u_i),
\end{align}
with $\widetilde{u}_{i}=\sum_{k=0}^{i-1} u_k$, $\Theta_i=(\Omega(i,\Pi_{i}+\cdot),\Xi_i,x_i)$ 
and $I_i=\{\widetilde{u}_{i} L_n,\dots,\widetilde{u}_{i+1} L_n-1\}$ for $i\in \{0,\dots,N-1\}$, 
with $\omega_{I}=(\omega_i)_{i\in I}$ for $I\subset \N$, where $\{0,\dots, N-1\}$ is partitioned 
into 
\be{defI}
\widetilde V_{u,m}\cup V_{u,m}\quad \text{with}\quad 
\widetilde V_{u,m}=\{i\in \{0,\dots,N-1\}\colon\; u_i>m\}.
\ee
For all $(N,\Theta_{\text{traj}},x,u)$, we rewrite $\widetilde V_{u,m}$ in the form of 
an increasing sequence $\{i_1,\dots,i_{\widetilde k}\}$ and we drop the $(u,m)$-dependence 
of $\widetilde{k}$ for simplicity. We also set $\widetilde{u}=u_{i_1}+\dots+u_{i_{
\widetilde{k}}}$, which is the total number of steps taken by a trajectory associated 
with $(N,\Theta_{\text{traj}},x,u)$ in those columns where more than $m L_n$ steps are 
taken. Finally, for $s\in \{1,\dots,\widetilde{k}\}$ we partition $I_{i_s}$ into
\begin{align}
\label{labint}
J_{i_s}\cup \widetilde{J}_{i_s} \quad \text{with} 
\quad J_{i_s}
&=\{\widetilde{u}_{i_s} L_n,\dots,(\widetilde{u}_{i_s}+M+2) L_n\}, \\ \nonumber
\quad \widetilde{J}_{i_s}
&=\{(\widetilde{u}_{i_s}+M+2) L_n+1,\dots,\widetilde{u}_{i_{s}+1} L_n-1\},
\end{align}
and we partition $\{1,\dots,n\}$ into 
\begin{align}
\label{twoeq}
& J\cup \widetilde{J} \quad \text{with} \quad 
\widetilde{J}=\cup_{s=1}^{\widetilde{k}} \widetilde{J}_{i_s}, 
\qquad J=\{1,\dots,n\}\setminus \widetilde{J},
\end{align}
so that $\widetilde{J}$ contains the label of the steps constituting the pieces of 
trajectory exeeding $(M+2)L_n$ steps in those columns where more than $m L_n$ steps 
are taken.  

 
\subsubsection{Step 1}
\label{step1}  
In this step we replace the pieces of trajectories in the columns indexed in 
$\widetilde V_{u,m}$  by shorter trajectories of length $(M+2) L_n$. To that aim, 
for every $(N,\Theta_{\text{traj}},x,u)$ we set 
\begin{align}
\widehat{A}_2=\prod_{i\in \widetilde V_{u,m} }\, 
Z_{L_n}^{\,\omega_{J_i}}(\Theta^{'}_i,M+2)
\end{align}
with $\Theta^{'}_i=(\Omega(i,\Pi_{i}+\cdot),\Xi_i,1)$. We will show that for all 
$\gep>0$ and for  $m$ large enough, the event 
\be{event}
B_n=\{\omega\colon\, \widetilde{A}_2\leq 
\widehat{A}_2\, e^{3\gep n} \ \text{for all}\ (N,\Theta_{\text{traj}},x,u)\}
\ee
satisfies $\P_{\omega}(B_n)\to 1$ as $n\to \infty$.

Pick, for each $s\in \{1,\dots,\widetilde{k}\}$, a trajectory $\pi_s$ in the set 
$\cW_{\Theta_{i_s},u_{i_s},L_n}$. By concatenating them we obtain a trajectory in 
$\cW_{\widetilde u L_n}$ satisfying $\pi_{\widetilde u L_n,1}=\widetilde{k}L_n$. 
Thus, the total entropy carried by those pieces of trajectories crossing the columns 
indexed in $\{i_1,\dots,i_ {\widetilde{k}}\}$ is bounded above by
\be{boundentro}
{\textstyle \prod_{s=1}^{\widetilde{k}}\, 
|\cW_{\Theta_{i_s},u_{i_s},L_n}|
\leq \big|\{\pi\in \cW_{\widetilde{u}L_n}\colon\,
\pi_{\widetilde{u}L_n,1}=\widetilde{k}L_n\}\big|.}
\ee
Since $\widetilde{u}/\widetilde{k}\geq m$, we can use Lemma \ref{convularge} in 
Appendix \ref{Path entropies} to assert that, for $m$ large enough, the right-hand side 
of \eqref{boundentro} is bounded above by $e^{\gep n}$.

Moreover, we note that an $\widetilde{u} L_n$-step trajectory satisfying $\pi_{\widetilde{u}
L_n,1}=\widetilde{k}L_n$ makes at most $\widetilde{k} L_n+ \widetilde{u}$ excursions in 
the $B$ solvent because such an excursion requires at least one horizontal step or at 
least $L_n$ vertical steps. Therefore, by using the inequalities $\widetilde{k} L_n
\leq n/m$ and $\widetilde{u}\leq n/L_n$ we obtain that, for $n$ large enough, the sum 
of the Hamiltonians associated with $(\pi_1,\dots,\pi_{\widetilde{k}})$ is bounded from 
above, uniformly in $(N,\Theta_{\text{traj}},x,u)$ and $(\pi_1,\dots,\pi_{\widetilde{k}})$, 
by
\be{upb}
{\textstyle \sum_{s=1}^{\widetilde{k}} 
H_{u_{i_s} L_n,L_n}^{\omega_{I_{i_s}},\Omega(i_s,\Pi_{i_s}+\cdot)}
(\pi_s)\leq \max\{\sum_{i\in I} \xi_i\colon\, I\in \cup_{r=1}^{2n/m}\cE_{n,r}\}},
\ee
with $\cE_{n,r}$ defined in \eqref{add26} in Appendix \ref{Computation} and $\xi_i=\beta 
1_{\{\omega_i=A\}}-\alpha 1_{\{\omega_i=B\}}$ for $i\in \N$. At this stage we use the 
definition in \eqref{add28} and note that, for all $\omega \in \cQ_{n,m}^{\gep/\beta,
(\alpha-\beta)/2+\gep}$, the right-hand side in \eqref{upb} is smaller than $\gep n$.
Consequently, for $m$ and $n$ large enough we have that, for all $\omega \in 
\cQ_{n,m}^{\gep/\beta,(\alpha-\beta)/2+\gep}$, 
\be{indif}
\widetilde{A}_2\leq e^{2\gep n}\quad \text{for all}\quad (N,\Theta_{\text{traj}},x,u).
\ee

Recalling \eqref{boundel} and noting that $\widetilde{k} L_n\leq n/m$, we can write
\be{boundbe}
\widehat{A}_2\geq e^{-\widetilde{k} (M+2) L_n C_{\text{uf}}(\alpha)}
\geq e^{-n \tfrac{M+2}{m}  C_{\text{uf}}(\alpha)},
\ee
and therefore, for $m$ large enough, for all $n$ and all $(N,\Theta_{\text{traj}},x,u)$ 
we have $ \widehat{A}_2\geq e^{-\gep n}$.

Finally, use \eqref{indif} and \eqref{boundbe} to conclude that, for $m$ and $n$ large 
enough, $\cQ_{n,m}^{\gep/\beta,(\alpha-\beta)/2+\gep}$ is a subset of $B_n$. Thus, Lemma
\ref{lele} ensures that, for $m$ large enough, $\lim_{n\to\infty} P_\omega(B_n)=1$.


\subsubsection{Step 2}
\label{step2}

Let $(\widetilde{w}_i)_{i\in\N}$ be an i.i.d.\ sequence of Bernouilli trials, independent 
of $\omega,\Omega$. For $(N,\Theta_{\text{traj}},x,u)$ we set $\widehat{u}=\widetilde{u}
-\widetilde{k} (M+2)$. In Step 1 we have removed $\widehat{u}L_n$ steps from the trajectories
associated with $(N,\Theta_{\text{traj}},x,u)$ so that they have become trajectories 
associated with $(N,\Theta_{\text{traj}},x^{'},u)$. In this step, we will concatenate 
the trajectories associated with $(N,\Theta_{\text{traj}},x^{'},u)$ with an 
$\widehat{u} L_n$-step trajectory to recover a trajectory that belongs to 
$\cW_{n,M}^{\,m}$.
 
For $\Omega\in \{A,B\}^{\N_0\times \Z}$, $t,N\in \N$ and $k\in \Z$, let
\be{mesemp}
P_{A}^\Omega(N,k)(t)=\frac{1}{t} \sum_{j=0}^{t-1} 1_{\{\Omega(N+j,k)=A\}}
\ee
be the proportion of $A$-blocks on the $k^{\text{th}}$ line and between the $N^{\text{th}}$ 
and the $(N+t-1)^{\text{th}}$ column of $\Omega$. Pick $\eta>0$ and $j\in \N$, and set 
\be{defsn}
S_{\eta,j}=\bigcup_{N=0}^{j}\bigcup_{k=-j}^{j} 
\bigcup_{t\geq \eta j}\Big\{ P_{A}^\Omega(N,k)(t)\leq \frac{p}{2}\Big\}.
\ee
By a straightforward application of Cramer's Theorem for i.i.d.\ random variables, we have
that $\sum_{j\in\N} P_\Omega(S_{\eta,j})<\infty$. Therefore, using the Borel-Cantelli Lemma, 
it follows that for $\P_\Omega$-a.e.\ $\Omega$, there exists a $j_\eta(\Omega)\in \N$ such 
that $\Omega\notin S_{\eta,j}$ as soon as $j\geq j_\eta(\Omega)$. In what follows, we consider 
$\eta=\gep/\alpha m$ and we take $n$ large enough so that  $n/L_n\geq j_{\gep/\alpha m}
(\Omega)$, and therefore $\Omega\notin S_{\frac{n}{L_n},\frac{\gep}{\alpha m}}$. 

Pick $(N,\Theta,x,u)$ and consider one trajectory $\widehat \pi$, of length $\widehat{u} L_n$, 
starting from $(N,\Pi_N+b_N)L_n$, staying in the coarsed-grained line at height $\Pi_N$, 
crossing the $B$-blocks in a straight line and the $A$-blocks in $m L_n$ steps. The number 
of columns crossed by $\widehat{\pi}$ is denoted by $\widehat{N}$ and satisfies $\widehat{N}
\geq \widehat{u}/m$. If $\widehat{u} L_n \leq \gep n/\alpha$, then the Hamiltonian associated 
with $\widehat{\pi}$ is clearly larger than $-\gep n$. If $\widehat{u} L_n \geq \gep n/\alpha$ 
in turn, then
\be{hampi}
H_{\widehat{u} L_n,L_n}^{\,\widetilde{w},
\Omega(N+\cdot,\Pi_N)}(\widehat{\pi})
\geq  -\alpha L_n  \widehat{N} \big[1-P_{A}^\Omega(N,\Pi_N)(\widehat{N})\big].
\ee
Since $N\leq n/L_n$, $|\Pi_N|\leq n/L_n$ and $\widehat{N}\geq \gep n/(\alpha m L_n)$, we can 
use the fact that $\Omega\notin S_{\frac{\gep}{\alpha m},\frac{n}{L_n}}$ to obtain 
\be{boundP}
P_{A}^\Omega(N,\Pi_N)(\widehat{N})\geq \frac{p}{2}.
\ee
At this point it remains to bound $\widehat{N}$ from above, which is done by noting that
\be{boundN}
\widehat{N} \big[m P_{A}^\Omega(N,\Pi_N)(\widehat{N})
+1-P_{A}^\Omega(N,\Pi_N)(\widehat{N})\big]
=\widehat{u}\leq \tfrac{n}{L_n}.
\ee
Hence, using \eqref{boundP} and \eqref{boundN}, we obtain $\widehat{N}\leq 2 n/p m L_n$
and therefore the right-hand side of \eqref{hampi} is bounded from below by $-\alpha (2-p)
n/p m$, which for $m$ large enough is larger than $-\gep n$.

Thus, for $n$ and $m$ large enough and for all $(N,\Theta,x,u)$, we have a trajectory 
$\widehat{\pi}$ at which the Hamiltonian is bounded from below by $-\gep n$ that can 
be concatenated with all trajectories associated with $(N,\Theta,x{'},u)$ to obtain a 
trajectory in $\cW_{n,M}^{\,m}$. Consequently, recalling \eqref{labint}, for $n$ and 
$m$ large enough we have 
\be{boundA2}
A_2\widehat{A_2}\leq e^{\gep n} 
Z_{\,n,L_n}^{(\omega_{J}\,,\,\widetilde \omega), \Omega} (M,m)
\qquad \forall\, (N,\Theta,x,u).
\ee


\subsubsection{ Step 3}
\label{step3}

In this step, we average over the microscopic disorders $\omega,\widetilde{\omega}$. 
Use \eqref{boundA2} to note that, for $n$ and $m$ large enough and all $\omega\in B_n$, 
we have
\begin{align}
\label{variaf222}
Z_{n,L_n}^{\omega,\Omega}(M)
&\leq e^{4\gep n} \sum_{N=1}^{n/L_n} 
\sum_{\Theta_{\text{traj}}\in \widetilde{\cD}_{L_n,N}^{M} } 
 \,\sum_{x\in \cX_{\Theta_{\text{traj}},\Omega}^{M,\infty} }
\,\sum_{u\in \,\cU_{\Theta_{\text{traj}},x,n}^{\,M,\infty,L_n}} 
Z_{\,n,L_n}^{(\omega_{J}\,,\,\widetilde{\omega}), \Omega} (M,m).
\end{align}  
We use \eqref{concmesut} to claim that there exists $C_1,C_2>0$ so that for all $n\in \N$, all 
$m\in \N$ and all $J$, 
\be{inegco}
\P_{\omega,\widetilde{\omega}}\Big(\Big|\tfrac1n
\log Z_{n,L_n}^{(\omega_{J}\,,\,\widetilde{\omega}),\Omega}(M,m)
-f_{1,n}^{\Omega}(M,m)\Big|\geq \gep \Big)\leq C_1 e^{-C_2 \gep^2 n}.
\ee 
We set also
\be{inegco2}
D_{n}=\bigcap_{(N,\Theta_{\text{traj}},x,u)} 
\Big\{\Big|\tfrac1n\log Z_{n,L_n}^{(\omega_{J}\,,
\,\widetilde{\omega}),\Omega}(M,m)-f_{1,n}^{\Omega}(M,m)\Big|\leq \gep\Big\},
\ee
recall the definition of $c_n$ in \eqref{carsum} (used with $(M,\infty)$), and use 
\eqref{inegco} and the fact that $c_n$ grows subexponentially, to obtain $\lim_{n\to\infty} 
\P_{\omega,\widetilde{\omega}}(D_n^c)= 0$. For all $(\omega,\widetilde{\omega})$ satisfying 
$\omega\in B_n$ and $(\omega,\widetilde{\omega})\in D_n$, we can rewrite \eqref{variaf222} as 
\begin{align}
\label{variafin}
Z_{n,L_n}^{\omega,\Omega}(M)
&\leq c_n\ e^{n  f_{1,n}^{\Omega}(M,m) +5\gep n}.
\end{align}  
As a consequence, recalling \eqref{boundel}, for $m$ large enough we have 
\be{endres}
f^{\Omega}_{n}(M;\alpha,\beta)\leq \P(B_n^c\cup D_n^c) 
\, C_{\text{uf}}(\alpha)+ \frac{\log c_n}{n}+\frac1n
\E\Big(1_{\{B_n\cup D_n\}}  \big(n f_{1,n}^{\Omega}(M,m) +5\gep n\big)\Big).
\ee
Since $\P(B_n^c\cup D_n^c)$ and $(\log c_n)/n$ vanish when $n\to \infty$, it suffices to 
apply Proposition \ref{pr:formimp} and to let $\gep\to 0$ to obtain \eqref{lims}. This 
completes the proof of Proposition \ref{pr:formimpp}.


\subsection{Proof of Proposition \ref{pr:varar}}
\label{svarar}

Note that, for all $m\geq M+2$, we have $\cR_{p,M}^{m}\subset  \cR_{p,M}$. Moreover, any 
$(u_\Theta)_{\Theta\in \overline\cV_M^{\,m}}\in \cB_{\overline \cV_M^{\,m}}$  can be 
extended to $\overline\cV_M$ so that it belongs to $\cB_{\overline \cV_M}$. Thus, 
\be{}
\sup_{m\geq M+2} f(M,m;\alpha,\beta)\leq \sup_{\rho\in \cR_{p,M}} 
\sup_{(u)\in \cB_{\overline\cV_M} }  V(\rho,u).
\ee
As a consequence, it suffices to show that for all $\rho\in \cR_{p,M}$ and $(u_\Theta)_{
\Theta\in \overline\cV_M}\in \cB_{\overline \cV_M}$,
\be{finalst}
V(\rho,u)\leq\sup_{m\geq M+2} \sup_{\rho\in \cR_{p,M}^{m}} 
\sup_{(u)\in \cB_{\overline\cV_M^{\,m}} } V(\rho,u).
\ee
If $\int_{\overline \cV_M} u_\Theta\,\rho(d\Theta)=\infty$, then \eqref{finalst} is 
trivially satisfied since $V(\rho,u)=-\infty$. Thus, we can assume that $\rho(\overline
\cV_M\setminus D_M)=1$, where $D_M=\{\Theta\in\overline \cV_M\colon\,\chi_\Theta
\in\{A^{\Z},B^{\Z}\}, x_\Theta=2\}$. Since $\int_{\overline \cV_M} u_\Theta\,\rho(d\Theta)
<\infty$ and since (recall \eqref{boundel}) $\psi(\Theta,u)$ is uniformly bounded by 
$C_{\text{uf}}(\alpha)$ on $(\Theta,u)\in \overline \cV_M^{\,*}$, we have by dominated 
convergence that for all $\gep>0$ there exists an $m_0\geq M+2$ such that, for all 
$m\geq m_0$, 
\be{mamj}
V(\rho,u)\leq \frac{\int_{\overline\cV_M^{\,m}} 
u_\Theta \psi(\Theta,u_\Theta) \rho(d\Theta)}{\int_{\overline\cV_M^{\,m}} 
u_\Theta  \rho(d\Theta)}+\tfrac{\gep}{2}.
\ee
Since $\rho(\overline \cV_M\setminus D_M)=1$ and since $\cup_{m\geq M+2} \overline 
\cV_M^{\,m}= \overline \cV_M\setminus D_M$, we have $\lim_{m\to\infty}\rho(\overline
\cV_M^{\,m})=1$. Moreover, for all $m\geq m_0$ there exists a $\widehat{\rho}_{m}\in
\cR_{p,M}^{m}$ such that $\widehat\rho_{m}=\rho_{m}+\overline{\rho}_{m}$, with $\rho_{m}$ 
the restriction of $\rho$ to $\overline \cV_M^{\,m}$ and $\overline \rho_m$ charging only 
those $\Theta$ satisfying $x_\Theta=1$. Since all $\Theta \in \overline\cV_M$ with 
$x_\Theta=1$ also belong to $\overline \cV_M^{\,M+2}$, we can state that $\overline
\rho_m$ only charges $\overline \cV_M^{\,M+2}$ . Therefore
\be{malmj}
V(\widehat\rho_m,u) 
=\frac{\int_{\overline\cV_M^{\,m}} u_\Theta \psi(\Theta,u_\Theta) \rho(d\Theta)
+\int_{\overline \cV_M^{\,M+2}} u_\Theta \psi(\Theta,u_\Theta) 
\overline \rho_m(d\Theta)}{\int_{\overline\cV_M^{\,m}} 
u_\Theta  \rho(d\Theta)+\int_{\overline \cV_M^{\,M+2}} u_\Theta \overline\rho_m(d\Theta)}.
\ee
Since $\Theta\mapsto u_\Theta$ is continuous on $\overline\cV_M$, there exists an $R>0$ 
such that $u_\Theta\leq R$ for all $\Theta\in \overline\cV_M^{M+2}$. Therefore we can 
use \eqref{mamj} and \eqref{malmj} to obtain, for $m\geq m_0$,
\be{malmj2}
V(\widehat\rho_m,u)\geq (V(\rho,u)-\tfrac{\gep}{2}) 
\frac{\int_{\overline\cV_M^{\,m}} u_\Theta \rho(d\Theta)}{\int_{\overline\cV_M^{\,m}} 
u_\Theta  \rho(d\Theta)+\int_{\overline  \cV_M^{\,M+2}} 
u_\Theta \overline\rho_m(d\Theta)}
- R\, C_{\text{uf}}(\alpha)\, (1-\rho(\overline\cV_M^{\,m})).
\ee
The fact that $\overline \rho_m(\overline \cV_M^{\, M+2})=\rho(\overline\cV_M\setminus
\overline\cV_M^{\,m})$ for all $m\geq m_0$ implies that $\lim_{m\to\infty}\overline 
\rho_m(\cV_M^{M+2})=0$. Consequently, the right-hand side in \eqref{malmj2} tends 
to $V(\rho,u)-\gep/2$ as $m\to\infty$. Thus, there exists a $m_1\geq m_0$  such 
that $V(\widehat{\rho}_{m_1},u)\geq V(\rho,u)-\gep$. Finally, we note that there exists 
a $m_2\geq m_1+1$ such that $u_\Theta\leq m_2$ for all $\Theta\in \overline\cV_M^{\,m_1}$, 
which allows us to extend $(u_\Theta)_{\Theta\in \overline \cV_M^{\,m_1}}$ to 
$\overline \cV_M^{\,m_2}$ such that $(u_\Theta)_{\Theta\in \overline \cV_M^{\,m_2}}\in 
\cB_{\overline \cV_M^{\,m_2}}$. It suffices to note that $\widehat{\rho}_{m_1} \in 
\cR_{p,M}^{m_1}\subset  \cR_{p,M}^{m_2}$ to conclude that 
\be{finaalt}
V(\rho,u)\leq f(M,m_2;\,\alpha,\beta)+\gep.
\ee


\subsection{Proof of Proposition \ref{pr:formimppp}}
\label{Minfinity}

It remains to remove the $M$-truncation from the variational formula in 
Proposition~\ref{pr:varar}. To that aim it suffices to show that 
\begin{align}
\label{limss}
\limsup_{n\to \infty} f^{\Omega}_{n}(\alpha,\beta)
\leq \sup_{M\geq 1 } \limsup_{n\to \infty} 
f^{\Omega}_{n}(M;\alpha,\beta)\quad \text{ for } \P-a.e.\,\Omega.
\end{align}

The proof of \eqref{limss} is similar to that of \eqref{lims} in Section~\ref{sMinf}.  In the latter, the pieces 
of path inside the columns where too many steps $(\geq m L_n)$ were taken were replaced by a shorter 
path. However, the mesoscopic strategy of displacement was not changed. This is a major difference with 
the proof of Proposition~\ref{pr:formimppp} below, since we need to compare the contribution to the partition 
function of groups of trajectories that do not follow the same mesoscopic strategy of displacement. 

For $\pi\in \cW_{n}$, we recall that $N_\pi$ is the number of columns crossed by $\pi$ after $n$ steps. We 
recall (\ref{variaf22}--\ref{twoeq}) and use the same notations with $M=\infty$ to rewrite the full partition function  
as
\begin{align}
\label{variaf22*}
Z_{n,L_n}^{\omega,\Omega}&=\sum_{N=1}^{n/L_n} \sum_{\Theta_{\text{traj}}\in \widetilde{\cD}_{L_n,N}^{\infty} } 
 \,\sum_{x\in \cX_{\Theta_{\text{traj}},\Omega}^{\infty,\infty} }
\,\sum_{u\in \,\cU_{\Theta_{\text{traj}},x,n}^{\infty,\infty, L_n}}  A_1.
\end{align} 

We pick $N \in \{1,\dots,n/L_n\}$, and with each $\Theta_{\text{traj}}\in \widetilde{\cD}_{L_n,N}^{\, \infty}$ and 
$x\in \cX_{\Theta_{\text{traj}},\Omega}^{\, \infty,\infty}$ we associate an auxiliary mesocopic strategy denoted 
by $\widetilde \Theta_{\text{traj}}\in \widetilde{\cD}_{L_n,N}^{M}$ and $\widetilde x\in \cX_{\widetilde 
\Theta_{\text{traj}},\Omega}^{\, M,\infty}$ that is built as follows. Let $i_1$ be the index of the first column in 
which the mesoscopic displacement of $\Theta_{\text{traj}}$ is strictly larger than $M$, i.e., $(|\Delta \Pi _{i_1}|>M)$. 
Until $i_1$,  both strategies $(\Theta_{\text{traj}},x)$ and $(\widetilde \Theta_{\text{traj}},\widetilde x)$ are equal, 
i.e., 
\be{}
\widetilde  \Theta_i=(\Omega(i,\widetilde \Pi_{i}+\cdot),\widetilde \Xi_i,\widetilde x_i)
=(\Omega(i, \Pi_{i}+\cdot),\Xi_i,x_i)=\Theta_i\quad  \text{for} \ i\leq i_1-1.
\ee
The mesoscopic displacement $\Delta \Pi_{i_1}$ of $ \Theta_{\text{traj}}$ is large and $\widetilde  \Theta_{\text{traj}}$ 
starts making mesoscopic steps of size $M$ to catch up with $\Theta_{\text{traj}}$ as soon as possible. This takes 
$r_1\in \N$ columns indexed in $\{i_1,\dots, i_1+r_1-1\}$ for which $|\Delta \widetilde \Pi_i|=M, \widetilde x_i=1$, 
except for the  very last column ($i=i_1+r_1-1$), which is used to end the catch up between $(\widetilde 
\Theta_{\text{traj}},\widetilde x)$ and $(\Theta_{\text{traj}},x)$. We note that there may be other columns 
among $\{i_1,\dots, i_1+r_1-1\}$ in which the mesoscopic displacement of  $\Theta_{\text{traj}}$ is $>M$. 

After $(\widetilde \Theta_{\text{traj}},\widetilde x)$ catches up with  $(\Theta_{\text{traj}},x)$, it remains equal to 
$(\Theta_{\text{traj}},x)$ until a new column appears (indexed by $i_2\geq i_1+r_1$) with a large mesoscopic 
displacement, i.e., $|\Delta \Pi_{i_2}|>M$. Thus, $\widetilde  \Theta_i=\Theta_i$ for  $i \in \{i_1+r_1,\dots i_2-1\}$,
and so on. The resulting $\widetilde \Theta_{\text{traj}}$ and $\widetilde x$ belong to $\widetilde{\cD}_{L_n,N}^{M}$
and  $\cX_{\widetilde \Theta_{\text{traj}},\Omega}^{\, M,\infty}$, respectively, and $ \Theta_i=\widetilde \Theta_i$, 
except on $k$ groups of consecutive columns denoted by $\{i_1,\dots,i_1+r_1-1\}, \dots,\{i_{k},\dots, i_k+r_k-1\}$ 
and referred to as the catch-up columns in what follows. For simplicity, the dependence in $ \Theta_{\text{traj}}$ 
of $k, i_1, r_1, \dots, i_k,r_k$ is omitted.

We can give a crude upper bound on the number of  columns on which $\widetilde \Theta_{\text{traj}}$ differs 
from $\widetilde \Theta_{\text{traj}}$. The sum of the absolute values of the large mesoscopic jumps (i.e., 
$\sum_{i=1}^N |\Delta \Pi_i| 1\{|\Delta \Pi_i|>M\}$) performed by $ \Theta_{\text{traj}}$ indeed cannot exceed 
$n/L_n$. Moreover, the number of columns in which the mesoscopic displacement is larger than $M/2$ is 
bounded above by $2n/M L_n$,  and in each catch-up column where the mesoscopic displacement of 
$\Theta_{\text{traj}}$ is smaller than $M/2$,  $\widetilde \Theta_{\text{traj}}$ scores at least $M/2$ blocks in 
its race against $\Theta_{\text{traj}}$. Therefore, the number of catch-up columns $r_1+\dots+r_k$ is bounded 
above by  $4n/ML_n$.

In order to discriminate between the catch-up columns and the columns on which $\Theta_{\text{traj}}$ and 
$\widetilde \Theta_{\text{traj}}$ are equal, we keep the notations of (\ref{defI}--\ref{twoeq}) and we rewrite 
$A_1$ as $A_2 \widetilde{A}_2$ with
\begin{align}
\label{spint}
A_2 &=\prod_{i\in V_{\Pi,M} }\, 
Z_{L_n}^{\omega_{I_i}}(\Theta_i,u_i),
\qquad 
\widetilde{A}_2=\prod_{i\in \widetilde V_{\Pi,M} }\, 
Z_{L_n}^{\omega_{I_i}}(\Theta_i,u_i),
\end{align}
where $\{0,\dots, N-1\}$ is partitioned into $\widetilde V_{\Pi,M}\cup V_{\Pi,M}$ and $\widetilde V_{\Pi,M}
=\cup_{s=1}^k \{i_s,\dots,i_{s}+r_s-1\}$ gathers the indices of the $k$ groups of catch-up columns. 

We also set $\bar{u}_s= u_{i_s}+\dots+u_{i_{s}+r_s-1}$, which is the total number of steps taken by a trajectory 
associated with $(N,\Theta_{\text{traj}},x,u)$ in the $s$-th group of  catch-up columns. Finally, for each $j\in 
\widetilde V_{\Pi,M}$ we let $v_j L_n$ be the minimal number of steps that are required to cross a column of 
type $\widetilde \Theta_j$. Even though it is not necessarily true that $v_j\leq u_j$ for all $j\in \widetilde V_{\Pi,M}$, 
it is true by construction that for $s\in \{1,\dots,k\}$ we have $\bar u_s\geq v_{i_s}+\dots+v_{i_s+r_s-1} = \bar v_s$. 
For each  $s\in \{1,\dots,k\}$ and each $t\in \{0,\dots,r_s-1\}$, we define 
\be{Jil}
J_{i_s+t}=\{(\widetilde{u}_{i_s}+v_{i_s}+\dots+v_{i_s+t-1}) L_n, (\widetilde{u}_{i_s}+v_{i_s}+\dots+v_{i_s+t}) L_n-1\},
\ee
so that we partition $I_{i_s}\cup\dots\cup I_{i_s+r_s}$ into
\begin{align}
\label{labint1}
K_s\cup \widetilde{K}_{s} \quad \text{with} 
\quad K_{s}
&=\{\widetilde{u}_{i_s} L_n,\dots,(\widetilde{u}_{i_s}+v_{i_s}+\dots+v_{i_s+r_s-1}) L_n\}, \\ \nonumber
\quad \widetilde{K}_{s}
&=\{(\widetilde{u}_{i_s}+v_{i_s}+\dots+v_{i_s+r_s-1}) L_n+1,\dots,\widetilde{u}_{i_{s}+r_s} L_n-1\},
\end{align}
and we partition $\{1,\dots,n\}$ into 
\begin{align}
\label{twoeq1}
& T\cup \widetilde{T} \quad \text{with} \quad 
\widetilde{T}=\cup_{s=1}^{k} \widetilde K_s,
\qquad T=\{1,\dots,n\}\setminus \widetilde{T}.
\end{align}


\subsubsection{Step 1}
\label{step11} 
 
In this step, we aim at replacing the Hamiltonian in the catch-up columns by an auxiliary coarse-grained 
version of the Hamiltonian, which simply assigns an energetic penalty $\frac{\beta-\alpha}{2}$ to each 
monomers placed in solvent $B$. To that aim, for $\chi\in \{A,B\}^\Z$ and $\pi\in \cW_{u L}$ such that 
$\pi_{u L,1}=L$, we set  
\be{alterh}
\widehat H_{u L,L}^{\chi}
(\pi)= \tfrac{\beta-\alpha}{2} \sum_{i=1}^{u L} 1\{\chi^{L}_{(\pi_{i-1},\pi_i)}=B\}.  
\ee
and we recall that $\chi^{L}_{(\pi_{i-1},\pi_i)}$ denotes the label of the block the step $(\pi_{i-1},\pi_i)$
lies in. With the help of \eqref{alterh} and recalling \eqref{partfunc2}, we define the partition function
associated with those trajectories crossing a block-column of type $\Theta=(\chi,\Xi,x)$ in $uL$ steps as
\be{alt ham}
 \quad \widehat Z_L(\Theta,u)
=\sum_{\pi \in \cW_{\Theta,u,L}} e^{\,\widehat H_{uL,L}^{\, \chi}(\pi)},
\ee
and we note that $\widehat Z_L(\Theta,u)$ does not depend on the microscopic disorder $\omega$canymore.
Thus, we can set 
\begin{align}
\label{sstg}
\widehat{A}_2=\prod_{i\in \widetilde V_{\Pi,M} }\, 
\widehat Z_{L_n}( \Theta_i,u_i).
\end{align}
In the rest of this proof, we will often state results that hold uniformly on $(N,\Theta_{\text{traj}},x,u)$ without 
recalling that $N\in \{1,\dots,n/L_n\}$, $\Theta_{\text{traj}}\in \widetilde{\cD}_{L_n,N}^{\infty}$,  $x\in 
\cX_{\Theta_{\text{traj}},\Omega}^{\infty,\infty}$ and $u\in \,\cU_{\Theta_{\text{traj}},x,n}^{\infty,\infty, L_n}$. 

Our aim is to show that, for all $\gep>0$, $M$ large enough and $\Omega\in \{A,B\}^{\N_{0}\times \Z}$, the set
\be{eventt}
B^1_{n,M}=\{\omega\colon\, \widetilde{A}_2
\leq \widehat{A}_2\, e^{\gep n} \ \text{for all}\ (N,\Theta_{\text{traj}},x,u)\}
\ee
satisfies $\lim_{n\to\infty} \P_{\omega}(B^1_{n,M}) = 1$. We consider a given $(N,\Theta_{\text{traj}},x,u)$, and 
we set  $\widetilde u=\bar u_1+\dots+\bar u_s$. We then pick for each $i \in \widetilde{V}_{\Pi,M}$ a trajectory 
$\pi_i$ in the set $\cW_{\Theta_{i},u_{i},L_n}$.  By concatenating these trajectories, we obtain a trajectory 
$\widehat \pi \in \cW_{\widetilde u L_n}$ satisfying $\widehat \pi_{\widetilde u L_n,1}=(r_1+\dots+r_k)L_n$. 
The difference between the Hamiltonian associated with $\widehat \pi$ in $\widetilde A_2$ and the one associated 
with $\widehat \pi$ in $\widehat A_2$ equals
\begin{align}
\label{upb*}
{\textstyle \sum_{s=1}^{k} \sum_{x=0}^{r_s-1}}
 H_{u_{i_s+x} L_n,L_n}^{\omega_{I_{i_s+x}},\, \Omega(i_s+x,\Pi_{i_s+x}+\cdot)}
&(\pi_{i_s+x})-
\widehat H_{u_{i_s+x} L_n,L_n}^{\, \Omega(i_s+x,\Pi_{i_s+x}+\cdot)}.
(\pi_{i_s+x}).
\end{align}
Either $\widehat \pi$ takes in $B$ a number of steps that is $\leq \gep n/(2 \alpha)$ and the Hamiltonian difference 
in \eqref{upb*} is bounded above by $\gep n$, or the number of steps in $B$ is larger than $\gep n/ 2\alpha$. In the 
latter case, since  $\widehat \pi_{\widetilde{u} L_n,1}=(r_1+\dots+r_k)L_n$, $\pi$ makes at most $(r_1+\dots+r_k) 
L_n+ \widetilde{u}$ excursions in $B$ because each such excursion requires at least one horizontal step or at 
least $L_n$ vertical steps. Therefore, by using the inequalities $(r_1+\dots+r_k) L_n \leq 4n/M$ and $\widetilde{u}
\leq n/L_n$, we can claim that, as soon as $L_n\geq M$,  $\widehat \pi$ does not perform more than $5n/M$ 
excursions in $B$, and hence
\begin{align}
\label{upb2}
{\textstyle \sum_{s=1}^{k} \sum_{x=0}^{r_s-1}}
 H_{u_{i_s+x} L_n,L_n}^{\omega_{I_{i_s+x}},\, \Omega(i_s+x,\Pi_{i_s+x}+\cdot)}
&(\pi_{i_s+x})-
\widehat H_{u_{i_s+x} L_n,L_n}^{\, \Omega(i_s+x,\Pi_{i_s+x}+\cdot)}
(\pi_{i_s+x})
\\
\nonumber 
&\leq \max\big\{\, \textstyle \sum_{i\in I} \ \big(\xi_i-\tfrac{\beta-\alpha}{2}\big)\colon\, 
I\in \cup_{r=1}^{5n/M}\cE_{n,r}\big\},
\end{align}
with $\cE_{n,r}$ defined in \eqref{add26} in Appendix~\ref{Computation} and $\xi_i=\beta 1_{\{\omega_i=A\}}
-\alpha 1_{\{\omega_i=B\}}$ for $i\in \N$. At this point we use the definition in \eqref{add28} and note that, for 
all $\omega \in \cQ_{n,M/5}^{\gep/(2\alpha),\, \gep}$, the right-hand side in \eqref{upb2} is smaller than $\gep n$.
Consequently, for $M$ and $n$ large enough we have that, for all $\omega \in \cQ_{n,M/5}^{\gep/2\alpha,\gep}$, 
\be{indif*}
\frac{\widetilde{A}_2}{\widehat A_2}\leq e^{\gep n}\quad \text{for all}\quad (N,\Theta_{\text{traj}},x,u).
\ee
It remains to use \eqref{indif*} and \eqref{eventt} to conclude that, for $M$ and $n$ large 
enough, $\cQ_{n,M/5}^{\gep/2\alpha,\gep}$ is a subset of $B^1_{n,M}$. Thus, Lemma
\ref{lele} ensures that, for $M$ large enough, $\lim_{n\to\infty} P_\omega(B^1_{n,M})=1$,
which completes Step 1.


\subsubsection{Step 2}
\label{step22}  

In this step we further simplify the expression of $\widehat A_2$ introduced in \eqref{sstg} by setting 
\begin{align}
\label{ssttt}
\widehat{A}_3=\prod_{i\in \widetilde V_{\Pi,M} } e^{\frac{\beta-\alpha}{2}\,  \mathfrak{N}_B(\Theta_i) L_n},
\end{align}
where $\mathfrak{N}_B(\Theta_i)$ is the number of $B$-blocks located in between the entrance block and the exit 
block that have to be crossed entirely in the vertical direction by any trajectory that crosses a block of type $\Theta_i$. 
We note that $\mathfrak{N}_B(\Theta_i)$ only depends on $\Delta \Pi_i$ (the mesoscopic displacement in the column) 
and on  the disorder in the column seen from the entrance block $\Omega(i,\Pi_i+\cdot)$. Our aim is to show that, for 
$\gep >0$ and for $M$ and $n$ large enough, we have for all $\Omega\in \{A,B\}^{\N_0\times \Z}$ that $\widehat{A}_2
<\widehat{A}_3 e^{\gep n}$ uniformly in $(N,\Theta,x,u)$.

For a given $(N,\Theta,x,u)$ we pick, for each $i \in \widetilde{V}_{\Pi,M}$, a trajectory $\pi_i \in \cW_{\Theta_{i},u_{i},L_n}$. 
We recall that $\beta-\alpha\leq 0$, since $(\alpha,\beta)\in \CONE$. In $\widehat A_2$ the Hamiltonian associated 
with $(\pi_i)_{i\in \widetilde{V}_{\Pi,M}}$ is bounded above by 
\be{bbel}
\sum_{i\in \widetilde V_{\Pi,M}} \widehat H_{u_{i} L_n,L_n}^{\, \Omega(i,\Pi_{i}+\cdot)}
(\pi_{i})\leq  \tfrac{\beta-\alpha}{2} L_n \sum_{i\in \widetilde V_{\Pi,M}} \mathfrak{N}_B(\Theta_i)
\ee
because, for each $i\in \widetilde V_{\Pi,M}$, $\pi_i$ must cross vertically at least $\mathfrak{N}_B(\Theta_i)$ blocks of 
type $B$.  In the right-hand side of \eqref{bbel}, we recognise the exponential factor in \eqref{ssttt}, and therefore this 
step will be complete once we control the entropy carried by those pieces of trajectories that cross the columns indexed 
in $\widetilde V_{\Pi, M}$. To that aim, we recall that  $\widetilde u=\bar u_1+\dots+\bar u_s$ and we note that, by 
concatenating the paths $(\pi_i)_{i\in \widetilde V_{\Pi,M}}$, we obtain a trajectory $\widehat \pi \in \cW_{\widetilde u L_n}$ 
satisfying $\pi_{\widetilde u L_n,1}=(r_1+\dots+r_k)L_n$. Thus,
\be{boundentro*}
{\textstyle \prod_{i\in \widetilde{V}_{\Pi,M}}\, 
|\cW_{\Theta_{i},u_{i},L_n}|
\leq \big|\{\pi\in \cW_{\widetilde{u}L_n}\colon\,
\pi_{\widetilde{u}L_n,1}=(r_1+\dots+r_k)L_n\}\big|,}
\ee
and either $\widetilde u\leq \gep n/\log(3) L_n$ and the right-hand side in \eqref{boundentro*} is smaller than $e^{\gep n}$,
or $\widetilde u\geq  \gep n/\log(3) L_n$ and the crude bound $r_1+\dots+r_k\leq 4n/M L_n$ allows us to write 
$\widetilde{u}/(r_1+\dots,r_k)\geq M\gep/(4 \log(3))$, and we can use Lemma~\ref{convularge} in Appendix~\ref{Path entropies} 
to assert that, for $M$ large enough, the right-hand side of \eqref{boundentro*} is bounded above by $e^{\gep n}$.


\subsubsection{Step 3}
\label{step33}  

In this step, we link each mesoscopic strategy $\Theta_{\text{traj}}$ to its auxiliary counterpart $\widetilde\Theta_{\text{traj}}$ 
by replacing $\widehat A_3$ in \eqref{ssttt} by
\be{sstttt}
\widehat{A}_4=\prod_{i\in \widetilde V_{\Pi,M} }\, 
e^{\frac{\beta-\alpha}{2} \, \mathfrak{N}_B(\widetilde \Theta_i) L_n}.
\ee
As in the previous step, $\mathfrak{N}_B(\widetilde \Theta_i)$ is the number of $B$-blocks located in between the entrance 
and the exit blocks that have to be crossed entirely in the vertical direction by any trajectory crossing the $\widetilde\Theta_i$ 
column. Our aim is to show that for all $\gep>0$ there exists an $M\in \N$ such that, for $\P-a.e.\,\Omega$ and $n$ is large 
enough, that  $\widehat A_3\leq \widehat A_4 e^{\gep n}$ for all $(N,\Theta,x,u)$. To that aim, it suffices to prove that for all 
$\gep>0$ there exists an $M\in \N$ such that, for $\P-a.e.\,\Omega$ and $n$ large enough, that 
\be{}
\sum_{i\in \widetilde{V}_{\Pi,M}} \mathfrak{N}_B(\Theta_i)
\geq \sum_{i\in \widetilde{V}_{\Pi,M}} \mathfrak{N}_{B}(\widetilde \Theta_i)
-\frac{\gep n}{L_n}, \ \text{for all}\ (N,\Theta_{\text{traj}},x,u).
\ee
We set 
\begin{align}
R_{n,M}=\bigg\{\Omega\in \{A,B\}^{\N_0\times \Z} \colon\, \exists 
& N\in \{1,\dots,\tfrac{n}{L_n}\}, \  \exists \Theta_{\text{traj}}\in \widetilde{\cD}_{L_n,N}^{\, \infty}, \\
\nonumber & \text{and}\ \sum_{i\in \widetilde{V}_{u,M}} \mathfrak{N}_B(\Theta_i)
\leq  \sum_{i\in \widetilde{V}_{u,M}} \mathfrak{N}_B(\widetilde \Theta_i)-\gep \tfrac{n}{L_n}\bigg\}
\end{align}
and we aim at showing that, for $M$ large enough, $\sum_{n\geq 1} \P_{\Omega}(R_{n,M}\neq \emptyset)<\infty$.

We need to simplify the expression for $R_{n,M}$. As explained earlier, for each $(N,\Theta_{\text{traj}},x,u)$, the 
location of the catch-up columns $\widetilde V_{\Pi,M}$ only depends on $\Pi=(\Pi_i)_{i=0}^{N}$ and the subsequence 
$(\mathfrak{N}_B(\Theta_i))_{i\in \widetilde V_{\Pi,M}}$ only depends on $\Omega$ and $(\Pi_i,\Pi_{i+1})_{i\in \widetilde 
V_{\Pi,M}} $. Moreover, in the catch-up columns, the associated mesoscopic strategy of displacement
$(\widetilde \Pi_i,\widetilde \Pi_{i+1})_{i\in \widetilde V_{\Pi,M}} $ is completely determined by  $(\Pi_i,\Pi_{i+1})_{i\in \widetilde 
V_{\Pi,M}}$ and  $(\mathfrak{N}_B(\widetilde \Theta_i))_{i\in \widetilde V_{\Pi,M}}$ only depends on $\Omega$ and 
$(\widetilde \Pi_i,\widetilde \Pi_{i+1})_{i\in \widetilde V_{\Pi,M}}$. As a consequence, for all $i\in \widetilde V_{\Pi,M}$ we 
can rewrite $\mathfrak{N}_B(\Theta_i)$ and $\mathfrak{N}_B(\widetilde \Theta_i)$ as $\mathfrak{N}_B(\Omega(i,\Pi_i+\cdot), 
\Delta \Pi_i)$ and $\mathfrak{N}_B(\Omega(i,\widetilde \Pi_i+\cdot), \Delta \widetilde \Pi_i)$, respectively, and we obtain 
\be{rnm}
R_{n,M}\subset\bigcup_{N=1}^{n/L_n} \bigcup_{k=1}^{4n/M L_n}
\bigcup_{V \subset \{0,\dots,N-1\} \colon |V|=k} \ \ \bigcup_{\bar Y\in \mathfrak{V}_{N,V}}  \bar R(N,k,V,\bar Y),
\ee
where 
\begin{align}
\label{defvn}
\nonumber \mathfrak{V}_{N,V}=\bigg\{
\bar Y:= (Y_{i}^0,Y_{i}^1&)_{i\in V}\in (\Z^2)^{V}\colon\, Y_{i}^{1}=Y_{i+1}^{0} \ 
\text{if}\ (i,i+1)\in V^2 \, \text{and}\ \exists\,   \Pi\in \{0\}\times \Z^N \colon\\
&\sum_{i=0}^{N-1} |\Delta \Pi_i|\leq \tfrac{n}{L_n}, \, \widetilde V_{\Pi,M}=V, \, 
(\Pi_{i},\Pi_{i+1})_{i\in V}=(Y_{i}^0,Y_{i}^1)_{i\in V} \bigg\},
\end{align}
where
\begin{align}
\bar R(N,k,V,\bar Y)=\bigg\{&\Omega\in \{A,B\}^{\N_0\times \Z} \colon\\
\nonumber & \sum_{i\in V} \mathfrak{N}_B(\Omega(i,Y_i+\cdot), \Delta Y_i)
\leq  \sum_{i\in V}\mathfrak{N}_B(\Omega(i,\widetilde Y_i+\cdot), \Delta \widetilde Y_i)-\gep \tfrac{n}{L_n}\bigg\}
\end{align}
and where, with each $\bar Y\in \mathfrak{V}_{N,V}$, we associate $\widetilde Y= (\widetilde \Pi_i,\widetilde 
\Pi_{i+1})_{i\in V}$ with $\widetilde \Pi$ the mesoscopic displacement strategy associated with the $\Pi$, which 
in \eqref{defvn} guarantees that $\bar Y\in   \mathfrak{V}_{N,V}$.

For $N\in \{1,\dots,n/L_n\}$, $k\in \{1,\dots,4n/M L_n\}$, $V \subset \{0,\dots,N-1\} \colon |V|=k$ and 
$\bar Y\in  \mathfrak{V}_{N,V}$, we let $\mathfrak{N}(\bar Y)$ be the number of blocks that have to be crossed 
entirely in the vertical direction in the catch-up columns (i.e., those columns indexed in $V$). By construction, 
we note that $\mathfrak{N}(\bar Y)\geq \mathfrak{N}(\widetilde Y)$, so that the number of blocks that have to be 
crossed vertically in the catch-up columns for the mesoscopic strategy of displacement $\bar Y$ is not smaller 
than its counterpart for the auxiliary mesoscopic strategy of displacement $\widetilde Y$. We then note that 
$\Omega\in R_{n,M}$ necessarily implies that there exists  $N,k,V, \bar Y$ such that  $\mathfrak{N}(\widetilde  Y)
\geq \gep n/L_n$, and therefore $\mathfrak{N}(\bar Y)\geq \gep n/L_n$ (since it is always the case that  
$\mathfrak{N}(\bar Y)\geq \mathfrak{N}(\widetilde Y)\geq \sum_{i\in V} \mathfrak{N}_B(\Omega(i,\widetilde Y_i+\cdot), 
\Delta \widetilde Y_i)$. As a consequence, we can bound  $\P_{\Omega}(R_{n,M}\neq \emptyset)$ as follows:
\begin{align}
\label{bsupR}
 \P_{\Omega}(R_{n,M}\neq & \emptyset)\\ \nonumber  
 \leq & \sum_{N=1}^{n/L_n} 
\sum_{k=1}^{4n/M L_n} \ \ \ \sum_{V \subset \{0,\dots,N-1\} \colon |V|=k}\ \  \  
\sum_{\bar Y\in \mathfrak{V}_{N,V}\colon \mathfrak{N}(\bar Y)\geq \gep n/L_n}   \P_\Omega\big(\bar R(N,k,V,\bar Y)\big).
 \end{align}
By a standard application of Cramer's Theorem we obtain that the probability under the sum of the right-hand side
in \eqref{bsupR} is uniformly bounded by $e^{-c_\gep n/L_n}$ with $c_\gep>0$.  At this stage we note that, uniformly 
in $N$ and $k$, we can bound $| V \subset \{0,\dots,N-1\} \colon |V|=k|$ from above by $\binom{n/L_n}{4n/M L_n}$, 
which for $M$ large enough has an exponential growth rate that is smaller than $c_\gep$. Moreover, uniformly in 
$N,k,V$ 
\be{bcmv}
|\mathfrak{V}_{N,V}|\leq 2^k \binom{n/L_n}{2k} \Big(\frac{n}{k L_n}\Big)^k\leq  2^{\frac{4n}{M L_n}} 
\binom{n/L_n}{8n/M L_n} \Big(\frac{M}{8}\Big)^{\frac{8n}{M L_n}}.
\ee
The upper bound in \eqref{bcmv} can be understood as follows. First, in each catch-up columns we have to choose 
the length of the mesoscopic displacement and this gives rise to the term $\binom{n/L_n}{2k}$, since the sum 
of all mesoscopic increments is bounded above by $n/L_n$. Next, we have to choose the sign of these $k$ 
increments and this gives a factor $2^k$. Finally, in each catch-up columns we have to choose the height of the 
entrance block ($\Pi_i$). Once again, the fact that the sum of all mesoscopic displacement is smaller than $n/L_n$
tells us that the difference between the height of the exit block of a given catch-up column and the height of the 
entrance block of the following catch-up column is bounded by the sum of the absolute value of the mesoscopic 
increments that have been made in between these two columns. But once again, since the sum of these mesoscopic 
displacements in absolute value is smaller than $n/L_{n}$ the number of choices for the heights of all entrance blocks 
in catch-up columns is bounded above by $(n/k L_n)^k$. This completes the proof because when $M$ is chosen 
large enough the exponential growth rate in the right-hand side of \eqref{bcmv} is smaller than $c_\gep$.  


\subsubsection{Step 4}
\label{step44} 

In this step, we recall the coarse-grained version of the Hamiltonian defined in \eqref{alterh} and we use it to introduce, 
in the catch-up columns, those trajectories moving according to the auxiliary mesoscopic strategy, i.e., for 
$(N,\Theta_{\text{traj}},x,u)$ we set  
\begin{align}
\label{defA5*}
\widehat{A}_5=\prod_{i\in \widetilde V_{\Pi,M} }\, 
\widehat Z_{L_n}(\widetilde \Theta_i,v_i)
\end{align}
with $\widetilde \Theta_i=(\Omega(i,\widetilde \Pi_{i}+\cdot),\widetilde \Xi_i,1)$. Our aim is to prove that, for $\gep >0$ 
and $M$ large enough, we have for all $\Omega\in \{A,B\}^{\N_0\times \Z}$ and all $n\in \N$ that $\widehat{A}_4\leq
\widehat{A}_5 e^{\gep n}$ uniformly in $(N,\Theta,x,u)$.

For each $i\in \widetilde V_{\Pi,M}$, we pick $\pi_i\in \cW_{\widetilde \Theta_i, v_i,L_n}$.  Since $v_i L_n$ is the minimal 
number of steps required to cross the column indexed by $i$, and since $\mathfrak{N}_B(\widetilde\Theta_i)$ is the 
number of $B$-blocks that have to be crossed vertically by any trajectory crossing a block column of type $\widetilde \Theta_i$,
we can assert that the number of steps performed by $\pi_i$ in the $B$-blocks belongs to $\{\mathfrak{N}_B(\widetilde\Theta_i) 
L_n,\dots,\mathfrak{N}_B(\widetilde\Theta_i) L_n+3 L_n\}$. Therefore, recalling the definition \eqref{alterh} and the crude bound 
$r_1+\dots+r_k\leq 4n/M L_n$, we can assert that, for all $(N,\Theta_{\text{traj}},x,u)$ and all $(\pi_i)_{i\in \widetilde V_{\Pi,M}}
\in  \bigotimes_{i\in \widetilde V_{\Pi,M}} \cW_{\widetilde \Theta_i, v_i,L_n}$,
\be{upppb}
 \sum_{s=1}^{k} \sum_{x=0}^{r_s-1}\widehat H_{v_{i_s+x} L_n,L_n}^{\, 
 \Omega(i_s+x,\widetilde \Pi_{i_s+x}+\cdot)}(\pi_{i_s+x})\geq  \tfrac{\beta-\alpha}{2}\, 
 \sum_{i\in \widetilde{V}_{\Pi,M}} \mathfrak{N}_{B}(\widetilde \Theta_i) L_n- 12 \tfrac{\alpha n}{M L_n}.
\ee
Thus, it suffices to choose $M$ so large that $12 \alpha/M\leq \gep$ to complete the proof of the step.


\subsubsection{Step 5}
\label{step55}  

In this step we replace, for each of the pieces of trajectories crossing the catch-up column, the coarse-grained Hamiltonian 
$\widehat H$ by the original Hamiltonian. Thus, we set 
\begin{align}
\label{defA6*}
\widehat{A}_6=\prod_{i\in \widetilde V_{\Pi,M} }\, 
Z_{L_n}^{\,\omega_{J_i}}(\widetilde \Theta_i,v_i).
\end{align}
Our aim is to show that, for $\gep >0$ and $M$ large enough, we have for all $\Omega\in \{A,B\}^{\N_0\times \Z}$ that 
\be{event*}
B^2_{n,M}=\{\omega\colon\, \widehat{A}_5\leq 
\widehat{A}_6\, e^{\gep n} \ \text{for all}\ (N,\Theta_{\text{traj}},x,u)\}
\ee
satisfies $\P_{\omega}(B^2_{n,M})\to 1$ as $n\to \infty$. We will not give the details of the proof, because it is completely 
similar to that of Step 1. The only difference is that we replace $\widetilde u=\bar u_1+\dots+\bar u_r$ by $\widetilde v
=\bar v_1+\dots+\bar v_r$.


\subsubsection{Step 6}
\label{step66}

Let $(\widetilde{w}_i)_{i\in\N}$ be an i.i.d.\ sequence of Bernouilli trials, independent of $\omega,\Omega$. For 
$(N,\Theta_{\text{traj}},x,u)$ we set $\widehat{u}=\sum_{s=1}^k \bar u_s-\bar v_s$. By changing the $u_i$ into $v_i$ 
in those catch-up columns, we remove $\widehat{u}L_n$ steps from the trajectories associated with $(N,
\Theta_{\text{traj}},x,u)$, so that they have become trajectories associated with $(N,\widetilde \Theta_{\text{traj}},
\widetilde x,v)$. In this step, we will concatenate the trajectories associated with $(N,\widetilde \Theta_{\text{traj}},
\widetilde x,v)$ with an $\widehat{u} L_n$-step trajectory to recover a trajectory that belongs to $\cW_{n,M}$.

Therefore, our aim is to show that for all $\gep>0$ there exists an $M\in \N$ such that, for $\P-a.e.\,\Omega$ and 
for all $\omega,\widetilde \omega \in \{A,B\}^{\N}$, we have for $n$ large enough that  $\widetilde A_2\, 
\widehat A_6\leq  e^{\gep n} Z_{\,n,L_n}^{(\omega_{J}\,,\,\widetilde \omega), \Omega} (M)$ for all $(N,\Theta,x,u)$. 
The proof is completely similar to that of Step 2 in the proof of Proposition~\ref{pr:formimpp} (see
Section~\ref{step2}). For this reason, we will not repeat the details.


\subsubsection{ Step 7}
\label{step77}

In this step, we average over the microscopic disorders $\omega,\widetilde{\omega}$. We recall \eqref{eventt} 
and \eqref{event}, and we set $B_{n,M}=B^{1}_{n,M}\cap B^2_{n,M}$.  With the help of Steps 1--6 above we 
can state that for every $\gep>0$ there exists an $M\in \N$ such that, for  $\P-a.e.\,\Omega$ and for 
$\omega\in B_{n,M}$ and $\widetilde \omega\in \{A,B\}^\N$, we have (recall \eqref{spint})
\be{upfin}
A_1 = A_2 \widetilde A_2 \leq  e^{6\gep n} Z_{\,n,L_n}^{(\omega_{J}\,,\,\widetilde \omega), \Omega} (M)
\quad \text{for all $(N,\Theta,x,u)$}.
\ee 
Next, we recall \eqref{partfunc11} and  \eqref{variaf22}, and we use \eqref{upfin} to state that for all $\gep>0$ 
there exists an $M\in \N$ such that, for $\P-a.e.\,\Omega$ and $n$ large enough and for $\omega\in B_{n,M}$ 
and $\widetilde \omega\in \{A,B\}^\N$, 
\begin{align}
\label{variaf222*}
Z_{n,L_n}^{\omega,\Omega}
&\leq e^{6\gep n} \sum_{N=1}^{n/L_n} 
\sum_{\Theta_{\text{traj}}\in \widetilde{\cD}_{L_n,N}^{\infty} } 
 \,\sum_{x\in \cX_{\Theta_{\text{traj}},\Omega}^{\infty,\infty} }
\,\sum_{u\in \,\cU_{\Theta_{\text{traj}},x,n}^{\,\infty,\infty,L_n}} 
Z_{\,n,L_n}^{(\omega_{J}\,,\,\widetilde{\omega}), \Omega} (M).
\end{align}  
We use \eqref{concmesut} to claim that there exists $C_1,C_2>0$ such that, for all for all $\Omega$, all 
$n\in \N$, all $M\in \N$ and all $J$, 
\be{inegco*}
\P_{\omega,\widetilde{\omega}}\Big(\Big|\tfrac1n
\log Z_{n,L_n}^{(\omega_{J}\,,\,\widetilde{\omega}),\Omega}(M)
-f_{n}^{\Omega}(M)\Big|\geq \gep \Big)\leq C_1 e^{-C_2 \gep^2 n}.
\ee 
We set also
\be{inegco2*}
D_{n,M}=\bigcap_{(N,\Theta_{\text{traj}},x,u)} 
\Big\{\Big|\tfrac1n\log Z_{n,L_n}^{(\omega_{J}\,,
\,\widetilde{\omega}),\Omega}(M)-f_{n}^{\Omega}(M)\Big|\leq \gep\Big\},
\ee
recall the definition of $c_n$ in \eqref{carsum} (with $M=m=\infty)$), and use \eqref{inegco*} and the fact 
that $c_n$ grows subexponentially in $n$, to obtain $\lim_{n\to\infty} \P_{\omega,\widetilde{\omega}}(D_{n,M}^{c})= 0$. 
For all $(\omega,\widetilde{\omega})$ satisfying $\omega\in B_{n,M}$ and $(\omega,\widetilde{\omega})\in D_{n,M}$, 
we can rewrite \eqref{variaf222*} as 
\begin{align}
\label{variafin*}
Z_{n,L_n}^{\omega,\Omega}
&\leq c_n\ e^{n  f_{n}^{\Omega}(M) +7\gep n}.
\end{align}  
Consequently, recalling \eqref{boundel}, for $M$ large enough we have 
\be{endres*}
f^{\Omega}_{n}(\alpha,\beta)\leq \P(B_{n,M}^{c} \cup D_{n,M}^{c}) 
\, C_{\text{uf}}(\alpha)+ \frac{\log c_n}{n}+\frac1n
\E\Big(1_{\{B_{n,M}\cup D_{n,M}\}}  \big(n f_{n}^{\Omega}(M) +7\gep n\big)\Big).
\ee
Since $\P(B_{n,M}^c\cup D_{n,M}^c)$ and $\lim_{n\to\infty} (\log c_n)/n=0$, the proof of Proposition~\ref{pr:formimppp} 
is complete.


\section{Proof of Theorem \ref{varformula2}: slope-based variational formula}
\label{varfo2}

We are now ready to show how the variational formula in \eqref{genevarold} can be 
transformed into the variational formula in \eqref{genevar}. We recall that, by the 
definition of $\bar \cR_p$ in \eqref{setbarp}, the variational formula in \eqref{genevar} 
can also be written as
\be{genevarte}
f(\alpha,\beta;p) =\sup_{M\geq 1} \, \sup_{\bar{\rho}\in \bar{\cR}_{p,M}}\,
\sup_{v\,\in\,\bar{\cB}}\,\,\frac{\bar N(\bar{\rho},v)}{\bar D(\bar{\rho},v)}.
\ee

Let $\cF_{\overline\cV_M}$ and $\bar{\cF}$ be the counterparts of $\cB_{\overline\cV_M}$ 
and $\bar{\cB}$ for Borel functions instead of continuous functions, i.e., 
\be{BVdef2}
\cF_{\overline\cV_M}=\big\{(u_\Theta)_{\Theta\in \overline\cV_M}\in\R^{\overline\cV_M}
\colon\,u_\Theta\geq t_\Theta\,\,\forall\,\Theta \in \overline\cV_M,\,
\Theta \mapsto u_\Theta \mbox{ Borel}\big\}
\ee 
and
\be{newB1}
\bar{\cF}=\{v=(v_A,v_B,v_\cI)\in \bar\cD\times\bar\cD\times[1,\infty)\},
\ee
where
\be{newC1}
\bar\cD=\left\{l \mapsto v_l \text{ on } [0,\infty) \text{ Lebesgue measurable with } 
v_l\geq 1+l\quad \forall\,l \geq 0\right\}.
\ee
The proof of Theorem \ref{varformula2} is divided into 4 steps, organized as 
Sections~\ref{Step1}--\ref{Step4}. In Step 1 we show that the supremum over 
$\cB_{\overline \cV_M}$ in \eqref{genevarold} may be extended to $\cF_{\overline\cV_M}$, 
i.e.,
\begin{align}
\label{ineg3}
\sup_{\rho\in \cR_{p,M}}\,\sup_{(u_\Theta)_{\Theta\in \overline\cV_M}\,
\in\,\cB_{\,\overline\cV_M}}\,\,\frac{N(\rho,u)}{D(\rho,u)}
=\sup_{\rho\in \cR_{p,M}}\,\sup_{(u_\Theta)_{\Theta\in \overline\cV_M}\,
\in\,\cF_{\,\overline\cV_M}}\,\,\frac{N(\rho,u)}{D(\rho,u)}.
\end{align}
In Step 2 we show that the supremum over $\bar \cB$ in \eqref{genevar} may be extended
to $\bar\cF$, i.e.,
\begin{align}
\label{ineg4}
\sup_{\bar{\rho}\in \bar{\cR}_{p,M}}\,
\sup_{v\in\,\bar{\cB}}\,\,\frac{\bar N(\bar{\rho},v)}{\bar D(\bar{\rho},v)}
=\sup_{\bar{\rho}\in \bar{\cR}_{p,M}}\,\sup_{v\in\bar{\cF}}\,\,
\frac{\bar N(\bar\rho,v)}{\bar D(\bar\rho,v)}.
\end{align}
Then, the proof of Theorem \ref{varformula2} is achieved with the help of 
 Steps 3 and 4 which, combined with Theorem~\ref{varformula}, allow us to show 
\begin{align}
\label{ineg1}
&f(\alpha,\beta;p) \geq 
\sup_{M\geq 1}  \sup_{\bar{\rho}\in \bar{\cR}_{p,M}}\,\sup_{v\in\bar{\cF}}\,
\frac{\bar N(\bar\rho,v)}{\bar D(\bar\rho,v)},\\
\label{ineg2}
&f(\alpha,\beta;p) \leq
\sup_{M\geq 1} \sup_{\bar{\rho}\in \bar{\cR}_{p,M}}\,\sup_{v\in\bar{\cF}}\,
\frac{\bar N(\bar\rho,v)}{\bar D(\bar\rho,v)}.
\end{align}
Along the way we will need a few technical facts, which are collected in 
Appendices~\ref{B}--\ref{appA}. 


\subsection{Step 1: extension of the variational formula}
\label{Step1}

For $c\in (0,\infty)$, let $u(c)=(u_\Theta(c))_{\Theta\in \overline\cV_M}$ be the counterpart 
of the function $v(c)$ introduced in (\ref{defucp1}-\ref{defucp3}). For $\Theta\in \overline\cV_M$ 
and $c\in (0,\infty)$, set 
\be{defucp}
u_\Theta(c)= \left\{
\begin{array}{ll}
\vspace{.1cm}
t_\Theta
& \mbox{if } \partial^+_u (u\,\psi(\Theta,u))(t_\Theta) \leq c, \\
\vspace{.1cm}
z
&\mbox{otherwise, with } z \mbox{ such that } \partial^-_u (u\,\psi(\Theta,u))(z) 
\geq c \geq \partial^+_u (u\,\psi(\Theta,u))(z),
\end{array}
\right.
\ee
where  $z$ exists and is finite by Lemma~\ref{fep1} in Appendix~\ref{B}, and is unique 
by the strict concavity of $u\to u\psi(\Theta,u)$ for $\Theta\in \overline \cV_M$ (see
Lemma~\ref{concav} in Appendix~\ref{B}). The fine properties of $\Theta\mapsto u_\Theta(c)$ 
are given in Lemma \ref{reguc} in Appendix \ref{appBB}.

For $(\alpha,\beta)\in \CONE$ and $\rho\in \cM_1(\overline\cV_M)$ such that $\int_{\overline{\cV}_M} 
t_\Theta\,  \rho(d\Theta)<\infty$, set
\be{defgg}
g(\rho; \alpha,\beta)=\sup_{u\in  \cF_{\overline{\cV}_M}}\,\frac{ N(\rho,u)}{ D(\rho,u)},
\ee
with the convention that  $N(\rho,u)/ D(\rho,u)=-\infty$ when $D(\rho,u)=\infty$. The equality in
\eqref{ineg3} is a straightforward consequence of the following lemma.

\bl{maximo}
For $(\alpha,\beta)\in \CONE$ and $\rho\in \cM_1(\overline\cV_M)$ such that  $g(\rho; \alpha,\beta)>0$, 
\be{tery}
g(\rho;\alpha,\beta)=\frac{ N(\rho,\bar u)}{ D(\rho,\bar u)} \quad \text{with }  
\bar u=u(g(\rho;\alpha,\beta)).
\ee
Moreover, $u= \bar u$ for $\rho$-a.e.\  $\Theta\in \overline \cV_{M}$ for all $u\in \cF_{\overline{\cV}_M}$ 
satisfying $g(\rho;\alpha,\beta)=\frac{N(\rho,u)}{D(\rho,u)}$. 
\el

\begin{proof}
The following lemma will be needed in the proof.

\bl{maximo*}
For $(\alpha,\beta)\in \CONE$ and $\gep>0$ there exists a $t_\gep>0$ such that, for all  $\rho\in 
\cM_1(\overline\cV_M)$ and all $u\in \cF_{\overline{\cV}_M}$ satisfying $D(\rho,u)\in (t_\gep,\infty)$,
\be{tery*}
\frac{N(\rho, u)}{D(\rho, u)} \leq \gep.
\ee
\el

\begin{proof}
Pick $\gep>0$. By Lemma \ref{boundunif}, there exists a $C_\gep>0$ such that $\psi(\Theta,u)\leq \gep/2$ 
for $\Theta\in\overline{\cV}_M$ and $u\geq \max \{C_\gep, t_\Theta\}$. For $R\in (0,\infty)$, set $B^-(R)
=\{\Theta\in\overline{\cV}_M\colon u_\Theta\leq R\}$ and $B^+(R)=\{\Theta\in \overline{\cV}_M\colon 
u_\Theta> R\}$, and write
\be{subdi}
\frac{ N(\rho, u)}{ D(\rho, u)}=\frac{\int_{B^-(C_\gep)} u_\Theta \psi(\Theta,u_\Theta) \rho(d\Theta)}{ D(\rho, u)}
+\frac{\int_{B^+(C_\gep)} u_\Theta \psi(\Theta,u_\Theta) \rho(d\Theta)}{ D(\rho, u)}.
\ee
By the definition of $C_\gep$ and since $u_\Theta\geq t_\Theta$ for all $\Theta\in \overline{\cV}_M$, we 
can bound the second term in the right-hand side of \eqref{subdi} by $\gep/2>0$. The first term in the 
right-hand side of \eqref{subdi} in turn can be bounded from above by $C_\gep C_{\text{uf}}(\alpha)
/D(\rho, u)$ (recall \eqref{boundel}). Consequently, it suffices to choose $t_\gep= 2 C_\gep 
C_{\text{uf}}(\alpha)/\gep$ to complete the proof.
\end{proof}

We resume the proof of Lemma \ref{maximo}. By assumption, we know that $g(\rho)>0$, which entails 
that $\int_{\overline \cV_M} t_\Theta \rho(d\Theta)<\infty$. Thus, Lemma \ref{reguc}(iv) tells us that 
$D(\rho, u(c))<\infty$ for all $c>0$. We argue by contradiction. Suppose that $\frac{ N(\rho,\bar u)}
{D(\rho,\bar u)}<g(\rho)$, and pick $u\in \cF_{\overline\cV_M}$ such that $D(\rho, u)<\infty$. Write
\be{dec}
\frac{ N(\rho,u)}{ D(\rho,u)}
= \frac{ N(\rho,\bar u)+ [ N(\rho, u)- N(\rho,\bar u)]}{ D(\rho,\bar u)
+[ D(\rho,u)- D(\rho,\bar u)]},
\ee
where
\begin{equation}
\begin{aligned}
N(\rho, u)- N(\rho,\bar u)
&=\int_{\overline{\cV}_M}  u_\Theta \psi(\Theta,u_\Theta)
-\bar u_\Theta \psi(\Theta,\bar u_\Theta)\, \rho(d\Theta).
\end{aligned}
\end{equation}
The strict concavity of $u\mapsto u\,\psi(\Theta,u)$ on $[t_\Theta,\infty)$ for every $\Theta\in
\overline{\cV}_M$, together with the definition of $\bar u$ in \eqref{tery}, allows us to estimate 
\be{refg}
N(\rho, u)- N(\rho,\bar u)\leq  g(\rho) \int_{\overline{\cV}_M}\,
(u_\Theta-\bar u_{\Theta})\,\,\rho(d\Theta).
\ee
Consequently, \eqref{dec} becomes 
\be{uni}
\frac{ N(\rho,u)}{ D(\rho,u)}
\leq\frac{ N(\rho,\bar u)+ g(\rho) [\bar D(\rho,u)-\bar D(\rho,\bar u)]}
{ D(\rho,\bar u)+[ D(\rho,u)- D(\rho,\bar u)]}.
\ee
Define $G=x\mapsto [N(\rho,\bar u)+ g(\bar\rho) x]/[D(\rho,\bar u)+x]$ on $(-D(\rho,\bar u),\infty)$. 
Note that $N(\rho,\bar u)/ D(\rho,\bar u)$ $<g(\rho)$ implies that $G$ is strictly increasing with 
$\lim_{x\to \infty} G(x)=g(\rho)$. Use Lemma \ref{maximo*} to assert that $\N(\rho,u)/D(\rho,u)
\leq \tfrac12 g(\rho)$ when $D(\rho,u) \geq t_{\tfrac12 g(\rho)}$. But then, for all $u$ satisfying 
$D(\rho,u)\leq t_{\frac{g(\rho)}{2}}$, \eqref{uni} gives
\be{uni*}
\frac{ N(\rho,u)}{ D(\rho,u)}
\leq G\bigg(t_{\tfrac{g(\rho)}{2}}-D(\rho,\bar u)\bigg)<g(\rho).
\ee
Consequently,
\be{}
\sup_{u\in \cF_{\overline\cV_M} } \frac{ N(\rho,u)}{ D(\rho,u)}\leq 
\max\bigg\{\tfrac{g(\rho)}{2},  G\Big(t_{\frac{g(\rho)}{2}}-D(\rho,\bar u)\Big)\bigg\}<g(\rho),
\ee 
which is a contradiction, and so $g(\rho)=N(\rho,\bar u)/D(\rho,\bar u)$.

It remains to prove that if $u\in \cF_{\overline{\cV}_M}$ satisfies $g(\rho)=N(\rho,u)/D(\rho,u)$, 
then $u= \bar u$ for $\rho$-a.e.\  $\Theta\in \overline \cV_{M}$.  We proceed again by contradiction, 
i.e., we suppose that a such $u$ is not equal to $\bar u$ for  $\rho$-a.e.\  $\Theta\in \overline \cV_{M}$. 
In this case, both inequalities in \eqref{refg} and \eqref{uni} are  strict, which immediately yields 
that $\frac{N(\rho,u)}{D(\rho,u)}<g(\rho)$.
\end{proof}


\subsection{Step 2: extension of the reduced variational formula}
\label{Step2}

Recall (\ref{defucp1}--\ref{defucp3}) and, for $(\alpha,\beta)\in \CONE$ and 
$\bar \rho\in \cM_1\big(\R_+\cup \R_+\cup\{\cI\}\big)$ such that 
$\int_{0}^\infty (1+l)\,[\bar\rho_A+\bar\rho_B](dl)<\infty$, set 
\be{defhh}
h(\bar\rho; \alpha,\beta)=\sup_{v\in \bar\cF} \, \frac{\bar N(\bar \rho,  v)}{\bar D(\bar\rho, v)}.
\ee 
Recall (\ref{defucp1}-\ref{defucp3}). The equality in \eqref{ineg4} is a straightforward consequence 
of the following lemma.

\bl{reducc}
For $(\alpha,\beta)\in \CONE$ and $\bar \rho\in \cM_1\big(\R_+\cup \R_+\cup\{\cI\}\big)$ such that 
$h(\bar \rho;\,\alpha,\beta)>0$,
\be{tery**}
h(\bar \rho;\alpha,\beta)=\frac{\bar N(\bar \rho,\bar v)}{\bar D(\bar\rho,\bar v)},\quad \text{with}\ 
\bar v=v(h(\bar\rho;\alpha,\beta)).
\ee
For $v\in \bar\cF$ satisfying $h(\bar \rho; \alpha,\beta)=\frac{\bar N(\bar\rho,v)}{\bar D(\bar\rho,v)}$, 
$v= \bar v$ for $\bar \rho$-a.e.\  $(k,l)\in \{A,B\}\times [0,\infty)$ or $k=\cI$.
\el

\begin{proof}
The proof is similar to that of Lemma \ref{maximo}. The counterpart of Lemma \ref{maximo*} is 
obtained by showing that for $(\alpha,\beta)\in \CONE$ and $\gep>0$ there exists a $t_\gep>0$ 
such that, for all $\bar\rho\in \cM_1\big(\R_+\cup \R_+\cup\{\cI\}$ and all $v\in \bar \cF$ satisfying 
$\bar D(\bar \rho,v)\in (t_\gep,\infty)$,
\be{teryext}
\frac{ \bar N(\bar \rho, v)}{ \bar D(\bar \rho, v)}\leq \gep.
\ee
The proof of \eqref{tery**} is similar to the proof of Lemma \ref{maximo*} and relies mainly on 
Lemmas \ref{l:lemconv2}(ii--iii)  and on the limit given in Lemma \ref{l:lemconv}(ii). 

It remains to show that $h(\bar \rho;\alpha,\beta)=\frac{\bar N(\bar \rho,\bar v)}{\bar D(\bar\rho,\bar v)}$
and that $v\in \bar\cF$ satisfying $h(\bar \rho; \alpha,\beta)=\frac{\bar N(\bar\rho,v)}{\bar D(\bar\rho,v)}$
necessarily satisfies $v= \bar v$ for $\bar \rho$-a.e.\  $(k,l)\in \{A,B\}\times [0,\infty)$ or $k=\cI$. The proofs 
are similar to their counterparts in Lemma \ref{maximo} and require the strict concavity of $u\mapsto u 
\tilde\kappa (u,l)$ for $l\in \R$ and of $u\mapsto u \phi_\cI(u)$, as well as the definition of $\bar v$ in 
(\ref{defucp1}--\ref{defucp3}). 
\end{proof}


\subsection{Step 3: lower bound}
\label{Step3}

The inequality in \eqref{ineg1} is a straightforward consequence of the following lemma.

\bl{GGG}
For all $(\alpha,\beta)\in \CONE$, $\bar \rho\in \bar\cR_{p,M}$ and $v=(v_A,v_B,v_\cI)\in\bar{\cF}$ there exists 
$\rho\in \cR_{p,M}$ and $u=(u_\Theta)_{\Theta\in \overline{\cV}_M} \in  \cF_{\,\overline\cV_M}$ satisfying 
\be{rere}
\frac{N(\bar\rho,v)}{D(\bar\rho,v)}\leq \frac{N(\rho,u)}{D(\rho,u)}.
\ee
\el

\begin{proof}
Since $\bar\rho\in\bar\cR_{p,M}$, there exist $\rho\in \cR_{p,M}$ and $h\in \cE$ such that
$\bar\rho=G_{\rho,h}$. For $\Theta\in \overline{\cV}_M$ and $k\in \{A,B\}$, set 
$d_{k,\Theta}=l_{k,\Theta}/h_{k,\Theta}$ if $h_{k,\Theta}>0$ and $d_{k,\Theta}=0$ 
otherwise. Put
\be{defuu}
u_\Theta=h_{A,\Theta}\,v_{A,d_{A,\Theta}}+h_{B,\Theta}\,v_{B,\,d_{A,\Theta}}
+h_{\cI,\Theta}\, v_{\cI},\quad \Theta\in \overline{\cV}_M.
\ee
To prove \eqref{rere},  we recall \eqref{Gdef} and integrate \eqref{defuu} against $\rho$. 
Since $\bar\rho=G_{\rho,h}$, it follows that 
\be{equ}
D(\bar\rho, v) = \int_{\overline{\cV}_M} u_\Theta\, \rho(d\Theta)=D(\rho,u).
\ee
Since $h\in \cE$ we can assert that 
\be{appart}
(h_{A,\Theta},h_{B,\Theta},h_{\cI,\Theta}), (h_{A,\Theta}\,v_{A,\,d_{A,\Theta}}, 
h_{B,\Theta}\,v_{B,\,d_{B,\Theta}},h_{\cI,\Theta}\,v_{\cI})\in \cL(\Theta;\,u_\Theta),
\quad \Theta\in \overline{\cV}_M,
\ee
which, with the help of \eqref{Bloc of type I}, allows us to write
\begin{align}
\label{defu3}
u_\Theta\,\Psi(\Theta,u_\Theta) \geq h_{A,\Theta} \
&v_{A,\,d_{A,\Theta}}\ 
\tilde{\kappa}(v_{A,\,d_{A,\Theta}}, d_{A,\Theta})\\
+\nonumber &h_{B,\Theta}\,v_{B,\,d_{B,\Theta}}\, 
\left [\tilde{\kappa}(v_{B,\,d_{B,\Theta}},d_{B,\Theta})+\tfrac{\beta-\alpha}{2}\right]
+ h_{\cI,\Theta}\, v_{\cI}\,\phi_{\cI}(v_{\cI};\alpha,\beta).
\end{align}
After integrating \eqref{defu3} against $\rho$ and using that $\bar \rho=G_{\rho,u}$,
we obtain 
\begin{align}
\label{fin*}
\int_{\overline{\cV}_M} u_\Theta \psi(\Theta,u_\Theta) \rho(d\Theta)
\geq
\bigg[\int_{0}^{\infty} &v_{A,l}\,\kappa(v_{A,l},l)\,\rho_{A}(dl)\\
\nonumber 
&+ \int_{0}^{\infty} v_{B,l}\,\big[\kappa(v_{B,l},l)+\tfrac{\beta-\alpha}{2}\big]\,
\rho_{B}(dl)+\rho_{\cI}\, v_{\cI}\,\phi_{\cI}(v_{\cI};\alpha,\beta)\bigg].
\end{align}
Thus, \eqref{rere} is immediate from \eqref{equ} and \eqref{fin*}.
\end{proof}


\subsection{Step 4: upper bound}
\label{Step4}

The proof of \eqref{ineg2}  is a straightforward consequence of the following lemma.

\bl{ABC}
For all $(\alpha,\beta)\in \CONE$,  $\rho\in \cR_{p,M}$ and $u\in\cB_{\,\overline\cV_M}$, 
there exist $\bar \rho\in \bar\cR_{p,M}$ and $v\in\bar{\cF}$ such that 
\be{prstep2}
\frac{N(\rho,u)}{D(\rho,u)}\leq \frac{N(\bar\rho,v)}{D(\bar\rho,v)}.
\ee
\el

\begin{proof}
Since $u \in \cB_{\,\overline\cV_M}$, Proposition~\ref{Borel} in Appendix~\ref{appA} allows 
us to state that there exist $h\in \cE$ and $r\in\cU(h)$ such that  
\begin{align}
\label{targ}
u_\Theta\,\psi(\Theta,u_\Theta)=h_{A,\Theta}\, 
r_{A,\Theta}\,\tilde{\kappa}\big(r_{A,\Theta},
\tfrac{l_{A,\Theta}}{h_{A,\Theta}}\big)&+h_{B,\Theta}\,
r_{B,\Theta}\,\big[\tilde{\kappa}\big(r_{B,\Theta},
\tfrac{l_{B,\Theta}}{h_{B,\Theta}}\big)+\tfrac{\beta-\alpha}{2}\big]\\
&\nonumber
+h_{\cI,\Theta}\, r_{\cI,\Theta}\,\phi_\cI(r_{\cI,\Theta}), 
\quad \quad \forall \Theta \in \overline\cV_M,
\end{align}
and 
\be{targ1}
h_{A,\Theta}\, r_{A,\Theta}+ h_{B,\Theta}\, r_{B,\Theta}+h_{\cI,\Theta}\, 
r_{\cI,\Theta}=u_\Theta, \quad \quad \forall \Theta \in \overline\cV_M.
\ee
Define $\rho_{A,h},\rho_{B,h},\rho_{\cI,h}$ to be the probability measures on $\overline
\cV_M$ given by
\be{measure}
\frac{d\rho_{k,h}}{d\rho}(\Theta)
=\frac{h_{k,\Theta}}{\int_{\overline\cV_M} h_{k,\Theta}\,\rho(d\Theta)},
\quad k\in \{A,B,\cI\}. 
\ee
For $l\in\R_+$, set 
\be{defv*}
v_{A,l}=E_{\rho_{A,h}}\big[r_{A,\Theta}\,\big|\,\tfrac{l_{A,\Theta}}{h_{A,\Theta}}=l\,\big],
\quad 
v_{B,l}=E_{\rho_{B,h}}\big[r_{B,\Theta}\,\big|\,\tfrac{l_{B,\Theta}}{h_{B,\Theta}}=l\,\big],
\ee
and 
\be{defvI}v_{\cI}=E_{\rho_{\cI,h}}\big[r_{\cI,\Theta}\,\big].
\ee
The fact that $r\in \cU(h)$ implies that $v_\cI\geq 1$ and $v_{k,l}\geq 1+l$ for $l\in\R_+$ 
and $k\in\{A,B\}$. Moreover, the Borel measurability of $\Theta\mapsto h_{k,\Theta}$ for 
$k\in \{A,B\}$ implies the Lebesgue measurability of $l\mapsto v_{k,l}$ for $k\in\{A,B\}$. 
Therefore, $(v_A,v_B,v_\cI)\in \bar\cF$. 

By the concavity of $a\mapsto a\tilde{\kappa}(a,b)$ and $\mu\mapsto \mu\phi_{\cI}(\mu)$, 
we obtain that 
\begin{align}
\label{rdf}
&E_{\rho_{A,h}}\Big[r_{A,\Theta}\,\tilde{\kappa}(r_{A,\Theta},l)\,\big|\,
\tfrac{l_{A,\Theta}}{h_{A,\Theta}}=l\Big]\leq v_{A,l}\,\widetilde \kappa(v_{A,l},l),\\
\nonumber 
&E_{\rho_{B,h}}\Big[r_{B,\Theta}\,\big(\tilde{\kappa}(r_{B,\Theta},l)
+\tfrac{\beta-\alpha}{2}\big)\,\big|\,\tfrac{l_{B,\Theta}}{h_{B,\Theta}}
=l\Big]\leq v_{B,l}\,\big[\widetilde \kappa(v_{B,l},l)+\tfrac{\beta-\alpha}{2}\big],\\
\nonumber 
&E_{\rho_{\cI,h}}\Big[r_{\cI,\Theta}\,\phi_{\cI}(r_{\cI,\Theta})\,\Big]
\leq v_{\cI}\,\phi_{\cI}(v_\cI). 
\end{align}
Integrate $\eqref{targ}$ against $\rho$, to obtain
\begin{align}
\label{inte}
\int_{\overline\cV_M} u_\Theta\,\psi(\Theta,u_\Theta)\,\rho(d\Theta)
&=  {\textstyle \int_{\overline\cV_M} h_{A,\Theta}\,\rho(d\Theta)} \ 
E_{\rho_{A,h}}\big[r_{A,\Theta}\,\widetilde \kappa\big(r_{A,\Theta},
\tfrac{l_{A,\Theta}}{h_{A,\Theta}}\big)\big]\ \\
\nonumber &\quad + {\textstyle \int_{\overline\cV_M} h_{\cI,\Theta}\,\rho(d\Theta)} \  
E_{\rho_{\cI,h}}\big[r_{\cI,\Theta}\,\phi_{\cI}(r_{\cI,\Theta})\big]\\
\nonumber 
&\quad +{\textstyle \int_{\overline\cV_M} h_{B,\Theta}\,\rho(d\Theta)}\  
E_{\rho_{B,h}}\big[r_{B,\Theta}\,\big(\widetilde \kappa\big(r_{B,\Theta},
\tfrac{l_{B,\Theta}}{h_{B,\Theta}}\big) +\tfrac{\beta-\alpha}{2}\big)\big].
\end{align}
Set $\bar\rho=G_{\rho,h}$. In the right-hand side of \eqref{inte} take the 
conditional expectation with respect to $\tfrac{l_{A,\Theta}}{h_{A,\Theta}}$ and 
$\tfrac{l_{B,\Theta}}{h_{B,\Theta}}$ in the first term and the second term, 
respectively. Then use the inequalities in \eqref{rdf}, to obtain
\begin{align}
\label{inte2}
\int_{\overline\cV_M} u_\Theta\,\psi(\Theta,u_\Theta)\,\rho(d\Theta)
\leq \int_{0}^{\infty} 
&v_{A,l}\, \widetilde \kappa(v_{A,l},l)\,\bar \rho_{A}(dl)\\
\nonumber 
&+\int_{0}^{\infty} v_{B,l}\,\big[\tilde{\kappa}(v_{B,l},l)+\tfrac{\beta-\alpha}{2}\big]\,
\bar\rho_{B}(dl)+\bar\rho_{\cI}\,v_{\cI}\,\phi_{\cI}(v_{\cI},\alpha,\beta).
\end{align}
Similarly, integrate \eqref{targ1} against $\rho$ and take the conditional expectation 
with respect to $\tfrac{l_{A,\Theta}}{h_{A,\Theta}}$ and $\tfrac{l_{B,\Theta}}{h_{B,\Theta}}$, 
to obtain
\be{intete}
\int_{\overline \cV_M} u_\Theta\,\rho(d\Theta)
= \int_{0}^{\infty} v_{A,l}\,\bar \rho_{A}(dl)
+\int_{0}^{\infty} v_{B,l}\,\bar\rho_{B}(dl)+\bar\rho_{\cI}\,v_{\cI}.
\ee
At this point, \eqref{inte} and \eqref{intete} allow us to conclude that $N(\rho,u)/D(\rho,u) 
\leq N(\bar \rho,v)/D(\bar \rho,v)$. Since $v\in \bar\cF$, this completes the proof.
\end{proof}


\section{Phase diagrams: proof of Theorems~\ref{phgene}, \ref{phdgsup} and \ref{phdgsub}}
\label{phdiagsupsub}


\subsection{Proof of Theorem~\ref{phgene}}

We first state and prove a proposition that compares $f$, $f_\cD$ and $f_{\cD_2}$, 
and deals with the regularity and the monotonicity of $f_{\cD}$. Recall the definition 
of $\alpha^*$ in \eqref{alpha}.

\begin{proposition}
\label{propa}
(i) $f(\alpha,\beta)=f_{\cD}(\alpha,\beta)$ for $(\alpha,\beta)\in \CONE\colon\,
\beta\leq 0$.\\
(ii) $x\mapsto f_{\cD}(x,0)$ is continuous, convex and non-increasing on $[0,\infty)$.\\
(iii) $f_{\cD}(x,0)>f_{\cD_2}$ for $x\in [0,\alpha^*)$ and $f_{\cD}(x,0)=f_{\cD_2}$ 
for $x\in [\alpha^*,\infty)$.
\end{proposition}

\begin{proof}
(i) Note that for $(\alpha,\beta)\in \CONE\colon \beta\leq 0$ and $v\geq 1$ we have 
$\phi^\cI(v,\alpha,\beta)=\tilde\kappa(v,0)$, because the Hamiltonian in \eqref{Zinf} 
is always non-positive. Thus, \eqref{genevar} and \eqref{genevarD1d} imply (i).

\medskip\noindent 
(ii) Since $(\alpha,\beta)\mapsto f(\alpha,\beta)$ is convex on $\R^2$ (being the 
pointwise limit of a sequence of convex functions; see \eqref{felimdeff}) and is 
everywhere finite, it is also continuous. Therefore (i) implies that $x\in[0,\infty)
\mapsto f_{\cD}(x,0)$ is continous and convex. The monotonicity of $x\mapsto f_{\cD}(x,0)$ 
can be read off directly from \eqref{genevarD1d}.

\medskip\noindent 
(iii) It is obvious from \eqref{genevarD1d} and \eqref{genevarD2d} that $f_{\cD}(x,0)
\geq f_{\cD_2}$ for every $x\in [0,\infty)$. Recall \eqref{alpha}. Since $x\mapsto 
f_{\cD}(x,0)$ is continuous and non-increasing, it follows that $f_{\cD}(x,0)>f_{\cD_2}$ 
for $x\in [0,\alpha^*)$ and $f_{\cD}(x,0)=f_{\cD_2}$ for $x\in [\alpha^*,\infty)$. 
\end{proof}

We are now ready to give the proof of Theorem~\ref{phgene}.

\begin{proof}
(a) Pick $\alpha\geq 0$ and note that every element of $J_\alpha$ can be written in the 
form $(\alpha+\beta,\beta)$ (with $\beta\geq -\alpha/2)$, so that $f_\cD$ is constant and
equal to $f_{\cD}(\alpha,0)$ on $J_\alpha$. By the convexity of $(\alpha,\beta)\mapsto
f(\alpha,\beta)$ and by Proposition \ref{propa}(i), we know that $g_\alpha\colon\,\beta
\mapsto f(\alpha+\beta,\beta)-f_{\cD}(\alpha,0)$ is convex and equal to $0$ when $\beta 
\leq 0$. Therefore $g_\alpha$ is non-decreasing, and we can define 
\be{defb}
\beta_c(\alpha)=\inf\{\beta\geq 0\colon\,f(\alpha+\beta,\beta)>f_{\cD}(\alpha,0)\},
\ee
so that $(\alpha+\beta,\beta)\in \cD$ if and only if $\beta\leq \beta_c(\alpha)$.
It remains to check that $\beta_c(\alpha)<\infty$. 

To that aim, pick any $\bar \rho\in \bar \cR_p$ such that $\bar \rho_\cI>0$ and any $v\in 
\bar\cB$ such that $v_\cI>1$ and $\bar D(\bar \rho, v)<\infty$, recall \eqref{varformod}, 
and note that $\lim_{\beta\to\infty} \bar N(\alpha+\beta,\beta;\,\bar\rho,v)=\infty$ 
because $\lim_{\beta\to \infty}\phi_{\cI}(v_\cI;\alpha+\beta,\beta)=\infty$. The last 
observation is obtained by considering a trajectory in $\cW_{v_\cI L}$ that starts at 
$(0,0)$ ends at $(L,0)$, and in between stays in the $A$-solvent except when the 
microscopic disorder $\omega$ has 3 consecutive $B$-monomers, in which case the trajectory 
makes an excursion of size $3$: one step south, one step east and one step north, inside 
the $B$-solvent. Such a trajectory has energy $\beta c L$ for some $c>0$.

\medskip\noindent 
(b) This is a straightforward consequence of the fact that $f_{\cD}(\alpha,\beta)=f_{\cD}
(\alpha-\beta,0)$ for $(\alpha,\beta)\in \CONE$.
\end{proof} 


\subsection{Proof of Theorem \ref{phdgsup}}\label{secpdsup}

\begin{proof}
(a) 
We want to show that $\alpha^*\in (0,\infty)$. To that aim, we first prove that 
$f_{\cD}(0,0)>f_{\cD_2}$, which by the continuity of $x \mapsto f_{\cD}(x,0)$ implies 
that $\alpha^*>0$. It is easy to see that $p \delta_{A,0}(dl)+(1-p) \delta_{B,0}(dl)
\in \bar \cR_p$, since this corresponds to trajectories travelling along the $x$-axis 
while staying on one side. Thus, \eqref{genevarD1d} implies that $f_{\cD}(0,0)\geq 
\tilde\kappa(u^*,0)$, where $u^*$ is the unique maximizer of $u \mapsto\tilde{\kappa}
(u,0)$ on $[1,\infty)$. Moreover, by Lemma~\ref{l:lemconv2}(ii), we have $\tilde{\kappa}
(u,l)\leq\tilde\kappa(u^*,0)$ for every $l\in [0,\infty)$, $u\geq 1+l$ and $(u,l)\neq (u^*,0)$. Since 
$\delta_{A,0}(dl)$ does not belong to $\bar\cR_p$, it follows that $f_{\cD_2}<f_{\cD}(0,0)$, 
and therefore the continuity of $x\mapsto f_{\cD}(x,0)$ implies that $\alpha^*>0$.

It remains to show that $\alpha^*<\infty$. Recall Hypothesis \ref{hyp2}. We argue by 
contradiction. Assume that $f_{\cD}(n,0)>f_{\cD_2}$ for all $n\in\N$. Then Hypothesis \ref{hyp1}
and Lemma \ref{reducc}  tell us that there exists a sequence $(\bar \rho_n)_{n\in \N}$ 
in $\cT_{p}$ such that
\be{assu1}
f(n,0) = f_{\cD}(n,0) 
= \frac{\bar{N}_{\cD}(\bar{\rho}_n,v_n;\,n,0)}{\bar D_{\cD}(\bar{\rho}_n,v_n)}
>f_{\cD_2}>0, \quad n\in \N,
\ee
with $v_n=v(f_\cD(n,0))$, where we recall (\ref{defucp1}--\ref{defucp3}). For simplicity, 
we write $f_2=f_{\cD_2}$ and $\bar v=v(f_2)$ (recall (\ref{defucp1}-\ref{defucp3})) until the end of the proof. Since $f_\cD(n,0)>f_2$ for $n\in \N$,  
Lemma~\ref{regvc}(ii) yields $v_{n,A,l}\leq  \bar v_{A,l}$ for $ l\in [0,\infty), n\in \N$.
Note that Lemma~\ref{regvc} is stated for fixed $(\alpha,\beta)\in \CONE$, which is 
not the case here because $(\alpha,\beta)=(n,0)$. However, in the present setting 
Lemma~\ref{regvc}(ii) remains true for $v_A$ since, by definition, the value taken by 
$v_{A,l}(c)$ for $l\in [0,\infty)$ and $c \in (0,\infty)$ does not depend on $(\alpha,\beta)$.

We can write
\begin{align}
\label{iyt}
f_{\cD}(n,0)-f_{2}
=& \frac{\int_0^{\infty} v_{n,A,l}
[\tilde{\kappa}(v_{n,A,l},l)-f_2](\bar\rho_{n,A}+\bar\rho_{n,\cI}\,\delta_0)(dl)}
{D_\cD(\bar\rho_n,v_n)}\\
\nonumber &\qquad\qquad\qquad \qquad\qquad \qquad
+\frac{\int_0^{\infty} v_{n,B,l} [\tilde{\kappa}(v_{n,B,l},l)
-\frac{n}{2}-f_2] \bar \rho_{n,B}(dl)}{D_\cD(\bar\rho_n,v_n)},
\end{align}
and the concavity of $v\mapsto v\tilde\kappa(v,l)$, together with the fact that $v_{n,A,l}
\leq \bar v_{A,l}$ for all $l\in [0,\infty)$ and $\partial_v (v\tilde\kappa(v,l))(\bar v_{A,l})
=f_2$, implies that
\be{concavt*}
\bar v_{A,l} \tilde{\kappa}(\bar v_{A,l},l)-v_{n,A,l} \tilde{\kappa}(v_{n,A,l},l)
\geq f_2 (\bar v_{A,l}- v_{n,A,l}).
\ee
Since $\tilde\kappa$ is uniformly bounded from above and $v_{n,B,l}\geq 1+l$ for every 
$l\in[0,\infty)$, we can claim that, for $n$ large enough,
\be{bouco}
v_{n,B,l} [\tilde{\kappa}(v_{n,B,l},l)-\tfrac{n}{2}-f_2]
\leq -\tfrac{n}{4} (1+l), \quad l\in [0,\infty).
\ee
Consequently, \eqref{assu1} and (\ref{iyt}--\ref{bouco}) allow us to write 
\be{teri}
\int_0^{\infty} \bar v_{A,l} [
\tilde{\kappa}(\bar v_{A,l},l)-f_2](\bar\rho_{n,A}+\bar\rho_{n,\cI} \delta_0)(dl)  
-\frac{n}{4}\int_0^{\infty} 1+l\,\bar \rho_{n,B}(dl)>0 \ \text{for $n$ large enough},
\ee
which clearly contradicts Hypothesis \ref{hyp2} because $\bar \rho_n\in \cT_p$ for 
$n\in \N$. The proof is therefore complete.

\medskip \noindent 
(b-c) By the definition of $\cD$, $\cD_1$ and $\cD_2$ in \eqref{defD}, \eqref{genevarD1} 
and \eqref{genevarD21}, we know that $\cD=\cD_1\cup \cD_2$ and $\cD_1\cap \cD_2
=\emptyset$. Thus, Theorem~\ref{phgene}(a) implies that (b) and (c) will be proven 
once we show that $J_\alpha\cap \cD_2=\emptyset$ for $\alpha\in [0,\alpha^*)$ and 
$J_\alpha\cap \cD_1=\emptyset$ for $\alpha\in [\alpha^*,\infty)$. Moreover, 
Theorem ~\ref{phgene}(b) tells us that $f_\cD$ is constant and equal to $f_{\cD}(\alpha,0)$
on each $J_\alpha$ with $\alpha\in [0,\infty)$. Consequently it suffices to show that      
$f_{\cD}(\alpha,0)>f_{\cD_2}$ for $\alpha\in [0,\alpha^*)$ and $f_{\cD}(\alpha,0)=f_{\cD_2}$ 
for $\alpha\in [\alpha^*,\infty)$. But this is precisely what Proposition~\ref{propa}(iii) 
states.

\medskip\noindent 
(d) Pick $\alpha\in [0,\infty)$ and assume that Hypothesis \ref{hyp1} holds. Then there
exists a $\bar\rho_\alpha\in\cO_{p,\alpha,0}$ such that $\bar\rho_{\alpha,\cI}>0$. Set
$\bar v=v(f_\cD(\alpha,0))$ and
\be{defbetac}
\tilde \beta_c(\gamma(\alpha)) = \inf\big\{\beta>0\colon\,
\phi_{\cI}(\bar v_{A,0};\beta+\alpha,\beta)
>\tilde \kappa (\bar v_{A,0},0)\big\}.
\ee
The proof will be complete as soon as we show that $\tilde\beta_c(\gamma(\alpha))
=\beta_c(\gamma(\alpha))$ (recall \eqref{intjalpha1}). Note that, by the convexity of 
$\beta\to\phi_{\cI}(\bar v_{A,0};\alpha+\beta,\beta)$, and since $ \phi_{\cI}(\bar v_{A,0}; 
\beta+\alpha,\beta)=\tilde \kappa (\bar v_{A,0},0)$ for $\beta\leq 0$, we necessarily 
have that $\phi_{\cI}(\bar v_{A,0}; \alpha+\beta,\beta)>\tilde \kappa(\bar v_{A,0},0)$ 
for all $\beta>\tilde\beta_c(\gamma(\alpha))$. From Propositions \ref{propa}(i) and 
\ref{maxv}(2), we have that
\be{eqqf}
f(\alpha,0)=f_{\cD}(\alpha,0)=\tfrac{\bar{N}_{\cD}(\bar{\rho}_\alpha,
\bar v)}{\bar{D}_{\cD}(\bar{\rho}_\alpha,\bar v)},
\ee 
and
\be{derc}
\bar{N}_{\cD}(\bar{\rho}_\alpha,\bar v) 
= \int_{0}^{\infty} \bar v_{A,l}\,\tilde\kappa(\bar v_{A,l},l)\,
[\bar{\rho}_{\alpha,A}+\bar{\rho}_{\alpha,\cI}\,\delta_0](dl)
+ \int_{0}^{\infty} \bar v_{B,l}\,\big[\tilde\kappa(\bar v_{B,l},l)
-\tfrac{\alpha}{2}\big]
\,\bar \rho_{\alpha,B}(dl).
\ee
By the definition of $\bar{v}=v(f_{\cD}(\alpha,0))$ in (\ref{defucp1}--\ref{defucp3}), 
we have that $\partial_v( v\,\tilde{\kappa}(v,0))(\bar v_{A,0})=f_{\cD}(\alpha,0)$. 
For notational reasons we suppress the dependence on $\alpha$ of $f_{\cD}$.

First, assume that $\phi_{\cI}(\bar v_{A,0};\beta+\alpha,\beta)=\tilde \kappa (\bar v_{A,0},0)$ 
(we also suppress the dependence on $(\beta+\alpha,\beta)$). Then, since 
$v \to  v \phi_{\cI}(v)$ and $v \to v \tilde{\kappa}(v,0)$ are both concave and $\phi_{\cI}
(v)\geq \tilde{\kappa}(v,0)$ for all $v\geq 1$, we have that $v\to v \phi_{\cI}(v)$ is 
differentiable at $\bar v_{A,0}$ and 
\be{trs}
\partial_v [ v\,\tilde{\kappa}(v,0)](\bar v_{A,0})
= \partial_v [ v\,\phi_{\cI}(v)](\bar v_{A,0})=f_{\cD}.
\ee
Thus, for any $\bar\rho\in \bar\cR_{p}$ and $v\in \bar \cB$, we set $\tilde v\in \bar \cB$ 
such that $\tilde{v}\equiv v$, except for $\tilde{v}_\cI$, which takes the value 
$\bar{v}_{A,0}$. In other words, 
\begin{equation}
\label{ecco}
\begin{aligned}
 \frac{\bar{N}(\bar{\rho}, v)}{\bar{D}(\bar{\rho},v)}
&= \frac{\bar{N}_{\cD}(\bar{\rho}, \tilde v)
+\bar \rho_\cI  [v_{\cI}  \phi_{\cI}(v_{\cI}) 
-\bar v_{A,0} \tilde \kappa(\bar v_{A,0},0) ]}{\bar{D}_\cD(\bar{\rho},\tilde v)
+\bar \rho_\cI  [v_{\cI}  -\bar v_{A,0} ]}\\
&\leq \frac{\bar{N}_{\cD}(\bar{\rho}, \tilde v)
+\bar\rho_\cI f_{\cD} (v_{\cI}-\bar v_{A,0})}{\bar{D}_\cD(\bar{\rho},\tilde v)
+\bar \rho_\cI (v_{\cI}-\bar v_{A,0})},
\end{aligned}
\end{equation}
where we use \eqref{trs}, the concavity of $v\to v \phi_{\cI}(v)$ and the fact that 
$\phi_{\cI}(\bar v_{A,0})=\tilde \kappa (\bar v_{A,0},0)$ by assumption. At this stage 
we recall that, by definition, $\frac{\bar{N}_{\cD}(\bar{\rho}, \tilde v)}{\bar{D}_\cD
(\bar{\rho}, \tilde v)}\leq f_{\cD}$. Hence \eqref{ecco} entails that $\frac{N(\bar{\rho},v)}
{D(\bar{\rho},v)}\leq f_{\cD}$. Thus, $\beta_c(\gamma(\alpha))\geq \tilde \beta_c
(\gamma(\alpha))$.

The other inequality is much easier. Indeed, if we consider $\beta$ such that $\phi_{\cI}
(\bar v_{A,0}; \alpha+\beta,\beta)>\tilde \kappa (\bar v_{A,0},0)$, then $\bar{N}
(\bar{\rho}_\alpha, \bar v)>\bar{N}_{\cD}(\bar{\rho}_\alpha, \bar v)$ because 
$\bar \rho_{\cI,\alpha}>0$. As a consequence, $f(\alpha+\beta,\beta)>f_{\cD}(\alpha,0)$, 
so that $\beta>\beta_c(\gamma(\alpha))$, and therefore $\beta_c(\gamma(\alpha))\leq 
\tilde\beta_c(\gamma(\alpha))$.

\medskip\noindent 
(e) We recall that for $\alpha\in [\alpha^*,\infty)$ we have $\bar{v}=v(f_{D_2})$ and 
therefore $\bar v_{A,0}$ is constant. In (c) we proved that $\beta_c(\gamma(\alpha))
=\tilde\beta_c(\gamma(\alpha))$ on $[\alpha^*,\infty)$. The definition of $\tilde\beta_c
(\gamma(\alpha))$ in \eqref{defbetac} can be extended to $\alpha\in [0,\infty)$. Since 
$\alpha^*>0$, the proof of (d) will be complete once we show that $\alpha\mapsto\tilde 
\beta_c(\gamma(\alpha))$ is concave, continuous and non-decreasing on $(0,\infty)$ 
and that $\lim_{\alpha\to \infty} \tilde\beta_c(\gamma(\alpha))<\infty$.   

By using the same argument as the one we used in the proof of Theorem~\ref{phgene}(a), 
we can claim  that $\lim_{\beta\to \infty} \phi_{\cI}(\bar v_{A,0};\alpha+\beta,\beta)
=\infty$ for every $\alpha\in [0,\infty)$. Consequently, $\tilde \beta_c(\gamma(\alpha))
\in [0,\infty)$ for every $\alpha\in [0,\infty)$. Moreover, the convexity of $(\alpha,\beta)
\mapsto \phi_{\cI}(\bar v_{A,0};\alpha,\beta)$ implies the convexity of $(\alpha,\beta)
\mapsto \phi_{\cI}(\bar v_{A,0};\alpha+\beta,\beta)-\tilde{\kappa}(\bar v_{A,0},0)$, 
which is also non-negative. Therefore, the set $\{(\alpha,\beta)\colon\,\alpha\in 
[0,\infty),\beta\in [-\tfrac{\alpha}{2},\tilde \beta_c(\gamma(\alpha))]\}$ is convex, and 
consequently $\alpha\mapsto \tilde\beta_c(\gamma(\alpha))$ is concave on $[0,\infty)$. This 
concavity yields that $\alpha\mapsto \tilde\beta_c(\gamma(\alpha))$ is continuous on $(0,\infty)$,
and since it is bounded from below by $0$, also that it is non-decreasing. 

It remains to show that $\lim_{\alpha\to \infty} \tilde{\beta}_c(\gamma(\alpha))<\infty$. To that 
aim,  we define $\tilde \beta_c(\infty)$ by choosing  $\alpha=\infty$ in \eqref{defbetac}.
Since  $\phi_{\cI}(\bar v_{A,0};\infty,\beta)\leq \phi_{\cI}(\bar v_{A,0};\alpha+\beta,
\beta)$ for every $\alpha\geq 0$ and $\beta\in [-\tfrac{\alpha}{2},\infty)$, it follows 
that $\tilde \beta_c(\gamma(\alpha))\leq\tilde\beta_c(\infty)$ for every $\alpha\in(0,\infty)$. 
Therefore  it suffices to prove that $\tilde \beta_c(\infty)<\infty$. But this is a 
consequence of the fact that $\lim_{\alpha\to \infty} \phi_{\cI}(\bar v_{A,0};\infty,
\beta)=\infty$. This limit is obtained  by using again the same argument as the one we 
used in the proof of Theorem \ref{phgene}(a).

\medskip\noindent 
(f) This is a straightforward consequence of the fact that $f=f_{\cD}$ on $\cD_1$ and 
$f_{\cD}$ is a function of $\alpha-\beta$.

\medskip\noindent 
(g) This is a direct consequence of the definition of the $\cD_2$-phase in \eqref{genevarD21} 
and the fact that $f_{\cD_2}$ does not depend on $\alpha$ and $\beta$ (see \eqref{genevarD2d}).
\end{proof}


\subsection{Proof of Theorem \ref{phdgsub}}

The proof of Theorem~\ref{phdgsub} has much in common with that of 
Theorem~\ref{phdgsup} in Section~\ref{secpdsup}. For this reason we 
only focus on the points that need to be adapted from the proof of 
Theorem~\ref{phdgsup}.

\begin{proof}
(a) The proof of $\bar\alpha^*\in (0,\infty)$ follows the same scheme as the proof of 
Theorem~\ref{phdgsup}(a). The bound $f_{\cD}(0,0) \geq \tilde\kappa(u^*,0)$ remains 
valid ($u^*$ being the unique maximizer of $u \mapsto\tilde{\kappa}(u,0)$). Moreover,  
$\{\bar \rho\in \bar \cR_p\colon\; K_B(\bar \rho)=K_p\}$ does not contain any element 
of the form $x\delta_{A,0}(dl)+(1-x) \delta_{B,0}(dl)$, since the fraction of horizontal 
steps taken in solvent $B$ can obviously be reduced by allowing the path to sometimes
travel in solvent $A$ with a non-zero slope. This implies that $f_{\cD}(0,0)>f_{\cD_2}(0,0)$, 
and therefore $\bar \alpha^*>0$.

The upper bound is also similar to that of Theorem~\ref{phdgsup}(a). The only difference i
s that $f_{\cD_2}$ depends on $n$, so that we write $f_{2}(n)$ as well as $\bar v_n=v(f_2(n))$. 
Both \eqref{iyt} and \eqref{concavt*} are still true, whereas some attention is needed to adapt 
\eqref{bouco} since $f_2$ depends on $n$. However, it suffices to pick any $\bar \rho\in 
\bar\cR_p\setminus \cT_p$ such that $K_A(\bar \rho)+K_B(\bar \rho)<\infty$ and 
$\bar v^*\in \cB$, and such that $\bar v^*_{k,l}=1+l$ for $(k,l)\in \{A,B\}\times [0,\infty)$ 
and $\bar v^*_{\cI}=1$, to obtain that 
\be{eqqfff}
f_2(n)=f_{\cD}(n,0)\geq \tfrac{\bar{N}_{\cD}(\bar{\rho},
\bar v^*)}{\bar{D}_{\cD}(\bar{\rho},\bar v^*)}\geq c_1-\frac{c_2}{2} n,
\ee
where $c_1\in \R$ and $c_2=\int_0^\infty (1+l)\,\bar \rho_B(dl)/\bar{D}_{\cD}(\bar{\rho},\bar v^*)$.  
Since $p<p_c$ and $\bar \rho\in  \bar \cR_p\setminus \cT_p$, it follows that $K_A(\bar \rho)>0$ 
and $K_B(\bar \rho)>0$, and hence $c_2\in (0,1)$. Thus, \eqref{bouco} still holds with a 
right-hand side of the form $-(\frac14-\frac{c_2}{4}) (1+l)$, which contradicts 
Hypothesis~\ref{hyp3} and completes the proof.

\medskip \noindent 
(b) The proof is literally the same as that of Theorem \ref{phdgsup}(b-c-d).

\medskip \noindent 
(c) This is again a consequence of the fact that $f=f_{\cD}$ on $\cD$ and 
that $f_{\cD}$ is a function of $\alpha-\beta$.
\end{proof}


\begin{appendix}


\section{Uniform convergence of  path entropies}
\label{Path entropies}

In Appendix~\ref{A.1} we state a basic lemma (Lemma~\ref{conunif1}) about uniform 
convergence of path entropies in a single column. This lemma is proved with the
help of three additional lemmas (Lemmas~\ref{convularge}--\ref{convularge3}),
which are proved in Appendix~\ref{A.2}. The latter ends with an elementary lemma 
(Lemma~\ref{l:lemconv2}) that allows us to extend path entropies from rational to 
irrational parameter values. In Appendix~\ref{A.3}, we extend Lemma~\ref{conunif1} 
to entropies associated with sets of paths fullfilling certain restrictions on 
their vertical displacement. 


\subsection{Basic lemma}
\label{A.1}

We recall the definition of $\widetilde{\kappa}_L$, $L\in \N$, in \eqref{ttrajblock} 
and $\widetilde{\kappa}$ in \eqref{conventr}.

\begin{lemma}
\label{conunif1}
For every $\gep>0$ there exists an $L_\gep\in \N$ 
such that 
\begin{align}
\label{unif}
|\tilde{\kappa}_L(u,l)-\tilde{\kappa}(u,l)|
&\leq \gep \text{ for } L\geq L_\gep \text{ and } (u,l)\in \cH_L.
\end{align}
\end{lemma}

\begin{proof} 
With the help of Lemma~\ref{convularge} below we get rid of those $(u,l) \in 
\cH\cap\mathbb{Q}^2$ with $u$ large, i.e., we prove that $\lim_{u\to\infty}
\kappa_L(u,l)=0$ uniformly in $L\in \N$ and $(u,l)\in \cH_L$. Lemma~\ref{convularge2} 
in turn deals with the moderate values of $u$, i.e., $u$ bounded away from infinity and 
$1+|l|$. Finally, with~Lemma \ref{convularge3} we take into account the small values 
of $u$, i.e., $u$ close to $1+|l|$. To ease the notation we set, for $\eta\geq 0$ and 
$U>1$, 
\be{defHM}
\cH_{L,\eta,U}=\{(u,l)\in \cH_L\colon  1+|l|+\eta\leq u\leq U\},
\qquad 
\cH_{\eta,U}=\{(u,l)\in \cH\colon  1+|l|+\eta \leq u\leq U\}.
\ee 

\bl{convularge}
For every $\gep>0$ there exists an $U_\gep>1$ such that 
\be{convularge1}
\tfrac{1}{uL}\log \big|\{\pi\in \cW_{uL}\colon \pi_{uL,1}=L\}\big|
\leq \gep \qquad \forall\,L\in \N,\,u\in 1+\tfrac{\N}{L}\colon u\geq U_\gep.
\ee
\el

\bl{convularge2}
For every $\gep>0$, $\eta>0$ and $U>1$ there exists an $L_{\gep,\eta,U}\in \N$ 
such that
\begin{align}
\label{unif1}
|\tilde{\kappa}_L(u,l)-\tilde{\kappa}(u,l)|
&\leq \gep \qquad \forall\,L\geq L_{\gep,\eta,U},\,(u,l)\in \cH_{L,\eta,U}.
\end{align}
\el

\bl{convularge3} 
For every $\gep>0$ there exist $\eta_\gep \in (0,\tfrac12)$ and $L_\gep\in \N$ 
such that
\be{convunf4alt}
|\tilde{\kappa}_L(u,l)-\tilde{\kappa}_{L}(u+\eta,l)|
\leq \gep \qquad \forall\,L\geq L_\gep,\,(u,l)\in \cH_{L},\, 
\eta\in(0,\eta_\gep)\cap \tfrac{2\N}{L}.
\ee
\el

\noindent
Note that, after letting $L\to\infty$ in Lemma \ref{convularge3}, we get 
\be{complco}
|\tilde{\kappa}(u,l)-\tilde{\kappa}(u+\eta,l)|
\leq \gep \qquad \forall\,(u,l)\in \cH\cap\mathbb{Q}^2,\, 
\eta\in(0,\eta_\gep)\cap \mathbb{Q}.
\ee

Pick $\gep>0$ and $\eta_\gep\in (0,\tfrac12)$ as in Lemma~\ref{convularge3}. Note 
that Lemmas~\ref{convularge}--\ref{convularge2} yield that, for $L$ large enough, 
\eqref{unif} holds on $\{(u,l)\in \cH_L\colon u\geq 1+|l|+\frac{\eta_\gep}{2}\}$. 
Next, pick $L\in \N$, $(u,l)\in\cH_{L}\colon  u\leq 1+|l|+\frac{\eta_\gep}{2}$ 
and $\eta_L\in (\frac{\eta_\gep}{2},\eta_\gep)\cap\frac{2\N}{L}$, and write
\be{bounfpr}
|\tilde{\kappa}_L(u,l)-\tilde{\kappa}(u,l)|\leq A+B+C,
\ee
where
\be{abc}
A=|\tilde{\kappa}_L(u,l)-\tilde{\kappa}_L(u+\eta_L,l)|, 
\quad 
B=|\tilde{\kappa}_L(u+\eta_L,l)-\tilde{\kappa}(u+\eta_L,l)|,
\quad 
C=|\tilde{\kappa}(u+\eta_L,l)-\tilde{\kappa}(u,l)|.
\ee
By \eqref{complco}, it follows that $C\leq \gep$. As mentioned above, the fact 
that $(u+\eta_L,l)\in \cH_L$ and $u+\eta_L\geq |l|+\frac{\eta_\gep}{2}$ implies 
that, for $L$ large enough, $B\leq \gep$ uniformly in $(u,l)\in\cH_{L}\colon\,
u\leq 1+|l|+\frac{\eta_\gep}{2}$. Finally, from Lemma~\ref{convularge3} we obtain 
that $A\leq \gep$ for $L$ large enough, uniformly in $(u,l)\in\cH_{L}\colon\,
u\leq 1+|l|+\frac{\eta_\gep}{2}$. This completes the proof of Lemma~\ref{conunif1}.
\end{proof}


\subsection{A generalization of Lemma \ref{conunif1}}
\label{A.3} 

In Section~\ref{proofofgene} we sometimes needed to deal with subsets of trajectories 
of the following form. Recall \eqref{add4}, pick $L\in \N$, $(u,l)\in \cH_L$ and 
$B_0, B_1\in \tfrac{Z}{L}$ such that
\be{condb}
B_1\,\geq 0\vee l\,\geq\, 0\wedge l\,\geq B_0 \quad \text{and}\quad B_1-B_0\geq 1.
\ee 
Denote by $\widetilde\cW_L(u,l,B_0,B_1)$ the subset of $\cW_L(u,l)$ containing those 
trajectories that are constrained to remain above $B_0 L$ and below $B_1L$ (see 
Fig.~\ref{figtratra}), i.e.,  
\begin{align}
\label{trajblockalt}
\widetilde\cW_L(u,l,B_0,B_1) &= \big\{\pi\in \cW_L(u,l) \colon\, 
B_0 L< \pi_{i,2}<B_1 L  \text{ for } i\in \{1,\dots,uL-1\}\big\},
\end{align}
and let 
\be{ttraj}
\widetilde\kappa_L(u,l,B_0,B_1) = \frac{1}{uL} \log |\widetilde\cW_L(u,l,B_0,B_1) |
\ee 
be the entropy per step carried by the trajectories in $\widetilde\cW_L(u,l,B_0,B_1)$. 
With Lemma~\ref{conunifalt1} below we prove that the effect on the entropy of the 
restriction induced by $B_0$ and $B_1$ in the set $\widetilde\cW_L(u,l)$ vanishes 
uniformly as $L\to \infty$.

\begin{figure}[htbp]
\begin{center}
\includegraphics[width=.48\textwidth]{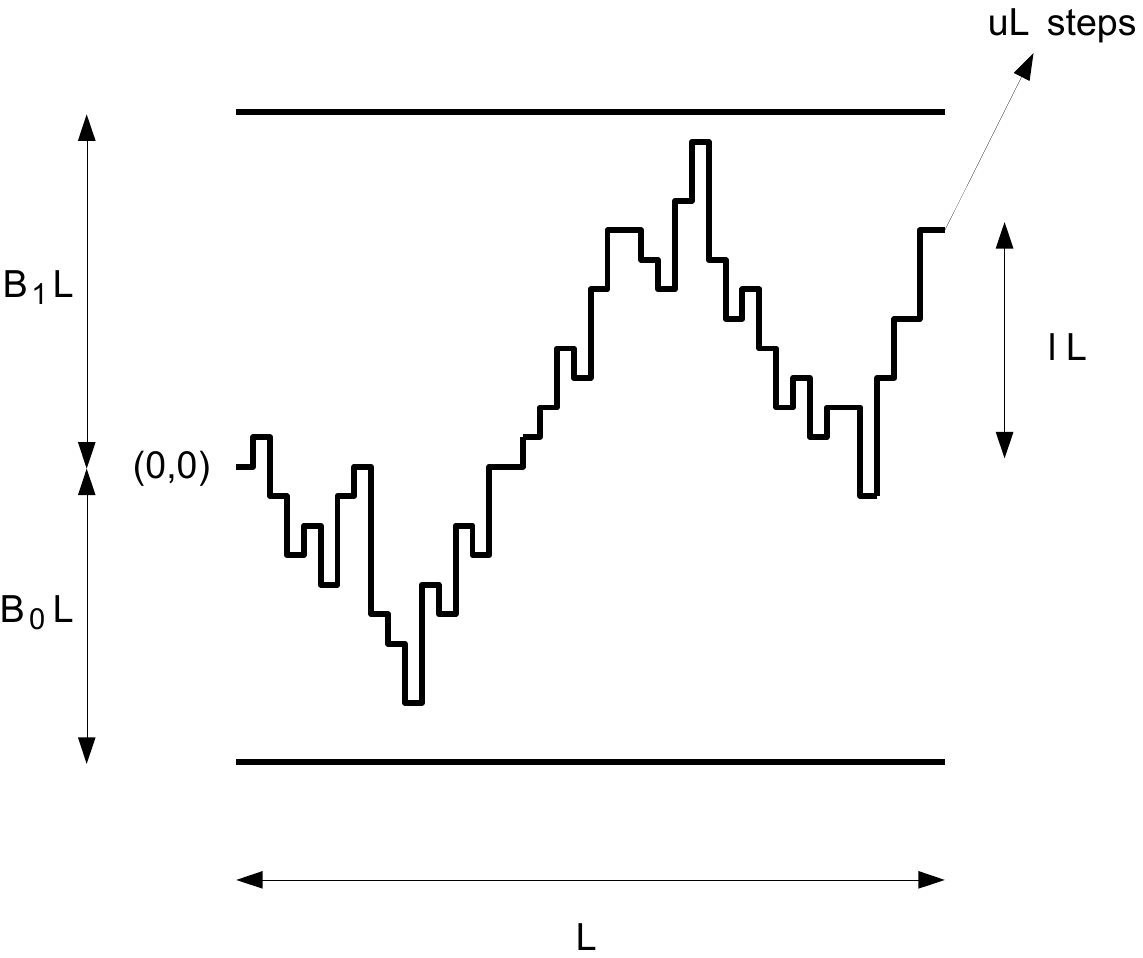}
\end{center}
\vspace{-4.3cm}
\caption{A trajectory in $\widetilde\cW_L(u,l,B_0,B_1)$.}
\label{figtratra}
\end{figure}

\begin{lemma}
\label{conunifalt1}
For every $\gep>0$ there exists an $L_\gep\in \N$ such that, for $L\geq L_\gep$, $(u,l)
\in \cH_L$ and $B_0,B_1\in \Z/L$ satisfying $B_1-B_0\geq 1$, $B_1\,\geq \max\{0,l\}$ 
and $B_0\leq \min\{0,l\}$,
\begin{align}
\label{unifalt}
& |\tilde{\kappa}_L(u,l,B_0,B_1)-\tilde{\kappa}_L(u,l)|
\leq \gep. 
\end{align}
\end{lemma}

\begin{proof}
The key fact is that $B_1-B_0 \geq 1$. The vertical restrictions $B_1\,\geq \max\{0,l\}$ 
and $B_0\leq \min\{0,l\}$ gives polynomial corrections in the computation of the 
entropy, but these corrections are harmless because $(B_1-B_0)L$ is large.  
\end{proof}


\subsection{Proofs of Lemmas \ref{convularge}--\ref{convularge3}}
\label{A.2} 


\subsubsection{Proof of Lemma~\ref{convularge}} 
\label{A.2.1}

The proof relies on the following expression:
\be{compex}
v_{u,L} = \big|\{\pi\in \cW_{uL}\colon \pi_{uL,1}=L\}\big|
=\sum_{r=1}^{L+1} \binom{L+1}{r} \binom{(u-1) L}{r} 2^r,
\ee
where $r$ stands for the number of vertical stretches made by the trajectory (a 
vertical stretch being a maximal sequence of consecutive vertical steps). Stirling's 
formula allows us to assert that there exists a $g\colon\,[1,\infty)\to (0,\infty)$ 
satisfying $\lim_{u\to \infty}g(u)=0$ such that
\be{stir}
\binom{uL}{L}\leq e^{g(u) uL}, \qquad u\geq 1,\,L\in \N.
\ee
Equations (\ref{compex}--\ref{stir}) complete the proof. 


\subsubsection{Proof of Lemma~\ref{convularge2}}
\label{A.2.2}

We first note that, since $u$ is bounded from above, it is equivalent to prove 
\eqref{unif1} with $\tilde{\kappa}_L$ and $\tilde{\kappa}$, or with $G_L$ and 
$G$ given by
\be{defG}
G(u,l) = u\tilde{\kappa}(u,l),\qquad
G_L(u,l) = u\tilde{\kappa}_L(u,l),\quad (u,l)\in \cH_L.
\ee
Via concatenation of trajectories, it is straightforward to prove that $G$ is 
$\mathbb{Q}$-concave on $\cH\cap \mathbb{Q}^2$, i.e., 
\be{Gconcav}
G(\lambda(u_1,l_1)+(1-\lambda)(u_2,l_2))
\geq \lambda G(u_1,l_1)+(1-\lambda) G(u_2,l_2),
\ \  \lambda\in \mathbb{Q}_{[0,1]},\,(u_1,l_1),(u_2,l_2)\in \cH\cap \mathbb{Q}^2.
\ee
Therefore $G$ is Lipschitz on every $K\cap\cH\cap \mathbb{Q}^2$ with $K \subset
\cH^0$ (the interior of $\cH$) compact. Thus, $G$ can be extended on $\cH^0$ to 
a function that is Lipschitz on every compact subset in $\cH^0$.

Pick $\eta>0$, $M>1$, $\gep>0$, and choose $L_\gep\in \N$ such that $1/L_\gep\leq 
\gep$. Since $\cH_{\eta,M}\subset\cH^0$ is compact, there exists a $c>0$ (depending 
on $\eta,M$) such that $G$ is $c$-Lipschitz on $\cH_{\eta,M}$. Moreover, any point 
in $\cH_{\eta,M}$ is at distance at most $\gep$ from the finite lattice $\cH_{L_\gep,
\eta,M}$. Lemma~\ref{lementr} therefore implies that there exists a $q_\gep\in \N$ 
satisfying
\be{convunf4}
|G_{qL_\gep}(u,l)-G(u,l)|
\leq \gep \qquad \forall\,(u,l)\in \cH_{L_\gep,\eta,M},\, q\geq q_\gep.
\ee
Let $L'=q_\gep L_\gep$, and pick $q\in \N$ to be specified later. Then, for $L\geq 
q L'$ and $(u,l)\in \cH_{L,\eta,M}$, there exists an $(u',l')\in\cH_{L_{\gep},\eta,M}$ 
such that $|(u,l)-(u',l')|_\infty \leq \gep$, $u>u'$, $|l|\geq |l'|$ and $u-u'\geq 
|l|-|l'|$. We recall \eqref{conventr} and write
\be{maj1}
0\leq G(u,l)-G_L(u,l)\leq A+B+C,
\ee
with
\begin{align}
A= |G(u,l)-G(u',l')|,\quad 
B= |G(u',l')-G_{L'}(u',l')|,\quad 
C= G_{L'}(u',l')-G_L(u,l).
\end{align}
Since $G$ is $c$-Lipschitz on $\cH_{\eta,M}$, and since $|(u,l)-(u',l')|_\infty\leq\gep$, 
we have $A\leq c\gep$. By \eqref{convunf4} we have that $B\leq \gep$. Therefore only $C$ 
remains to be considered. By Euclidean division, we get that $L=sL'+r$, where $s\geq q$ 
and $r\in \{0,\dots,L'-1\}$. Pick $\pi_1,\pi_2,\dots,\pi_s\in \cW_{L'}(u',|l'|)$, 
and concatenate them to obtain a trajectory in $\cW_{sL'}(u',|l'|)$. Moreover, 
note that 
\begin{align}
\label{impin}
uL-u'sL'&=(u-u')sL'+ur\\
\nonumber 
&\geq (|l|-|l'|)sL'+(1+|l|)r= (L-sL')+ (|l|L-s|l'|L'),
\end{align}
where we use that $L-sL'=r$, $u-u'\geq |l|-|l'|$ and $u\geq 1+|l|$. Thus, \eqref{impin} 
implies that any trajectory in $\cW_{L'}(u',|l'|)$ can be concatenated with 
an ($uL-u'sL'$)-step trajectory, starting at $(sL',s|l'|L')$ and ending at $(L,|l|L)$, 
to obtain a trajectory in  $\cW_{L}(u,|l|)$. Consequently,
\be{concat}
G_{L}(u,l)\geq \tfrac{s}{L}\log \kappa_{L'}(u',l')\geq \tfrac{s}{s+1} G_{L'}(u',l').
\ee
But $s\geq q$ and therefore $G_{L'}(u',l')-G_{L}(u,l)\leq \tfrac1q G_{L'}(u',l')\leq
\tfrac1q M \log 3$ (recall that $\log 3$ is an upper bound for all entropies per step). 
Thus, by taking $q$ large enough, we complete the proof.


\subsubsection{Proof of Lemma~\ref{convularge3}}
\label{A.2.3}

Pick $L\in \N$, $(u,l)\in \cH_L$, $\eta\in\frac{2\N}{L}$, and define the map $T\colon
\cW_{L}(u,l)\mapsto \cW_L(u+\eta,l)$ as follows. Pick $\pi\in  \cW_{L}(u,l)$, find 
its first vertical stretch, and extend this stretch by $\tfrac{\eta L}{2}$ steps. 
Then, find the first vertical stretch in the opposite direction of the stretch just
extended, and extend this stretch by $\tfrac{\eta L}{2}$ steps. The result of this 
map is $T(\pi)\in\cW_L(u+\eta,l)$, and it is easy to verify that $T$ is an injection, 
so that $|\cW_{L}(u,l)|\leq | \cW_{L}(u+\eta,l)|$. 

Next, define a map $\widetilde{T}\colon\,\cW_{L}(u+\eta,l)\mapsto \cW_L(u,l)$ as 
follows. Pick $\pi\in\cW_{L}(u+\eta,l)$ and remove its first $\tfrac{\eta L}{2}$ 
steps north and its first $\tfrac{\eta L}{2}$ steps south. The result is 
$\widetilde{T}(\pi)\in \cW_L(u,l)$, but $\widetilde{T}$ is not injective. However, 
we can easily prove that for every $\gep>0$ there exist $\eta_\gep>0$ and $L_\gep
\in \N$ such that, for all $\eta<\eta_\gep$ and all $L\geq l_\gep$, the number of 
trajectories in  $\cW_{L}(u+\eta,l)$ that are mapped by $\widetilde{T}$ to a 
particular trajectory in $\pi\in\cW_{L}(u,l)$ is bounded from above by $e^{\gep L}$, 
uniformly in $(u,l)\in\cH_L$ and $\pi\in\cW_{L},(u,l)$.

\medskip
This completes the proof of Lemmas~\ref{convularge}--\ref{convularge3}.


\section{Entropic properties}
\label{appBB}

 Recall Lemma~\ref{lementr}, 
where $(u,l)\mapsto \tilde{\kappa}(u,l)$ is defined on $\cH\cap\mathbb{Q}^2$. 

\bl{l:lemconv2}
(i) $(u,l)\mapsto u\tilde{\kappa}(u,l)$ extends to a continuous and strictly concave 
function on $\cH$.\\
(ii) For all $u\in [1,\infty)$, $l\mapsto \tilde{\kappa}(u,l)$ is strictly increasing 
on $[-u+1,0]$ and strictly decreasing on $[0,u-1]$.\\  
(iii) For all $l\in \R$,  $\lim_{u\to \infty} \tilde{\kappa}(u,l)=0$.\\
(iv) $\lim_{|l|\to \infty} \tilde{\kappa}(u,l)=0$ uniformly in $u\geq 1+|l|$.\\
(v) For all $l\in \R$,  $u\mapsto u\tilde{\kappa}(u,l)$ is continuous, strictly concave, 
strictly  increasing on $[1+|l|,\infty)$ and $\lim_{u\to \infty} u \tilde{\kappa}(u,l)
=\infty$.\\
(vi) For all $l\in \R$,  $u\mapsto u\tilde{\kappa}(u,l)$ is analytic on $(1+|l|,\infty)$ 
and
\begin{align}
\label{l1} 
\lim_{v\to \infty} \partial_u(u\tilde \kappa(u,l))
&(v)=0,\\
\label{l2}
\lim_{v\to 1+l} \partial_u(u\tilde \kappa(u,l))
&(v)=\partial^+_u(u\tilde \kappa(u,l))(1+|l|)=\infty.
\end{align}
\el

\bl{addit3}
For all $\gep>0$ there exists $R_\gep>0$ such that 
\be{addit3*}
\partial_u(u\tilde \kappa(u,l))(v)\leq \gep, 
\quad \text{for}\,\,l\in [0,\infty),\,v\geq R_\gep \vee 2+l. 
\ee
\el

Recall the definition of $\{v(c), c\in (0,\infty)\}$ in (\ref{defucp1}-\ref{defucp3}).

\bl{regvc}
(i) For all $c\in (0,\infty)$, $v(c)\in \bar \cB$.\\
(ii) For $(k,l)\in \{A,B\}
\times (0,\infty)$, $c\mapsto v_{k,l}(c)$ is strictly decreasing and $c\mapsto v_{\cI}(c)$ 
is non-increasing.\\
(iii) If $(c_n)_{n\in \N}\in (0,\infty)^\N$ satisfies $\lim_{n\to \infty}c_n= c_\infty
\in (0,\infty)$, then $v(c_n)$ converges pointwise to $v(c_\infty)$.\\
(iv) $D(\bar \rho, v(c))<\infty$ for all $\bar \rho\in \cM_1(\R_+\cup \R_+\cup\{\cI\})$ 
satisfying $\int_{0}^{\infty} (1+l) (\bar \rho_A+\bar \rho_B)(dl)<\infty$ and all 
$c\in (0,\infty)$.
\el

Recall the definition of  $\{u(c), c\in (0,\infty)\}$ in \eqref{defucp}.

\bl{reguc}
(i) For all $c\in (0,\infty)$, $u(c)\in  \cB_{\overline{\cV}_M}$.\\
(ii) For all $\Theta\in \overline \cV_M$,  $c\mapsto u_{\Theta}(c)$ is 
non-increasing on $(0,\infty)$.\\
(iii) If $(c_n)_{n\in \N}\in (0,\infty)^\N$ satisfies $\lim_{n\to \infty}c_n= c_\infty
\in (0,\infty)$, then $u(c_n)$ converges pointwise to $u(c_\infty)$.\\
(iv) $D(\rho, u(c))<\infty$ for all $\rho\in \cM_1(\overline{\cV}_M)$ 
satisfying $\int_{\overline{\cV}_M} t_\Theta\,  \rho(d\Theta)<\infty$ and all 
$c\in (0,\infty)$.
\el


\subsection{Proofs of Lemmas \ref{l:lemconv2}--\ref{reguc}}

\subsubsection{Proof of Lemma \ref{l:lemconv2}}
(i) In the proof of Lemma~\ref{conunif1} we have shown that $\tilde{\kappa}$ can be 
extended to $\cH^0$ in such a way that $(u,l)\mapsto u\tilde{\kappa}(u,l)$ is 
continuous and concave on $\cH^0$. Lemma~\ref{convularge3} allows us to extend 
$\tilde{\kappa}$ to the boundary of $\cH$, in such a way that continuity and 
concavity of $(u,l)\mapsto u\tilde{\kappa}(u,l)$ hold on all of $\cH$. To obtain 
the strict concavity, we recall the formula in \eqref{kapexplform}, i.e., 
\be{kapexplform1}
u\tilde{\kappa}(u,l) = \left\{\begin{array}{ll}
u\kappa(u/|l|,1/|l|), &\l \neq 0,\\
u\hat{\kappa}(u), &l = 0,
\end{array}
\right.
\ee
where $(a,b)\mapsto a\kappa(a,b)$, $a\geq 1+b$, $b\geq 0$, and $\mu\mapsto \mu
\hat{\kappa}(\mu)$, $\mu \geq 1$, are given in \cite{dHW06}, Section 2.1, and 
are strictly concave. In the case $l\neq 0$, \eqref{kapexplform1} provides 
strict concavity of $(u,l)\mapsto u\tilde{\kappa}(u,l)$ on $\cH^+=\{(u,l)
\in \cH\colon l>0\}$ and on $\cH^-=\{(u,l)\in \cH\colon l<0\}$, while in the 
case $l=0$ it provides strict concavity on $\overline\cH=\{(u,0),u\geq 1\}$. 
We already know that $(u,l)\mapsto u\tilde{\kappa}(u,l)$ is concave on $\cH$, 
which, by the strict concavity on $\cH^+$, $\cH^-$ and $\overline\cH$, implies 
strict concavity of $(u,l)\mapsto u\tilde{\kappa}(u,l)$ on $\cH$. 

\medskip\noindent
(ii) This follows from the strict concavity of $l\mapsto \tilde{\kappa}(u,l)$ and from the
fact that $\tilde{\kappa}(u,l) = \tilde{\kappa}(u,-l)$  for 
$(u,l)\in \cH$.

\medskip\noindent
(iii-iv) These are direct consequences of Lemma~\ref{convularge}.

\medskip\noindent
(v)  By (i) we have that  $u\mapsto u\tilde{\kappa}(u,l)$ is continuous and strictly concave 
on $[1+|l|,\infty)$. Therefore, proving that $\lim_{u\to \infty} u \tilde{\kappa}
(u,l)=\infty$ is sufficient to obtain that  $u\mapsto u\tilde{\kappa}(u,l)$ is 
strictly increasing. It is proven in \cite{dHW06}, Lemma 2.1.2 (iii), that 
$\lim_{\mu\to \infty} u\hat{\kappa}(u)=\infty$, so that \eqref{kapexplform1} 
completes the proof for $l=0$. If $l\neq 0$, then we use \eqref{kapexplform1} 
again and the variational formula in the proof of \cite{dHW06}, Lemma 2.1.1, 
to check that $\lim_{a\to \infty} a \kappa(a,b)=\infty$ for all $b>0$.

\medskip\noindent
(vi)  To get the analyticity 
on $(1+|l|,\infty)$, we use  \eqref{kapexplform1} and the analyticity of  $(a,b)\mapsto a\kappa(a,b)$ 
and $\mu\mapsto \mu\hat{\kappa}(\mu)$ inside their domain of definition (see \cite{dHW06}, 
Section 2.1).

We note that for every $l\in \R$, 
\be{inegcoalt}
u\phi_\cI(u)\geq u \tilde{\kappa}(u,0)\geq u\tilde \kappa(u,l), 
\quad u\in [1+|l|,\infty),
\ee
where the first inequality is well known and the second inequality comes from 
Lemma~\ref{l:lemconv2}(ii). Since, by Lemma \ref{l:lemconv2}(v), $u\mapsto u
\tilde\kappa(u,l)$ is concave and increasing on $[1+|l|,\infty)$, \eqref{l3} and 
\eqref{inegco} imply \eqref{l1}.

It remains to prove \eqref{l2}. To that aim, we recall that an explicit formula is available for 
$\tilde{\kappa}(u,l)$, namely, 
\be{kapexplform*}
\tilde{\kappa}(u,l) = 
\kappa(u/|l|,1/|l|),\quad \text{for}\  l \neq 0,
\ee
where $\kappa(a,b)$, $a\geq 1+b$, $b\geq 0$ is given in \cite{dHW06}, Section 2.1 
(in the proof of Lemmas 2.1.1--2.1.2). The latter formula allows us to compute
$\partial_u (u \tilde\kappa(u,l)) (1+l+\gep,l)= G\big(1+\tfrac{1}{l}+\tfrac{\gep}{l},\tfrac{1}{l}\big)$
with
\be{gege}
G(a,b)=\tfrac12 \log 
\bigg[\tfrac{(a+1-b) (a-1-b)}{(a+1-b-2 \delta _{a,b}) (a-1-b-2 \gep_{a,b})}\bigg]
\ee
and with 
\begin{align}
\label{gege2}
\nonumber \delta_{a,b} 
&=\tfrac{b}{2(1+b)} \Big[ (a+1)-\big( (a-b)^2+(b^2-1)\big)^{1/2}\Big]\\ 
\gep_{a,b}&=\tfrac{b}{2(1-b)} \Big[-(a-1)+\big((a-b)^2+b^2-1\big)^{1/2}\Big],
\end{align}
so that the proof of  \eqref{l2} will be complete once we show that for all $b>0$ it holds that
 $\lim_{\gep\to 0^+} G(1+b+\gep,b)=\infty$. The latter is achieved by using first \eqref{gege2}
to check that $\delta_{1+b+\gep,b}=\frac{b}{1+b}+\big(\frac12-\frac{1}{1+b}\big)\gep+o(\gep)$ 
and $\gep_{1+b+\gep,b}=\frac{\gep}{2}+o(\gep)$ as $\gep\to 0^+$, and then by substituting 
these two expansions into \eqref{gege} at $(a,b)=(1+b+\gep,b)$, which implies the result after 
a straightforward computation.


\subsubsection{Proof of Lemma \ref{addit3}}

The proof is based on the following lemma.

\bl{leddit}
\be{eq:ledit}
\lim_{l\to \infty} \partial_u \big[u \tilde{\kappa}(u,l)\big] (2+l,l)=0.
\ee
\el

\bpr
We recall (\ref{kapexplform*}--\ref{gege2}), and we note that $\partial_u (u \tilde\kappa(u,l))
(2+l,l)= G\big(1+\tfrac{2}{l},\tfrac{1}{l}\big)$. Thus, the proof of Lemma \ref{leddit} will be 
complete once we show that $\lim_{b\to 0^+} G(1+2b,b)=0$. The latter is achieved by using 
\eqref{gege} and \eqref{gege2} to compute
\be{gege1}
G(1+2b,b)=\tfrac12 \log 
\bigg[\tfrac{(2+b) b}{\big[2+b\big(1-\frac{2}{1+b}+o(b)\big)\big] (b+o(b))}\bigg]
\ee
which immediately implies the result. 
\epr

We resume the proof of Lemma \ref{addit3}. Once Lemma \ref{leddit} is proven, we use the concavity of 
$u\mapsto u \tilde\kappa(u,l)$ for $l\in \R$ to obtain that for $\gep>0$ there exists a $l_\gep>0$ such that 
$\partial_u [u \tilde{\kappa}(u,l)] (u,l)\leq \gep$ for all $l\colon\;|l|\geq l_\gep$ and $u\geq 2+l$. Thus, it 
remains to show that there exists a $R_\gep>0$ such that $\partial_u [u \tilde{\kappa}(u,l)] (u,l)\leq \gep$
for $l\in [0,l_\gep]$ and $u\geq R_\gep$. By contradiction, if we assume that the latter does not hold, then 
there exists $\gep>0$ and two sequences $(l_n)_{n\in \N}\in [0,l_\gep]^\N$ and $(u_n)_{n\in \N}$ such that 
$u_n\geq 1+l_n$ for $n\in \N$ and $\lim_{n\to \infty} u_n=\infty$ and such that $\partial_u [u \tilde{\kappa}(u,l)] 
(u_n,l_n)\geq \gep$ for $n\in \N$. As a consequence, we can write
\be{tret}
u_n \tilde{\kappa}(u_n,l_n)-(1+l_n) \tilde{\kappa}(1+l_n,l_n)\geq \gep (u_n-1-l_n),
\ee
and, with the help of Lemma \ref{l:lemconv2}(ii), we obtain 
\be{tret*}
u_n \tilde{\kappa}(u_n,0)\geq u_n \tilde{\kappa}(u_n,l_n)\geq \gep (u_n-1-l_\gep),\quad \text{for}\ n\in \N,
\ee
which clearly contradicts Lemma \ref{l:lemconv2}(iii) because $\lim_{n\to \infty} u_n=\infty$.


\subsubsection{Proof of Lemma \ref{regvc}}

(i) We must prove that $l\mapsto v_{A,l}(c)$ and $l\mapsto v_{B,l}(c)$ are continuous 
on $[0,\infty)$. We give the proof for $v_A$, the proof for $v_B$ being similar. Let 
$(l_n)_{n\in \N}$ be a sequence in $[0,\infty)$ such that $\lim_{n\to \infty} l_n
=l_\infty\in [0,\infty)$. We want to prove that $\lim_{n\to \infty} v_{A,l_n}(c)
= v_{A,l_\infty}(c)$.  For simplicity, we set $v_n=v_{A,l_n}(c)$ for $n\in \N$ and 
$v_\infty=v_{A,l_\infty}(c)$. We also set $g_n(u)=u \tilde{\kappa}(u,l_n)$ for $n\in \N$ 
and $u \geq 1+l_n$ and $g_\infty(u)=u \tilde{\kappa}(u,l_\infty)$ for $u \geq 1+l_\infty$. 
By Lemmas~\ref{l:lemconv2}(i) and (v), we know that $g_n$ converges pointwise to $g_\infty$ 
as $n\to \infty$, and that $g_n$ and $g_\infty$ are strictly concave. Consequently, 
$\partial_u(g_n)$ converges pointwise to $\partial_u(g_\infty)$. We argue by contradiction.
Suppose that $v_n$ does not converge to $v_\infty$. Then there exists an $\eta>0$ such 
that $v_n\geq v_\infty+\eta$ along a subsequence or $v_n\leq v_\infty-\eta$ along a 
subsequence. Suppose for simplicity that $v_n\leq v_\infty-\eta$ for $n\in \N$. Then 
the strict concavity of $g_n$ implies that $\partial_u(g_n)(v_\infty-\eta)\leq \partial_u
(g_n)(v_n)=c$, and therefore, letting $n\to \infty$ and using the strict concavity of 
$g_\infty$, we obtain $\partial_u(g_\infty)(v_\infty)<\partial_u(g_\infty)(v_\infty-\eta)
\leq c$. This provides the contradiction, because  $\partial_u(g_\infty)(v_\infty)=c$ by 
definition. The proof is similar when we assume that $v_n\geq v_\infty+\eta$ for $n\in \N$.

\medskip\noindent
(ii) For $(k,l)\in \{A,B\}
\times [0,\infty)$, this is a straightforward consequence of the definition of $v(c)$ in 
(\ref{defucp1}-\ref{defucp2}), of the strict concavity of $u\mapsto u\tilde\kappa(u,l)$ 
and of the continuity of $u\mapsto \partial_ u (u \tilde\kappa(u,l))$ for every $l\in [0,\infty)$ 
(see Lemma \ref{l:lemconv2}(v-vi)). For $c\mapsto v_\cI(c)$ we do not have strict monotonicity 
because $u\mapsto \partial_u (u\phi_\cI(u))$ is not proven to be continuous.

\medskip\noindent
(iii) Similarly to what we did in (i), we consider $(c_n)_{n\in \N}$ a sequence in $(0,\infty)$ 
such that $\lim_{n\to \infty} c_n=c_\infty\in (0,\infty)$, and we want to show that 
$\lim_{n\to \infty} v_{k,l}(c_n)=v_{k,l}(c_\infty)$ for $k\in \{A,B\}$ and $l\in [0,\infty)$ 
and $\lim_{n\to \infty} v_{\cI}(c_n)=v_{\cI}(c_\infty)$. Again we argue by contradiction.
Suppose, for instance, that $v_{\cI}(c_n)$ does not converge to $v_{\cI}(c_\infty)$. Then 
there exists an $\eta>0$ such that $v_{\cI}(c_n)\leq v_{\cI}(c_\infty)-\eta$ or 
$v_{\cI}(c_n)\geq v_{\cI}(c_\infty)+\eta$ along a subsequence. Suppose for simplicity that 
$v_{\cI}(c_n)\geq v_{\cI}(c_\infty)+\eta$. Then $\partial_u^-(u\phi_{\cI}(u))(v_{\cI}
(c_\infty)+\eta)\geq \partial_u^-(u\phi_{\cI}(u))(v_{\cI}(c_n))\geq c_n$ for $n\in \N$. 
Let $n\to \infty$ to obtain $\partial_u^+(u\phi_{\cI}(u))(v_{\cI}(c_\infty))>\partial_u^-
(u\phi_{\cI}(u))(v_{\cI}(c_\infty)+\eta) \geq c_\infty$, which contradicts the definition 
of $v_{\cI}(c_\infty)$ in (\ref{defucp1}-\ref{defucp3}). The proof is similar when we 
assume that $v_{\cI}(c_n)\leq v_{\cI}(c_\infty)-\eta$ for $n\in \N$.

\medskip\noindent
(iv) This is a consequence of Lemma~\ref{addit3}, which implies that for all $c\in
(0,\infty)$ there exists a $l_c\in [0,\infty)$ such that $v_{A,l}(c)\leq 2+l$ for 
all $l\geq l_c$. Moreover, (\ref{defucp1}-\ref{defucp2}) and the fact that $(\alpha,\beta)
\in \CONE$ entail that $v_{B,l}(c)\leq v_{A,l}(c)$ for $l\in [0,\infty)$, and therefore 
$\int_{0}^{\infty} (1+l) (\bar \rho_A+\bar \rho_B)(dl)<\infty$ combined with the finitness 
of $v_\cI(c)$ imply  $\bar D(\bar \rho,v(c))<\infty$.


\subsubsection{Proof of Lemma \ref{reguc}}

(i) The proof is similar to that of  Lemma \ref{regvc}(i), except for the fact that when we 
consider $\Theta_n\to \Theta_\infty$ as $n\to \infty$ in $\overline \cV_M$, we have (by 
Lemma \ref{concavt}) the pointwise convergence of  $g_n(u)=u\psi(\Theta_n,u)$ to 
$g_\infty(u)=u\psi(\Theta_\infty,u)$, but we do not have the pointwise convergence of 
$\partial g_n(u)$ to  $\partial g_\infty(u)$ since $ g_\infty$ is not a priori differentiable. However, 
the strict concavity and the pointwise convergence of $g_n$ towards $g_\infty$ gives us 
\be{}
\partial^- g_\infty(u)\geq \limsup_{n\to \infty} \partial^- g_n(u)
\geq \liminf_{n\to \infty} \partial^+ g_n(u)\geq   \partial^+ g_\infty(u),
\ee
with which we can easily mimick the proof in Lemma  \ref{regvc}(i)

\medskip\noindent
(ii) The proof is similar to that of  Lemma \ref{regvc}(ii), except for the fact that the monotonicity of
$c\mapsto u_\Theta(c)$ is not proven to be strict because $u\mapsto \partial (u \psi(\Theta,u))$ is 
not proven to be continuous.

\medskip\noindent
(iii) We mimick the proof of Lemma \ref{regvc}(iii).  Let $(c_n)_{n\in \N}$ be a sequence in 
$(0,\infty)$ such that $\lim_{n\to \infty} c_n=c_\infty\in (0,\infty)$, and assume that there exists 
an $\eta>0$ such that $u_\Theta(c_n)\geq u_\Theta(c_\infty)+\eta$ along a subsequence. 
Then $\partial_u^-(u\psi(\Theta,u))(u_\Theta(c_\infty)+\eta)\geq \partial_u^-(u\psi(\Theta,u))
(u_\Theta(c_n)) \geq c_n$ for $n\in \N$. Let $n\to \infty$ to obtain $\partial_u^+(u\psi(\Theta,u))
(u_\Theta(c_\infty))> \partial_u^-(u\psi(\Theta,u))(u_\Theta(c_\infty)+\eta)  \geq c_\infty$, which 
contradicts the definition of $u_{\Theta}(c_\infty)$ in \eqref{defucp}.

\medskip\noindent
(iv) The proof is similar to that of Lemma \ref{regvc}(iv). The role of Lemma \ref{addit3} is taken
over by Lemma \ref{fep2}


\section{Properties of free energies}
\label{B}

 
\subsection{Free energy along a single linear interface}
\label{B.1}

Also the free energy $\mu\mapsto\phi^\cI(\mu;\alpha,\beta)$ defined in 
Proposition~\ref{l:feinflim} can be extended from $\mathbb{Q}\cap [1,\infty)$ to
$[1,\infty)$, in such a way that $\mu\mapsto \mu\phi^\cI(\mu;\alpha,\beta)$ is 
concave and continous on $[1,\infty)$.  By concatenating trajectories, we can indeed check 
that $\mu\mapsto \mu\phi^\cI(\mu;\alpha,\beta)$ is concave on $\mathbb{Q}\cap [1,\infty)$. 
Therefore it is Lipschitz on every compact subset of $(1,\infty)$ and can be extended 
to a concave and continuous function on $(1,\infty)$. The continuity at $\mu=1$ comes 
from the fact that $\phi^\cI(1;\alpha,\beta)=0$ and $\lim_{\mu \downarrow 1} 
\phi^\cI(\mu)=0$, which is obtained by using Lemma~\ref{lele} below.

\bl{l:lemconv}
For all $(\alpha,\beta)\in\CONE$:\\
(i) $\mu \mapsto \mu \phi^\cI(\mu;\alpha,\beta)$ is strictly increasing on $[1,\infty)$ 
and $\lim_{\mu\to \infty} \mu\phi^\cI(\mu;\alpha,\beta)=\infty$.\\
(ii) $\lim_{\mu\to \infty}\phi^\cI(\mu;\alpha,\beta)=0$.\\
(iii)
\begin{align}
\label{l3}
\lim_{v\to \infty} \partial^-_{u} (u \phi_{\cI}(u;\, \alpha,\beta))(v)
&=0,\\
\label{l4}
\lim_{v\to 1} \partial^+_{u} (u \phi_{\cI}(u;\, \alpha,\beta))(v)
&=\partial^+_{u} (u \phi_{\cI}(u;\, \alpha,\beta))(1)=\infty. 
\end{align}
\el

\begin{proof}
(i) Clearly, $\phi^\cI(\mu;\alpha,\beta)\geq \widetilde{\kappa}(\mu,0)$ for $\mu\geq 1$.
Therefore Lemma \ref{l:lemconv2}(iv) implies that $\lim_{\mu\to \infty} \mu\phi^\cI
(\mu;\alpha,\beta)=\infty$. Thus, the concavity of $\mu\mapsto \mu \phi^\cI(\mu;\alpha,
\beta)$ is sufficient to obtain that it is strictly increasing on $[1,\infty)$.\\
(ii) See \cite{dHP07b}, Lemma 2.4.1(i).\\
(iii)
To prove \eqref{l3}, we pick $\chi\in\{A,B\}^\Z$ such that $\chi(0)=A$ and $\chi(-1)=B$. We recall 
\eqref{partcv} and consider $\Theta=(\chi,0,0,0,2)\in \bar \cV_{\nAB,A,2,M}$ such that $l_A(\Theta)
=l_B(\Theta)=0$.  By Proposition \ref{energ}, we have
\be{2in}
u\psi(\Theta_2,u)\geq  u\phi_\cI(u), \quad u\in [1,\infty),
\ee
and \eqref{2in}, together with Lemma~\ref{fep1} and the concavity and monotonicity 
of $u\mapsto u\phi_\cI(u)$, imply \eqref{l3}. 

It remains to prove \eqref{l4}. For all $(\alpha,\beta) \in \CONE$ we know that 
$u\mapsto u\phi_\cI(u;\alpha,\beta)$ is continuous and strictly concave on $[1,\infty)$. 
Therefore we necessarily have
\be{egc}
\lim_{v\to 1^+} \partial^+_u(u\phi_{\cI}(u))(v)= \partial^+_u(u\phi_{\cI}(u))(1).
\ee
Moreover, since $(u\phi_\cI(u))(1)=(u\tilde\kappa(u,0))(1)=0$ and since $\phi_{\cI}(u)
\geq \tilde\kappa(u,0)$ for $u\geq 1$, we have $\partial^+_u(u\phi_{\cI}(u))(1)\geq
\partial^+_u(u \tilde\kappa(u,0))(1)$ and \eqref{l2} gives $\partial^+_u(u \tilde\kappa
(u,0))(1)=\infty$, which completes the proof of \eqref{l4}.
\end{proof}

\noindent
Recall Assumption \ref{assu}, in which we assumed that $\mu \mapsto \mu \phi^\cI(\mu;
\alpha,\beta)$ is strictly concave on $[1,\infty)$. The next lemma states that the 
convergence of the average quenched free energy $\phi^\cI_L$ to $\phi^\cI$ as 
$L\to\infty$ is uniform on $\mathbb{Q} \cap [1,\infty)$.

\bl{l:feinflim1} 
For every $(\alpha,\beta)\in\CONE$ and $\gep>0$ there exists an $L_\gep\in \N$ 
such that 
\be{fesainfalt}
|\phi_L(\mu)-\phi(\mu)|\leq \gep \qquad \forall\,\mu\in 1+\tfrac{2\N}{L},\,
L\geq L_\gep.
\ee
\el

\begin{proof}
Similarly to what we did for Lemma~\ref{conunif1}, the proof can be done by 
treating separately the cases $\mu$ large, moderate and small. We leave the 
details to the reader.
\end{proof}


\subsection{Free energy in a single column}
\label{B.2}

We can extend $(\Theta,u)\mapsto \psi(\Theta,u)$ from $\cV_M^{*}$ to $\overline\cV_M^{*}$ 
by using the variational formula in \eqref{Bloc of type I}
and by recalling that $\widetilde{\kappa}$ and $\phi^{\cI}$ have been extended to $\cH$ 
and $[1,\infty)$ in Appendices \ref{A.2} and \ref{B.1}.

Pick $M\in \N$ and recall \eqref{set2}. Define a distance $d_M$ on $\overline\cV_M$ 
as follows. Pick $\Theta_1,\Theta_2 \in \overline \cV_M$, abbreviate 
\be{defth1*}
\Theta_1=(\chi_1,\Delta\Pi_1, b_{0,1},b_{1,1},x_1),
\qquad \Theta_2=(\chi_2,\Delta\Pi_2, b_{0,2},b_{1,2},x_2),
\ee 
and define
\be{dist}
d_M(\Theta_1,\Theta_2)= \sum_{j\in \Z} \frac{1_{\{\chi_1(j)\neq \chi_2(j)\}}}{2^{|j|}} 
+|\Delta \Pi_1-\Delta \Pi_2|+|b_{0,1}-b_{0,2}|+|b_{1,1}-b_{1,2}|+|x_1-x_2|
\ee
so that $\widetilde{d}_M((\Theta_1,u_1),(\Theta_2,u_2))=\max \{|u_1-u_2|,d_M
(\Theta_1,\Theta_2)\}$ is a distance on $\overline{\cV}^{\,*,m}_M$ for which 
$\overline{\cV}^{\,*,m}_M$ is compact.

Lemmas   \ref{concavt} and \ref{concav} below are proven in Section~\ref{proofiny}.
\bl{concavt}
For every $(M,m)\in \EIGH$ and $(\alpha,\beta)\in \CONE$,
\begin{align}
\label{fuun}
(u,\Theta) \mapsto u\,\psi(\Theta,u;\alpha,\beta)
\end{align}
is uniformly continuous on $\overline{\cV}^{\,*,m}_M$ endowed with $\widetilde{d}_M$. 
\el

\bl{concav}
For every $\Theta\in \overline\cV_M$, the function $u \mapsto u\psi(\Theta,u)$ 
is continuous and strictly concave on $[t_\Theta,\infty)$.
\el

Below we list several results that were used in Section~\ref{varfo2}. The proofs of these result
are given in Section~\ref{proofiny}. Proposition~\ref{base} below says that the free energy 
per column associated with the Hamiltonian given by $(\beta-\alpha)/2$ times the time spent 
by the copolymer in the $B$-solvent is a good aproximation of $\psi(\Theta,u)$ when $u\to\infty$ 
uniformly in $\Theta\in \overline\cV_M$. This proof of this proposition will be given in 
Section~\ref{proofbase}.

\begin{proposition}
\label{base}
For all $(\alpha,\beta)\in \CONE$ and all $\gep>0$ there exists $R_\gep>0$ and 
$L_\gep\in \N$ such that  
\begin{align}
\label{defuc2}
\Big|\psi(\Theta,u)&-\tfrac{1}{uL} \log \sum_{\pi\in \cW_{\Theta,u,L}} 
e^{T(\pi) \tfrac{\beta-\alpha}{2}}\Big| \leq\gep,\ \  
\Theta\in \overline\cV_M, \ \ u\geq t_\Theta\vee R_\gep, \ \ L\geq L_\gep,
\end{align}
where $T(\pi)=\sum_{i=1}^{uL} 1\{\chi^{L}_{(\pi_{i-1},\pi_i)}=B\}$ is the time spent 
by $\pi$ in solvent $B$.
\end{proposition}

Lemmas~\ref{boundunif}--\ref{fep2} below are consequences of Lemma~\ref{concav} and 
Proposition~\ref{base}. The proofs of Lemmas~\ref{boundunif} and \ref{fep2} will
be given in Sections~\ref{proofboundunif} and \ref{prooffep2}. Lemma~\ref{boundunif} 
shows that $\psi(\Theta,u)$ is bounded from above uniformly in $\Theta\in\overline\cV_M$ 
as $u\to \infty$. Lemma~\ref{fep1} identifies the limit of $\partial^{-}_u(u\,\psi(\Theta,u))$ 
as $u\to \infty$ for $\Theta\in \overline{\cV}_M$. Lemma~\ref{fep2} is the counterpart 
of Lemma~\ref{boundunif} for $\partial^{-}_u(u\psi(\Theta,u))$ instead of $\psi(\Theta,u)$. 

\begin{lemma}
\label{boundunif}
For all $(\alpha,\beta)\in \CONE$ and $\gep>0$ there exists a $C_\gep>0$ such that 
\be{defuc}
\psi(\Theta,u)\leq \left\{
\begin{array}{ll}
\vspace{.1cm}
\gep
& \mbox{if} \ \ \Theta\in \overline \cV_M\setminus \overline \cV_{\nAB,B,1,M},
\quad  u\geq t_\Theta\vee C_\gep, \\
\vspace{.1cm}
\frac{\beta-\alpha}{2}+\gep
&\mbox{if} \ \ \Theta\in \overline \cV_{\nAB,B,1,M},
\quad u\geq t_\Theta\vee C_\gep,
\end{array}
\right.
\ee
\end{lemma}

\begin{lemma}
\label{fep1}
For all $(\alpha,\beta)\in \CONE$, 
\be{defucalt}
\lim_{v\to \infty} \partial^{+}_u (u\psi(\Theta,u))(v)= \left\{
\begin{array}{ll}
\vspace{.1cm}
0
& \mbox{if  }\  \ \Theta\in \overline \cV_M\setminus \overline \cV_{\nAB,B,1,M}, \\
\vspace{.1cm}
\frac{\beta-\alpha}{2}
&\mbox{if } \  \ \Theta\in \overline \cV_{\nAB,B,1,M}.
\end{array}
\right.
\ee
\end{lemma}

\begin{lemma}
\label{fep2}
For all $(\alpha,\beta)\in \CONE$ and $\gep>0$ there exists a $V_\gep>0$ such that  
\be{defucaltalt}
\partial^{-}_u (u\psi(\Theta,u))(v)\leq \left\{
\begin{array}{ll}
\vspace{.1cm}
\gep
& \mbox{if  }\  \ \Theta\in \overline \cV_M\setminus \overline \cV_{\nAB,B,1,M}, 
\quad v\geq 2 t_\Theta\vee V_\gep,\\
\vspace{.1cm}
\frac{\beta-\alpha}{2}+\gep
&\mbox{if } \  \ \Theta\in \overline \cV_{\nAB,B,1,M}, 
\quad  v\geq 2 t_\Theta\vee V_\gep.
\end{array}
\right.
\ee
\end{lemma}

\subsection{Proof of Lemmas \ref{concavt}--\ref{fep2} }\label{proofiny}
\subsubsection{Proof of Lemma \ref{concavt}}
Pick $(M,m)\in \EIGH$. By the compactness of $\overline{\cV}^{\,*,m}_M$, it suffices to show 
that $(u,\Theta) \mapsto u\,\psi(\Theta,u)$ is continuous on $\overline{\cV}^{\,*,m}_M$. 
Let $(\Theta_n,u_n)=(\chi_n,\Delta\Pi_n,b_{0,n},b_{1,n},u_n)$ be the general term of 
an infinite sequence that tends to $(\Theta,u)=(\chi,\Delta\Pi,b_{0},b_{1},u)$ 
in $(\overline{\cV}^{\,*,m}_M,\widetilde{d}_M)$. We want to show that $\lim_{n\to\infty}
u_n\psi(\Theta_n,u_n)=u \psi(\Theta,u)$. By the definition of $\widetilde{d}_M$, we 
have $\chi_n=\chi$ and $\Delta\Pi_n=\Delta\Pi$ for $n$ large enough. We assume that 
$\Theta\in \cV_{\AB}$, so that $\Theta_n\in \cV_{\AB}$ for $n$ large enough as well. 
The case $\Theta\in \cV_{\text{nint}}$ can be treated similarly.

Set
\be{defrm}
\cR_m=\{(a,h,l)\in [0,m]\times [0,1]\times \R\colon h+|l|\leq a\}
\ee
and note that $\cR_m$ is a compact set. Let $g\colon\,\cR_m\mapsto [0,\infty)$ be 
defined as $g(a,h,l)=a\,\widetilde{\kappa}(\tfrac{a}{h},\tfrac{l}{h})$ if $h>0$ 
and $g(a,h,l)=0$ if $h=0$. The continuity of $\widetilde{\kappa}$, stated in 
Lemma~\ref{l:lemconv2}(i), ensures that $g$ is continuous on $\{(a,h,l)\in
\cR_m\colon h>0\}$. The continuity at all $(a,0,l)\in \cR_m$ is obtained by 
recalling that $\lim_{u\to\infty}\tilde{\kappa}(u,l)=0$ uniformly in $l\in
[-u+1,u-1]$ (see Lemma~\ref{l:lemconv2}(ii-iii)) and that $\widetilde{\kappa}$ 
is bounded on $\cH$.

In the same spirit, we may set $\cR'_m=\{(u,h)\in [0,m]\times[0,1]\colon\,
h\leq u\}$ and define $g'\colon\,\cR'_m\mapsto [0,\infty)$ as $g'(u,h)
= u\,\phi^{\cI}(\tfrac{u}{h})$ for $h>0$ and $g'(u,h)=0$ for $h=0$. With the 
help of Lemma~\ref{l:lemconv} we obtain the continuity of $g'$ on $\cR'_m$ 
by mimicking the proof of the continuity of $g$ on $\cR_m$.

Note that the variational formula in \eqref{Bloc of type I} can be rewriten as 
\begin{align}
\label{Bloc of type IIa} 
u \,\psi(\Theta,u)
&=\sup_{(h),(a) \in \cL(l_A,\, l_B;\,u)} Q((h),(a),l_A,l_B),
\end{align}
with
\be{defQ}
Q((h),(a),l_A,l_B)=g(a_A,h_A,l_A)+g(a_B,h_B,l_B)+a_B \,
\tfrac{\beta-\alpha}{2}+g^{'}(a^\cI,h^\cI),
\ee
and with $l_A$ and $l_B$ defined in \eqref{bl}. Note that $\cL(l_A,\,l_B;\,u)$ is 
compact, and that $(h),(a)\mapsto Q((h),(a),l_A,l_B)$ is continuous on 
$\cL(l_A,\, l_B;\,u)$ because $g$ and $g'$ are continuous on $\cR_m$ and 
$\cR^{'}_m$, respectively. Hence, the supremum in \eqref{Bloc of type IIa} is 
attained.

Pick $\gep>0$, and note that $g$ and $g'$ are uniformly continuous on $\cR_m$ 
and $\cR'_m$, which are compact sets. Hence there exists an $\eta_\gep>0$ such 
that $|g(a,h,l)-g(a',h',l')|\leq \gep$ and $|g'(u,b)-g'(u',b')| \leq \gep$ when 
$(a,h,l),(a',h',l')\in  \cR_m$ and $(u,b),(u',b')\in \cR'_m$ are such that 
$|a-a'|,|h-h'|,|l-l'|,|u-u'|$ and $|b-b'|$ are bounded from above by $\eta_\gep$.

Since $\lim_{n\to \infty}(\Theta_n,u_n)=(\Theta,u)$ we also have that 
$\lim_{n\to\infty} b_{0,n}=b_0$, $\lim_{n\to \infty} b_{1,n}=b_1$ and 
$\lim_{n\to \infty} u_{n}=u$. Thus, $\lim_{n\to\infty} l_{A,n}= l_A$ and 
$\lim_{n\to \infty} l_{B,n}= l_B$, and therefore $|l_{A,n}- l_A|\leq \eta_\gep$, 
$|l_{B,n}- l_B|\leq \eta_\gep$ and $|u_n-u|\leq \eta_\gep$ for $n\geq n_\gep$ 
large enough.

For $n\in \N$, let $(h_n),(a_n)\in\cL(l_{A,n},\, l_{B,n};\,u_n)$ be a maximizer 
of \eqref{Bloc of type IIa} at $(\Theta_n,u_n)$, and note that, for $n\geq n_\gep$, 
we can choose $(\widetilde{h}_n),(\widetilde{a}_n)\in\cL(l_A,\,l_B;\,u)$ such 
that $|\widetilde{a}_{A,n}-a_{A,n}|$, $|\widetilde{a}_{B,n}-a_{B,n}|$, 
$|\widetilde{a}_{n}^\cI-a_{n}^\cI|$, $|\widetilde{h}_{A,n}-h_{A,n}|$, 
$|\widetilde{h}_{B,n}-h_{B,n}|$ and $|\widetilde{h}_{n}^\cI-h_{n}^\cI|$ are 
bounded above by $\eta_\gep$. Consequently,
\be{boundonpsi}
u_n \psi(\Theta_n,u_n)- u\psi(\Theta,u)
\leq Q((h_n),(a_n), l_{A,n},l_{B,n})-Q((\widetilde{h}_n),
(\widetilde{a}_n), l_{A},l_{B})\leq 3\gep.
\ee
We bound $u\psi(\Theta,u)-u_n \psi(\Theta_n,u_n)$ from above in a similar manner, 
and this suffices to obtain the claim.


\subsubsection{Proof of lemma \ref{concav}}
The continuity is a straightforward consequence of Lemma~\ref{concavt}: simply 
fix $\Theta$ and let $m\to\infty$. To prove the strict concavity, we note that 
the cases $\Theta\in \cV_{\AB,M}$ and $\Theta\in \cV_{\nAB,M}$ can be treated 
similarly. We will therefore focus on $\Theta\in \cV_{\AB,M}$.

For $l\in \R$, let
\be{defrmalt}
\cN_l=\{(a,h)\in [0,\infty) \times [0,1]\colon a\geq h+|l|\},
\quad \cN_l^+=\{(a,h)\in \cN_l\colon h>0\},
\ee
and let $g_l\colon\,\cN_l\mapsto [0,\infty)$ be defined as $g_l(a,h)=a\,
\tilde{\kappa}(\tfrac{a}{h},\tfrac{l}{h})$ for $h>0$ and $g_l(a,h)=0$ for 
$h=0$.  For $l\neq 0$, the strict concavity of $(u,l)\mapsto u\tilde{\kappa}(u,l)$ on 
$\cH$, stated in Lemma \ref{l:lemconv2}(i), immediately yields that $g_l$ 
is strictly concave on $\cN_l^+$ and concave on $\cN_l$. Consequently, for 
all $(a_1,h_1)\in \cN_l^+$ and $(a_2,h_2)\in \cN_l\setminus \cN_l^+ $, $g_l$ is 
strictly concave on the segment $[(u_1,h_1),(u_2,h_2)]$.

Define also $\widetilde g\colon\, \cN_0\mapsto [0,\infty)$ as $\widetilde g(a,h)
= a\,\phi^{\cI}(\tfrac{a}{h})$  for $h>0$ and $\widetilde g(a,h)=0$ for $h=0$. The 
strict concavity of $u\mapsto u\phi^{\cI}(u)$ and of $u\mapsto u\tilde{\kappa}(u,0)$ 
on $[1,\infty)$, stated in Assumption~\ref{assu} and in Lemma \ref{l:lemconv2}, immediately 
yield that $\widetilde g$ and $g_0$ are concave on $\cN_0$ and that, for $h>0$, 
$a\mapsto \widetilde{g}(a,h)$ and $a\mapsto g_0(a,h)$ are strictly concave on $[h,\infty)$

Similarly to what we did in \eqref{Bloc of type IIa}, we can rewrite the variational 
formula in \eqref{Bloc of type I} as 
\begin{align}
\label{Bloc of type IIaa} 
u \,\psi(\Theta,u)
&=\sup_{(h),(a) \in \cL(l_A,\, l_B;\,u)} \widetilde{Q}((h),(a))
\end{align}
with
\be{defQtil}
\widetilde{Q}((h),(a))=g_{l_A}(a_A,h_A)+g_{l_B}(a_B,h_B)+a_B \,
\tfrac{\beta-\alpha}{2}+\widetilde{g}(u-a_A-a_B,1-h_A-h_B),
\ee
and the supremum in \eqref{Bloc of type IIaa} is attained. In what follows we will restrict 
the proof to the case $l_A,l_B>0$ for the following reason. If $l_k=0$ for $k\in \{A,B\}$, 
then the inequality $g_0\leq \widetilde{g}$ and the concavity of $\widetilde{g}$ ensure 
that there exists a $(h),(a)\in \cL(l_A,\, l_B;\,u)$ maximizing \eqref{Bloc of type IIaa} and 
satisfying $h_k=a_k=0$, which allows to copy the proof below after removing the $k$-th 
coordinate in $(h),(a)$.

Next, we show that if $(h),(a)\in \cL(l_A,\,l_B;\,u)$ realizes the maximum in 
\eqref{Bloc of type IIaa}, then $(h),(a)\notin \widetilde{\cL}(l_A,\,l_B;\,u)$ with
\be{restor}
\widetilde{\cL}(l_A,\, l_B;\,u) 
= \widetilde{\cL}_A(l_A,\, l_B;\,u)\cup\widetilde{\cL}_B(l_A,\, l_B;\,u)
\cup\widetilde{\cL}^{\, \cI}(l_A,\, l_B;\,u)
\ee
and
\begin{align}
\label{restore}
\nonumber \widetilde{\cL}_A(l_A,\, l_B;\,u)
&=\{(h),(a)\in \cL(l_A,\, l_B;\,u)\colon\,h_A=0 \ \ \text{and}\ \ a_A>l_A\},\\
\nonumber \widetilde{\cL}_B(l_A,\, l_B;\,u)
&=\{(h),(a)\in \cL(l_A,\, l_B;\,u)\colon\, h_B=0 \ \ \text{and}\ \ a_B>l_B\},\\
\widetilde{\cL}^{\,\cI}(l_A,\, l_B;\,u)
&=\{(h),(a)\in \cL(l_A,\, l_B;\,u)\colon\, h_I=0 \ \ \text{and}\ \ a_I>0\}.
\end{align}
Assume that $(h),(a)\in \widetilde{\cL}(l_A,\, l_B;\,u)$, and that $h_A>0$ or $h^\cI>0$. 
For instance, $(h),(a)\in \widetilde{\cL}^\cI(l_A,\, l_B;\,u)$ and $h_A>0$. Then, by 
Lemma~\ref{l:lemconv2}(iv), $\widetilde{Q}$ strictly increases when $a_A$ is replaced 
by $a_A+a^\cI$ and $a^\cI$ by $0$. This contradicts the fact that $(h),(a)$ is a maximizer. 
Next, if $(h),(a)\in \widetilde{\cL}(l_A,\,l_B;\,u)$ and $h_A=h^\cI=0$, then $h_B=1$, 
and the first case is $(h),(a)\in \widetilde{\cL}_A(l_A,\,l_B;\,u)$, while the second 
case is $(h),(a)\in \widetilde{\cL}^\cI(l_A,\, l_B;\,u)$. In the second case, as before, 
we replace $a_A$ by $a_A+a^\cI$ and $a^\cI$ by $0$, which does not change $\widetilde{Q}$ 
but yields that $a_A>l_A$ and therefore brings us back to the first case. In this first 
case, we are left with an expression of the form
\be{fca}
Q((h),(a))=g_{l_B}(a_B,1) + a_B\, \tfrac{\beta-\alpha}{2} 
\ee
with $h_A=h^\cI=0$ and $a_A>l_A$. Thus, if we can show that there exists an $x\in (0,1)$ 
such that 
\be{fca1}
g_{l_A}(a_A,x)+g_{l_B}(a_B,1-x)>g_{l_B}(a_B,1), 
\ee
then we can claim that $(h),(a)$ is not a maximizer of \eqref{Bloc of type IIaa} and 
the proof for $(h),(a)\notin \widetilde{\cL}(l_A,\,l_B;\,u)$ will be complete. 

To that end, we recall \eqref{kapexplform}, which allows us to rewrite the left-hand 
side in \eqref{fca1} as
\be{fca2}
g_{l_A}(a_A,x)+g_{l_B}(a_B,1-x)=a_A \,\kappa\big(\tfrac{a_A}{l_A},\tfrac{x}{l_A}\big)+
a_B \,\kappa\big(\tfrac{a_B}{l_B},\tfrac{1-x}{l_B}\big)+ a_B\, \tfrac{\beta-\alpha}{2}.
\ee
We recall \cite{dHW06}, Lemma 2.1.1, which claims that $\kappa$ is defined on 
$\DOM=\{(a,b)\colon a\geq 1+b, b\geq 0\}$, is analytic on the interior of $\DOM$  
and is continuous on $\DOM$. Moreover, in the proof of this lemma, an expression 
for $\partial_b\, \kappa(a,b)$ is provided, which is valid on the interior of $\DOM$. 
From this expression we can easily check that if $a>1$, then $\lim_{b\to 0} \partial_b\,
\kappa(a,b)=\infty$. Therefore, by the continuity of $\kappa$ on $(a_A/l_A,0)$ with 
$a_A/l_A>1$ we can assert that the derivative with respect to $x$ of the left-hand side 
in \eqref{fca2} at $x=0$ is infinite, and therefore there exists an $x>0$ such that 
\eqref{fca1} is satisfied.

It remains to prove the strict concavity of $u\mapsto u \psi(\Theta,u)$ with $\Theta\in 
\cV_{\AB,M}$. Pick $u_1>u_2\geq t_\Theta$, and let $(h_1),(a_1) \in \cL(l_A,\, l_B;\,u_1)$ 
and $(h_2),(a_2) \in \cL(l_A,\, l_B;\,u_2)$ be maximizers of \eqref{Bloc of type IIaa} 
for $u_1$ and $u_2$, respectively. We can write
\begin{align}
\label{aa}
\nonumber (a_1),(h_1)
&=\big(a_{A,1},a_{B,1},a^{\cI}_1),(h_{A,1},h_{B,1},h^{\cI}_1\big),\\
(a_2),(h_2)
&=\big(a_{A,2},a_{B,2},a^{\cI}_2),(h_{A,2},h_{B,2},h^{\cI}_2\big).
\end{align}
Thus, $(\tfrac{a_1+a_2}{2}),(\tfrac{h_1+h_2}{2})\in \cL(l_A,\,l_B;\,\tfrac{u_1+u_2}{2})$ 
and, with the help of the concavity of $g_{l_A}, g_{l_B},\widetilde{g}$ proven above, 
we can write
\be{aaalt}
\tfrac{u_1+u_2}{2}\,\psi(\Theta,\tfrac{u_1+u_2}{2})
\geq \widetilde{Q}((\tfrac{a_1+a_2}{2}), (\tfrac{h_1+h_2}{2}))
\geq \tfrac12 \big(u_1\,\psi(\Theta,u_1)+u_2\,\psi(\Theta,u_2)\big).
\ee
At this stage, we assume that the right-most inequality in \eqref{aaalt} is an equality and 
show that this leads to a contradiction, after which Lemma \ref{concav} will be proven. 

We have proven above that $(a_1),(h_1) \notin \widetilde{\cL}(l_A,\, l_B;\,u_1)$ and 
$(a_2),(h_2) \notin \widetilde{\cL}(l_A,\, l_B;\,u_2)$. Thus, we can use \eqref{defQtil} 
and the strict concavity of $g_{l_A}, g_{l_B}$ on $\cN_{l_A}^+,\cN_{l_B}^+$ and the 
concavity of $\widetilde{g}$ on $\cN_0$  to conclude that necessarily
\begin{align}
\label{2imp}
&(a_{A,1},h_{A,1}) = (a_{A,2},h_{A,2}) \quad \text{and} \quad 
(a_{B,1},h_{B,1}) = (a_{B,2},h_{B,2}).
\end{align}
As a consequence, we recall that $u_1>u_2$ and we can write
\be{onut}
u_1^\cI=u_1-a_{A,1}-a_{B,2}>u_2-a_{A,2}-a_{B,2}=u_2^\cI\geq 0,
\ee
and therefore, since $(a_1),(h_1) \notin \widetilde{\cL}^{\cI}(l_A,\, l_B;\,u_1)$, it follows 
that $h_1^\cI>0$ such that (recall \eqref{2imp})
\be{onut2}
h_1^\cI=1-h_{A,1}-h_{B,1}=1-h_{A,2}-h_{B,2}=h_2^\cI>0.
\ee 
Hence we can use the strict concavity of $a\mapsto \widetilde{g}(a,h_1^\cI)$ to conclude 
that $u_1^\cI=u_2^{\cI}$, which clearly contradicts \eqref{onut}.


\subsubsection{Proof of Proposition \ref{base}}
\label{proofbase}

The proof is performed with the help of Lemma \ref{lele} stated in section \ref{Computation}. For
this reason we use some notations introduced in Lemma \ref{lele}.

We pick $\gamma, \eta>0$ (which will be specified later), and we let $\widehat{K}\in\N$ 
be the integer in Lemma~\ref{lele} associated with $\alpha,\beta,\eta,\gamma$. For 
$\Theta\in \overline\cV_M$, $u\geq t_\Theta$ and $\pi\in \cW_{\Theta,u,L}$, we let 
$N_\pi$ be the number of excursions of $\pi$ in solvent $B$ in columns of type $\Theta$. 
We further let  also $(I_\pi)=(I_\pi(1),\dots, I_\pi(N_\pi))$ be the sequence of consecutive 
intervals in $\{1,\dots,uL\}$ on which $\pi$ makes these $N_\pi$ excursions in $B$, so 
that  $(I_\pi)\in\cE_{uL,N_\pi}$ and $T(\pi)=\sum_{i=1}^{N_\pi} |I_\pi(i)|$. 

Pick $\Theta\in \overline\cV_M$, $u\geq t_\Theta$ and partition $\cW_{\Theta,u,L}$ into
two parts:
\begin{align}
V_{u,L,\gamma}^{\Theta,+}&=\{\pi\in \cW_{\Theta,u,L}\colon\,
T(\pi)\geq \gamma uL\}\quad \text{and}\quad V_{u,L,\gamma}^{\Theta,-}
=\{\pi\in \cW_{\Theta,u,L}\colon\, T(\pi)\leq \gamma uL\}.
\end{align}
There exists a $c>0$, depending on $\alpha,\beta$ only, such that 
\be{pide}
\big|H_L^{\Theta,\omega}(\pi)-T(\pi) \tfrac{\beta-\alpha}{2}\big|
\leq c T(\pi)\leq c \gamma u L, \quad \pi\in V_{u,L,\gamma}^{\Theta,-}.
\ee
Since any excursion in solvent $B$ requires at least $1$ horizontal steps or $L$ vertical 
steps, we have that $N_\pi\leq u+L$ for $\pi\in \cW_{\Theta,u,L}$. Since $u+L\leq uL/
\widehat K$ as soon as $u,L\geq 2\widehat K$, it follows that 
\be{Ipi}
I(\pi)\in\cup_{N=1}^{uL/\widehat K} \{I\in \cE_{uL,N}\colon\, 
T(I)\geq \gamma u L\},\ \  L\geq 2\widehat K,\ \ u\geq t_\Theta\vee 2\widehat K,\ \ 
\pi\in V_{u,L,\gamma}^{\Theta,+},
\ee
and therefore $\omega\in Q_{uL, \widehat K}^{\gamma,\eta}$ implies that $|H_L^{\Theta,\omega}
(\pi)-T(\pi) \tfrac{\beta-\alpha}{2}|\leq \eta u L$ for $\pi \in V_{u,L,\gamma}^{\Theta,+}$.
Consequently, for $\omega\in Q_{uL, \widehat K}^{\gamma,\eta}$, we have 
\be{pide2}
\big|H_L^{\Theta,\omega}(\pi)-T(\pi) \tfrac{\beta-\alpha}{2}\big|
\leq  u L (\eta+c \gamma),\ \ \Theta\in \overline{\cV}_M,\ \ 
u\geq 2\widehat K\vee t_\Theta,\ \ L\geq 2 \widehat K, \ \ \pi\in \cW_{\Theta,u,L}.
\ee
Rewrite
\begin{align}
\label{inegva2}
\psi_L(\Theta,u)=\E\bigg[\tfrac{1}{uL} \log &\sum_{\pi\in \cW_{\Theta,u,L}} 
e^{H_L^{\Theta,\omega}(\pi)} \big| \cQ_{uL,\widehat K}^{\gamma,\eta}\bigg]
+ \P\Big(\big(\cQ_{uL,\widehat K}^{\gamma,\eta}\big)^c\Big)\ \Delta,
\end{align}
where $\Delta$ is an error term given by 
\begin{align}
\label{inegva22}
\Delta= \E\Big[\tfrac{1}{uL} \log \sum_{\pi\in \cW_{\Theta,u,L}} 
e^{H_L^{\Theta,\omega}(\pi)} \big| \big(\cQ_{uL,\widehat K}^{\gamma,\eta}\big)^c\Big]
-\E\Big[\tfrac{1}{uL} \log \sum_{\pi\in \cW_{\Theta,u,L}} 
e^{H_L^{\Theta,\omega}(\pi)} \big| \cQ_{uL,\widehat K}^{\gamma,\eta}\Big].
\end{align}
By \eqref{boundel}, we obtain that $|\Delta|\leq 2 C_{\text{uf}}$.

To conclude, we set $\eta=\gep/3$, $\gamma=\gep/3c$. By Lemma~\ref{lele}, there exists 
an $L_\gep\in \N$ such that, for $u\geq 2\widehat{K}\vee t_\Theta$ and $L\geq L_\gep$, 
we have $\P\big(\big(\cQ_{uL,\widehat K}^{\gamma,\eta}\big)^c\big)\leq \gep/6\,
C_{\text{uf}}$. Thus, we can use \eqref{pide2} and \eqref{inegva2} to complete the 
proof of Proposition~\ref{base}.


\subsubsection{Proof of Lemma \ref{boundunif}}\label{proofboundunif}

Pick $\gep>0$. By applying  Proposition~\ref{base} with $\gep/2$, we see that there 
exists an $R_{\gep/2}>0$ such that 
\be{fejuo}
\psi(\Theta,u)\leq \limsup_{L\to \infty} \tfrac{1}{uL} \log 
\sum_{\pi\in \cW_{\Theta,u,L}} e^{T(\pi) \tfrac{\beta-\alpha}{2}} 
+\tfrac{\gep}{2},\quad \Theta\in \overline \cV_M,\quad u\geq t_\Theta\vee R_{\gep/2}.
\ee
We first consider the case $\Theta\in \overline \cV_M\setminus \overline \cV_{\nAB,B,1,M}$. 
Since $(\alpha,\beta)\in \CONE$, we can use \eqref{fejuo} to obtain 
\be{fejuo1}
\psi(\Theta,u)\leq \limsup_{L\to \infty} \tfrac{1}{uL} \log |\cW_{\Theta,u,L}| 
+ \tfrac{\gep}{2},\quad u\geq t_\Theta\vee R_{\gep/2}.
\ee
Thus, \eqref{fejuo1} and Lemma~\ref{convularge} imply that there exists a $C_\gep\geq 
R_{\gep/2}$ such that $\psi(\Theta,u)\leq \gep$ when $u\geq t_\Theta\vee C_\gep$ and 
$\Theta\in \overline \cV_M\setminus \overline \cV_{\nAB,B,1,M}$. The case $\Theta \in 
\overline \cV {\nAB,B,1,M}$ can be treated similarly after noticing that $T(\pi)=uL$ 
for $\pi\in \cW_{\Theta,u,L}$ and $\Theta\in \overline \cV_{\nAB,B,1,M}$.


\subsubsection{Proof of Lemma \ref{fep1}}
\label{prooffep1}

The proof is a straightforward consequence of the strict concavity of $u\mapsto u\psi
(\Theta,u)$ for $\Theta\in \overline\cV_M$, Proposition~\ref{base} and 
Lemma~\ref{convularge}.


\subsubsection{Proof of Lemma \ref{fep2}}
\label{prooffep2}

Pick $\gep>0$. The proof will be complete once we show the following two properties:
\begin{itemize}
\item[(1)]
There exists a $T_{\gep}>0$ such that 
\be{defu22}
\partial^{-}_u (u\psi(\Theta,u))(2 t_\Theta)\leq \left\{
\begin{array}{ll}
\vspace{.1cm}
\gep
& \mbox{if  }\  \ \Theta\in \overline \cV_M\setminus
\overline\cV_{\nAB,B,1,M}\colon t_\Theta\geq T_\gep,\\
\vspace{.1cm}
\frac{\beta-\alpha}{2}+\gep
&\mbox{if } \  \ \Theta\in\overline\cV_{\nAB,B,1,M}
\colon t_\Theta\geq T_\gep.
\end{array}
\right.
\ee
\item[(2)] 
For all $T>0$ there exists a $V_{\gep,T}>0$ such that 
\be{defu23}
\partial^{-}_u (u\psi(\Theta,u))(v)\leq \left\{
\begin{array}{ll}
\vspace{.1cm}
\gep
& \mbox{if }\  \ \Theta\in \overline \cV_M\setminus \overline 
\cV_{\nAB,B,1,M}\colon t_\Theta\leq T, \quad v\geq  t_\Theta\vee V_{\gep,T},\\
\vspace{.1cm}
\frac{\beta-\alpha}{2}+\gep
&\mbox{if } \  \ \Theta\in \overline \cV_{\nAB,B,1,M}\colon 
t_\Theta\leq T, \quad v\geq  t_\Theta\vee V_{\gep,T}.
\end{array}
\right.
\ee
\end{itemize}

We prove \eqref{defu23} for the case $\Theta\in \cV_M\setminus \overline \cV_{\nAB,B,1,M}$ 
(the case $\Theta\in\cV_{\nAB,B,1,M}$ can be treated similarly). To that aim, we assume 
that there exists a sequence $(\Theta_n)_{n\in\N}$ in $\cV_M\setminus \overline
\cV_{\nAB,B,1,M}$ such that $t_{\Theta_n}\leq T$ for $n\in \N$ and a sequence 
$(u_n)_{n\in\N}$ such that $u_n\geq t_{\Theta_n}$ for $n\in \N$, $\lim_{n\to\infty} 
u_n=\infty$ and
\be{trf}
\partial^{-}_u (u\psi(\Theta_n,u))(u_n)\geq \gep,\quad n\in \N.
\ee
By concavity of $u\mapsto u\psi(\Theta_n,u)$ for $n\in \N$ (see Lemma \ref{concav}), we 
have
\be{trf2}
u_n \psi(\Theta_n,u_n)-t_{\Theta_n} \psi(\Theta_n,t_{\Theta_n})\geq 
\gep \,(u_n-t_{\Theta_n}),\quad n\in \N.
\ee
Therefore, the uniform bound on free energies in \eqref{boundel} and the inequality 
$t_{\Theta_n}\leq T$ allow us to rewrite \eqref{trf2} as
\be{bbv}
\psi(\Theta_n,u_n)\geq \gep-\frac{T (C_{\text{uf}}+\gep)}{u_n}, \quad n\in \N,
\ee
which contradicts Lemma~\ref{boundunif} because $\lim_{n\to\infty} u_n=\infty$.

It remains to prove \eqref{defu22}. This is done in a similar manner for the case 
$\Theta\in \cV_M\setminus \overline \cV_{\nAB,B,1,M}$ (the case $\Theta\in\cV_{\nAB,B,1,M}$ 
can again be treated similarly), by assuming that there exists a sequence 
$(\Theta_n)_{n\in\N}$ in $\cV_M\setminus \overline \cV_{\nAB,B,1,M}$ such that
$\lim_{n\to\infty} t_{\Theta_n}=\infty$ and
\be{trf3}
\partial^{-}_u (u\psi(\Theta_n,u))(2 t_{\Theta_n})\geq \gep,\quad n\in \N.
\ee
Thus, similarly as in (\ref{trf2}--\ref{bbv}), the concavity of $u\mapsto u\psi
(\Theta_n,u)$ and \eqref{trf3} give
\be{trf4}
\psi(\Theta_n,2 t_{\Theta_n})\geq \frac{\gep}{2}
+\frac{\psi(\Theta_n,t_{\Theta_n})}{2},\quad n\in \N.
\ee
At this point we use Proposition~\ref{base} to assert that there exist $R_\gep>0$ and 
$L_\gep\in \N$ such that, for $n$ satisfying $t_{\Theta_n}\geq R_\gep$ and $L\geq L_\gep$, 
we have 
\begin{align}
\label{deff22}
\psi(\Theta_n,t_{\Theta_n})&\geq \tfrac{1}{t_{\Theta_n}L} 
\log \sum_{\pi\in \cW_{\Theta,t_{\Theta_n},L}} 
e^{T(\pi) \tfrac{\beta-\alpha}{2}} -\tfrac{\gep}{4},\\
\nonumber \psi(\Theta_n,2 t_{\Theta_n})
&\leq  \tfrac{1}{2 t_{\Theta_n}L} \log \sum_{\pi\in \cW_{\Theta_n,2 t_{\Theta_n},L}} 
e^{T(\pi) \tfrac{\beta-\alpha}{2}} +\tfrac{\gep}{4}.
\end{align}
By using (\ref{trf4}--\ref{deff22}), we obtain that, for $t_{\Theta_n}\geq R_\gep$ and 
$L\geq L_\gep$,
\begin{align}
\label{defg22}
&\tfrac{1}{2 t_{\Theta_n}L} \log \sum_{\pi\in \cW_{\Theta_n,2 t_{\Theta_n},L}} 
e^{T(\pi) \tfrac{\beta-\alpha}{2}}\geq \tfrac{1}{2 t_{\Theta_n}L} 
\log \sum_{\pi\in \cW_{\Theta,t_{\Theta_n},L}} 
e^{T(\pi) \tfrac{\beta-\alpha}{2}} +\tfrac{\gep}{8},
\end{align}uses
some key ingredients that are provided
which we can rewrite as
\begin{align}
\label{defh22}
\tfrac{1}{2 t_{\Theta_n}L} \log |\cW_{\Theta_n,2 t_{\Theta_n},L}|
+ \tfrac{\beta-\alpha}{4 t_{\Theta_n}L} & \min\{T(\pi), 
\pi\in \cW_{\Theta_n,2 t_{\Theta_n},L}\} \\
\nonumber &\geq \tfrac{\beta-\alpha}{4t_{\Theta_n}L} 
\min\{T(\pi), \pi\in \cW_{\Theta_n, t_{\Theta_n},L}\} +\tfrac{\gep}{8}.
\end{align}
Since $\Theta_n\in \cV_M\setminus \overline \cV_{\nAB,B,1,M}$, there exist $\pi_1\in
\cW_{\Theta_n,t_{\Theta_n},L}$ and $\pi_2\in \cW_{\Theta_n, 2t_{\Theta_n},L}$ 
such that
\begin{align}
\label{minreach}
T(\pi_1)
&=l_B(\Theta_n)=\min\{T(\pi), \pi\in \cW_{\Theta_n, t_{\Theta_n},L}\},\\
\nonumber 
T(\pi_2)&=l_B(\Theta_n)=\min\{T(\pi), \pi\in \cW_{\Theta_n, 2 t_{\Theta_n},L}\}.
\end{align}
Thus, for $t_{\Theta_n}\geq R_\gep$ and $L\geq L_\gep$, the inequality in 
\eqref{defh22} becomes
\begin{align}
\label{defh25}
\tfrac{1}{2 t_{\Theta_n}L} \log |\cW_{\Theta_n,2 t_{\Theta_n},L}|\geq\tfrac{\gep}{8},
\end{align}
which obviously contradicts Lemma \ref{convularge}.


\section{Concentration of measure}
\label{Ann2}

Let $\cS$ be a finite set and let $(X_i,\cA_i,\mu_i)_{i\in \cS}$ be a family of 
probability spaces. Consider the product space $X=\prod_{i\in \cS} X_i$ endowed 
with the product $\sigma$-field $\cA=\otimes_{i\in \cS}\cA_i$ and with the product 
probability measure $\mu=\otimes_{i\in \cS} \mu_i$.  

\begin{theorem}
\label{theoco} {\rm (Talagrand~\cite{T96})}
Let $f\colon\,X\mapsto \R$ be integrable with respect to $(\cA,\mu)$ and, for 
$i\in \cS$, let $d_i>0$ be such that $|f(x)-f(y)|\leq d_i$ when $x,y\in X$ differ 
in the $i$-th coordinate only. Let $D=\sum_{i\in \cS} d_i^2$. Then, for all $\gep>0$,
\be{gteo}
\mu\left\{x\in X\colon \left|f(x)-\int fd\mu\right|>\gep\right\}
\leq 2 e^{-\frac{\gep^2}{2D}}.
\ee
\end{theorem}

The following corollary of Theorem~\ref{theoco} was used several times in the paper. 
Let $(\alpha,\beta)\in \CONE$ and let $(\xi_i)_{i\in\N}$ be an i.i.d.\ sequence 
of Bernouilli trials taking the values $-\alpha$ and $\beta$ with probability $\tfrac12$ 
each. Let $l\in \N$, $T\colon\,\,\{(x,y)\in\Z^2\times \Z^2\colon |x-y|=1\}\to\{0,1\}$ 
and $\Gamma\subset\cW_l$ (recall \eqref{defw}). Let $F_l\colon\,[-\alpha,\alpha]^l
\to \R$ be such that
\be{defF}
F_l(x_1,\dots,x_l) = \log \sum_{\pi\in \Gamma} e^{\sum_{i=1}^{l}
x_i\, T( (\pi_{i-1},\pi_i))}. 
\ee
For all $x,y\in [-\alpha,\alpha]^l$ that differ in one coordinate only we have 
$|F_l(x)-F_l(y)|\leq 2\alpha$. Therefore we can use Theorem~\ref{theoco} with 
$\cS=\{1,\dots,l\}$, $X_i=[-\alpha,\alpha]$ and $\mu_i=\tfrac{1}{2} (\delta_{-\alpha} 
+ \delta_{\beta})$ for all $i\in \cS$, and $D=4\alpha^2 l$, to obtain that there 
exist $C_1,C_2>0$ such that, for every $l\in\N$, $\Gamma\subset \cW_n$ and 
$T\colon\,\{(x,y)\in\Z^2\times \Z^2\colon |x-y|=1\}\to \{0,1\}$,
\begin{equation}
\label{concmesut}
\P\big(|F_l(\xi_1,\dots,\xi_m)-\E(F_l(\xi_1,\dots,\xi_m))|>\eta\big)
\leq C_1e^{-\tfrac{C_2\eta^2}{l}}.
\end{equation}  


\section{Large deviation estimate}
\label{Computation}

Let $(\xi_i)_{i\in\N}$ be an i.i.d.\ sequence of Bernouilli trials taking values 
$\beta$ and $-\alpha$ with probability $\frac12$ each. For $N\leq n\in \N$, denote 
by $\cE_{n,N}$ the set of all ordered sequences of $N$ disjoint and non-empty 
intervals included in $\{1,\dots,n\}$, i.e.,
\begin{align}
\label{add26}
\nonumber\cE_{n,N} &= \big\{(I_j)_{1\leq j\leq N}\subset\{1,\dots,n\}\colon 
I_j=\{\min I_j,\dots,\max I_j\} \,\,\forall\,1\leq j\leq N,\\
& \max I_j<\min I_{j+1} 
\,\,\forall\,1\leq j\leq N-1\ \text{and}\ I_j \neq 
\emptyset\,\,\forall\,1\leq j\leq N\big\}.
\end{align}
For $(I)\in \cE_{n,N}$, let $T(I)=\sum_{j=1}^N |I_j|$ be the cumulative length of 
the intervals making up $(I)$. Pick $\gamma>0$ and $K\in \N$, and denote by 
$\widehat{\cE}_{n,K}^{\,\gamma}$ the set of those $(I)$ in $\cup_{1\leq N\leq (n/K)}\,
\cE_{n,N}$ that have a cumulative length larger than $\gamma n$, i.e.,
\be{add27}
\widehat{\cE}_{n,K}^{\,\gamma} = \cup_{N=1}^{n/K}\big\{(I)\in \cE_{n,N}
\colon\,T(I)\geq \gamma n\big\}.
\ee
Next, for $\eta>0$ set 
\be{add28}
\cQ_{n,K}^{\gamma,\eta} = \bigcap_{\,(I)\in \widehat{\cE}_{n,K}^{\,\gamma}} 
\left\{\bigg|\sum_{j=1}^N \sum_{i\in I_j} \big (\xi_i-\tfrac{\beta-\alpha}{2}\big)\bigg|
\leq \eta\,T(I)\right\}.
\ee

\begin{lemma}
\label{lele}
For all $(\alpha,\beta)\in \CONE$, $\gamma>0$ and $\eta>0$ there exists an $\widehat{K}
\in \N$ such that, for all $K\geq \widehat{K}$,
\be{convsubset}
\lim_{n\to \infty} P((\cQ_{n,K}^{\gamma,\eta})^c)=0.
\ee
\end{lemma} 

\begin{proof}
An application of Cram\'er's theorem for i.i.d.\ random variables gives that
there exists a $c_\eta>0$ such that, for every $(I)\in \widehat{\cE}_{n,K}^{\,\gamma}$,
\be{add29}
\P_\xi\bigg(\bigg|\sum_{j=1}^N \sum_{i\in I_j} 
\big(\xi_i -\tfrac{\beta-\alpha}{2}\big) \bigg|\geq \eta\, T(I)\bigg)\leq\, 
e^{-c_\eta T(I)}\leq\, e^{-c_\eta\gamma n},
\ee
where we use that $T(I)\geq \gamma n$ for every $(I)\in \widehat{\cE}_{n,K}^{\,\gamma}$. 
Therefore
\begin{align}
\P_\xi((\cQ_{n,K}^{\gamma,\eta})^c)&\leq  |\widehat{\cE}_{n,K}^{\,\gamma}| 
e^{-c(\eta)\gamma n },
\end{align}
and it remains to bound  $|\widehat{\cE}_{n,K}^{\,\gamma}|$ as
\be{domiE}
\widehat{\cE}_{n,K}^{\,\gamma} = \sum_{N=1}^{n/K}\big|\big\{(I)\in \cE_{n,N}
\colon T(I)\geq \gamma n\big\}\big|
\leq \sum_{N=1}^{n/K} \binom{n}{2N}, 
\ee
where we use that choosing $(I)\in \cE_{n,N}$ amounts to choosing in $\{1,\dots,n\}$ 
the end points of the $N$ disjoint intervals. Thus, the right-hand side of \eqref{domiE} 
is at most $(n/K) \binom{n}{2n/K}$, which for $K$ large enough is 
$o( e^{c(\eta)\gamma n})$ as $n\to \infty$. 
\end{proof}


\section{On the maximizers of the slope-based variational formula}
\label{appT}

In this appendix we prove that, when restricted to $\bar \cR_{p,M}$, the supremum 
of the variational formula in \eqref{genevar}, which equals the truncated free energy 
$f(\alpha,\beta;M,p)$, is attained at some $\bar \rho\in \bar\cR_{p,M}$ and for a unique 
$\bar v\in \bar\cB$. For ease of notation we suppress the $M,p$-dependence of 
$f(\alpha,\beta;M,p)$ in the proofs of this section.

Recall \eqref{defhh} and for $M\in \N$, $p\in (0,1)$ and $(\alpha,\beta)\in \CONE$, 
let $\cO_{p,M,\alpha,\beta}$ be the subset of $\bar \cR_{p,M}$ containing those 
$\bar\rho$ that maximize the variational formula in \eqref{genevar}, i.e., 
\be{rap**}
f(\alpha,\beta;M,p)=h(\bar\rho;\,\alpha,\beta)
= \sup_{v\in \bar \cB}\,\frac{\bar N(\bar\rho,v)}{\bar D(\bar\rho,v)}
\quad \text{for $\bar\rho\in \cO_{p,M,\alpha,\beta}$.}
\ee
Recall (\ref{defucp1}--\ref{defucp3}) and set 
\be{defbv}
\bar v= v(f(\alpha,\beta;M,p)).
\ee

\begin{theorem}
\label{maxv}
For all $M\in \N$,  $p\in (0,1)$ and $(\alpha,\beta)\in \CONE$ the following hold:\\
(1) The set $\cO_{p,M,\alpha,\beta}$ is non-empty.\\
(2) For all $\bar\rho\in \cO_{p,M,\alpha,\beta}$ and all $v\in \bar\cB$ satisfying 
$f(\alpha,\beta;M,p)=\bar N(\bar\rho,v)/\bar D(\bar\rho,v)$, $v= \bar v$ for 
$\bar \rho$-a.e.\  $(k,l)\in \{A,B\}\times [0,\infty)$ or $k=\cI$.
\end{theorem}

\begin{proof}
The following proposition will be proven in Section \ref{reached1} below and tells us 
that the maximum of the old variational formula in \eqref{genevarD1} is attained for 
some $\rho \in \cR_{p,M}$. Recall the definition of $g(\rho;\alpha,\beta)$ for 
$\rho \in \cR_{p,M}$ in \eqref{defgg}.

\begin{theorem}
\label{reached}
For all $(\alpha,\beta)\in \CONE$, there exists a $\rho\in \cR_{p,M}$ such that 
$f(\alpha,\beta;M,p)=g(\rho;\alpha,\beta)$.
\end{theorem}

We give the proof of Theorem \ref{maxv} subject to Theorem \ref{reached}. 
To that aim, we pick $(\alpha,\beta)\in \CONE$ and note that, by Theorem \ref{reached}, 
there exists a $\hat \rho\in \cR_{p,M}$ such that $f(\alpha,\beta)=g(\hat\rho;\alpha,\beta)$. 
In what follows, we suppress the $(\alpha,\beta)$-dependence of $g(\hat\rho;\alpha,\beta)$. 

Since $f(\alpha,\beta)=g(\hat\rho)$, \eqref{posi} ensures that $g(\hat \rho)>0$, and by applying 
Lemma \ref{maximo} we obtain that 
\be{tgk}
f(\alpha,\beta)=\frac{N(\hat\rho,u(f(\alpha,\beta)))}{D(\hat\rho,u(f(\alpha,\beta)))}.
\ee
Apply Lemma \ref{ABC}, which ensures that there exist a $\bar \rho\in  \bar \cR_p$ and a 
$v\in \bar \cF$ such that
\be{tgk2}
\frac{N(\hat \rho,u(f(\alpha,\beta)))}{D(\hat \rho,u(f(\alpha,\beta)))}\leq 
\frac{\bar N(\bar \rho,v)}{\bar D(\bar \rho, v)}.
\ee
Then $h(\bar \rho)>0$, and we use Lemma \ref{reducc}, which tells us that
\be{tgk3}
\frac{\bar N(\bar \rho,v)}{\bar D(\bar \rho, v)}\leq  
\frac{\bar N(\bar \rho,v(h(\bar \rho)))}{\bar D(\bar \rho, v(h(\bar \rho)))}.
\ee 
Now (\ref{tgk}--\ref{tgk3}) and the variational formula in \eqref{genevar} are sufficient to complete 
the proof of (1). The proof of (2) is a straightforward consequence of Lemma \ref{maximo}.
\end{proof}


\subsection{Proof of Theorem \ref{reached}}
\label{reached1}

We give the proof of Theorem \ref{reached} subject to the following lemma, which will be proven 
in Section \ref{lefon1} below.

\bl{lefon}
For all $t>0$ and $u\in \cB_{\overline{\cV}_M}$ there exists an $m_0\in \N$ such that, 
for all $\rho\in \cR_{p,M}$ and $v\in \cB_{\overline{\cV}_M}$ satisfying $v\leq u$ and 
$N(\rho,v)/D(\rho,v) \geq t$, there exists a $\tilde{\rho}\in \cR_{p,M}^{m_0}$ such that 
$N(\tilde\rho,v)/D(\tilde \rho,v) \geq N(\rho,v)/D(\rho,v)$.
\el

Let $(\rho_n)_{n\in \N}$ in $\cR_{p,M}$ be such that $n \mapsto g(\rho_n;\alpha,\beta)$ is increasing 
with $\lim_{n\to \infty} g(\rho_n;\alpha,\beta)=f(\alpha,\beta)$. Obviously we can choose $(\rho_n)_{n\in \N}$ 
such that $g(\rho_n;\alpha,\beta) \geq f(\alpha,\beta)/2$ for all $n\in \N$. Thus, with the help of Lemma 
\ref{maximo}, we obtain 
\be{}
g(\rho_n;\alpha,\beta)=\frac{N(\rho_n,u(g(\rho_n)))}{ D(\rho_n,u(g(\rho_n)))},  \quad  n\in \N.
\ee
Apply Lemma \ref{lefon} to see that there exists an $m_0\in \N$ such that for all $n\in \N$ there exists an 
$\hat \rho_n \in \cR_{p,M}^{m_0}$ such that 
\be{ccu}
\frac{N(\hat\rho_n,u(g(\rho_n)))}{D(\hat\rho_n,u(g(\rho_n)))}
\geq \frac{N(\rho_n,u(g(\rho_n)))}{D(\rho_n,u(g(\rho_n)))}.
\ee
A straightforward consequence of \eqref{ccu} is that 
\be{}
\lim_{n\to \infty} \frac{N(\hat\rho_n,u(g(\rho_n)))}{D(\hat\rho_n,u(g(\rho_n)))}=f(\alpha,\beta).
\ee 
Moreover, $\hat \rho_n\in \cM_1(\overline{\cV}_M^{\,m_0})$ for all $n\geq n_0$, and since 
$\overline{\cV}_M^{\,m_0}$ is compact we have that $\hat \rho_n$ converges weakly to 
$\rho_\infty  \in \cR_{p,M}^{m_0}$ along a subsequence. Lemma \ref{reguc} implies that 
$n \mapsto u(g(\rho_n))$ is non-increasing and converges pointwise to $u(f(\alpha,\beta))$ 
as $n\to \infty$. Since $\overline{\cV}_M^{\,m_0}$ is compact, Dini's Theorem tells us that 
the convergence of $u(g(\rho_n))$ to $u(f(\alpha,\beta))$ is uniform on $\overline{\cV}_M^{\,m_0}$. 
Therefore, using the uniform continuity of $(u,\Theta)\mapsto u \psi(\Theta,u)$ (see Lemma 
\ref{concavt}), we obtain 
\be{}
f(\alpha,\beta)=\frac{N(\rho_\infty,u(f( \alpha,\beta)))}{D(\rho_\infty,u(f(\alpha,\beta)))},
\ee 
which completes the proof of Theorem \ref{reached}.


\subsubsection{Proof of Lemma \ref{lefon}}
\label{lefon1}

First, we state and prove Claim \ref{lemaint} below, which will be needed to prove Lemma \ref{lefon}. 
Pick $m\geq M+2$, and note that for $\Theta=(\chi,\Delta\Pi,b_0,b_1,x)\in \overline{\cV}_M\setminus
\overline{\cV}_M^{\,m}$ we necessarily have $x_\Theta=2$. Define  $T_m\colon\,\overline{\cV}_M
\mapsto \overline{\cV}_M^{\,m}$ as
\be{defT}
T_m(\Theta)= \left\{
\begin{array}{ll}
\vspace{.1cm}
\Theta
& \mbox{if } \Theta\in \overline{\cV}_M^m, \\
\vspace{.1cm}
\widetilde \Theta=(\chi,\Delta\Pi,b_0,b_1,1)
& \mbox{if } \Theta=(\chi,\Delta\Pi,b_0,b_1,2)\in \overline{\cV}_M\setminus \overline{\cV}_M^m,
\end{array}
\right.
\ee

\begin{claim}
\label{lemaint}
For all $\rho\in \cR_{p,M}$ and $m\in \N \colon m\geq M+2$, $\rho\, \circ\, T_m^{-1}\in \cR_{p,M}^m$.
\end{claim}

\begin{proof}
First note that $T_m\colon\,\overline{\cV}_M\mapsto \overline{\cV}_M^{\,m}$ is continuous with respect 
to the $d_M$-distance. Next, pick $\rho\in \cR_{p,M}$. By the definition of $\cR_{p,M}$, there exists a 
strictly increasing sequence $(N_k)_{k\in \N}$ and $(\Pi^k_j)_{j\in \N_0}$, $(b^k_j)_{j\in \N_0}$, 
$(x^k_j)_{j\in \N_0}$ such that $\rho=\lim_{k\to \infty} \rho_{N_k}(\Omega,\Pi^k,b^k,x^k)$. The 
continuity of $T_m$ implies that
\be{}
\rho\, \circ\, T_m^{-1}=\lim_{k\to \infty} \rho_{N_k}(\Omega,\Pi^k,b^k,x^k)\, \circ\, T_m^{-1},
\ee
and we can easily check that  
\be{}
\rho_{N_k}(\Omega,\Pi^k,b^k,x^k)\, \circ\, T_m^{-1}=\rho_{N_k}(\Omega,\Pi^k,b^k,\tilde x^k), 
\ee
where for $j,k\in \N_0$ we define
\be{defxx}
\tilde x^k_j= \left\{
\begin{array}{ll}
\vspace{.1cm}
  x^k_j
& \mbox{if } \quad (\Omega(j,\cdot),\Delta\Pi^k_j,b^k_j,b^k_{j+1}\tilde x^k_j)\in \overline{\cV}_M^m, \\
\vspace{.1cm}
1
&\mbox{otherwise.} 
\end{array}
\right.
\ee
Consequently, $\rho\, \circ\, T_m^{-1}\in \cR_{p,M}$.
\end{proof}

We resume the proof of Lemma \ref{lefon}. Pick $t>0$, $\rho\in \cR_{p,M}$, $u\in \cB_{\overline{\cV}_M}$ 
and $v\in \cB_{\overline{\cV}_M}$ satisfying $v\leq u$ and $N(\rho,v)/D(\rho,v) \geq t$. Pick $m\in \N
\colon m\geq M+2$, whose value will be specified later, and set $\rho_m=\rho\, \circ\, T_m^{-1}$, 
which belongs to $\cR_{p,M}$ by Claim \ref{lemaint}. Write
\be{defG1}
\frac{ N(\rho_m,v)}{D(\rho_m,v)}-\frac{N(\rho,v)}{D(\rho,v)}=\int_{0}^{1} G'(t) dt
\quad \text{with} \quad G(t)=\frac{A+tB}{c+tD}
\ee
with 
\begin{align}
\label{defABCD}
A&=\int_{\overline{\cV}_M} v_\Theta \psi(\Theta,v_\Theta)\,\rho(d\Theta)\quad \quad 
B=\int_{\overline{\cV}_M\setminus \overline{\cV}_M^m } v_{\widetilde\Theta}\,  
\psi(\widetilde\Theta,v_{\widetilde\Theta})-v_\Theta \psi(\Theta,v_\Theta)\,\rho(d\Theta)\\
C&=\int_{\overline{\cV}_M} v_\Theta \,\rho(d\Theta)\qquad \qquad \qquad D
=\int_{\overline{\cV}_M\setminus \overline{\cV}_M^m } v_{\widetilde\Theta}-v_\Theta \,\rho(d\Theta).
\end{align}
Note that the sign of the derivative $G'(t)$ is constant and equal to the sign of
\be{BBB}
B-\frac{A}{C} D=\int_{\overline{\cV}_M\setminus \overline{\cV}_M^m } 
v_\Theta \bigg[\frac{A}{C}\,  \Big(1-\frac{v_{\widetilde\Theta}}{v_\Theta}\Big) -\psi(\Theta,v_\Theta)
+\frac{v_{\widetilde\Theta}}{v_\Theta}\,  \psi(\widetilde\Theta,v_{\widetilde\Theta}) \bigg]\rho(d\Theta).
\ee
Therefore Lemma \ref{lefon} will be proven once we check that for $m$ large enough the right-hand 
side of \eqref{BBB} is strictly positive, uniformly in $v\leq u$. To that aim, we recall Lemma \ref{boundunif}, 
which tells us that $\psi(\Theta,v_\Theta)\leq t/2$ for every $\Theta \in \overline{\cV}_M\setminus
\overline{\cV}_M^{m}$, provided $m$ is chosen large enough (because $v_\Theta\geq t_\Theta\geq m$), 
and we recall \eqref{boundel}, which tells us that $\psi(\widetilde\Theta,v_{\widetilde\Theta})\leq 
C_{\text{uf}}(\alpha)$ for $\Theta \in \overline{\cV}_M\setminus \overline{\cV}_M^{m}$. We further 
note that 
\be{bornn}
v_{\widetilde\Theta}\leq \max\Big\{u_\Theta\colon\, \Theta\in \overline{\cV}_M^{M+2}\Big\}<\infty
\quad \text{for every}\ \Theta\in \overline{\cV}_M,
\ee
which, together with the fact that $\frac{A}{C}=N(\rho,v)/D(\rho,v) \geq t>0$ and $v_\Theta \geq 
t_\Theta\geq m$ for $\Theta\in\overline{\cV}_M\setminus \overline{\cV}_M^{m}$, ensures that for 
$m$ large enough the right-hand side of \eqref{BBB} is strictly positive, uniformly in $v\leq u$. 
This completes the proof of Lemma \ref{lefon}.

 
\section{Uniqueness of the maximizers of the variational formula}
\label{appA}

In this appendix we first prove, with the help of Lemma~\ref{lemare}, that for $\Theta\in
\overline\cV_M$ and $u\geq t_\Theta$ the variational formula in Proposition~\ref{energ} 
has unique maximizers. This uniqueness implies that, for a given column type and a given 
time spent in the column, the copolymer has a unique way to move through the column. We 
next use this uniqueness to show, with the help of Proposition~\ref{Borel1}, that for 
$u\in \cB_{\,\overline \cV_M}$ the maximizers of \eqref{Bloc of type I} are Borel functions 
of $\Theta\in \overline\cV_M$. 

Recall \eqref{setE} and pick $h\in \cE$. Set 
\begin{align}
\label{setU}
\nonumber \cU(h)=\big\{(r_{A,\Theta},r_{B,\Theta},
r_{\cI,\Theta})_{\Theta\in \overline{\cV}_M}\in ([0,\infty)^3)^{\overline\cV_M}\colon
&\ r_{k,\Theta}\geq 1+\tfrac{l_{k, \Theta}}{h_{k, \Theta}}\ \forall\,k\in \{A,B\} \ 
\forall \Theta\in \overline\cV_M,\\
\nonumber 
&\ r_{\cI,\Theta}\geq 1 \,\,\forall\,k\in \{A,B\}\ \forall \Theta\in \overline\cV_M,\\
&\ \Theta \mapsto r_{k,\Theta}\ \text{Borel}\,\,\forall\,k\in \{A,B,\cI\}\big\},
\end{align}
where we recall that $\tfrac{l_{k,\Theta}}{h_{k, \Theta}}=0$ by convention when 
$l_{k,\Theta}=h_{k,\Theta}=0$.

\begin{proposition}
\label{Borel}
For all $u\in \cB_{\,\overline \cV_M}$ there exist $h\in \cE$ and $r\in\cU(h)$
such that, for all $\Theta\in \overline\cV_M$,
\begin{align}\label{Borel1}
u_\Theta\,\psi(\Theta,u_\Theta)=h_{A,\Theta}\, 
&r_{A,\Theta}\,\tilde{\kappa}\big(r_{A,\Theta},
\tfrac{l_{A,\Theta}}{h_{A,\Theta}}\big)\\
\nonumber 
&+h_{B,\Theta}\,r_{B,\Theta}\,\big[\tilde{\kappa}\big(r_{B,\Theta},
\tfrac{l_{B,\Theta}}{h_{B,\Theta}}\big) +\tfrac{\beta-\alpha}{2}\big]
+h_{\cI,\Theta}\, r_{\cI,\Theta}\,\phi_\cI(r_{\cI,\Theta}),
\end{align}
and
\be{Borel2}
h_{A,\Theta}\, r_{A,\Theta} + h_{B,\Theta}\, r_{B,\Theta}
+ h_{\cI,\Theta}\, r_{\cI,\Theta}=u_\Theta.
\ee
\end{proposition}

\begin{proof}
For $l\in\R$, let
\be{defrmalt1}
\cN_l=\{(a,h)\in [0,\infty) \times [0,1]\colon a\geq h+|l|\},
\quad \cN_l^+=\{(a,h)\in \cN_l\colon h>0\},
\ee
let $g_l\colon\,\cN_l\mapsto [0,\infty)$ be defined as $g_l(a,h)=a\,\tilde{\kappa}
(\tfrac{a}{h},\tfrac{l}{h})$ for $h>0$ and $g_l(a,h)=0$ for $h=0$, and let $\tilde{g}
\colon\,\cN_0\mapsto [0,\infty)$ be defined as $\tilde{g}(a,h)=a\,\phi_{\cI}(\tfrac{a}{h})$ 
for $h>0$ and $\tilde{g}(a,h)=0$ for $h=0$. We can rewrite \eqref{Bloc of type I} as
\begin{align}
\label{Bloc of type II}
&u \psi(\Theta,u;\alpha,\beta)
= \sup_{(h),(a) \in \cL(\Theta;\,u)} f_{\,l_A,l_B}\big[(h),(a)\big]
\end{align}
with
\be{redef2}
f_{\,l_A,l_B}\big[(h),(a)\big]=g_{l_A}(a_A,h_A)+g_{l_B}(a_B,h_B)
+a_B\, \tfrac{\beta-\alpha}{2}+\tilde{g}(a_\cI,h_{\cI}).
\ee
Lemma~\ref{lemare} shows that, subject to some additional conditions, the maximizer in the 
right-hand side of \eqref{Bloc of type II} is unique. This allows us to prove the continuity 
of this maximizer as a function of $\Theta$ on each subset of a finite partition of 
$\overline{\cV}_M$, which implies the Borel measurability of this maximizer and completes 
the proof of Proposition~\ref{Borel}. 

\bl{lemare}
For all $\Theta\in \overline \cV_M$ and $u\geq t_\Theta$ there exists a unique 
$(\bar h), (\bar a) \in \cL (\Theta;u)$ satisfying:\\
(i) $u \,\psi(\Theta,u;\alpha,\beta) =  f_{\,l_A,l_B}[(\bar h),(\bar a)]$.\\
(ii) $\bar h_k>0$ if $\bar a_k>0$ for $k\in \{A,B,\cI\}$.\\
(iii) $\bar a_k=\bar h_k=0$ if $\bar l_k=0$ for $\Theta\in \overline{\cV}_{\AB,M}$ and 
$k\in \{A,B\}$.\\
(iv) $\bar a_k=\bar h_k=0$ if $\bar l_k=0$ for $\Theta\in \overline{\cV}_{\nAB,k,2,M}$
and $k\in \{A,B\}$.
\el

\begin{proof}
We prove existence and uniqueness.

\medskip\noindent
{\bf Existence.}
The existence of a $(h_1), (a_1) \in \cL (\Theta;u)$ satisfying (i) is ensured by the 
continuity of $f_{l_A,l_B}$ and the compactness of $\cL (\Theta;u)$. Assume that $\Theta
\in \overline{\cV}_{\AB,M}$, $l_A=0$ and $(h_{1,A}, a_{1,A})\neq (0,0)$. Then 
\begin{align}\label{co}
g_0(a_{1,A},h_{1,A})+\tilde{g}(a_{1,\cI},h_{1,\cI})
&\leq \tilde g(a_{1,A},h_{1,A})+\tilde{g}(a_{1,\cI},h_{1,\cI})\\
\nonumber 
& \leq 2\,\tilde{g}\big(\tfrac{a_{1,A}+a_{1,\cI}}{2},\tfrac{h_{1,A}
+h_{1,\cI}}{2}\big)=\tilde{g}\big(a_{1,A}+a_{1,\cI},h_{1,A}+h_{1,\cI}\big),
\end{align}
where we use the inequality $g_0\leq \tilde{g}$ and the concavity of $\tilde{g}$. Thus, 
by setting $(h_2),(a_2)=(0,h_{1,B},h_{1,A}+h_{1,\cI}), (0,a_{1,B},a_{1,A}+a_{1,\cI})$,
we obtain that $(h_2), (a_2) \in \cL (\Theta;u)$, satisfies (iii) and 
\be{derq}
f_{l_A,l_B}((h_2),(a_2))\geq f_{l_A,l_B}((h_1),(a_1)),
\ee
which implies that $(h_2),(a_2)$ also satisfies (i). The case $\Theta\in\overline{\cV}_{
\AB,M}$, $l_B=0$ and the case $\Theta\in \overline{\cV}_{\nAB,k,2,M}$, $l_k=0$, $k\in
\{A,B\}$, can be treated similarly, to conclude that there exist $(h),(a)\in\cL(\Theta;u)$ 
satisfying (i), (iii--iv). We will show that (ii) follows from these as well. The proof 
will be given for the case $\Theta\in \overline{\cV}_{\AB,M}$ and $l_A,l_B>0$, since (iii) 
already indicates that $h_k=a_k=0$ if $l_k=0$ for $k\in \{A,B\}$ and $\Theta\in\overline{
\cV}_{\AB,M}$. The case $\Theta\in \overline{\cV}_{\nAB,M}$ can be treated similarly. 

In the proof of Lemma \ref{concav} we showed that $(h),(a)\in \cL(\Theta,u)$ maximizing 
\eqref{Bloc of type II} necessarily satisfies $h_k>0$ if $a_k>l_k$ for $k\in \{A,B\}$ and 
$h_{\cI}>0$ if $a_{\cI}>0$. Thus, we only need to exclude the cases $h_k=0$ and $a_k=l_k>0$ 
for $k\in \{A,B\}$. We will therefore assume that $h_B=0$ and $a_B=l_B$, and prove that 
this leads to a contradiction. The case $h_A=0$ and $a_A=l_A$ is easier to deal with. We 
finally assume that $a_\cI>h_\cI>0$ (the case $a_\cI=h_\cI$ being easier). We pick $c>1$ 
and $x>0$ small enough to ensure that $a_\cI-cx>h_\cI-x>0$, and we set $(h)_x,(a)_x=(h_A,
x,h_\cI-x),(a_A,l_B+cx, a_I-cx)$. The proof will be complete once we show that for $x$ small 
enough the quantity
\be{lin}
f_{l_A,l_B}((h)_x,(a)_x)-f_{l_A,l_B}((h),(a))
=g_{l_B}(l_B+cx,x)-V_x+cx\left(\tfrac{\beta-\alpha}{2}\right)
\ee
is strictly positive with $V_x=\tilde{g}(a_\cI,h_\cI)-\tilde{g}(a_\cI-cx,h_\cI-x)$.

At this stage, we note that $\mu\mapsto \mu\phi_{\cI}(\mu)$ is concave on $[1,\infty)$, and 
therefore is Lipshitz on any interval $[r,t]$ with $r>1$. Since $a_\cI/h_\cI>0$, there exists 
a $C>0$, depending on $(a_\cI,h_\cI)$ only, such that $V_x\leq C x$ for $x$ small enough. 
Therefore \eqref{lin} becomes
\be{linalt}
f_{l_A,l_B}((h)_x,(a)_x)-f_{l_A,l_B}((h),(a))
\geq g_{l_B}(l_B+cx,x)-\big(C+c\,\tfrac{\beta-\alpha}{2}\big)\,x
\ee
for $x$ small enough. By the concavity of $g_{l_B}$, and since $g_{l_B}(l_B+cx,0)=0$, we can 
write $g_{l_B}(l_B+cx,x)\geq x\,\partial_2g_{l_B}(l_B+cx,x)$ for $x>0$. By the definition of 
$g_{l_B}$, and with \eqref{kapexplform}, we obtain that
\be{partde}
\partial_2 g_{l_B}(l_B+cx,x)
= \big(1+\tfrac{cx}{l_B}\big) \partial_2 \kappa\big(1+\tfrac{cx}{l_B},\tfrac{cx}{l_B}\big).
\ee 
We now recall \cite{dHW06}, Lemma 2.1.1, which claims that $\kappa$ is defined on $\DOM
=\{(a,b)\colon a\geq 1+b, b\geq 0\}$ and is analytic on the interior of $\DOM$. Moreover, 
in the proof of this lemma, an expression for $\partial_b\, \kappa(a,b)$ is provided that 
is valid on the interior of $\DOM$. From this expression, and since $c>1$, we can check 
that $\lim_{s \downarrow 0} \partial_2 \kappa(1+cs,s)=\infty$, which suffices to conclude 
that the right-hand side of \eqref{lin} is strictly positive for $x$ small enough. This 
completes the proof of the existence in Lemma \ref{lemare}.

\medskip\noindent
{\bf Uniqueness.}
The uniqueness of $(\bar h),(\bar a)$ is a straightforward consequence of the strict concavity 
of $g_{l_A}$ and $g_{l_B}$ when $l_A\neq 0$ and $l_B\neq 0$ and of the concavity of $g_0$ and
$\widetilde{g}$. We will not write out the proof in detail, because it requires us to 
distinguish between the cases $\Theta\in\overline\cV_{\AB,M}$ and $\Theta\in\overline
\cV_{\nAB,M}$, between $l_k=0$ and $l_k\neq 0$, $k\in \{A,B\}$, and also between $x_\Theta=1$ 
and $x_\Theta=2$. The latter distinctions are tedious, but no technical diffulties arise.
\end{proof}

We resume the proof of Proposition~\ref{Borel}. We pick $u\in \cB_{\,\overline \cV_M}$, and 
for each $\Theta\in \overline \cV_M$ we apply Lemma~\ref{lemare} at $\Theta,u_\Theta$, to 
obtain a $(\bar h)_\Theta,(\bar a)_\Theta \in \cL(\Theta;u_\Theta)$ satisfying (i--iv). We 
set $(\bar h)\colon\, \Theta\in \overline{\cV}_M\mapsto \bar h_\Theta$ and $(\bar a)\colon\, 
\Theta\in \overline{\cV}_M\mapsto \bar a_\Theta$, and we recall \eqref{setE}. If we can show 
that $\Theta\mapsto (\bar h)_\Theta$ is Borel, then it follows that $(\bar h)\in \cE$, 
because (ii) and the fact that $(\bar h)_\Theta,(\bar a)_\Theta \in \cL(\Theta;u_\Theta)$ 
for $\Theta\in \overline{\cV}_M$ ensure that the other conditions required to belong to 
$\cE$ are fulfilled by $(\bar h)$. Moreover, if we can we show that $\Theta\mapsto(\bar
a)_\Theta$ is Borel, then the proof of Proposition \ref{Borel} will be complete, because 
we can set  
\be{defr}
(\bar r_A(\Theta),\bar r_B(\Theta),\bar r_\cI(\Theta))
=\Big(\tfrac{\bar a_A(\Theta)}{\bar h_A(\Theta)},
\tfrac{\bar a_B(\Theta)}{\bar h_B(\Theta)},
\tfrac{\bar a_\cI(\Theta)}{\bar h_\cI(\Theta)}\Big), 
\quad \Theta\in \overline \cV_M,
\ee
with the convention $\bar r_k(\Theta)=1$ when $\bar a_k(\Theta)=\bar h_k(\Theta)=0$ for 
$k\in \{A,B,\cI\}$, after which $(\bar r)\in \cU(h)$ and $(\bar h),(\bar r)$ satisfy 
\eqref{Borel1} and \eqref{Borel2}. 

To complete the proof it remains to show that $\Theta\mapsto (\bar h)_\Theta,
(\bar a)_{\Theta}$ is Borel. Recall the partition  
\be{papar}
\overline \cV_M=\overline \cV_{\AB,M}\cup 
\big(\cup_{(x,k)\in \{1,2\}\times \{A,B\}} \overline \cV_{\AB,k,x,M}\big),
\ee
and partition these five subsets in the right-hand side of \eqref{papar} into smaller 
subsets depending on the values taken by $l_A$ and $l_B$. For $\overline \cV_{\AB,M}$, 
this gives 
\begin{align}
\label{papar2}
\overline{\cV}_{\AB,M}
=& \{\Theta\in \overline\cV_{\AB,M}\colon l_A,l_B>0\}
\cup\{\Theta\in \overline\cV_{\AB,M}\colon l_A>0, l_B=0\}\\ \nonumber
&\cup \{\Theta\in \overline\cV_{\AB,M}\colon l_A=0,l_B>0\}
\cup \{\Theta\in \overline\cV_{\AB,M}\colon l_A=l_B=0\},
\end{align}
and on each of these subsets the fact that $(\bar h)_\Theta,(\bar a)_\Theta$ are the 
unique elements in $\cL(\Theta;u_\Theta)$ satisfying (i--iv) implies that $\Theta\mapsto
(\bar h)_\Theta,(\bar a)_\Theta$ are continuous and therefore Borel. Since each subsets 
in the right-hand side of \eqref{papar2} belongs to the Borel $\sigma$-field generated 
by $d_M$ (recall \eqref{dist}), we can conclude that $\Theta\mapsto (\bar h)_\Theta,
(\bar a)_\Theta$ are Borel on $\overline\cV_M$. This completes the proof of 
Proposition~\ref{Borel}.
\end{proof}


\end{appendix}



\end{document}